\documentclass[11pt]{article}
\usepackage{amsmath,amssymb,amsthm}
\usepackage[margin=1in]{geometry}
\usepackage{enumitem}
\usepackage{cite}
\usepackage{hyperref}
\hypersetup{colorlinks=true,linkcolor=blue,citecolor=blue,urlcolor=blue}
\usepackage{etoolbox}
\makeatletter
\renewcommand\section{\@startsection{section}{1}{\z@}{-3.5ex \@plus -1ex \@minus -.2ex}{2.3ex \@plus.2ex}%
  {\normalfont\Large\bfseries\boldmath}}
\renewcommand\subsection{\@startsection{subsection}{2}{\z@}{-3.25ex\@plus -1ex \@minus -.2ex}{1.5ex \@plus .2ex}%
  {\normalfont\large\bfseries\boldmath}}
\patchcmd{\l@section}{\bfseries}{\bfseries\boldmath}{}{}
\makeatother
\newtheorem{theorem}{Theorem}[section]
\newtheorem{proposition}[theorem]{Proposition}
\newtheorem{lemma}[theorem]{Lemma}
\newtheorem{corollary}[theorem]{Corollary}

\theoremstyle{definition}
\newtheorem{definition}[theorem]{Definition}
\newtheorem{remark}[theorem]{Remark}
\numberwithin{equation}{section}

\newcommand{\step}[2]{\par\smallskip\noindent\textbf{Step #1.}\ #2\par\nobreak\noindent\ignorespaces}

\newcommand{\Z}{\mathbb{Z}}
\newcommand{\Q}{\mathbb{Q}}
\newcommand{\F}{\mathbb{F}}
\newcommand{\R}{\mathbb{R}}

\newcommand{\m}{\mathfrak{m}}
\newcommand{\QG}{\mathcal{Q}_G}
\newcommand{\Tr}{\operatorname{Tr}}
\newcommand{\Res}{\operatorname{Res}}

\newcommand{\rank}{\operatorname{rank}}
\newcommand{\corank}{\operatorname{corank}}
\newcommand{\coker}{\operatorname{coker}}
\newcommand{\supp}{\operatorname{supp}}
\newcommand{\wt}{\operatorname{wt}}
\newcommand{\disc}{\operatorname{disc}}

\newcommand{\E}{\mathbb{E}}

\title{Almost all graphs are determined by their generalized spectrum}
\author{
  Ziqing Xiang \thanks{Email: \texttt{xiangzq@sustech.edu.cn}} \\
  Southern University of Science and Technology \\
}
\date{}

\begin{document}

\maketitle

\begin{abstract}
Haemers' conjecture for the generalized spectrum asserts that almost all graphs are determined by their generalized spectrum, that is, for
almost all graphs $G$, every graph with the same spectrum as $G$ whose complement has the same spectrum as the complement of $G$ is
isomorphic to $G$. We prove this conjecture by showing that the random graph $G(n,1/2)$ is determined by its generalized spectrum with
probability tending to one.

The proof combines a local theory of the denominators of the rational orthogonal matrices that relate generalized cospectral graphs, Fourier
analysis of random symmetric matrices over finite rings, Seidel switching, inverse Littlewood--Offord theory over finite fields, rank
estimates for random symmetric matrices modulo primes, and a study of the rational factors of the characteristic polynomial of a random
symmetric $\pm1$ matrix with zero diagonal.
\end{abstract}

\setcounter{tocdepth}{1}
\tableofcontents

\section{Introduction}\label{sec:intro}

Which graphs are determined by their spectrum? The question goes back to the early days of spectral graph theory. It was raised by
G\"unthard--Primas (1956)~\cite{GP56} in connection with H\"uckel's theory of molecular orbitals. Schwenk (1973)~\cite{Sch73} showed that
almost all trees have a cospectral mate, and Godsil--McKay (1982)~\cite{GM82} gave a general switching construction of cospectral graphs. In
their surveys, van Dam--Haemers (2003, 2009)~\cite{vDH03,vDH09} suggested that, unlike trees, almost all graphs are determined by their
spectrum, and Haemers (2016)~\cite{Hae16} stated this as a conjecture, now known as Haemers' conjecture, which is still open. Koval--Kwan
(2024)~\cite{KK24} recently showed that exponentially many graphs on $n$ vertices are determined by their spectrum. The author
(2026)~\cite{Xia26} found a natural graph matrix whose spectrum determines almost all graphs up to isomorphism.

A well-studied variant of the question concerns the \emph{generalized spectrum} of a graph $G$, the pair formed by the spectrum of its adjacency
matrix $A$ and the spectrum of the adjacency matrix $J-I-A$ of its complement. A graph is \emph{determined by its generalized spectrum}
(DGS) if every graph with the same generalized spectrum is isomorphic to it, that is, if it has no \emph{generalized cospectral mate}, a
non-isomorphic graph with the same generalized spectrum. \emph{Haemers' conjecture for the generalized spectrum} asserts that almost all
graphs are DGS. Johnson--Newman (1980)~\cite{JN80} showed that two graphs $G$ and $H$ have the same generalized spectrum if and only if
$A(H)=Q^{\top}A(G)Q$ for an orthogonal matrix $Q$ that fixes the all-ones vector $e$.

Wei Wang pioneered the study of DGS graphs through arithmetic conditions on walk matrices. The \emph{walk matrix} of $G$ is
$W:=[e,Ae,\dots,A^{n-1}e]$, and $G$ is called \emph{controllable} if $W$ is nonsingular. Wang--Xu (2006)~\cite{WX06a,WX06b} observed that
for a controllable graph the matrix $Q$ above is unique and rational, and that its denominators are constrained by the arithmetic of $W$. In
a long series of works, Wang and his coauthors turned this observation into arithmetic criteria for DGS
graphs~\cite{W13,W17,WY,WQH19,QWWZ23,WWZ23,WWY23,GWW25,WLW26}. The best-known criterion is the theorem of Wang (2013, 2017)~\cite{W13,W17}
that $G$ is DGS whenever $2^{-\lfloor n/2\rfloor}\det W$ is odd and squarefree. Joint works of the two authors named Wei Wang and their
coauthors later brought geometric ideas into this approach, such as invariant subspaces of the adjacency matrix modulo a
prime~\cite{QWWZ23,WWZ23,GWW25}. Analogous criteria were found for other graph matrices, for instance by Qiu--Ji--Wang (2019)~\cite{QJW19}
for the signless Laplacian.

Van Werde and Lvov developed a probabilistic approach to the question of how often such criteria hold. The theorem of Wang above can be
stated in terms of the cokernel of the walk matrix, which is a module over $\Z[x]$. Van Werde (2025)~\cite{VW25} connected such criteria
with the theory of cokernels of random integral matrices, and for random directed graphs with random edge weights he proved limit theorems
for this module. Lvov--Van Werde (2026)~\cite{LVW26} turned this into a general framework. They rephrased such criteria in terms of these
modules, and they used a random matrix model to predict how often the criteria hold for random symmetric matrices.

Other probabilistic results concern the cospectral mates of random graphs directly. Wang--Zhao (2025)~\cite{WZ25} considered the random
graph $G(n,p)$ for fixed $p\in(0,1)$. They proved that with probability tending to one it has no cospectral mate related to it by a rational
orthogonal matrix of a fixed level, that is, whose entries have a fixed least common denominator. Zhao (2026)~\cite{Z26} extended this to
rational orthogonal matrices whose entries have denominators of size $o((n/\log n)^{1/10})$. In the sparse regime, Van de Berg--Van Werde
(2026)~\cite{VdBVW26} showed for $G(n,\lambda/n)$ with fixed $\lambda>1$ that the giant component typically has cospectral mates, while with
high probability no local switching produces a cospectral mate of its $2$-core.

Wang--Wang (2025)~\cite{WW22} designed an algorithm that finds all generalized cospectral mates of a given graph. Their experiments support
Haemers' conjecture for the generalized spectrum, since they indicate that for $22\le n\le50$ at least $99\%$ of the graphs on $n$ vertices
are DGS. The arithmetic criteria of Wang and his coauthors, however, do not suffice to prove this. They apply to controllable graphs, and
O'Rourke--Touri (2016)~\cite{OT16} proved a conjecture of Godsil (2012)~\cite{God12} that $G(n,1/2)$ is controllable with probability
$1-o(1)$. Beyond controllability, they need conditions at the prime factors of $\det W$. For instance, the theorem of Wang above requires
$2^{-\lfloor n/2\rfloor}\det W$ to be odd and squarefree, and the probability of squarefreeness is not known to be positive. Moreover,
$2^{-\lfloor n/2\rfloor}\det W$ is even for $G(n,1/2)$ with probability at least $0.31$ for large $n$ (Remark~\ref{rem:wangfails}). It was
not even known that DGS graphs have positive density.

Building on the arithmetic and probabilistic approaches described above, we prove Haemers' conjecture for the generalized spectrum in the
following form.

\begin{samepage}
\begin{theorem}\label{thm:main}
The random graph $G(n,1/2)$ is determined by its generalized spectrum with probability $1-o(1)$, that is,
\[
\lim_{n\to\infty}\Pr\bigl(G(n,1/2)\text{ is determined by its generalized spectrum}\bigr)=1 .
\]
\end{theorem}
\end{samepage}

Being DGS is a property of the isomorphism class of a graph. Almost all graphs are asymmetric by Erd\H{o}s--R\'enyi (1963)~\cite{ER63}, and
there are $(1+o(1))2^{n(n-1)/2}/n!$ isomorphism classes of graphs on $n$ vertices (see Harary--Palmer (1973)~\cite[Chapter~9]{HP73}). Hence
the isomorphism class of $G(n,1/2)$ is close in total variation distance to a uniformly random one, and Theorem~\ref{thm:main} also holds
for a uniformly random unlabelled graph.

Section~\ref{sec:plan} contains the proof of Theorem~\ref{thm:main}, assuming five component theorems, Theorems~\ref{thm:largedegrees}
to~\ref{thm:huge}, and it explains the ideas behind their proofs. The starting point is the notion of a local failure at a prime $p$
(Definition~\ref{def:Ep}), which asks for a special invariant subspace of the adjacency matrix modulo $p$. Every generalized cospectral mate
of a controllable graph creates a local failure at each odd prime that divides a denominator of the corresponding rational orthogonal
matrix. The components treat the primes according to their size relative to $n$, except that local failures of large degree are excluded at
all odd primes at once. Three components exclude local failures, one controls the denominators directly at the small primes, and one is an
arithmetic statement about random symmetric matrices that is needed at the largest primes.

The remaining sections contain the proofs of the components. After the preliminaries of Section~\ref{sec:prelim},
Section~\ref{sec:reduction} excludes local failures of large degree at all odd primes (Theorem~\ref{thm:largedegrees}), and
Sections~\ref{sec:local}--\ref{sec:smooth} treat the small primes, which are at most $n^{1/2-\varepsilon}$ for a fixed small
$\varepsilon>0$ (Theorem~\ref{thm:smooth}). Sections~\ref{sec:tools}--\ref{sec:deg3} treat the moderate primes, up to $e^{n^{1/10}}$
(Theorem~\ref{thm:moderate}). Section~\ref{sec:nofactor} shows that, with high probability, the characteristic polynomial of a random
symmetric $\pm1$ matrix with zero diagonal has no rational factor of small degree (Theorem~\ref{thm:nofactor}), and
Section~\ref{sec:huge} uses this to treat the huge primes, above $e^{n^{1/10}}$ (Theorem~\ref{thm:huge}).

The proof relates to earlier work as follows. It combines the arithmetic approach of Wang and his coauthors with methods from the theory of
random matrices over finite fields and over the integers. The connection between the two, through the cokernel of the walk matrix, was made
by Van Werde (2025)~\cite{VW25} and Lvov--Van Werde (2026)~\cite{LVW26}. At the small primes, we study this cokernel as a module over
$\Z[x]$. At the large primes, it suffices to exclude local failures. Unlike squarefreeness, local failures are conditions modulo $p$.
Conditions of this kind can be controlled at the large primes, as Nguyen--Wood (2022, 2025)~\cite{NW22,NW25} showed for cokernels of random
integral matrices.

On the arithmetic side, the passage from generalized cospectral mates to local failures rests on the exclusion criterion of Wang--Xu
(2006)~\cite{WX06a}, in the form given by Guo--Wang--Wang (2025)~\cite{GWW25}. We show that for a controllable graph the denominators of the
rational orthogonal matrices that relate it to its generalized cospectral mates are controlled, one prime at a time, by the local components
of a finite ring attached to the walk matrix (Section~\ref{sec:local}). The cokernel of the walk matrix is an ideal of this ring. This local
theory is a lattice-theoretic form of the approach of Wang and his coauthors. At odd primes, its criterion for excluding a prime from the
denominators is equivalent to that of Guo--Wang--Wang (2025)~\cite{GWW25}, and it recovers, for instance, the theorem of Wang (2013,
2017)~\cite{W13,W17} quoted above.

On the probabilistic side, the parts of the cokernel of the walk matrix at the maximal ideals of $\Z[x]$ were studied by Van Werde
(2025)~\cite{VW25} for directed graphs and by Lvov--Van Werde (2026)~\cite{LVW26} for a profinite model of random symmetric matrices. At the
small primes, we study them for $G(n,1/2)$ with ideas from the work of Wood (2017)~\cite{Woo17} and Nguyen--Wood (2025)~\cite{NW25} on
cokernels of random integral matrices. We show that the parts that can contribute to the denominators are few and have exponents bounded in
probability. Together with the local theory, this bounds the denominators whose prime factors are all small, and mates with such
denominators are excluded through sparse integer vectors.

To exclude local failures at the large primes, we build on the methods of Nguyen--Wood (2022, 2025)~\cite{NW22,NW25} for ranks modulo large
primes. As they do, we treat all huge primes at once by lifting to characteristic zero. Modulo a huge prime, a local failure of small degree
whose vectors are all structured forces a factor of small degree over $\Q$ of the characteristic polynomial, and Theorem~\ref{thm:nofactor}
excludes such factors. We further use Seidel switching, after van Lint--Seidel (1966)~\cite{vLS66}, and ideas from the work of
Ferber--Jain--Sah--Sawhney (2023)~\cite{FJSS} on ranks of random symmetric matrices. We also use the methods of Costello--Tao--Vu
(2006)~\cite{CTV06} and Vershynin (2014)~\cite{Ver14} for the singularity of random symmetric matrices, whose singularity probability is now
known to be exponentially small by Campos--Jenssen--Michelen--Sahasrabudhe (2025)~\cite{CJMS}, and the simple spectrum theorem of Tao--Vu
(2017)~\cite{TV17}. For anti-concentration we use the inverse Littlewood--Offord theory of Tao--Vu (2009)~\cite{TV09}, Rudelson--Vershynin
(2008)~\cite{RV08}, Nguyen--Vu (2011)~\cite{NV11} and Ferber--Jain--Luh--Samotij (2021)~\cite{FJLS}, together with Hal\'asz's inequality
(1977)~\cite{Hal77}. For the arithmetic of characteristic polynomials at the huge primes, we use a theorem of Schur (1912)~\cite{Sch12} on
the prime divisors of the values of a polynomial, in an effective form that we prove by the method of Chebyshev (1852)~\cite{Che52}, which
Nagell (1921)~\cite{Nag21} applied to such values.

One result that we prove along the way may be of independent interest. For every fixed ${\eta\in(0,\frac1{27})}$, the characteristic
polynomial of a uniformly random symmetric $\pm1$ matrix of order $m$ with zero diagonal has no factor over $\Q$ of degree between $1$ and
$m^{\eta}$, with probability at least $1-\exp(-m^{1/9-\eta}/9)$ for large $m$ (Theorem~\ref{thm:nofactor}). For random symmetric $\pm1$
matrices, irreducibility with probability tending to one is known only for matrices with random diagonal and under the extended Riemann
hypothesis, by Ferber--Jain--Sah--Sawhney (2023)~\cite{FJSS}. As far as we know, no unconditional result excluding rational factors of
growing degree was known for random symmetric $\pm1$ matrices. For matrices with independent entries, see Eberhard (2022)~\cite{Eb22}
and Remark~\ref{rem:nofactor}.

\subsection*{Notation}\label{subsec:notation}
We write $[n]:=\{1,\dots,n\}$ and $|X|$ for the cardinality of a finite set $X$. The indicator of an event or a condition $\mathcal A$ is
$\mathbf 1_{\mathcal A}$, also written $\mathbf 1\{\mathcal A\}$. Graphs have vertex set $[n]$, and for a graph $G$ the letter $A$ denotes
its adjacency matrix. The random graph $G(n,1/2)$ is uniformly distributed on the $2^{n(n-1)/2}$ graphs on $[n]$, that is, its edges are
present independently with probability $\frac12$. We write $e\in\Z^n$ for the all-ones vector of length $n$ and $J:=ee^{\top}$ for the
all-ones matrix of order $n$. The letter $I$ denotes the identity matrix of order $n$, and more generally the identity matrix of the same
order as the square matrices that it is combined with, as in $M-\lambda I$. All vectors are column vectors, and in every dimension $e_i$
denotes the $i$th standard basis vector. Polynomials are in the variable $x$, and for a square matrix $M$ with entries in $\Z$ or in a
field, $\chi_M(x):=\det(xI-M)$ is its characteristic polynomial and $\Delta_M:=\disc\chi_M$ its discriminant. For a scalar $\lambda$ we
write $M-\lambda:=M-\lambda I$. For vectors $y,y'$ over a field we use the standard bilinear form $y^{\top}y'$, which is symmetric and not
Hermitian. We write $y\perp U$ if $y^{\top}u=0$ for all $u$ in a set $U$, also when $y$ has entries in a subfield, and $U^{\perp}$ for the
subspace of these $y$, with $u^{\perp}:=\{u\}^{\perp}$. The support $\supp y$ of a vector $y$ is the set of indices of its nonzero
coordinates, and $\wt(y):=\lvert\supp y\rvert$. For a matrix $M$ and sets $K,L$ of indices, $M_{K,L}$ is the submatrix of $M$ with rows in $K$ and
columns in $L$.

Throughout the paper, $p$ denotes a prime, and bars denote reduction modulo $p$ of integers, vectors, matrices and polynomials. For an
integer matrix $M$, $\rank_pM$ is the rank of $\bar M$ over $\F_p$. For $d\ge1$ we write $q:=p^d$ and $\Tr:=\Tr_{\F_q/\F_p}$. We write
$\overline{\F}_p$ for an algebraic closure of $\F_p$ (this bar is not a reduction), $\Z_p$ for the ring of $p$-adic integers, and $M_m(R)$
for the ring of $m\times m$ matrices over a commutative ring $R$. We write $v_p$ for the $p$-adic valuation and $v_\beta(f)$ for the
exponent of an irreducible polynomial $\beta$ in $f\in\F_p[x]$. Further, $\|\cdot\|$ is the Euclidean norm, $\|\theta\|_{\R/\Z}$ is the
distance from $\theta\in\R$ to the nearest integer, $\log$ is the natural logarithm, $\log_2$ is the binary logarithm, and
$H(t):=-t\log_2t-(1-t)\log_2(1-t)$ is the binary entropy function.

We write $\Pr(\cdot)$ and $\E(\cdot)$ for probability and expectation, and we put ${\E(X;\ \mathcal A):=\E(X\mathbf 1_{\mathcal A})}$ and
${\Pr(\mathcal B;\ \mathcal A):=\Pr(\mathcal B\cap\mathcal A)}$. The letter $\xi$ denotes a uniformly random vector in $\{\pm1\}^m$, for the
relevant order $m$, which is independent of all other random objects unless stated otherwise. When other random objects are present,
$\Pr_\xi(\cdot)$ and $\E_\xi(\cdot)$ denote the conditional probability and expectation given all of them, that is, the probability and the
expectation with respect to $\xi$ alone, and we use the same notation, such as $\Pr_u$ and $\E_u$, for other random vectors. For instance,
if $U\subseteq\F_p^m$ is a subspace that depends on a random matrix $T$, then $\Pr_\xi(\bar\xi\perp U)$ is the probability that
$\bar\xi\perp U$ when $T$ is fixed, and it is a random variable that depends on $T$. Asymptotic notation such as $o(1)$ refers to
$n\to\infty$, or to $m\to\infty$ when $m$ is the order of the matrices under consideration, and ``for large $n$'' means for all sufficiently
large $n$.

\subsection*{Acknowledgments}

This research was supported by grants from the NSFC (12350710787) and NSFC (12471311).

\subsection*{AI Disclousure}

This paper is a collaboration with Claude Opus 5.5. The lean formalization of is available at Palomar Registry \url{https://palomar-registry.org/entry?id=PALOMAR-2026-10-23-000011}.

\section{Proof strategy}\label{sec:plan}

In this section we prove Theorem~\ref{thm:main}, assuming five theorems that we call the components of the proof, and we describe the main
ideas behind the proofs of the components. Section~\ref{subsec:plan-mates} reduces Theorem~\ref{thm:main} to the study of local failures,
events defined by special invariant subspaces of the adjacency matrix modulo an odd prime. Section~\ref{subsec:plan-components}
states the five components and deduces Theorem~\ref{thm:main} from them. Sections~\ref{subsec:plan-switch}--\ref{subsec:plan-huge} then
explain the ideas of their proofs, one component at a time, in the order in which the components are proved in
Sections~\ref{sec:reduction} to~\ref{sec:huge}. In these explanations we postpone some definitions to later sections, and each time we
indicate where the precise definition is given.

\subsection{From cospectral mates to local failures}\label{subsec:plan-mates}
Let $G$ be a graph with adjacency matrix $A$, and recall that $G$ is controllable if its walk matrix $W=[e,Ae,\dots,A^{n-1}e]$ is
nonsingular. We consider the set
\[
\QG:=\big\{Q\in O_n(\Q):\ Qe=e\text{ and }Q^{\top}AQ\text{ is a symmetric $(0,1)$-matrix with zero diagonal}\big\}
\]
of rational orthogonal matrices that fix $e$ and map $A$ to the adjacency matrix of a graph. The set $\QG$ contains the permutation
matrices, and the \emph{level} $\ell(Q)$ of $Q\in\QG$ is the least integer $\ell\ge1$ such that $\ell Q$ is integral. By the work of
Johnson--Newman (1980)~\cite{JN80} and Wang--Xu (2006)~\cite{WX06a}, if $G$ is controllable, then every graph with the same generalized
spectrum as $G$ has adjacency matrix $Q^{\top}AQ$ for a unique matrix $Q\in\QG$. Moreover, a controllable graph $G$ is DGS if and only if
every matrix in $\QG$ is a permutation matrix, which happens if and only if every matrix in $\QG$ has level $1$ (Theorem~\ref{thm:wangxu}).
Since $G(n,1/2)$ is controllable with probability $1-o(1)$ by the theorem of O'Rourke--Touri (2016)~\cite{OT16} (Theorem~\ref{thm:OT}),
Theorem~\ref{thm:main} is equivalent to the statement that, with probability $1-o(1)$, no matrix in $\QG$ has level at least $2$.

We examine the prime factors of the levels separately, and for a given prime $p$ we ask whether $p$ divides the level of some matrix in
$\QG$. If an odd prime $p$ divides the level $\ell$ of some $Q\in\QG$, then the columns of $\ell Q$, reduced modulo $p$, span a nonzero
subspace of $\F_p^n$ that is invariant under $\bar A$, orthogonal to $\bar e$, and \emph{totally isotropic}, that is, orthogonal to itself
for the bilinear form $y^{\top}y'$. This observation is at the heart of the arithmetic criteria of Wang (2013, 2017)~\cite{W13,W17}, and
Guo--Wang--Wang (2025)~\cite{GWW25} formulated it in essentially this form. To describe the minimal subspaces with these properties we use
the following standard language. For $M\in M_m(\F_p)$, the space $\F_p^m$ is an $\F_p[x]$-module in which $x$ acts as $M$, and its
submodules are the $M$-stable subspaces. A \emph{simple $M$-submodule} is a nonzero $M$-stable subspace that contains no other nonzero
$M$-stable subspace. It is isomorphic to $\F_p[x]/(\beta)$ for a unique monic irreducible $\beta\in\F_p[x]$, called its \emph{type}, and its
dimension is $\deg\beta$. A minimal nonzero $\bar A$-stable subspace of the column space of $\ell Q$ modulo $p$ is a simple
$\bar A$-submodule, and this leads to the following event, which plays a central role in the paper.

\begin{definition}\label{def:Ep}
Let $p$ be an odd prime and $d\ge1$. A graph $G$ has a \emph{local failure of degree $d$ at $p$} if some simple $\bar A$-submodule
$U\subseteq\F_p^n$ of dimension $d$ is totally isotropic and orthogonal to $\bar e$. We denote this event by $E_p^{(d)}$.
\end{definition}

Local failures are defined only at odd primes, because their analogue at $p=2$ always occurs (Section~\ref{sec:two}).

In terms of eigenvectors, $E_p^{(d)}$ holds if and only if $\bar A$ has an eigenvector $w$ over $\F_{p^d}$, with an eigenvalue of degree $d$
over $\F_p$, such that $w^{\top}w=0$ and $e^{\top}w=0$ (Lemma~\ref{lem:simple}). By the discussion before the definition, a matrix $Q\in\QG$
whose level is divisible by an odd prime $p$ forces a local failure at $p$ (Lemma~\ref{lem:matesfail}). The converse is far from true, and
this is where the probabilistic difficulty lies. Heuristically, a local failure of degree $d$ at $p$ has probability of order $p^{-2d}$
(compare Theorem~\ref{thm:firstmoment}). This is not small at a fixed prime, and we cannot hope to exclude local failures there. Summed over
all $d$ and all primes above $n^{1/2-\varepsilon}$, where $\varepsilon>0$ is small and fixed, these probabilities tend to $0$, and we show
that with probability $1-o(1)$ no local failure occurs at any of these primes (Section~\ref{subsec:plan-components}).

\subsection{The five components and the proof of Theorem~\ref{thm:main}}\label{subsec:plan-components}
In the statements of the components, and in all later statements about probabilities of events that involve a graph, $G:=G(n,1/2)$ is the
random graph and $A$ is its adjacency matrix, so that $\QG$ and the events $E_p^{(d)}$ are random. The components treat the primes according
to their size relative to $n$. Let $\varepsilon\in(0,\frac1{10})$ be fixed. We call a prime $p$ \emph{small} if $p\le n^{1/2-\varepsilon}$,
\emph{moderate} if $n^{1/2-\varepsilon}<p\le e^{n^{1/10}}$, and \emph{huge} if $p>e^{n^{1/10}}$. At the small primes, including $2$, we control the levels
directly, since local failures are not expected to be rare there. At the moderate and at the huge primes we exclude local failures
altogether, by two different methods. These ranges reflect the reach of our tools. Fourier analysis modulo $p$ works only for $p$ somewhat
below $\sqrt n$. The arguments at the moderate primes need $p\ge n^{2/5}$ in the degrees $d\ge3$, which is why $\varepsilon<\frac1{10}$, and
the argument at the huge primes works for all $p>e^{(\log n)^6}$. Local failures of large degree, namely of degree at least $3\log_2n$, are
excluded with probability $1-o(1)$ at all odd primes at once by a simpler argument, and this is the first component.

\begin{theorem}\label{thm:largedegrees}
The probability that $E_p^{(d)}$ holds for some odd prime $p$ and some $d\ge3\log_2n$ tends to $0$.
\end{theorem}

The second component shows that, with probability $1-o(1)$, the level of every matrix in $\QG$ other than a permutation matrix has a prime
factor that is not small. It holds more generally for every fixed $\varepsilon\in(0,\frac12)$, and its proof uses
Theorem~\ref{thm:largedegrees} at the small odd primes and a separate argument at the prime $2$ (Section~\ref{sec:two}).

\begin{theorem}\label{thm:smooth}
Fix $\varepsilon\in(0,\frac12)$. Then the probability that some $Q\in\QG$ other than a permutation matrix has a level whose prime factors
are all at most $n^{1/2-\varepsilon}$ tends to $0$.
\end{theorem}

The third component excludes local failures at the moderate primes in the degrees below $3\log_2n$.

\begin{theorem}\label{thm:moderate}
Fix $\varepsilon\in(0,\frac1{10})$. Then the probability that $E_p^{(d)}$ holds for some odd prime $p$ with
$n^{1/2-\varepsilon}<p\le e^{n^{1/10}}$ and some $1\le d<3\log_2n$ tends to $0$.
\end{theorem}

The fourth component is an arithmetic statement about random symmetric matrices, which is the main input at the huge primes. A \emph{Seidel
matrix} of order $m$ is a symmetric matrix with zero diagonal and off-diagonal entries $\pm1$. We write $S_m$ for a \emph{uniformly random
Seidel matrix} of order $m$, whose entries above the diagonal are independent uniform signs.

\begin{theorem}\label{thm:nofactor}
Fix $\eta\in(0,\frac1{27})$. Then, for all sufficiently large $m$,
\[
\Pr\bigl(\chi_{S_m}\text{ has a monic factor in }\Z[x]\text{ of degree in }[1,m^{\eta}]\bigr)\ \le\ \exp\bigl(-m^{1/9-\eta}/9\bigr).
\]
\end{theorem}

The fifth component excludes local failures at the huge primes, in all degrees.

\begin{theorem}\label{thm:huge}
The probability that $E_p^{(d)}$ holds for some odd prime $p>e^{(\log n)^6}$ and some $d\ge1$ tends to $0$.
\end{theorem}

We now deduce Theorem~\ref{thm:main} from the components, together with Theorems~\ref{thm:wangxu} and~\ref{thm:OT} and
Lemma~\ref{lem:matesfail} of Section~\ref{sec:prelim}, which make the discussion of Section~\ref{subsec:plan-mates} precise.
Theorem~\ref{thm:nofactor} enters only through the proof of Theorem~\ref{thm:huge}.

\begin{proof}[Proof of Theorem~\ref{thm:main}]
Put $\varepsilon:=\frac1{20}$ and $G:=G(n,1/2)$. Let $\mathcal A_n$ be the event that $G$ is controllable and that every $Q\in\QG$ other
than a permutation matrix has a level divisible by some prime $p>n^{1/2-\varepsilon}$. By Theorems~\ref{thm:OT} and~\ref{thm:smooth},
$\Pr(\mathcal A_n)\to1$. Let $\mathcal B_n$ be the event that $E_p^{(d)}$ holds for some odd prime $p$ and some $d\ge3\log_2n$, let
$\mathcal C_n$ be the event that $E_p^{(d)}$ holds for some odd prime $p\in(n^{1/2-\varepsilon},e^{n^{1/10}}]$ and some $1\le d<3\log_2n$,
and let $\mathcal D_n$ be the event that $E_p^{(d)}$ holds for some odd prime $p>e^{n^{1/10}}$ and some $d\ge1$. By
Theorems~\ref{thm:largedegrees}, \ref{thm:moderate} and~\ref{thm:huge}, the probabilities of these three events tend to $0$, as
$e^{n^{1/10}}>e^{(\log n)^6}$ for large $n$.

Let $n$ be so large that $n^{1/2-\varepsilon}>2$. Suppose that $\mathcal A_n$ holds and that $\mathcal B_n$, $\mathcal C_n$ and
$\mathcal D_n$ fail, and let $Q\in\QG$. If $Q$ is not a permutation matrix, then by $\mathcal A_n$ some prime $p>n^{1/2-\varepsilon}$ divides
$\ell(Q)$. This prime is odd, so Lemma~\ref{lem:matesfail} shows that $E_p^{(d)}$ holds for some $d\ge1$. Since $\mathcal B_n$ fails we have
$d<3\log_2n$, since $\mathcal C_n$ fails we have $p>e^{n^{1/10}}$, and then $\mathcal D_n$ holds, a contradiction. Hence every $Q\in\QG$ is
a permutation matrix, and $G$ is DGS by Theorem~\ref{thm:wangxu}, since $G$ is controllable. Consequently
\[
\Pr(G\text{ is DGS})\ \ge\ 1-\Pr(\mathcal A_n^c)-\Pr(\mathcal B_n)-\Pr(\mathcal C_n)-\Pr(\mathcal D_n)\ \longrightarrow\ 1 .\qedhere
\]
\end{proof}

The rest of the paper proves the components. After the preliminaries of Section~\ref{sec:prelim}, Theorem~\ref{thm:largedegrees} is proved
in Section~\ref{sec:reduction}, Theorem~\ref{thm:smooth} in Sections~\ref{sec:local}--\ref{sec:smooth}, Theorem~\ref{thm:moderate} in
Sections~\ref{sec:tools}--\ref{sec:deg3}, Theorem~\ref{thm:nofactor} in Section~\ref{sec:nofactor}, and Theorem~\ref{thm:huge} in
Section~\ref{sec:huge}. Each of these proofs uses only results proved before it, and Theorems~\ref{thm:largedegrees} and~\ref{thm:nofactor}
are the only components that the proofs of other components use.

\subsection{Large degrees, Seidel switching and a reduction}\label{subsec:plan-switch}
The first component, Theorem~\ref{thm:largedegrees}, holds at all odd primes, including the small ones, where it enters the proof of
Theorem~\ref{thm:smooth}. Its proof in Section~\ref{sec:reduction} rests on two tools that we use again at the moderate and at the huge
primes, Seidel switching and a reduction to a single random Seidel matrix.

The first tool replaces the all-ones vector $e$ in the definition of a local failure by a random vector that is independent of the matrix.
For a fixed simple submodule, orthogonality to $\bar e$ is a condition on the matrix alone, and its probability is hard to estimate. For a
random sign vector that is independent of the matrix, on the other hand, the conditional probability of orthogonality given the matrix is
the probability that a sum of independent random signs, with coefficients determined by the matrix, takes a given value. Bounding such
probabilities is an \emph{anti-concentration problem}, and they are small as soon as the coefficients are suitably spread out.

The Seidel matrix of a graph $G$ is $S:=J-I-2A$, so that the Seidel matrix $S$ of $G(n,1/2)$ is distributed as $S_n$. Let $\Sigma$ be a
uniformly random diagonal sign matrix independent of $S$. Conjugating $S$ by $\Sigma$, an operation
known as Seidel switching after van Lint--Seidel (1966)~\cite{vLS66}, produces a uniformly random Seidel matrix $S':=\Sigma S\Sigma$ and a
uniform sign vector $u:=\Sigma e$ that are independent (Lemma~\ref{lem:switch}). Modulo an odd prime $p$ we have $\bar A=-\frac12(I+\bar
S)$ on $\bar e^{\perp}$, so a local failure becomes a totally isotropic simple
$\bar S'$-submodule orthogonal to $\bar u$ (Lemma~\ref{lem:Epswitch}).

From now on, the letter $T$ denotes a symmetric integer matrix of order $m$, in most cases a random or fixed Seidel matrix, as stated in
each case. A monic irreducible $\beta\in\F_p[x]$ is a \emph{repeated factor of $T$ at $p$} if $\beta^2$ divides $\chi_T$ modulo $p$. A
totally isotropic simple submodule of type $\beta$ exists only if $\beta$ is a repeated factor of $S'$ at $p$ (Lemma~\ref{lem:simple}(d)).
When the eigenspace of a root of a repeated factor is one-dimensional, there is only one simple submodule of its type, and the conditional
probability of a local failure of this type, given $S'$, is the probability that the random sign vector $u$ is orthogonal to a fixed
eigenvector. For $w\in\F_p^m$ and a uniform sign vector $\xi\in\{\pm1\}^m$, we call
\[
\rho_p(w):=\max_{a\in\F_p}\Pr\Bigl(\sum_{j=1}^m\xi_jw_j=a\Bigr)
\]
the \emph{concentration} of $w$.

Larger eigenspaces may contain many lines $\F_qw$ with $w^{\top}w=0$. Those of dimension at least three, and the two-dimensional ones on
which $w^{\top}w$ vanishes identically, are handled by deleting a coordinate in the support of an eigenvector (Lemma~\ref{lem:delete}). This
lowers the dimension of the eigenspace by one, at the price of an orthogonality condition for the deleted column, and the deleted column is
independent of the smaller matrix. Since the eigenvectors have linear support at all odd primes (Proposition~\ref{lem:support}), and the law
of a random Seidel matrix is invariant under simultaneous permutations of rows and columns, the cost of the deletion is a constant factor
(Lemma~\ref{lem:exch}). Every repeated factor of $S'$ at $p$ forces $p$ to divide the discriminant $\Delta_{S'}$. Hence, as long as
$\Delta_{S'}\ne0$, the number of pairs formed by a prime and a repeated factor is controlled by the bound ${|\Delta_{S'}|\le(2n)^{n(n-1)}}$
(Lemma~\ref{lem:budget}). By the simple spectrum theorem of Tao--Vu (2017)~\cite{TV17}, a random Seidel matrix has $n$ distinct eigenvalues
with probability $1-o(1)$, so that $\Delta_{S'}\ne0$ with probability $1-o(1)$ (Theorem~\ref{lem:simplespectrum}).

The result is a reduction to a single random Seidel matrix. To a Seidel matrix $T$ and a set $\Pi$ of pairs $(p,d)$ we attach a weighted
count $N_\Pi(T)$ of the repeated factors $\beta$ of $T$ at the primes $p$ with $(p,\deg\beta)\in\Pi$. The weight of $\beta$ is the largest
probability that a uniform sign vector is orthogonal modulo $p$ to a simple $\bar T$-submodule of type $\beta$, so it measures the
anti-concentration of the corresponding eigenvectors. We postpone the precise definition to Definition~\ref{def:Npi}. The reduction
(Proposition~\ref{prop:reduction}) shows that $E_p^{(d)}$ holds for some $(p,d)\in\Pi$ only with probability $o(1)$, as soon as the
expectation of $N_\Pi(S_m)$ for $m\in\{n-1,n\}$, restricted to a set of matrices of our choice that has probability $1-o(1)$, tends to $0$.
Theorem~\ref{thm:largedegrees} is its first application. A $d$-dimensional subspace is orthogonal to a random sign vector with probability
at most $2^{-d}$ by the lemma of Odlyzko (1988)~\cite{Odl88} (Lemma~\ref{lem:odlyzko}). Once $d\ge3\log_2n$, this probability is at most
$n^{-3}$, while the discriminant bound allows at most $n^2\log_2(2n)$ repeated factors at all odd primes together, so that the weighted
count $N_\Pi$ is at most $\log_2(2n)/n$.

\subsection{Small primes and the local theory of levels}\label{subsec:plan-small}
The second component, Theorem~\ref{thm:smooth}, is proved in Sections~\ref{sec:local}--\ref{sec:smooth}, with Theorem~\ref{thm:largedegrees}
and Proposition~\ref{lem:suppd} of Section~\ref{sec:prelim} as further inputs. Its first ingredient is a deterministic local theory of
levels, developed in Section~\ref{sec:local}, which we now describe informally. Let $G$ be controllable. The columns of the walk matrix then
span a lattice $C:=W\Z^n$ of rank $n$, and the dual lattice $C^*$ consists of the vectors $y\in\Q^n$ with $y^{\top}c\in\Z$ for all $c\in C$.
Since $C$ is $A$-stable and $A$ is symmetric, $A$ acts on the finite abelian group $D:=C^*/C$, the discriminant module of $C$
(Section~\ref{subsec:loc-disc}). As a module over $\Z[x]$, with $x$ acting as $A$, the group $D$ is isomorphic to the finite ring
$\mathcal T:=\Z[x]/(\varphi,\psi)$ (Proposition~\ref{prop:D}), where $\varphi:=\chi_A$ is the characteristic polynomial of $A$ and
$\psi(x):=e^{\top}\operatorname{adj}(xI-A)e$. For every $Q\in\QG$ the lattice $Q\Z^n$ is unimodular, that is, equal to its dual lattice, and
it lies between $C$ and $C^*$ (Lemma~\ref{lem:ideals}). Its image in $D\cong\mathcal T$ is therefore an ideal of $\mathcal T$ of order
$|\mathcal T|^{1/2}$, and the identity matrix gives the ideal $\Z^n/C=\coker W$. These ideals are even Lagrangian for a natural pairing on
$D$ (Remark~\ref{rem:loc-lagr}), but we only use their order. As a finite ring, $\mathcal T$ is the product of finite local rings
$\mathcal T_\m$, which we call its \emph{local components}, one for each maximal ideal $\m$. Each $\m$ lies above a prime $p$ and a common
irreducible factor $\beta$ of $\varphi$ and $\psi$ modulo $p$, and $\deg\beta$ is the \emph{degree} of the component. Hensel's lemma splits
$\Z_p^n$ into orthogonal summands that realize the local components above $p$ (Lemma~\ref{lem:loccomp}).

The main result of the local theory, Proposition~\ref{lem:levellocal}, shows that a prime $p$ divides $\ell(Q)$ if and only if the ideals
given by $Q$ and by the identity matrix differ at some local component above $p$. The $p$-part of $\ell(Q)$ is then at most the largest
exponent of the corresponding parts of $\coker W$. Here $\coker W$ is decomposed along the local components, and the exponent of a finite
abelian group is the largest order of its elements. For odd $p$, such a component also gives a local failure in its degree. Since the two
ideals have the same order, a local component which is a \emph{chain ring}, that is, a local ring whose ideals are totally ordered by
inclusion, cannot contribute to the level, and Corollary~\ref{cor:chain} shows that $\mathcal T_\m$ is a chain ring as soon as the ideal
$(\varphi,\psi)$ of $\Z[x]$ is not contained in $\m^2$. This is a lattice-theoretic form of the arithmetic approach of Wang--Xu and Wang.
Related criteria and bounds for levels were obtained by Wang--Wang--Zhu (2023)~\cite{WWZ23} and Wang--Luo--Wang (2026)~\cite{WLW26}.

With the local theory in hand, it suffices to show that, with high probability, only a bounded number of local components above the small
primes can contribute to a level, that their exponents and the multiplicities of the corresponding factors of $\bar\varphi$ are bounded, and
that the short integer vectors produced below by a matrix of bounded level do not exist. At the odd small primes we bound the probabilities
of local failures by Fourier analysis modulo $p$, and at every fixed prime, including $2$, we bound the exponents of $\coker W$ by Fourier
analysis modulo powers of $p$ (Sections~\ref{sec:fourier} and~\ref{sec:small}), with the character method for random matrices of Wood
(2017)~\cite{Woo17} and Nguyen--Wood (2025)~\cite{NW25}. For fixed matrices $Y$ and $Z$ we expand the probability that $AY\equiv Z$ modulo a
given power of $p$ in additive characters. Since the entries of $A$ above the diagonal are independent, the expectation of a character is a
product of factors indexed by the pairs $i<j$, and each factor has modulus less than $1$ unless a certain coefficient vanishes. The
characters for which all these coefficients vanish, which we call \emph{special}, correspond to antisymmetric matrices, and all other
characters have many nonvanishing coefficients, provided that the columns of $Y$ are linearly independent modulo $p$ and span no nonzero
vector with few nonzero coordinates modulo $p$ (Lemma~\ref{lem:special}). The special characters can be counted exactly. Together with a
count of the candidate submodules through the Gram data of random matrices (Lemma~\ref{lem:gram}), this proves the bound of order
$p^{-2d}$ of Section~\ref{subsec:plan-mates} for the probability of $E_p^{(d)}$ (Theorem~\ref{thm:firstmoment}). These bounds are summable
over all pairs $(p,d)$, so with high probability only a bounded number of pairs carry local failures. Local failures of large degree at the
small odd primes are excluded by Theorem~\ref{thm:largedegrees}.

At a fixed prime $p$, including $p=2$, the exponents of the components of $\coker W$ of a fixed degree, and the multiplicities of the
irreducible factors of $\bar\varphi$ of a fixed degree, are \emph{tight}, that is, bounded in probability uniformly in $n$
(Proposition~\ref{prop:tight}). This follows from a comparison with Haar-random matrices over the $p$-adic integers, in the spirit of
Cohen--Lenstra (1984)~\cite{CL84}, Friedman--Washington (1989)~\cite{FW89} and Wood (2017)~\cite{Woo17}.

The prime $2$ requires a separate argument (Section~\ref{sec:two}). Seidel switching (Section~\ref{subsec:plan-switch}) is useless at $p=2$,
the analogue of a local failure always occurs there, and the hypothesis under which the theorem of Wang (2017)~\cite{W17} excludes the prime
$2$ from all levels fails with probability at least $0.31$ (Remark~\ref{rem:wangfails}). Modulo $2$, the matrix $A$ is a uniformly random
alternating matrix, that is, a uniformly random symmetric matrix over $\F_2$ with zero diagonal. We expose the Krylov sequence
$e,Ae,A^2e,\dots$ modulo $2$ one vector at a time, building on the analysis of the walk matrix modulo $2$ by Wang (2013,
2017)~\cite{W13,W17}. This shows that, for every $d_0\ge2$, the local components above $2$ of degree at least $d_0$ are chain rings, except
with probability $O(2^{-d_0})+o(1)$ (Proposition~\ref{prop:twolarge}). The finitely many components above $2$ of smaller degree are handled
by tightness.

After these steps, outside an event of small probability, a matrix $Q\in\QG$ whose level has only small prime factors must have a bounded
level. For every prime $p$ dividing $\ell(Q)$, some column of $\ell(Q)Q$ is then an integer vector that is nonzero modulo $p$, has at most
$\ell(Q)^2$ nonzero coordinates modulo $p$, and is killed modulo $p$ by a polynomial of bounded degree in $A$ (Lemma~\ref{lem:shortcol}).
Such sparse vectors exist only with exponentially small probability. Indeed, with probability exponentially close to $1$, the nonzero
vectors killed modulo $p$ by a polynomial of degree $d\le n^{1/4}$ in $A$ have more than $n/(4(d+1))$ nonzero coordinates, simultaneously at
all primes $p\le e^{n^{1/8}}$, since a sparser vector would give a nonzero kernel vector of a block Krylov matrix built from two independent
blocks of $A$, and such kernel vectors are unlikely (Proposition~\ref{lem:suppd}). The same holds for a random Seidel matrix at the odd
primes, and we use it again at the moderate primes. This completes the proof of Theorem~\ref{thm:smooth} in Section~\ref{sec:smooth}. For
generalized cospectral mates of $G(n,1/2)$, Theorem~\ref{thm:smooth} contains the results of Wang--Zhao (2025)~\cite{WZ25} and Zhao
(2026)~\cite{Z26} on mates of a fixed level and of small height, the height being the largest denominator of an entry (see
Section~\ref{sec:smooth}).

\subsection{Moderate primes}\label{subsec:plan-moderate}
The third component, Theorem~\ref{thm:moderate}, is the heart of the paper, and it is the second application of the reduction of
Section~\ref{subsec:plan-switch}. At the moderate primes the discriminant allows up to $O(n^2)$ repeated factors, so we need two pieces of
information about a random Seidel matrix. First, its eigenvectors modulo $p$ must be equidistributed, in the sense that the signed sums
$\sum_j\xi_jw_j$ of their coordinates are close to uniformly distributed, so that a sign vector is orthogonal to a simple submodule
of dimension $d$ with probability close to $p^{-d}$. Second, its repeated factors must be rare.

The degrees $d=1$, $d=2$ and $3\le d<3\log_2n$ are treated in Sections~\ref{sec:deg1}, \ref{sec:deg2} and~\ref{sec:deg3}, with tools from
inverse Littlewood--Offord theory over finite fields. The Littlewood--Offord problem, studied by Littlewood--Offord (1943)~\cite{LO43} and
Erd\H{o}s (1945)~\cite{Erd45}, asks for upper bounds on the concentration of $\sum_j\xi_jw_j$. Inverse theorems, in the tradition of
Tao--Vu (2009)~\cite{TV09}, Rudelson--Vershynin (2008)~\cite{RV08} and Nguyen--Vu (2011)~\cite{NV11}, show conversely that coefficient
vectors $w$ with large concentration have a strong additive structure.

Our tools are collected in Section~\ref{sec:tools}. The first is a version of Hal\'asz's inequality (1977)~\cite{Hal77} for finite abelian
groups of odd order, with explicit constants, which bounds the concentration in terms of the density of vanishing signed sums of the
coefficients (Proposition~\ref{lem:halasz}). Its proof uses Kneser's theorem (1953)~\cite{Kne53} on sumsets. The second is a counting lemma
for product measures in the style of Ferber--Jain--Luh--Samotij (2021)~\cite{FJLS}, which shows that vectors with many vanishing signed sums
are rare (Lemma~\ref{lem:chains} and Corollary~\ref{lem:fjls}). Together they give a block estimate in dimension one
(Proposition~\ref{lem:blockdelta}) and one in dimensions two to four (Proposition~\ref{lem:blockd}), which bound the measure of the vectors
with coordinates in $\F_p^d$ whose concentration exceeds a given threshold. In dimension at least two, a block can also be concentrated
because most of its coordinates lie in a proper subspace, and the counting lemma then also uses witnesses saying that a coordinate lies in
such a heavy subspace. These bounds make union bounds over structured vectors affordable. The third tool is a squaring trick in the spirit
of the decoupling argument of Costello--Tao--Vu (2006)~\cite{CTV06}. It bounds the Fourier coefficients of a quadratic form in a random sign
vector by those of linear forms (Lemma~\ref{lem:squaring}).

For eigenvalues in $\F_p$ (Section~\ref{sec:deg1}), the main estimate says that the eigenvectors of a random Seidel matrix modulo $p$ have
no short integer lifts. Here an integer lift of $w\in\F_p^n$ is a vector $z\in\Z^n$ with $\bar z=w$. Theorem~\ref{thm:SEplus} shows that,
with high probability and simultaneously at all primes $23\le p\le e^{n^{1/8}}$, all integer lifts of all eigenvectors with eigenvalues in
$\F_p$ have Euclidean length greater than $\sqrt{c_1n}\,p$, where $c_1=10^{-5}$ is the constant of Theorem~\ref{thm:SEplus}. Up to the
constant, this is the length of the shortest lift of a uniformly random vector, and it implies that the eigenvectors are equidistributed
(Corollary~\ref{cor:ACplus}). The proof is a union bound over the possible integer lifts. A discrete Gaussian weight on the lifts $z$,
proportional to $e^{-\sigma\|z\|^2}$ for a suitable $\sigma>0$, makes their coordinates independent, as in the work of Banaszczyk
(1993)~\cite{Ban93} on lattice points, and the structured lifts are counted with the block estimate in dimension one.

To show that repeated roots are rare, we expose a uniformly random Seidel matrix $S_n$ one row and column at a time, as Costello--Tao--Vu
do. Its leading principal submatrix of order $m$ is a uniformly random Seidel matrix $S_m$, and for each $\lambda\in\F_p$ we follow the
corank, that is, the dimension of the kernel, of $\bar S_m-\lambda I$ as $m$ grows. The transitions are governed by linear forms and by
quadratic forms in the resolvent $B:=(\bar S_m-\lambda I)^{-1}$, which the squaring trick turns into linear problems. In this way, for every
prime $23\le p\le e^{n^{1/9}}$, $\lambda$ is an eigenvalue of $\bar S_n$ with probability at most $\frac1p+O(p^{-2})$, and the corank of
$\bar S_n-\lambda I$ is at least $2$ with probability $O(p^{-2})$ (Theorem~\ref{thm:rankp}). A repeated root whose eigenspace is a line
requires in addition an isotropic eigenvector $w$, that is, one with $w^{\top}w=0$ (Lemma~\ref{lem:simple}(d)). When a row and a column are
added, this produces simultaneous conditions on the two quadratic forms $\xi^{\top}B\xi$ and $\xi^{\top}B^2\xi$, where $\xi$ is the new
column. A two-dimensional version of the argument, for the combinations $t_1B+t_2B^2$ with $(t_1,t_2)\in\F_p^2\setminus\{(0,0)\}$, shows
that the expected number of repeated roots in $\F_p$ is $O(1/p)$ (Propositions~\ref{prop:Ptrue} and~\ref{prop:isocond}). The estimates
for these pencils are proved in a form that also covers the resolvents over $\F_{p^2}$ below
(Proposition~\ref{prop:Ptrue}). Results of this type were proved by Ferber--Jain--Sah--Sawhney (2023)~\cite{FJSS} for symmetric $\pm1$
matrices with random diagonal and, for cokernels of random matrices, by Nguyen--Wood (2025)~\cite{NW25}. Here the diagonal is zero, which
removes a source of randomness that those arguments use.

Eigenvalues in $\F_{p^2}$ (Section~\ref{sec:deg2}) are handled in the same way. By Lemma~\ref{lem:simple}, the eigenvectors whose eigenvalue
has degree $d$ over $\F_p$ correspond to the vectors $y\in\F_p^m$ with $\beta(\bar T)y=0$, where $\beta$ is the minimal polynomial of the
eigenvalue. With the block estimate in dimension $d$, the union bound over integer lifts extends to these vectors, and
Proposition~\ref{prop:SEd} shows that the eigenvectors are anti-concentrated at the level $p^{-d}$, for $d\in\{2,3,4\}$. A trace identity,
equation~\eqref{eq:trres}, expresses the trace of a resolvent over $\F_{p^2}$ as a matrix $P(\bar T)\beta(\bar T)^{-1}$ over $\F_p$ with
$\deg P\le1$. So the quadratic forms of resolvents are covered by the estimates for pencils, and a given element of $\F_{p^2}\setminus\F_p$
is an eigenvalue with probability at most about $p^{-2}$, as for a random matrix (Proposition~\ref{thm:rank2}).

For eigenvalues of degree $d\ge3$ (Section~\ref{sec:deg3}), three arguments are combined. In the degrees $d\ge5$ a crude bound suffices. The
vectors of the invariant subspaces of small dimension have large support by Proposition~\ref{lem:suppd}, and a Littlewood--Offord inequality
for subspaces (Lemma~\ref{lem:LObases}) bounds the probability that a random sign vector is orthogonal to a subspace whose nonzero vectors
have large support. This suffices against the discriminant bound on the number of repeated factors. In degree four, and in degree three at
the primes $p>\sqrt n$, Proposition~\ref{prop:SEd} and the trivial bound $n/(2d)$ on the number of repeated factors of degree $d$ at each
prime suffice, since $n\sum_pp^{-d}$ is small over these primes. For degree three at the primes $n^{1/2-\varepsilon}<p\le\sqrt n$ we need a
bound for the probability that a given element of $\F_{p^3}$ of degree three is an eigenvalue (Proposition~\ref{thm:rank3}), the analogue of
Proposition~\ref{thm:rank2}.

\subsection{Rational factors of small degree}\label{subsec:plan-nofactor}
The fourth component, Theorem~\ref{thm:nofactor}, is proved in Section~\ref{sec:nofactor} by a union bound over the candidate factors. A
monic irreducible factor of degree $j$ has all its roots among the eigenvalues of $S_m$, which are real and have absolute value less than
$m$. This leaves at most $m^{2j^2}$ candidates (Lemma~\ref{lem:candidates}). If a candidate $g$ has a root $a$ modulo a prime $p$, then
$g\mid\chi_{S_m}$ forces the matrix $\bar S_m-\bar aI$ to be singular. For every fixed $\lambda\in\F_p$, the rank estimate of
Section~\ref{sec:deg1} (Theorem~\ref{thm:rankp}) shows that $\bar S_m-\lambda I$ is singular with probability at most about $1/p$, for all
primes $23\le p\le e^{m^{1/9}}$. So it suffices to find, for every candidate of degree $j$, a prime in this range at which it has a root and
which is much larger than $m^{2j^2}$. By a theorem of Schur (1912)~\cite{Sch12}, every nonconstant polynomial in $\Z[x]$ has roots modulo
infinitely many primes. We prove an effective form of this theorem by Chebyshev's method (1852)~\cite{Che52}. Given a bound $A$, it compares
the size of the product $g(1)g(2)\cdots g(N)$, where $N:=\lceil A^3\rceil$, with the largest possible contribution of the primes below $A$,
and it provides a prime $p$ with $A\le p\le A^{4j}$ at which $g$ has a root (Proposition~\ref{lem:schur}). With $\log A=m^{1/9}/(4j)$, this
prime lies in the range of the rank estimate, and it is much larger than $m^{2j^2}$ as long as $j\le m^{\eta}$ with $\eta<\frac1{27}$. Apart
from Theorem~\ref{thm:rankp}, the proof uses only classical facts from number theory.

\subsection{Huge primes}\label{subsec:plan-huge}
The fifth component, Theorem~\ref{thm:huge}, is proved in Section~\ref{sec:huge}. At the huge primes there are too many candidate vectors
modulo $p$ for a union bound, and we need a global argument. A vector $w\in\F_p^n$ with concentration $\rho_p(w)\ge n^{-3}$ is a linear
combination of the reductions of at most $3\log_2n$ vectors with entries in $\{-1,0,1\}$ (Lemma~\ref{lem:dissoc}). This follows from an
elementary argument with dissociated sets, and it is a very weak form of the inverse Littlewood--Offord theorems mentioned in
Section~\ref{subsec:plan-moderate}. If $p$ is large enough, a nonzero invariant subspace modulo $p$ inside the span of such reductions
forces a nonzero invariant subspace over $\Q$ inside the span of the integer vectors (Lemma~\ref{lem:liftinv}), and hence a monic factor in
$\Z[x]$ of small degree of the characteristic polynomial. The argument is applied to the random Seidel matrices $S_m$, $m\in\{n-1,n\}$. As
in the large degrees, Odlyzko's lemma handles their invariant subspaces of dimension greater than $3\log_2m$. Modulo a prime
$p>e^{(\log m)^6}$, a smaller invariant subspace all of whose nonzero vectors have concentration at least $m^{-3}$ forces a rational factor
of $\chi_{S_m}$ of degree at most $(\log m)^3$ (Proposition~\ref{lem:acinv}), and Theorem~\ref{thm:nofactor} excludes such factors. Passing
to the integers to treat all huge primes at once is the idea of Nguyen--Wood (2025)~\cite{NW25} for cokernels. The number of repeated
factors at the huge primes is again controlled by the discriminant. Together with Corollary~\ref{cor:ACplus}, the argument also shows that,
with high probability, the eigenvectors of a random Seidel matrix with eigenvalues in $\F_p$ are anti-concentrated at all primes $p\ge23$ at
once (Remark~\ref{rem:hugeAC}).

\section{Mates, local failures and Seidel switching}\label{sec:prelim}

This section collects the facts on which the strategy of Section~\ref{sec:plan} rests. We recall the description of generalized cospectral
mates by rational orthogonal matrices (Theorem~\ref{thm:wangxu}) and the theorem of O'Rourke--Touri (2016)~\cite{OT16} on controllability
(Theorem~\ref{thm:OT}), and we show that a matrix of $\QG$ whose level is divisible by an odd prime $p$ forces a local failure at $p$
(Lemma~\ref{lem:matesfail}). After the lemma of Odlyzko (1988)~\cite{Odl88} on sign vectors in subspaces (Lemma~\ref{lem:odlyzko}), we
describe local failures in terms of eigenvectors (Lemma~\ref{lem:simple}). We then introduce Seidel switching, which replaces the all-ones
vector by a random sign vector that is independent of the matrix (Lemmas~\ref{lem:switch} and~\ref{lem:Epswitch}). Finally, we bound the
geometric multiplicities of a random symmetric matrix modulo a prime, and the kernels of polynomials of small degree in it
(Lemma~\ref{lem:apriori} and Corollary~\ref{cor:kernelpoly}), and we show that the vectors killed by such polynomials have large support
(Proposition~\ref{lem:suppd}). These results are used at the small primes (Sections~\ref{sec:local}--\ref{sec:smooth}) as well as for the
local failures of large degree and at the moderate and the huge primes (Sections~\ref{sec:reduction}, \ref{sec:tools}--\ref{sec:deg3}
and~\ref{sec:huge}).

\subsection{Generalized cospectral mates and levels}
Recall from Section~\ref{subsec:plan-mates} the set $\QG$ of rational orthogonal matrices $Q$ with $Qe=e$ such that $Q^{\top}AQ$ is the
adjacency matrix of a graph, and the level $\ell(Q)$, the least integer $\ell\ge1$ such that $\ell Q$ is integral. The following theorem
reduces the DGS property of a controllable graph to the levels of the matrices in $\QG$. The existence of the orthogonal matrix is due to
Johnson--Newman (1980)~\cite{JN80}, and its uniqueness and rationality to Wang--Xu (2006)~\cite{WX06a} (see also Wang
(2013)~\cite[Theorem~1]{W13}).

\begin{theorem}[Johnson--Newman 1980~\cite{JN80}, Wang--Xu 2006~\cite{WX06a}]\label{thm:wangxu}
Let $G$ be controllable and let $H$ be a graph on $[n]$ with the same generalized spectrum as $G$. Then there is a unique orthogonal matrix
$Q$ with $Qe=e$ and $Q^{\top}A(G)Q=A(H)$, namely $Q^{\top}=W(H)W(G)^{-1}$. In particular $Q$ is rational and $Q\in\QG$. Consequently $G$ is
DGS if and only if every $Q\in\QG$ is a permutation matrix. Moreover, a matrix $Q\in\QG$ is a permutation matrix if and only if
$\ell(Q)=1$.
\end{theorem}
\begin{proof}
The existence of $Q$ is the theorem of Johnson--Newman~\cite{JN80}. If $Q$ is orthogonal with $Qe=e$ and $Q^{\top}A(G)=A(H)Q^{\top}$,
then $Q^{\top}e=Q^{-1}e=e$, so $Q^{\top}A(G)^ke=A(H)^ke$ for all $k$, that is, $Q^{\top}W(G)=W(H)$. This gives the formula for $Q$ and its
uniqueness, as observed by Wang--Xu~\cite{WX06a}. If $Q$ is a permutation matrix, then $H\cong G$. Conversely, if $H\cong G$ via a
permutation matrix $P$ with $P^{\top}A(G)P=A(H)$, then $P\in\QG$, and $Q=P$ by uniqueness. Every $Q\in\QG$ arises in this way, as the
orthogonal matrix attached to some graph $H$ with the same generalized spectrum as $G$. Indeed, $Q^{\top}e=e$ gives $Q^{\top}JQ=J$, so the
graph $H$ with $A(H)=Q^{\top}AQ$ has the same generalized spectrum as $G$, because its adjacency matrix is similar to $A$ and the adjacency
matrix $J-I-A(H)=Q^{\top}(J-I-A)Q$ of its complement is similar to that of the complement of $G$. Hence $G$ is DGS if and only if
every $Q\in\QG$ is a permutation matrix. Finally, $\ell(Q)=1$ if and only if $Q$
is integral. An integral orthogonal matrix is a signed permutation matrix, and $Qe=e$ forces it to be a permutation matrix.
\end{proof}

We use controllability through the following theorem, which confirmed a conjecture of Godsil (2012)~\cite{God12}.

\begin{theorem}[O'Rourke--Touri 2016~{\cite[Theorem~1.4]{OT16}}]\label{thm:OT}
With probability $1-o(1)$, the random graph $G(n,1/2)$ is controllable.
\end{theorem}

\subsection{Local failures}
Recall from Definition~\ref{def:Ep} that $E_p^{(d)}$ is the event that some simple $\bar A$-submodule of $\F_p^n$ of dimension $d$ is totally
isotropic and orthogonal to $\bar e$. Every nonzero $\bar A$-stable subspace contains a simple submodule of some type $\beta$, and a simple
submodule of type $\beta$ has dimension $\deg\beta$. The following lemma is the starting point of the arithmetic approach. It is
essentially the key lemma of Guo--Wang--Wang (2025)~\cite{GWW25}, whose statement differs slightly, and we include its short proof.

\begin{lemma}[Guo--Wang--Wang 2025~{\cite[Lemma~11]{GWW25}}]\label{lem:matesfail}
Let $G$ be a graph, let $Q\in\QG$ have level $\ell:=\ell(Q)$, and let $p$ be an odd prime dividing $\ell$. Then the column space
$U_p\subseteq\F_p^n$ of $\ell Q\bmod p$ is a nonzero $\bar A$-stable totally isotropic subspace orthogonal to $\bar e$. In particular
$E_p^{(d)}$ holds for some $d\ge1$.
\end{lemma}
\begin{proof}
Let $B:=Q^{\top}AQ$, which is an integral matrix. Then $AQ=QB$, as $Q$ is orthogonal, so $\bar A$ maps the columns of $\ell Q$ modulo
$p$ into $U_p$, and $U_p$ is $\bar A$-stable.
Since $Q$ is orthogonal, $(\ell Q)^{\top}(\ell Q)=\ell^2I\equiv0\pmod p$, so $U_p$ is totally isotropic. Moreover
$e^{\top}(\ell Q)=\ell(Q^{\top}e)^{\top}={\ell e^{\top}\equiv0\pmod p}$, as $Q^{\top}e=Q^{-1}e=e$. Finally $\ell Q\not\equiv0\pmod p$, since
otherwise $(\ell/p)Q$ would be integral. As $U_p\ne0$ is $\bar A$-stable, it contains a simple $\bar A$-submodule $U$, which is totally
isotropic and orthogonal to $\bar e$, so $E_p^{(\dim U)}$ holds.
\end{proof}

\subsection{Sign vectors and simple submodules}
We use the following estimate of Odlyzko (1988)~\cite{Odl88} on sign vectors in subspaces. We state it over an arbitrary field $F$, for
vectors with entries in a set $\{a,b\}$ of two distinct elements of $F$. We apply it over $\Q$, and over $\F_p$ to the reduction $\bar u$
of a vector $u$ that is uniform on $\{a,b\}^N$ for integers ${a\not\equiv b\pmod p}$, such as $\{a,b\}=\{\pm1\}$ for odd $p$ and
$\{a,b\}=\{0,1\}$ for every $p$. Then $\bar u$ is uniform on $\{\bar a,\bar b\}^N$, and $\bar a\ne\bar b$.

\begin{lemma}[Odlyzko 1988~\cite{Odl88}]\label{lem:odlyzko}
Let $F$ be a field, let $a,b\in F$ be distinct, and let $V\subseteq F^N$ be a subspace of dimension $k$, or a translate of such a subspace.
Then $V$ contains at most $2^k$ vectors of $\{a,b\}^N$. Hence, if $u$ is uniform on $\{a,b\}^N$, then $\Pr(u\perp U)\le2^{-d}$ for every
subspace $U\subseteq F^N$ of dimension $d$, and a matrix whose $l$ columns are independent copies of $u$ has rank less than $l$ with
probability less than $2^{l-N}$.
\end{lemma}
\begin{proof}
Let $V_0$ be the subspace of which $V$ is a translate. There is a set $I$ of $k$ coordinates such that the restriction map $V_0\to F^I$ is
injective, for instance the pivot columns of a basis of $V_0$ in echelon form. Then the restriction map $V\to F^I$ is injective as well, so
the vectors of $\{a,b\}^N$ in $V$ have distinct restrictions to $I$, which lie in the set $\{a,b\}^I$ of size $2^k$. For the second
statement apply this to $V=U^{\perp}$, which has dimension $N-d$. For the third, if the columns $u_1,\dots,u_l$ have rank less than $l$,
then some $u_j$ lies in the span of $u_1,\dots,u_{j-1}$, a subspace of dimension at most $j-1$. Given $u_1,\dots,u_{j-1}$, this has
probability at most $2^{j-1-N}$, and $\sum_{j\le l}2^{j-1-N}<2^{l-N}$.
\end{proof}

Simple submodules of a symmetric matrix over $\F_p$ correspond to eigenvectors over an extension field, by the standard description below
(see Lidl--Niederreiter (1997)~\cite{LN97} for the facts on finite fields that we use). For $M\in M_m(\F_p)$ and a monic irreducible
$\beta\in\F_p[x]$ of degree $d$, the subspace $\ker\beta(M)$ is a vector space over the field $\F_p[x]/(\beta)\cong\F_q$, $q=p^d$, on
which $x$ acts as $M$.

\begin{lemma}\label{lem:simple}
Let $p$ be odd, let $M\in M_m(\F_p)$ be symmetric, and let $\beta\in\F_p[x]$ be monic irreducible of degree $d$ with a root
$\lambda\in\F_q$, where $q=p^d$. Put $K_\lambda:=\ker(M-\lambda)\subseteq\F_q^m$ and $g:=\dim_{\F_q}K_\lambda$, and for
$w\in K_\lambda$ let $U_w:=\{\Tr(cw):\ c\in\F_q\}$, where $\Tr$ acts on vectors coordinatewise.
\begin{enumerate}[label=(\alph*)]
\item The map $w\mapsto\Tr(w)$ is an $\F_p$-linear isomorphism $K_\lambda\to\ker\beta(M)$, which carries multiplication by $\lambda$ to $M$.
The simple $M$-submodules of type $\beta$ are exactly the subspaces $U_w$ with $w\ne0$, one for each $\F_q$-line $\F_qw\subseteq K_\lambda$.
They lie in $\ker\beta(M)$, and $\dim_{\F_p}U_w=d$.
\item For $c,c'\in\F_q$ we have $\Tr(cw)^{\top}\Tr(c'w)=\Tr(cc'\,w^{\top}w)$. Hence $U_w$ is totally isotropic if and only if $w^{\top}w=0$.
\item For $y\in\F_p^m$, $y\perp U_w$ if and only if $y^{\top}w=0$. In particular $\Pr_\xi(\bar\xi\perp U_w)=\Pr_\xi(\xi^{\top}w=0)$.
\item $\beta^2\mid\chi_M$ if and only if $(x-\lambda)^2\mid\chi_M$, and this holds if and only if $g\ge2$, or $g=1$ and $w^{\top}w=0$
for $0\ne w\in K_\lambda$. In particular, if $\beta^2\nmid\chi_M$, then no simple submodule of type $\beta$ is totally isotropic.
\end{enumerate}
\end{lemma}
\begin{proof}
(a) For $w\in K_\lambda$ let $w^{(i)}$ be the vector obtained by applying the Frobenius map $a\mapsto a^{p^i}$ to the coordinates of $w$.
Since $M$ has entries in $\F_p$, the map $w\mapsto w^{(i)}$ is a bijection from $K_\lambda$ onto $K_{\lambda^{p^i}}:=\ker(M-\lambda^{p^i})$,
and $\Tr(w)=\sum_{i<d}w^{(i)}$. Moreover $M\Tr(w)=\Tr(Mw)=\Tr(\lambda w)$ and $\beta(M)\Tr(w)=\Tr(\beta(\lambda)w)=0$, so $w\mapsto\Tr(w)$
is an $\F_p$-linear map $K_\lambda\to\ker\beta(M)$ that carries multiplication by $\lambda$ to $M$. The roots $\lambda^{p^i}$, $0\le i<d$,
of $\beta$ are distinct, so the sum of the spaces $K_{\lambda^{p^i}}$ is direct, and $\Tr(w)=0$ forces $w=0$. Thus the map is injective.
Over $\F_q$, the kernel of $\beta(M)$ is $\bigoplus_iK_{\lambda^{p^i}}$, which has dimension $dg$, and the dimension of a kernel does not
change under field extensions. Hence $\dim_{\F_p}\ker\beta(M)=dg=\dim_{\F_p}K_\lambda$, and the map is bijective. So $\ker\beta(M)$, as an
$\F_p[x]$-module, is isomorphic to $K_\lambda$ with $x$ acting as $\lambda$, that is, to an $\F_q$-vector space, as $\F_p[\lambda]=\F_q$. Its
submodules correspond to $\F_q$-subspaces, and the simple ones to the $\F_q$-lines $\F_qw$, whose images are the $U_w$. Every simple
submodule of type $\beta$ is killed by $\beta(M)$, hence lies in $\ker\beta(M)$.

(b) Eigenvectors $y,y'$ of the symmetric matrix $M$ for distinct eigenvalues $\mu\ne\nu$ are orthogonal for the bilinear form, because
$\mu\,y^{\top}y'=(My)^{\top}y'=y^{\top}My'=\nu\,y^{\top}y'$. Expanding the traces, we get
$\Tr(cw)^{\top}\Tr(c'w)=\sum_{i,j}(cw)^{(i)\top}(c'w)^{(j)}=\sum_i(cc'w^{\top}w)^{p^i}=\Tr(cc'w^{\top}w)$. As the trace form
$(c,c')\mapsto\Tr(cc')$ is nondegenerate, $U_w$ is totally isotropic if and only if $w^{\top}w=0$.

(c) For $y\in\F_p^m$ we have $y^{\top}\Tr(cw)=\Tr(c\,y^{\top}w)$, which vanishes for all $c$ if and only if $y^{\top}w=0$.

(d) As $\chi_M$ has coefficients in $\F_p$, applying the automorphism $c\mapsto c^{p^i}$ of $\F_q$ shows that ${(x-\lambda)^2\mid\chi_M}$
if and only if $(x-\lambda^{p^i})^2\mid\chi_M$. Since the roots $\lambda^{p^i}$, $0\le i<d$, of $\beta$ are distinct, this gives the
first equivalence. Next, $(x-\lambda)^2\mid\chi_M$ if and only if $\chi_M(\lambda)=\chi_M'(\lambda)=0$. Put $N:=\lambda I-M$. Then
$\chi_M(\lambda)=\det N$ vanishes if and only if $g\ge1$, and $\chi_M'(\lambda)=\operatorname{tr}\operatorname{adj}N$ by Jacobi's formula.
If $g\ge2$, all minors of order $m-1$ of $N$ vanish, so $\operatorname{adj}N=0$. If $g=1$ and $K_\lambda=\F_qw$, then
$\operatorname{adj}N\ne0$, and $N\operatorname{adj}N=\det(N)I=0$ shows that the columns of $\operatorname{adj}N$ lie in $\F_qw$. As
$\operatorname{adj}N$ is symmetric, like $N$, we get $\operatorname{adj}N=c\,ww^{\top}$ with $c\ne0$, and
$\chi_M'(\lambda)=c\,w^{\top}w$. Hence, if $g\ge1$, then $\chi_M'(\lambda)=0$ if and only if $g\ge2$, or $g=1$ and $w^{\top}w=0$, which
proves the second equivalence. Finally, if $\beta^2\nmid\chi_M$, then $K_\lambda=0$, or
$K_\lambda=\F_qw$ with $w^{\top}w\ne0$, and by (a) and (b) no simple submodule of type $\beta$ is totally isotropic.
\end{proof}

For $g=1$, the second equivalence in (d) is the symmetric case of the classical fact that an eigenvalue of geometric multiplicity one is a
simple root of the characteristic polynomial if and only if its left and right eigenvectors are not orthogonal, and the proof above works
over every field.

\subsection{Seidel switching}
A local failure involves orthogonality to the all-ones vector $e$, which is a fixed vector, so that for a fixed submodule this orthogonality
is a condition on the matrix alone. At odd primes, Seidel switching, which goes back to van Lint--Seidel (1966)~\cite{vLS66}, replaces $e$
by a random sign vector that is independent of the matrix, so that orthogonality can be estimated conditionally on the matrix. Recall that
the Seidel matrix $S=J-I-2A$ of $G(n,1/2)$ is a uniformly random Seidel matrix of order $n$.

\begin{lemma}\label{lem:switch}
Let $\Sigma$ be a uniformly random diagonal sign matrix independent of $A$, and put ${S':=\Sigma S\Sigma}$ and $u:=\Sigma e$. Then $S'$ and
$u$ are independent, $S'$ is a uniformly random Seidel matrix, and $u$ is uniform on $\{\pm1\}^n$.
\end{lemma}
\begin{proof}
For each fixed diagonal sign matrix $\sigma$ with diagonal entries $\sigma_1,\dots,\sigma_n$, the map $S\mapsto\sigma S\sigma$ multiplies
the entry $(i,j)$ by $\sigma_i\sigma_j$ and is a bijection of the set of Seidel matrices of order $n$. Hence
$(S,\Sigma)\mapsto(\Sigma S\Sigma,\Sigma)$ is a bijection of the product of this set with the set of diagonal sign matrices, and
$(S',\Sigma)$ is uniform on this product, like $(S,\Sigma)$. So $S'$ is uniform and independent of $\Sigma$, and $u=\Sigma e$, the vector
of diagonal entries of $\Sigma$, is uniform on $\{\pm1\}^n$.
\end{proof}

After switching, a local failure becomes a statement about the random Seidel matrix $S'$ and the independent sign vector $u$. Recall
from Section~\ref{subsec:plan-switch} that a monic irreducible $\beta\in\F_p[x]$ is a repeated factor of a symmetric integer matrix $T$
at $p$ if $\beta^2$ divides $\chi_T$ modulo $p$.

\begin{lemma}\label{lem:Epswitch}
Let $p$ be odd, and let $\Sigma$, $S'$ and $u$ be as in Lemma~\ref{lem:switch}. Then $E_p^{(d)}$ holds for $G$ if and only if $\bar S'$ has a
totally isotropic simple submodule $U$ of some type $\beta$ with $\deg\beta=d$ and $U\perp\bar u$. In that case $\beta$ is a repeated
factor of $S'$ at $p$.
\end{lemma}
\begin{proof}
On $\bar e^{\perp}$ we have $\bar A=-\frac12(I+\bar S)$. So a subspace $U\perp\bar e$ is $\bar A$-stable if and only if it is
$\bar S$-stable, and as the same holds for the subspaces of $U$, it is a simple $\bar A$-submodule if and only if it is a simple
$\bar S$-submodule. The map $\Sigma$ preserves the bilinear form, satisfies $\bar S'\Sigma=\Sigma\bar S$, and maps $\bar e^{\perp}$ to
$\bar u^{\perp}$. So $U\mapsto\Sigma U$ is a type-preserving bijection from the totally isotropic simple $\bar S$-submodules orthogonal to
$\bar e$ onto the totally isotropic simple $\bar S'$-submodules orthogonal to $\bar u$. As a simple submodule has the degree of its type as
its dimension, this proves the equivalence. By Lemma~\ref{lem:simple}(d), a totally
isotropic simple submodule of type $\beta$ exists only if $\beta^2$ divides $\chi_{S'}$ modulo $p$.
\end{proof}

\subsection{Geometric multiplicities and sparse vectors}\label{subsec:prelim-sparse}
We finish with two estimates for random symmetric matrices modulo a prime, which are used at the small primes (Sections~\ref{sec:fourier}
and~\ref{sec:smooth}) as well as at the moderate primes (Sections~\ref{sec:deg1}--\ref{sec:deg3}). In this subsection $a\ne b$ are integers,
and $M$ is a random symmetric matrix of order $m$ with zero diagonal whose entries above the diagonal are independent and uniform on
$\{a,b\}$. We use the estimates for the adjacency matrix of $G(m,1/2)$, where $\{a,b\}=\{0,1\}$, at every prime, and for a uniformly random
Seidel matrix, where $\{a,b\}=\{\pm1\}$, at every odd prime. The first estimate bounds the geometric multiplicity of a fixed eigenvalue. The
bound is crude, but it holds uniformly in the order of the matrix, the prime $p\nmid a-b$ and the eigenvalue $\lambda\in\overline{\F}_p$.
As in the work of Costello--Tao--Vu (2006)~\cite{CTV06} on the singularity of random symmetric matrices, we expose the matrix one row and
column at a time and follow the corank, that is, the dimension of the kernel. The following elementary lemma describes one step of this
exposure, and it is used again in Sections~\ref{sec:reduction}, \ref{sec:deg1} and~\ref{sec:deg2}.

\begin{lemma}\label{lem:border}
Let $F$ be a field, let $L\in F^{r\times r}$ be symmetric with kernel $K$, let $y\in F^r$ and $c\in F$, and put
\[
L^+:=\begin{pmatrix}L&y\\ y^{\top}&c\end{pmatrix}.
\]
We write the vectors of $F^{r+1}$ as pairs $(v,t)$ with $v\in F^r$ and $t\in F$.
\begin{enumerate}[label=(\alph*)]
\item If $y\not\perp K$, then $\ker L^+=(K\cap y^{\perp})\times\{0\}$, so that $\corank L^+=\corank L-1$.
\item If $y\perp K$, then $y=Ly_0$ for some $y_0\in F^r$, and $y^{\top}y_0$ does not depend on the choice of $y_0$. If moreover
$c\ne y^{\top}y_0$, then $\ker L^+=K\times\{0\}$, so that $\corank L^+=\corank L$. If $c=y^{\top}y_0$, then
$\ker L^+=\{(v-ty_0,t):\ v\in K,\ t\in F\}$, so that $\corank L^+=\corank L+1$.
\end{enumerate}
In particular $\lvert\corank L^+-\corank L\rvert\le1$, and $\corank L^+\ge\corank L$ if and only if $y\perp K$. If $L$ is invertible, then
$\corank L^+\le1$, with equality if and only if $c=y^{\top}L^{-1}y$, and in that case $\ker L^+$ is spanned by $(L^{-1}y,-1)$.
\end{lemma}

\begin{proof}
As $L$ is symmetric, its image is $K^{\perp}$, since it is contained in $K^{\perp}$ and both have dimension $r-\dim K$. A pair $(v,t)$
lies in $\ker L^+$ if and only if $Lv+ty=0$ and $y^{\top}v+tc=0$. (a) If $y\notin K^{\perp}=\operatorname{Im}L$, the first equation
forces $t=0$, and then the two equations say that $v\in K\cap y^{\perp}$, a hyperplane of $K$. (b) If $y=Ly_0$, then $y^{\top}y_0$ does not
depend on $y_0$, as two choices differ by a vector of $K$, which is orthogonal to $y$. The first equation says that $v+ty_0\in K$, and then
$y^{\top}v=y_0^{\top}Lv=-t\,y_0^{\top}Ly_0=-t\,y^{\top}y_0$, so that the second equation reads
$t(c-y^{\top}y_0)=0$. If $c\ne y^{\top}y_0$, this forces $t=0$, so that $\ker L^+=K\times\{0\}$. If $c=y^{\top}y_0$, it holds for every
$t\in F$, so that $\ker L^+$ is the space $\{(v-ty_0,t):\ v\in K,\ t\in F\}$, which has dimension $\dim K+1$. This proves (b), and the
remaining claims follow from (a) and (b), with $K=0$ and $y_0=L^{-1}y$ if $L$ is invertible.
\end{proof}

By Lemma~\ref{lem:border}, the corank decreases at each step of the exposure unless the new column is orthogonal to the current kernel, and
by Lemma~\ref{lem:odlyzko} this orthogonality is unlikely when the kernel is large. This gives the first estimate.

\begin{lemma}\label{lem:apriori}
Let $p$ be a prime with $p\nmid a-b$, let $\lambda\in\overline{\F}_p$, and put $g_\lambda(M):=\dim_{\overline{\F}_p}\ker(\bar M-\lambda I)$. Then
\[
\E\big(2^{g_\lambda(M)^2/4}\big)\le30,\qquad\text{hence}\qquad\Pr\bigl(g_\lambda(M)\ge r\bigr)\le30\cdot2^{-r^2/4}\quad(r\ge0).
\]
Moreover $\E\big(10^{g_\lambda(M)}\big)\le C_0:=1+30\sum_{r\ge1}10^r2^{-r^2/4}<\infty$, for all such $p$ and all $\lambda$ and $m$.
\end{lemma}

\begin{proof}
For $1\le k\le m$ let $M_k$ be the leading principal $k\times k$ submatrix of $M$, and put $L_k:=\bar M_k-\lambda I$,
$K_k:=\ker L_k\subseteq\overline{\F}_p^k$ and $r_k:=\dim K_k=g_\lambda(M_k)$. For $k<m$ let $y\in\{a,b\}^k$ be the part of column $k+1$ of
$M$ above the diagonal. It is uniform and independent of $M_k$, and
\[
L_{k+1}=\begin{pmatrix}L_k&\bar y\\ \bar y^{\top}&-\lambda\end{pmatrix}.
\]
By Lemma~\ref{lem:border}, applied over $\overline{\F}_p$ with $L=L_k$, we have $|r_{k+1}-r_k|\le1$, and $r_{k+1}\ge r_k$ only if
$\bar y\perp K_k$. The vectors $z\in\F_p^k$ with $z\perp K_k$ form an $\F_p$-subspace of dimension at most $k-r_k$ (vectors of
$\F_p^k$ that are linearly independent over $\F_p$ remain so over $\overline{\F}_p$, and $K_k^{\perp}$ has dimension $k-r_k$), so by
Lemma~\ref{lem:odlyzko}
\[
\Pr\bigl(r_{k+1}\ge r_k\mid M_k\bigr)\le\Pr\bigl(\bar y\perp K_k\mid M_k\bigr)\le2^{-r_k}.
\]
Put $\Upsilon(r):=2^{r^2/4}$, which is increasing on $r\ge0$. If $r_k\ge3$, the two facts above give
\[
\E\big(\Upsilon(r_{k+1})\mid M_k\big)\le\Upsilon(r_k-1)+2^{-r_k}\Upsilon(r_k+1)=2^{5/4-r_k/2}\,\Upsilon(r_k)\le2^{-1/4}\Upsilon(r_k),
\]
while if $r_k\le2$, then $r_{k+1}\le3$ and $\Upsilon(r_{k+1})\le\Upsilon(3)=2^{9/4}<4.76$. Hence $t_k:=\E(\Upsilon(r_k))$ satisfies
$t_{k+1}\le2^{-1/4}t_k+4.76$, and $t_1\le\Upsilon(1)=2^{1/4}$. By induction $t_k\le4.76/(1-2^{-1/4})<30$ for all $k$. The tail bound
follows by Markov's inequality, and then $\E\big(10^{g_\lambda(M)}\big)\le{\Pr(g_\lambda(M)=0)}+\sum_{r\ge1}10^r\Pr(g_\lambda(M)\ge r)\le C_0$.
\end{proof}

A union bound over the eigenvalues of small degree turns Lemma~\ref{lem:apriori} into a bound for the kernels of all polynomials of small
degree in $\bar M$ at once. We use it here and in Section~\ref{sec:deg3}.

\begin{corollary}\label{cor:kernelpoly}
Let $p$ be a prime with $p\nmid a-b$, and let $d,r\ge1$ be integers. Then, except with probability at most $60\,p^{d-1}2^{-r^2/4}$, every
nonzero $\alpha\in\F_p[x]$ of degree less than $d$ satisfies $\dim\ker\alpha(\bar M)\le(d-1)(r-1)$.
\end{corollary}

\begin{proof}
The elements $\lambda\in\overline{\F}_p$ of degree less than $d$ over $\F_p$ lie in the fields $\F_{p^j}$ with $1\le j<d$, so there are at
most $\sum_{1\le j<d}p^j\le2p^{d-1}$ of them. By Lemma~\ref{lem:apriori} and a union bound, except with probability at most
$60\,p^{d-1}2^{-r^2/4}$, all of them satisfy $g_\lambda(M)\le r-1$. Suppose that this holds, and let $\alpha\in\F_p[x]$ be nonzero of
degree less than $d$. Over $\overline{\F}_p$, $\ker\alpha(\bar M)$ is the direct sum of the spaces $\ker(\bar M-\lambda I)^{k_\lambda}$
over the roots $\lambda$ of $\alpha$, of multiplicities $k_\lambda$, and $\dim\ker N^k\le k\dim\ker N$ for every square matrix $N$ and
every $k\ge1$. As the roots of $\alpha$ have degree less than $d$ over $\F_p$ and $\sum_\lambda k_\lambda\le d-1$, we get
$\dim\ker\alpha(\bar M)\le\sum_\lambda k_\lambda g_\lambda(M)\le(d-1)(r-1)$.
\end{proof}

The second estimate says that the nonzero vectors killed by a polynomial of small degree in $\bar M$, such as the vectors of its invariant
subspaces of small dimension, have large support, simultaneously at all primes $p\le e^{m^{1/8}}$. A vector with small support gives a
nonzero kernel vector of a block Krylov matrix (after Krylov (1931)~\cite{Kry31}) built from two independent blocks of $M$, and this is
unlikely.

\begin{proposition}\label{lem:suppd}
For all large $m$, with probability at least $1-e^{-m/8}$ the following holds for every prime $p\le e^{m^{1/8}}$ with $p\nmid a-b$, every
$1\le d\le m^{1/4}$, every monic $g\in\F_p[x]$ of degree $d$ and every nonzero $y\in\F_p^m$ with $g(\bar M)y=0$:
$\lvert\supp y\rvert>m/(4(d+1))$.
\end{proposition}

\begin{proof}
\step{1}{We show that a nonzero vector with small support killed by $g(\bar M)$ gives a nonzero kernel vector of a block Krylov
matrix.} Let $S\subseteq[m]$, $R:=[m]\setminus S$ and $s:=|S|$, and write
\[
\bar M=\begin{pmatrix}\bar M_{S,S}&C^{\top}\\ C&D\end{pmatrix},\qquad C:=\bar M_{R,S},\quad D:=\bar M_{R,R}.
\]
Let $y\ne0$ be supported on $S$ with $g(\bar M)y=0$, where $g=\sum_{k\le d}h_kx^k$ with $h_d=1$, and write
$\bar M^ky=(u_k,w_k)\in\F_p^S\times\F_p^R$. Then $w_0=0$ and $w_{k+1}=Cu_k+Dw_k$, so $w_k=\sum_{j<k}D^{k-1-j}Cu_j$, and the $R$-part of
$g(\bar M)y=0$ reads
\[
\sum_{j=0}^{d-1}D^jC\,z_j=0,\qquad z_j:=\sum_{i=0}^{d-1-j}h_{i+j+1}u_i,\qquad z_{d-1}=u_0=y_S\ne0 .
\]
So the block Krylov matrix $K_S:=[\,C\,|\,DC\,|\,\cdots\,|\,D^{d-1}C\,]\in\F_p^{R\times ds}$ has a nonzero kernel vector. It suffices to
show that, with the stated probability, $K_S$ has trivial kernel for all $p$, all $d$ and all $S$ with $1\le s\le m/(4(d+1))$.

\step{2}{We show that, except with a very small probability, $\alpha(D)$ has a small kernel for every nonzero $\alpha\in\F_p[x]$ of
degree less than $d$.} The matrix $D$ is the reduction of $M_{R,R}$, a random matrix of the same kind as $M$ of order $m-s$. Put
$r_0:=\lceil m^{0.55}\rceil$, so that $r_0^2\ge m^{1.1}$. By Corollary~\ref{cor:kernelpoly}, applied to $M_{R,R}$ with $r=r_0$, except with
probability at most $60\,p^{d-1}2^{-m^{1.1}/4}$ we have $\dim\ker\alpha(D)\le(d-1)(r_0-1)<\kappa_0:=dr_0$ for all such~$\alpha$.

\step{3}{We expose the columns of $C$ one at a time and bound the probability that $K_S$ has a nonzero kernel vector.} The matrix $D$
and the columns $c_i$, $i\in S$, of $M_{R,S}$ are independent, as they consist of
different entries of $M$, and each $c_i$ is uniform on $\{a,b\}^R$. Fix $D$ with $\dim\ker\alpha(D)<\kappa_0$ for all $\alpha$ as above,
and add the column groups $(\bar c_i,D\bar c_i,\dots,D^{d-1}\bar c_i)$, $i\in S$, one at a time. If $V$ is the span of the groups added so
far, so that $\dim V\le d(s-1)$, the new group fails to add $d$ dimensions if and only if $\alpha(D)\bar c_i\in V$ for some nonzero
$\alpha$ of degree $<d$, which we may take monic (at most $2p^{d-1}$ choices). For a fixed $\alpha$ this says that $\bar c_i$ lies in the
subspace $\alpha(D)^{-1}(V)$, of dimension at most $\dim V+\dim\ker\alpha(D)$, which has probability at most
$2^{d(s-1)+\kappa_0-(m-s)}$ by Lemma~\ref{lem:odlyzko}. If every group adds $d$ dimensions, $K_S$ has rank $ds$. So $K_S$ has a nonzero
kernel vector with probability at most
\[
60\,p^{d-1}2^{-m^{1.1}/4}+2s\,p^{d-1}2^{(d+1)s+\kappa_0-m}.
\]

\step{4}{We take a union bound over the primes $p$, the degrees $d$ and the sets $S$.} There are at most $m^{1/4}e^{m^{1/8}}$ pairs
$(p,d)$, and $p^{d-1}\le e^{m^{3/8}}$ and $\kappa_0\le2m^{0.8}$. As $s\le m/(4(d+1))\le m/8$, for each pair there are at most
$2^{H(1/8)m}\le2^{0.544m}$ sets $S$, and $(d+1)s\le m/4$. Since moreover $m^{1.1}/4\ge m$ for large $m$, both terms of the bound of
Step~3 are at most $e^{o(m)}2^{-0.75m}$. So the total is at most $e^{o(m)}2^{(0.544-0.75)m}\le e^{-m/8}$ for large $m$.
\end{proof}

For $d=1$, Proposition~\ref{lem:suppd} says that the eigenvectors of $\bar M$ over $\F_p$ have linear support. Bounds of this type are a
standard first step in the study of random symmetric matrices over finite fields. For instance, Ferber--Jain--Sah--Sawhney
(2023)~\cite[Lemma~2.2]{FJSS} show that, for a random symmetric $\pm1$ matrix $Z$ and $\lambda\in\F_p$, the matrix $\bar Z-\lambda I$ maps
no sparse vector to a sparse vector. For exact kernel vectors, Proposition~\ref{lem:suppd} extends this to all vectors killed by a
polynomial of degree at most $m^{1/4}$ in $\bar M$, in particular to all invariant subspaces of dimension at most $m^{1/4}$,
simultaneously at all primes $p\le e^{m^{1/8}}$.

\section{Reduction to a single random Seidel matrix}\label{sec:reduction}

This section proves the first component of the proof, Theorem~\ref{thm:largedegrees}, and it provides the reduction that we use again for
the third and the fifth components, at the moderate and at the huge primes. By Seidel switching (Lemmas~\ref{lem:switch}
and~\ref{lem:Epswitch} of Section~\ref{sec:prelim}), a local failure $E_p^{(d)}$ (Definition~\ref{def:Ep}) is a statement about a uniformly
random Seidel matrix of order $n$ and an independent uniform sign vector. We then eliminate the sign vector as well. Conditionally on the
matrix, it enters only through the anti-concentration of the eigenvectors. The eigenspaces of dimension at least three and the totally
isotropic ones of dimension two may contain many isotropic lines, that is, lines $\F_qw$ with $w^{\top}w=0$, and they are handled by
deleting a coordinate and using the exchangeability of the coordinates. The result is the criterion of Proposition~\ref{prop:reduction},
which says that local failures $E_p^{(d)}$ with $(p,d)$ in a prescribed set are rare as soon as, for each of the orders $n$ and $n-1$, the
expected number of repeated factors of the characteristic polynomial of a single random Seidel matrix, each weighted by the
anti-concentration of its eigenvectors (Definition~\ref{def:Npi}) and restricted to typical matrices, tends to $0$. Local failures of degree
at least $3\log_2n$ need no information on the eigenvectors, and we exclude them at all odd primes at the end of the section
(Theorem~\ref{thm:largedegrees}).

The proof of Proposition~\ref{prop:reduction} rests on four facts. Eigenvectors have large support at all odd primes
(Proposition~\ref{lem:support}), deleting a coordinate lowers geometric multiplicities (Lemma~\ref{lem:delete}), the coordinates are
exchangeable (Lemma~\ref{lem:exch}), and an irreducible factor is repeated if and only if the eigenspace of one of its roots has dimension
at least two or is spanned by an isotropic eigenvector (Lemma~\ref{lem:simple}(d)). For Theorem~\ref{thm:largedegrees} we also need the
discriminant budget (Lemma~\ref{lem:budget}), which bounds the number of repeated factors at all primes above a given bound as long as the
discriminant does not vanish, and the simple spectrum theorem of Tao--Vu
(2017)~\cite{TV17}, by which the discriminant vanishes only with probability $o(1)$ (Theorem~\ref{lem:simplespectrum}). Several of these
results are also used directly in Sections~\ref{sec:deg1}--\ref{sec:deg3} and~\ref{sec:huge}.

\subsection{Eigenvectors have large support}
We first show that, with probability exponentially close to $1$, all eigenvectors of a random Seidel matrix modulo all odd primes have
support of linear size. In the proof of Proposition~\ref{prop:reduction} this provides many coordinates at which an eigenspace can be
reduced by deletion. An eigenvector supported on a small set $I$ makes the columns of $T$ indexed by $I$, restricted to the rows outside
$I$, linearly dependent modulo $p$. Over $\Q$ such a dependence is very unlikely, and if there is none, the primes at which one appears
divide a nonzero minor. So only a few primes matter, and for each of them we use a set of rows that is independent of the
minor.

\begin{proposition}\label{lem:support}
Let $T$ be a uniformly random Seidel matrix of order $m$. With probability ${1-O(2^{-m/8})}$ the following holds: for every odd
prime $p$, every $\lambda\in\overline{\F}_p$ and every nonzero $w\in\overline{\F}_p^{\,m}$ with $\bar Tw=\lambda w$, one has
$\lvert\supp w\rvert>m/20$.
\end{proposition}

\begin{proof}
We may assume that $m$ is large. Put $k:=\lfloor m/20\rfloor\ge1$. Every $w\ne0$ with $\lvert\supp w\rvert\le m/20$ has its support in a set
$I\subseteq[m]$ with $|I|=k$, and there are $\binom mk\le2^{H(k/m)m}\le2^{H(1/20)m}\le2^{0.287m}$ such sets, as $H$ is increasing on
$[0,\frac12]$. For a fixed $I$, if $\supp w\subseteq I$ and $\bar Tw=\lambda w$, then for every $i\notin I$
\[
0=\lambda w_i=(\bar Tw)_i=\sum_{j\in I}T_{ij}w_j ,
\]
so the $(m-k)\times k$ matrix $M:=T_{I^c,I}$ has the nonzero kernel vector $w_I$ over $\overline{\F}_p$. As the rank does not change
under field extensions, $\rank_pM<k$. The entries of $M$ are independent uniform signs, since none of them lies on the diagonal and no
two of them are symmetric to each other. Let $Y$ consist of the first $N_1:=\lfloor(m-k)/2\rfloor$ rows of $M$, and $Z$ of the
remaining $N_2:=m-k-N_1\ge N_1$ rows. Then $Y$ and $Z$ are independent, and $k-N_2\le k-N_1\le k-\frac{m-k-1}2\le\frac12-0.425m$.

The columns of $Y$ are independent and uniform on $\{\pm1\}^{N_1}$, so by Lemma~\ref{lem:odlyzko} over $\Q$, we have $\rank_\Q Y<k$
with probability at most $2^{k-N_1}$. Suppose that $\rank_\Q Y=k$, and choose, as a function of $Y$, a nonzero $k\times k$ minor $\kappa$
of $Y$. By Hadamard's inequality $|\kappa|\le k^{k/2}$, so $\kappa$ has at most $\frac k2\log_2k$ prime factors, and
$\rank_pM\ge\rank_pY=k$ for every prime $p\nmid\kappa$. For each odd prime $p\mid\kappa$, as $Z$ is independent of $Y$ and is a submatrix
of $M$, Lemma~\ref{lem:odlyzko} over $\F_p$, applied to the reductions of the columns of $Z$, gives
$\Pr(\rank_pM<k\mid Y)\le\Pr(\rank_pZ<k)\le2^{k-N_2}$. Summing over the at most $\frac k2\log_2k$ odd primes $p\mid\kappa$,
and adding the bound $2^{k-N_1}$ for the probability that $\rank_\Q Y<k$, we get
\begin{equation}\label{eq:supportI}
\Pr\bigl(\rank_pM<k\text{ for some odd prime }p\bigr)\ \le\ \Big(1+\frac k2\log_2k\Big)2^{\frac12-0.425m}\ \le\ m^2\,2^{-0.425m},
\end{equation}
and a union bound over the sets $I$ shows that the conclusion of the proposition fails with probability at most
$m^22^{(0.287-0.425)m}=O(2^{-m/8})$.
\end{proof}

For a single prime and eigenvalues in $\F_p$, lower bounds of this kind for the supports of eigenvectors are a standard first step in the
study of random symmetric matrices over finite fields (see the discussion after Proposition~\ref{lem:suppd}). The point of
Proposition~\ref{lem:support} is that it holds at all odd primes at once and for eigenvalues of every degree, with an error probability that
does not depend on $p$. For a fixed set $I$ of size
$\lfloor m/20\rfloor$, the proof shows that the greatest common divisor of the maximal minors of $T_{I^c,I}$ is a power of $2$, except
with probability at most $m^22^{-0.425m}$, by~\eqref{eq:supportI}. The argument over $\Q$ alone shows that, with probability
$1-O(2^{-m/8})$, every eigenvector of $T$ in $\R^m$ has more than $m/20$ nonzero coordinates. Over the real numbers much stronger forms
of delocalization are known for many ensembles, see for instance Rudelson--Vershynin (2016)~\cite{RV16}.

\subsection{Deleting a coordinate}
Deleting a coordinate in the support of an eigenvector lowers the geometric multiplicity by one, at the price of an orthogonality
condition for the deleted column. We state this for the kernel of a symmetric matrix over an arbitrary field. It is the special case of
Lemma~\ref{lem:border} in which the kernel of the bordered matrix contains a vector with nonzero last coordinate.

\begin{lemma}\label{lem:delete}
Let $F$ be a field, $B\in F^{m\times m}$ symmetric with kernel $K$, and $i\in[m]$ such that $w_i\ne0$ for some $w\in K$. Let
$B^{(i)}$ be the matrix obtained from $B$ by deleting row and column $i$, and $b_i\in F^{m-1}$ the $i$th column of $B$ with its $i$th
entry deleted. Then $\dim\ker B^{(i)}=\dim K-1$, $\ker B^{(i)}\subseteq b_i^{\perp}$, and $K\cap e_i^{\perp}$ consists of the vectors
obtained from the vectors of $\ker B^{(i)}$ by inserting a zero in position $i$.
\end{lemma}

\begin{proof}
After a simultaneous permutation of the rows and columns we may assume that $i=m$. Then $B=L^+$ in the notation of
Lemma~\ref{lem:border}, with $L:=B^{(m)}$, $y:=b_m$ and $c:=B_{mm}$, and some vector of $K=\ker L^+$ has nonzero last coordinate. This
excludes case (a) of Lemma~\ref{lem:border} and the first case of (b). So $b_m\perp\ker B^{(m)}$, $\dim K=\dim\ker B^{(m)}+1$, and the
vectors of $K$ with last coordinate $0$ are the vectors $(v,0)$ with $v\in\ker B^{(m)}$.
\end{proof}

\subsection{Exchangeability}
The deletion lemma will be applied at coordinates that depend on the matrix. Since the law of a random Seidel matrix is invariant under
simultaneous permutations of rows and columns, it suffices to delete the last coordinate, at the price of a factor $20$, provided that
more than $m/20$ coordinates are available. This averaging over the coordinates is the device of the invertibility via distance lemma of
Rudelson--Vershynin (2008)~\cite[Lemma~3.5]{RV08}, where the spread of incompressible vectors~\cite[Lemma~3.4]{RV08} plays the role of
Proposition~\ref{lem:support}.

\begin{lemma}\label{lem:exch}
Let $T$ be a uniformly random Seidel matrix of order $m\ge2$. For $i\in[m]$ let $T^{(i)}$ be the matrix obtained from $T$ by deleting
row and column $i$, and let $t_i\in\{\pm1\}^{m-1}$ be the $i$th column of $T$ with its $i$th entry deleted. Then $(T^{(m)},t_m)$ has the
law of $(S_{m-1},\xi)$, where $\xi$ is uniform in $\{\pm1\}^{m-1}$ and independent of $S_{m-1}$. Let $\Phi$ be a set of pairs consisting
of a Seidel matrix of order $m-1$ and a vector in $\{\pm1\}^{m-1}$, and let $\mathcal F_i$ be the event that $(T^{(i)},t_i)\in\Phi$. Then
$\Pr(\mathcal F_i)=\Pr(\mathcal F_m)$ for all $i$. Moreover, if $\mathcal E$ is an event on which more than $m/20$ indices $i\in[m]$
satisfy $\mathcal F_i$, then
\[
\Pr(\mathcal E)\le20\Pr(\mathcal F_m).
\]
\end{lemma}

\begin{proof}
The matrix $T^{(m)}$ is formed by the entries $T_{jk}$, $j<k<m$, and $t_m$ by the entries $T_{jm}$, $j<m$. These are disjoint sets of
independent uniform signs, which gives the first claim. Fix $i$ and let $\tilde T$ be obtained from $T$ by listing the rows and columns
in the order $1,\dots,i-1,i+1,\dots,m,i$. A simultaneous permutation of rows and columns is a bijection of the set of Seidel matrices of
order $m$, so $\tilde T$ is again uniformly random, and $(\tilde T^{(m)},\tilde t_m)=(T^{(i)},t_i)$. Hence $(T^{(i)},t_i)$ has the law
of $(T^{(m)},t_m)$, and $\Pr(\mathcal F_i)=\Pr(\mathcal F_m)$. Finally, with $N:=\sum_{i=1}^m\mathbf 1_{\mathcal F_i}$,
\[
\frac m{20}\Pr(\mathcal E)\le\E\big(N;\ \mathcal E\big)\le\E(N)=\sum_{i=1}^m\Pr(\mathcal F_i)=m\Pr(\mathcal F_m).
\qedhere
\]
\end{proof}

\subsection{The reduction}
We now weight each repeated factor by the anti-concentration of its eigenvectors.

\begin{definition}\label{def:Npi}
Let $T$ be a symmetric integer matrix of order $m$, $p$ an odd prime, and $\beta\in\F_p[x]$ monic irreducible. Put
\[
a_\beta(T):=\max_U\Pr_\xi\bigl(\bar\xi\perp U\bigr),
\]
the maximum over the simple $\bar T$-submodules $U\subseteq\F_p^m$ of type $\beta$, with $a_\beta(T):=0$ if there is none. For a set
$\Pi$ of pairs $(p,d)$, $p$ an odd prime and $d\ge1$, put
\[
N_\Pi(T):=\sum_{(p,\beta)}a_\beta(T)\ \in[0,\infty],
\]
the sum running over all pairs $(p,\beta)$ such that $(p,\deg\beta)\in\Pi$ and $\beta$ is a repeated factor of $T$ at $p$.
\end{definition}

By Lemma~\ref{lem:simple}(a), every simple submodule of type $\beta$ has dimension $\deg\beta$ over $\F_p$, so
Lemma~\ref{lem:odlyzko} gives
\begin{equation}\label{eq:abetaodl}
a_\beta(T)\le2^{-\deg\beta}.
\end{equation}
If $\lambda\in\F_q$, $q=p^{\deg\beta}$, is a root of $\beta$, then by Lemma~\ref{lem:simple}(a),(c)
\begin{equation}\label{eq:abetaeig}
a_\beta(T)=\max\bigl\{\Pr_\xi(\xi^{\top}w=0):\ 0\ne w\in\ker(\bar T-\lambda I)\subseteq\F_q^m\bigr\},
\end{equation}
the maximum of the empty set being $0$.

The main result of this section is a criterion for local failures to be rare. It involves the first moments of $N_\Pi$ for random
Seidel matrices of orders $n$ and $n-1$, restricted to sets $\mathcal G_m$ of typical matrices that are at our disposal. The set of
pairs $(p,d)$ may depend on $n$.

\begin{proposition}\label{prop:reduction}
For all sufficiently large $n$ let $\Pi_n$ be a set of pairs $(p,d)$ with $p$ an odd prime and $d\ge1$, and for $m\in\{n-1,n\}$ let
$\mathcal G_m$ be a set of Seidel matrices of order $m$ (equivalently, an event depending only on $S_m$), which may depend on $n$.
Suppose that, for $m\in\{n-1,n\}$,
\[
\Pr(S_m\notin\mathcal G_m)\to0\qquad\text{and}\qquad\E\big(N_{\Pi_n}(S_m);\ S_m\in\mathcal G_m\big)\to0\qquad(n\to\infty).
\]
Then, for the random graph $G(n,1/2)$, the probability that $E_p^{(d)}$ holds for some $(p,d)\in\Pi_n$ tends to~$0$.
\end{proposition}

\begin{proof}
Let $n$ be large and put $\Pi:=\Pi_n$. Let $\Sigma$, $S'=\Sigma S\Sigma$ and $u=\Sigma e$ be as in Lemma~\ref{lem:switch}. Thus $S'$ is a
uniformly random Seidel matrix of order $n$, and $u$ is uniform in $\{\pm1\}^n$ and independent of $S'$. Let $\mathcal S$ be the event
that the conclusion of Proposition~\ref{lem:support} holds for $S'$. Let $\mathcal C$ be the set of pairs $(p,\beta)$ such that
$(p,\deg\beta)\in\Pi$ and $\beta$ is a repeated factor of $S'$ at $p$. For $(p,\beta)\in\mathcal C$ fix a root $\lambda\in\F_q$ of
$\beta$, $q=p^{\deg\beta}$, and put $K_\lambda:=\ker(\bar S'-\lambda I)\subseteq\F_q^n$ and $g:=\dim_{\F_q}K_\lambda$. Note that $g\ge1$,
as $\beta$ divides $\chi_{S'}$ modulo $p$. Let $\mathcal C_1$ be the set of the pairs in $\mathcal C$ with $g=1$, let $\mathcal C_2$ be
the set of the pairs with $g=2$ such that $w^{\top}w\ne0$ for some $w\in K_\lambda$, and let $\mathcal C_3$ be the set of the remaining
pairs, those with $g\ge3$ or with $g=2$ and $w^{\top}w=0$ for all $w\in K_\lambda$.

By Lemma~\ref{lem:Epswitch}, if $E_p^{(d)}$ holds for some $(p,d)\in\Pi$, then there are $(p,\beta)\in\mathcal C$ and a totally
isotropic simple $\bar S'$-submodule $U$ of type $\beta$ with $U\perp\bar u$. By Lemma~\ref{lem:simple}(a),(b), the simple
submodules of type $\beta$ are the subspaces $U_w$, one for each $\F_q$-line $\F_qw\subseteq K_\lambda$, and $U_w$ is totally isotropic
if and only if $w^{\top}w=0$.

\step{1}{We bound the contribution of the pairs in $\mathcal C_1\cup\mathcal C_2$, which have at most two totally isotropic simple
submodules of type $\beta$.} For a pair in $\mathcal C_1$ there is exactly one simple submodule of type $\beta$. For a pair in
$\mathcal C_2$, choose a basis $w_1,w_2$ of $K_\lambda$. Then
$(\theta_1w_1+\theta_2w_2)^{\top}(\theta_1w_1+\theta_2w_2)=\alpha\theta_1^2+2\gamma\theta_1\theta_2+\alpha'\theta_2^2$ with
$\alpha:=w_1^{\top}w_1$, $\gamma:=w_1^{\top}w_2$ and $\alpha':=w_2^{\top}w_2$, and this binary form is not the zero polynomial, since it
takes a nonzero value. A nonzero binary quadratic form vanishes on at most two points of the projective line, as it is a product of two
linear forms over an algebraic closure. So $K_\lambda$ contains at most two lines $\F_qw$ with $w^{\top}w=0$. In both cases at most two
simple submodules of type $\beta$ are totally isotropic. Since $u$ is uniform and independent of $S'$, each of them is orthogonal to
$\bar u$ with conditional probability at most $a_\beta(S')$ given $S'$. Hence, by a union bound given $S'$, the probability that
$S'\in\mathcal G_n$ and that some $(p,\beta)\in\mathcal C_1\cup\mathcal C_2$ admits a totally isotropic simple submodule $U$ of type
$\beta$ with $U\perp\bar u$ is at most
\[
\E\Bigl(\sum_{(p,\beta)\in\mathcal C}2a_\beta(S');\ S'\in\mathcal G_n\Bigr)=2\,\E\big(N_\Pi(S_n);\ S_n\in\mathcal G_n\big).
\]

\step{2}{We bound the contribution of the pairs in $\mathcal C_3$ by deleting a coordinate and using exchangeability.} Let $\mathcal F$
be the event that $\mathcal C_3\ne\emptyset$. This event depends only on $S'$, and we do not use $u$ in this step. The role of the sign
vector is played by a column of $S'$, which we delete together with the corresponding row. For $i\in[n]$ let $T_i$ be the matrix obtained
from $S'$ by deleting row and column $i$, a Seidel matrix of order $n-1$, and let $s_i\in\{\pm1\}^{n-1}$ be the $i$th column of $S'$ with
its $i$th entry deleted, so that $(T_i,s_i)$ is the pair $(T^{(i)},t_i)$ of Lemma~\ref{lem:exch} for $T=S'$. Consider the event
\begin{quote}
$\mathcal F_i$: for some $(p,\beta)$ with $(p,\deg\beta)\in\Pi$, $\beta$ is a repeated factor of $T_i$ at $p$ and
${\bar s_i\perp\ker\beta(\bar T_i)}$.
\end{quote}
On $\mathcal F\cap\mathcal S$ take $(p,\beta)\in\mathcal C_3$, with $\lambda$, $K_\lambda$ and $g$ as above. On $\mathcal S$ every
nonzero $w\in K_\lambda$ has more than $n/20$ nonzero coordinates, so the set $\mathcal I$ of indices $i$ such that $w_i\ne0$ for some
$w\in K_\lambda$ has more than $n/20$ elements. For $i\in\mathcal I$, Lemma~\ref{lem:delete}, applied over $\F_q$ to the symmetric matrix
$B=\bar S'-\lambda I$ (so that $B^{(i)}=\bar T_i-\lambda I$ and $b_i=\bar s_i$), shows that $\dim\ker(\bar T_i-\lambda I)=g-1$, that
$\ker(\bar T_i-\lambda I)\subseteq\bar s_i^{\perp}$, and that inserting a zero in position $i$ maps $\ker(\bar T_i-\lambda I)$ into
$K_\lambda$. Hence $\beta$ is a repeated factor of $T_i$ at $p$, by Lemma~\ref{lem:simple}(d) for $\bar T_i$. Indeed, if $g\ge3$, then
$\dim\ker(\bar T_i-\lambda I)=g-1\ge2$, and if $g=2$, then $\ker(\bar T_i-\lambda I)$ is a line $\F_qw'$ with
$w'^{\top}w'=\hat w^{\top}\hat w=0$, where $\hat w\in K_\lambda$ is obtained from $w'$ by inserting a zero. Moreover, by
Lemma~\ref{lem:simple}(a) every vector of $\ker\beta(\bar T_i)$ has the form $\Tr(v)$ with $v\in\ker(\bar T_i-\lambda I)$, and
$\bar s_i^{\top}\Tr(v)=\Tr(\bar s_i^{\top}v)=0$, as $\bar s_i$ has entries in $\F_p$. So $\bar s_i\perp\ker\beta(\bar T_i)$. Thus, on
$\mathcal F\cap\mathcal S$, the event $\mathcal F_i$ holds for more than $n/20$ indices $i$. As $\mathcal F_i$ is the event that
$(T_i,s_i)$ lies in a fixed set of pairs, Lemma~\ref{lem:exch} (for $T=S'$ and $\mathcal E=\mathcal F\cap\mathcal S$) gives
\[
\Pr(\mathcal F\cap\mathcal S)\le20\Pr(\mathcal F_n).
\]
By the same lemma, $T_n$ is a uniformly random Seidel matrix of order $n-1$, and $s_n$ is uniform in $\{\pm1\}^{n-1}$ and independent of
$T_n$. If $\beta$ is a repeated factor of $T_n$ at $p$, then $\bar T_n$ has an eigenvector for a root of $\beta$, so $\ker\beta(\bar T_n)$
contains a simple submodule $U$ of type $\beta$ (Lemma~\ref{lem:simple}(a)). Then ${\bar s_n\perp\ker\beta(\bar T_n)}$ implies
$\bar s_n\perp U$, which has conditional probability at most $a_\beta(T_n)$ given $T_n$. Summing over the pairs, the conditional
probability of $\mathcal F_n$ given $T_n$ is at most $N_\Pi(T_n)$. As $T_n$ has the law of $S_{n-1}$, we get
\[
\Pr(\mathcal F_n)\le\Pr(S_{n-1}\notin\mathcal G_{n-1})+\E\big(N_\Pi(S_{n-1});\ S_{n-1}\in\mathcal G_{n-1}\big).
\]

\step{3}{We combine the bounds of Steps~1 and~2 and conclude the proof.} The event that $E_p^{(d)}$ holds for some $(p,d)\in\Pi$ is
contained in the union of the event $S'\notin\mathcal G_n$, the event bounded in Step~1, the event $\mathcal F\cap\mathcal S$, and
$\mathcal S^c$. Since $S'$ has the law of $S_n$, adding the four bounds gives
\begin{multline*}
\Pr\bigl(E_p^{(d)}\text{ holds for some }(p,d)\in\Pi\bigr)\ \le\ \Pr(S_n\notin\mathcal G_n)+2\,\E\big(N_\Pi(S_n);\ S_n\in\mathcal G_n\big)\\
+20\Bigl(\Pr(S_{n-1}\notin\mathcal G_{n-1})+\E\big(N_\Pi(S_{n-1});\ S_{n-1}\in\mathcal G_{n-1}\big)\Bigr)+\Pr(\mathcal S^c).
\end{multline*}
All terms but the last tend to $0$ by hypothesis, and the last is $O(2^{-n/8})$ by Proposition~\ref{lem:support}.
\end{proof}

\subsection{The discriminant}
For Theorem~\ref{thm:largedegrees}, and again in Sections~\ref{sec:deg3} and~\ref{sec:huge}, we need to count repeated factors. The
discriminant of the characteristic polynomial bounds them at all primes at once, as long as it does not vanish.

\begin{lemma}\label{lem:budget}
Let $T$ be a symmetric integer matrix of order $m$ with $\Delta_T\ne0$.
\begin{enumerate}[label=(\alph*)]
\item For every prime $p$, $\sum_\beta\deg\beta\le v_p(\Delta_T)$, the sum running over the repeated factors $\beta$ of $T$ at~$p$.
\item If all eigenvalues of $T$ lie in $[-m,m]$ (for instance, if $T$ is a Seidel matrix), then ${|\Delta_T|\le(2m)^{m(m-1)}}$. Hence, for
every real $P\ge2$, the number of pairs $(p,\beta)$ with $p>P$ prime and $\beta$ a repeated factor of $T$ at $p$ is at most
$m^2\log(2m)/\log P$.
\end{enumerate}
\end{lemma}

\begin{proof}
(a) Let $\lambda_1,\dots,\lambda_m$ be the eigenvalues of $T$, and put $\chi:=\chi_T$ and $N:=\chi'(T)\in\Z^{m\times m}$. Since
$\chi'(\lambda_i)=\prod_{j\ne i}(\lambda_i-\lambda_j)$, we have
$\det N=\prod_i\chi'(\lambda_i)=\pm\prod_{i<j}(\lambda_i-\lambda_j)^2=\pm\Delta_T$.
By the Smith normal form of $N$, $p^r$ divides $\det N$, where $r:=\dim\ker\bar N$. If $\beta$ is a repeated factor of $T$ at $p$, then
$\beta\mid\bar\chi'$, as $\beta^2\mid\bar\chi$, so $\ker\beta(\bar T)\subseteq\ker\bar N$. Moreover $\ker\beta(\bar T)\ne0$, as the roots
of $\beta$ are eigenvalues of $\bar T$, and $\ker\beta(\bar T)$ is a vector space over the field $\F_p[x]/(\beta)$, with $x$ acting as
$\bar T$, so its dimension is at least $\deg\beta$. Distinct monic irreducible polynomials are coprime, so the sum of the subspaces
$\ker\beta(\bar T)$ over the repeated factors $\beta$ of $T$ at $p$ is direct. Hence $\sum_\beta\deg\beta\le r\le v_p(\Delta_T)$.

(b) The eigenvalues $\lambda_1,\dots,\lambda_m$ of $T$ are real, and $\Delta_T=\prod_{i<j}(\lambda_i-\lambda_j)^2$ with
$|\lambda_i-\lambda_j|\le2m$, so $|\Delta_T|\le(2m)^{m(m-1)}$. (The eigenvalues of a Seidel matrix have absolute value at most its largest
absolute row sum, $m-1$.) By (a), the number of pairs in question is at most
$\sum_{p>P}v_p(\Delta_T)\le\sum_{p>P}v_p(\Delta_T)\frac{\log p}{\log P}\le\frac{\log|\Delta_T|}{\log P}\le\frac{m(m-1)\log(2m)}{\log P}$.
\end{proof}

Lemma~\ref{lem:budget}(a) also shows that an odd prime can divide the level of a matrix in $\QG$ only if it divides $\Delta_A$. Indeed, if
an odd prime $p$ divides the level of some $Q\in\QG$, then by Lemmas~\ref{lem:matesfail} and~\ref{lem:simple}(d) the adjacency matrix $A$
has a repeated factor at $p$, and so $p$ divides $\Delta_A$ by part~(a) (trivially if $\Delta_A=0$).
Criteria for the generalized spectral characterization in terms of $\Delta_A$ were given by Wang--Yu (2016)~\cite{WY}.

By the simple spectrum theorem of Tao--Vu (2017)~\cite{TV17}, the discriminant of a random Seidel matrix vanishes only with probability
$o(1)$.

\begin{theorem}[Tao--Vu 2017~{\cite[Theorem~1.3]{TV17}}]\label{lem:simplespectrum}
$\Pr(\Delta_{S_m}=0)\to0$ as $m\to\infty$.
\end{theorem}

\begin{proof}
The discriminant $\Delta_{S_m}=\prod_{i<j}(\lambda_i-\lambda_j)^2$, where $\lambda_1,\dots,\lambda_m$ are the eigenvalues of $S_m$,
vanishes if and only if $S_m$ has a repeated eigenvalue. By a theorem of Tao--Vu~\cite[Theorem~1.3]{TV17}, a real symmetric random
matrix whose entries above the diagonal are independent copies of a fixed non-degenerate random variable, and whose diagonal entries
are independent of them (they may be deterministic, and here they are $0$), has simple spectrum with probability $1-O(m^{-C})$ for every
fixed $C>0$.
\end{proof}

\subsection{Large degrees}
We now prove Theorem~\ref{thm:largedegrees}, which states that the probability that $E_p^{(d)}$ holds for some odd prime $p$ and some
$d\ge3\log_2n$ tends to $0$. For eigenvalues of large degree no information on the eigenvectors is needed. By the discriminant budget, a
Seidel matrix of order $m\le n$ with nonzero discriminant has at most $n^2\log_2(2n)$ repeated factors at all odd primes together, and by
the bound~\eqref{eq:abetaodl}, which comes from Odlyzko's lemma, each of them of degree at least $3\log_2n$ has weight $a_\beta\le n^{-3}$.

\begin{proof}[Proof of Theorem~\ref{thm:largedegrees}]
Let $n\ge2$, let $\Pi_n$ be the set of pairs $(p,d)$ with $p$ an odd prime and ${d\ge3\log_2n}$, and for $m\in\{n-1,n\}$ let $\mathcal G_m$ be
the set of Seidel matrices $T$ of order $m$ with $\Delta_T\ne0$. By Theorem~\ref{lem:simplespectrum}, $\Pr(S_m\notin\mathcal G_m)\to0$. Let
$T\in\mathcal G_m$. By Lemma~\ref{lem:budget}(b) with $P=2$, there are at most $m^2\log(2m)/\log2\le n^2\log_2(2n)$ pairs $(p,\beta)$ such
that $p$ is an odd prime and $\beta$ is a repeated factor of $T$ at $p$. By~\eqref{eq:abetaodl}, each of them that is counted in
$N_{\Pi_n}(T)$ contributes $a_\beta(T)\le2^{-\deg\beta}\le2^{-3\log_2n}=n^{-3}$. Hence $N_{\Pi_n}(T)\le\log_2(2n)/n$, so
$\E(N_{\Pi_n}(S_m);\ S_m\in\mathcal G_m)\to0$ for $m\in\{n-1,n\}$, and Proposition~\ref{prop:reduction} proves the theorem.
\end{proof}

\section{Lattices, local components and levels}\label{sec:local}

This section begins the proof of the second component, Theorem~\ref{thm:smooth}, which occupies Sections~\ref{sec:local}--\ref{sec:smooth}.
At the small primes $p\le n^{1/2-\varepsilon}$, local failures are not expected to be rare, so instead of excluding them we show directly
that, with probability $1-o(1)$, no $Q\in\QG$ other than a permutation matrix has a level whose prime factors are all small. For this we
must understand how the arithmetic of the walk matrix $W$ at a prime $p$ controls the $p$-part of a level, and this section, which is
entirely deterministic, provides that understanding. Its main result, Proposition~\ref{lem:levellocal}, shows that if $p$ divides $\ell(Q)$,
then some local component above $p$ of a finite ring attached to $W$ is not a chain ring. Moreover, $p^{v_p(\ell(Q))}$ is bounded by the
exponents of the corresponding components of $\coker W$, and for odd $p$ the local failure $E_p^{(d)}$ of Definition~\ref{def:Ep} occurs in
the corresponding degree $d$. In Section~\ref{sec:smooth} this proposition is combined with probabilistic estimates
(Sections~\ref{sec:fourier}--\ref{sec:two} and Theorem~\ref{thm:largedegrees}) showing that such components are unlikely outside a fixed
finite set of pairs (prime, degree), and that the exponents of the remaining ones are bounded in probability.

Throughout this section $G$ is a fixed controllable graph on $[n]$. We study the matrices $Q\in\QG$ through the lattices
$Q\Z^n\subseteq\Q^n$. As in Section~\ref{subsec:plan-small}, the lattice $C=W\Z^n$ is $A$-stable. The finite quotient $D=C^*/C$ of its dual
lattice, the \emph{discriminant module} of $C$, is isomorphic, as a module over $\Z[x]$ with $x$ acting as $A$, to the finite ring
$\mathcal T=\Z[x]/(\varphi,\psi)$ (Proposition~\ref{prop:D}), where $\varphi=\chi_A$ is the characteristic polynomial of $A$ and
$\psi(x)=e^{\top}\operatorname{adj}(xI-A)e$. Every $Q\in\QG$ gives a unimodular lattice $Q\Z^n$ between $C$ and $C^*$, and the identity
matrix gives $\Z^n$, whose image in $D$ is $\coker W$ (Lemma~\ref{lem:ideals}). The ring $\mathcal T$ is the product of its local components
$\mathcal T_\m$, where $\m=(p,\tilde\beta)$ runs over the maximal ideals of $\Z[x]$ that contain $(\varphi,\psi)$
(Section~\ref{subsec:loc-lagr}), and Hensel's lemma splits $\Z_p^n$ accordingly into orthogonal summands (Lemma~\ref{lem:loccomp}). In each
summand, $Q$ and the identity matrix give two unimodular lattices, and the $p$-part of $\ell(Q)$ is governed by the summands in which they
differ (Proposition~\ref{lem:levellocal}). Corollary~\ref{cor:chain} gives a simple sufficient condition for $\mathcal T_\m$ to be a chain
ring, which Section~\ref{sec:two} uses at the prime~$2$. Finally, the summands give a Krylov presentation of the components of $\coker W$
(Lemma~\ref{lem:hensel}), which Section~\ref{sec:small} uses to bound their exponents, and they turn a matrix $Q\in\QG$ of small level into
a short integer vector killed modulo $p$ by a polynomial in $A$ of small degree (Lemma~\ref{lem:shortcol}).

The set $\QG$ and the fact that $\ell(Q)$ divides the last invariant factor $d_n$ of $W$ are due to Wang--Xu (2006)~\cite{WX06a} (see also
Wang (2013)~\cite[Theorem~3]{W13}). Arithmetic criteria excluding a prime from all levels were given by Wang (2013)~\cite{W13} at odd primes
and by Wang (2017)~\cite{W17} at $p=2$ (see also the Smith normal form approach of Qiu--Wang--Wang--Zhang (2023)~\cite{QWWZ23}). Further
criteria of this kind were given by Wang--Yu (2016)~\cite{WY} in terms of the discriminant $\Delta_A$, by Wang--Wang--Zhu
(2023)~\cite{WWZ23} at odd primes $p$ with $p\,\|\,d_n$, and by Guo--Wang--Wang (2025)~\cite{GWW25} in terms of the primary decomposition of
$A$ modulo $p$. Bounds on $v_p(\ell(Q))$ itself, such as $v_p(\ell(Q))\le\frac12v_p(\det W)$ for odd $p$ with $\rank_pW=n-1$, were proved by
Wang--Luo--Wang (2026)~\cite{WLW26}, and Proposition~\ref{lem:levellocal}(i) is a local bound of this kind, which contains the divisibility
$\ell(Q)\mid d_n$ of Wang--Xu. As far as we know, the description below through the discriminant module and its local components is new. It
gives a common local framework for criteria of this kind, as we illustrate after Corollary~\ref{cor:chain}.

\subsection{The discriminant module}\label{subsec:loc-disc}

We first recall some facts on lattices, over $\Z$ and over the ring $\Z_p$ of $p$-adic integers, and then identify the discriminant module
of $C$. Let $\mathcal O$ be $\Z$ or $\Z_p$, with field of fractions $F=\Q$ or $\Q_p$, and let $V\subseteq F^n$ be a subspace on which the
bilinear form $y^{\top}y'$ is nondegenerate. A \emph{lattice} in $V$ is an $\mathcal O$-submodule $\Lambda\subseteq V$ generated by a basis
of $V$. Its dual $\Lambda^*:=\{y\in V:\ y^{\top}z\in\mathcal O\text{ for all }z\in\Lambda\}$ is again a lattice in $V$,
$\Lambda^{**}=\Lambda$, and duality reverses inclusions. If $\Lambda_1\subseteq\Lambda_2$ are lattices in $V$, then
\begin{equation}\label{eq:loc-index}
|\Lambda_1^*/\Lambda_2^*|=|\Lambda_2/\Lambda_1| ,
\end{equation}
as one sees from a basis $b_1,\dots,b_k$ of $\Lambda_2$ such that $\delta_1b_1,\dots,\delta_kb_k$ is a basis of $\Lambda_1$ (elementary
divisors), together with the dual basis. A lattice $\Lambda$ is \emph{integral} if $\Lambda\subseteq\Lambda^*$, \emph{unimodular} if
$\Lambda=\Lambda^*$, and \emph{$A$-stable} if $A\Lambda\subseteq\Lambda$. In this section, unless stated otherwise, lattices are lattices in
$\Q^n$.

Write $A=\sum_{i=1}^n\lambda_iv_iv_i^{\top}$ with orthonormal $v_i\in\R^n$, and put $\alpha_i:=e^{\top}v_i$. As $\det W\ne0$, the vectors
$A^ke$ $(k\ge0)$ span $\Q^n$, so the minimal polynomial of $A$ has degree $n$. Hence the $\lambda_i$ are pairwise distinct, and $\alpha_i\ne0$
for all $i$ (otherwise $v_i\perp A^ke$ for all $k$). Since $\psi(x)/\varphi(x)=e^{\top}(xI-A)^{-1}e=\sum_i\alpha_i^2/(x-\lambda_i)$, we get
\begin{equation}\label{eq:loc-psi}
\psi(\lambda_i)=\alpha_i^2\,\varphi'(\lambda_i)\neq0\qquad(1\le i\le n).
\end{equation}
Let $R:=\Z[x]/(\varphi)$ and $R_\Q:=R\otimes\Q=\Q[x]/(\varphi)$. As $\varphi$ is separable, $R_\Q$ is a product of number fields, and its
trace over $\Q$ is given by $\Tr_{R_\Q/\Q}(f)=\sum_if(\lambda_i)$. By~\eqref{eq:loc-psi},
$\varphi'$ and $\psi$ are units of $R_\Q$. The map
\[
\iota:R_\Q\to\Q^n,\qquad f\mapsto f(A)e,
\]
is a $\Q$-linear isomorphism (as $e,Ae,\dots,A^{n-1}e$ is a basis of $\Q^n$), $\iota(xf)=A\,\iota(f)$, and $\iota(R)=\Z[A]e=W\Z^n=C$ by the
Cayley--Hamilton theorem. Note that $C\subseteq\Z^n\subseteq C^*$. Since $C$ is $A$-stable and $A$ is symmetric, $C^*$ is $A$-stable, so
that $A$ acts on $D=C^*/C$. Under $\iota$, the standard inner product on $\Q^n$ becomes a twisted trace form on $R_\Q$.

\begin{lemma}\label{lem:form}
For $f,g\in R_\Q$ we have $\iota(f)^{\top}\iota(g)=\Tr_{R_\Q/\Q}\bigl(fg\,\psi/\varphi'\bigr)$.
\end{lemma}

\begin{proof}
Since $e=\sum_i\alpha_iv_i$, we have $\iota(f)=\sum_i\alpha_if(\lambda_i)v_i$, so by~\eqref{eq:loc-psi}
\[
\iota(f)^{\top}\iota(g)=\sum_i\alpha_i^2f(\lambda_i)g(\lambda_i)=\sum_if(\lambda_i)g(\lambda_i)\,
\frac{\psi(\lambda_i)}{\varphi'(\lambda_i)}=\Tr_{R_\Q/\Q}\bigl(fg\,\psi/\varphi'\bigr).\qedhere
\]
\end{proof}

With this formula we can identify the dual lattice $C^*$ and the discriminant module $D$.

\begin{proposition}\label{prop:D}
We have $C^*=\iota(\psi^{-1}R)$, and $f\mapsto\iota(f/\psi)+C$ induces an isomorphism of $\Z[x]$-modules
\[
\mathcal T=\Z[x]/(\varphi,\psi)\ \xrightarrow{\ \sim\ }\ D ,
\]
where $x$ acts on $D$ as $A$. Moreover $|\mathcal T|=|D|=(\det W)^2$.
\end{proposition}

\begin{proof}
By Euler's lemma (see Serre (1979)~\cite[Ch.~III, \S6]{Serre}, whose proof applies to every monic separable $\varphi$) the trace dual $\{g\in
R_\Q:\ \Tr_{R_\Q/\Q}(gR)\subseteq\Z\}$ of $R$ is $\varphi'^{-1}R$. For $g\in R_\Q$, Lemma~\ref{lem:form} shows that $\iota(g)\in C^*$ if and
only if $\Tr_{R_\Q/\Q}(gf\psi/\varphi')\in\Z$ for all $f\in R$, that is, if and only if $g\psi/\varphi'\in\varphi'^{-1}R$, or equivalently
$g\psi\in R$. Hence $C^*=\iota(\psi^{-1}R)$, and $f\mapsto\iota(f/\psi)+C$ maps $R$ onto $D$ with kernel $\psi R$. As $R/\psi R=\mathcal T$
and $\iota$ commutes with multiplication by $x$, this is the asserted isomorphism. Finally $y\mapsto W^{\top}y$ maps
$C^*=(W^{\top})^{-1}\Z^n$ onto $\Z^n$ and $C$ onto $W^{\top}W\Z^n$, so $|D|=\det(W^{\top}W)=(\det W)^2$.
\end{proof}

The last step of the proof shows that $y\mapsto W^{\top}y$ identifies $D$ with the cokernel of $W^{\top}W=(e^{\top}A^{i+j}e)_{0\le i,j<n}$,
the Hankel matrix of walk counts of $G$. So Proposition~\ref{prop:D} says that this cokernel is isomorphic, as an abelian group, to the finite ring
$\Z[x]/(\varphi,\psi)$. Multiplication by $\psi$ on the free $\Z$-module $R=\Z[x]/(\varphi)$ has determinant
$\prod_i\psi(\lambda_i)=\Res(\varphi,\psi)$, so comparing orders gives the identity $(\det W)^2=\lvert\Res(\varphi,\psi)\rvert$.

Through Proposition~\ref{prop:D} we regard $D$ as a $\mathcal T$-module, which is free of rank one. Its $A$-stable subgroups are its
$\Z[x]$-submodules, and they correspond to the ideals of $\mathcal T$. The matrices $Q\in\QG$ give lattices between $C$ and $C^*$, as
follows.

\begin{lemma}\label{lem:ideals}
Let $Q\in\QG$. The lattice $\Lambda_Q:=Q\Z^n$ is $A$-stable and unimodular, and $C\subseteq\Lambda_Q\subseteq C^*$.
\end{lemma}

\begin{proof}
We have $e=Qe\in\Lambda_Q$. The matrix $B:=Q^{\top}AQ$ is integral and $AQ=QB$, so $A\Lambda_Q=QB\Z^n\subseteq\Lambda_Q$, and hence
$C=\Z[A]e\subseteq\Lambda_Q$. As $Q$ is orthogonal, $\Lambda_Q^*=(Q^{\top})^{-1}\Z^n=\Lambda_Q$, and dualizing $C\subseteq\Lambda_Q$ gives
$\Lambda_Q\subseteq C^*$.
\end{proof}

Lemma~\ref{lem:ideals} has a sharper form in terms of a pairing on $D$, which we record for completeness.

\begin{remark}\label{rem:loc-lagr}
The discriminant pairing $b(y+C,y'+C):=y^{\top}y'\bmod\Z$ is a well-defined nondegenerate symmetric pairing on $D$ with values in $\Q/\Z$,
because $C^{**}=C$, and by Proposition~\ref{prop:D} and Lemma~\ref{lem:form} it corresponds on $\mathcal T$ to
$b(f,g)=\Tr_{R_\Q/\Q}\bigl(fg/(\psi\varphi')\bigr)\bmod\Z$. As $b(f,g)$ only depends on $fg$, nondegeneracy shows that the orthogonal of an
ideal $\mathfrak a$ of $\mathcal T$ is its annihilator. For a lattice $C\subseteq\Lambda\subseteq C^*$, the orthogonal of $\Lambda/C$ for $b$ is
$\Lambda^*/C$. So Lemma~\ref{lem:ideals} says that $\Lambda_Q/C$ is a \emph{Lagrangian ideal} of $\mathcal T\cong D$, that is, an ideal
equal to its orthogonal, or equivalently to its annihilator, and the identity matrix gives the Lagrangian ideal $\Z^n/C=\coker W$. We shall
not use the pairing, only the consequence $|\Lambda/C|^2=|D|$ of~\eqref{eq:loc-index} for unimodular $\Lambda$, in the local form of
Lemma~\ref{lem:loccomp}(b) and of Step~1 of the proof of Proposition~\ref{lem:levellocal}.
\end{remark}

\subsection{Local components and Hensel summands}\label{subsec:loc-lagr}

Next we decompose these objects into local components. The maximal ideals of $\Z[x]$ are the ideals $\m=(p,\tilde\beta)$, where $p$ is a
prime, $\beta\in\F_p[x]$ is monic irreducible and $\tilde\beta\in\Z[x]$ is a monic lift of $\beta$. The residue field $\Z[x]/\m$ has
$p^{\deg\beta}$ elements. Every finite $\Z[x]$-module $X$ is the direct sum $X=\bigoplus_\m X_\m$ of its \emph{$\m$-parts}
$X_\m:=\{u\in X:\ \m^ku=0\text{ for some }k\}$, and $X_{(p)}:=\bigoplus_{\m\ni p}X_\m$ is the $p$-primary part of $X$, which we identify with
$X\otimes\Z_p$. Taking $\m$-parts commutes with submodules and quotients, and an element of $\Z[x]\setminus\m$ acts invertibly on $X_\m$. If
$pX=0$, then $X$ is an $\F_p[x]$-module and $X_\m$ is its $\beta$-primary part. We also write $X_{(p,\beta)}:=X_\m$ for
$\m=(p,\tilde\beta)$.

The finite ring $\mathcal T$ is the product of the finite local rings $\mathcal T_\m$ over the maximal ideals $\m\supseteq(\varphi,\psi)$,
that is, those with $\beta\mid\gcd(\bar\varphi,\bar\psi)$ in $\F_p[x]$, while $\mathcal T_\m=0$ for all other $\m$. We call the
$\mathcal T_\m$ the \emph{local components} of $\mathcal T$, and $\deg\beta$ is the \emph{degree} of $\mathcal T_\m$. By
Proposition~\ref{prop:D}, $D_\m\cong\mathcal T_\m$ as $\mathcal T_\m$-modules. A finite local ring is a \emph{chain ring} if its ideals are
totally ordered by inclusion. In this section and in Sections~\ref{sec:small} and~\ref{sec:smooth} we write
$\mathfrak M:=\Z^n/C=\coker W\subseteq D$ and $\exp(X)$ for the exponent of a finite abelian group $X$, so that $d_n=\exp(\mathfrak M)$.

Now fix a prime $p$ and write $\bar\varphi=\prod_\beta\beta^{a_\beta}$ in $\F_p[x]$. By Hensel's lemma (see Neukirch (1999)~\cite[Ch.~II,
(4.6)]{Neu99}) there is a factorization $\varphi=\prod_\beta\Phi_\beta$ in $\Z_p[x]$ into monic polynomials with
$\bar\Phi_\beta=\beta^{a_\beta}$. It is unique, because two monic polynomials in $\Z_p[x]$ with coprime reductions generate the unit ideal
(see the beginning of the proof of Lemma~\ref{lem:loccomp}), so that $\Phi_\beta$ divides, and hence equals, every monic factor of
$\varphi$ that reduces to $\beta^{a_\beta}$. We call $\Phi_\beta$ the \emph{Hensel factor} of $\varphi$ at $\beta$. In the rest of this
section and in Section~\ref{sec:small}, duals of lattices in a subspace $V\subseteq\Q_p^n$ are taken inside $V$. The local components
above $p$ are described by the following orthogonal decomposition of $\Z_p^n$.

\begin{lemma}\label{lem:loccomp}
Let $p$ be a prime, and for each irreducible factor $\beta$ of $\bar\varphi$ put $N_\beta:=\ker\Phi_\beta(A)\subseteq\Z_p^n$ and
$V_\beta:=N_\beta\otimes\Q_p\subseteq\Q_p^n$.
\begin{enumerate}[label=(\alph*)]
\item $\Z_p^n=\bigoplus_\beta N_\beta$ is an orthogonal decomposition into $A$-stable unimodular sublattices, $N_\beta$ has rank
$a_\beta\deg\beta$, and the reduction of $N_\beta$ modulo $p$ is $\bar N_\beta=\ker\beta(\bar A)^{a_\beta}$. The orthogonal projection
$\pi_\beta$ of $\Q_p^n$ onto $V_\beta$ is a polynomial in $A$ with coefficients in $\Z_p$.
\item Put $e_\beta:=\pi_\beta e$ and $C_\beta:=\Z_p[A]e_\beta$, and let $C_\beta^*$ be the dual of $C_\beta$ in $V_\beta$. Then
$C_\beta\subseteq N_\beta\subseteq C_\beta^*$, and for $\m:=(p,\tilde\beta)$
\[
\mathfrak M_\m=N_\beta/C_\beta,\qquad D_\m=C_\beta^*/C_\beta\cong\mathcal T_\m,\qquad |D_\m|=|\mathfrak M_\m|^2 .
\]
\end{enumerate}
\end{lemma}

\begin{proof}
(a) For $\beta\ne\beta'$ the finitely generated $\Z_p$-module $\Z_p[x]/(\Phi_\beta,\Phi_{\beta'})$ (finitely generated as $\Phi_\beta$ is
monic) vanishes modulo $p$, because $\beta^{a_\beta}$ and $\beta'^{a_{\beta'}}$ are coprime. By Nakayama's lemma it vanishes, so the
$\Phi_\beta$ are pairwise comaximal. By the Chinese remainder theorem there are $f_\beta\in\Z_p[x]$ with $f_\beta\equiv1\pmod{\Phi_\beta}$
and $f_\beta\equiv0\pmod{\Phi_{\beta'}}$ for $\beta'\ne\beta$. Then $\sum_\beta f_\beta\equiv1$ and $f_\beta f_{\beta'}\equiv0$ for
$\beta\ne\beta'$, modulo every $\Phi_{\beta''}$ and hence modulo $\varphi$. As $\varphi(A)=0$, the matrices $\pi_\beta:=f_\beta(A)$ are
pairwise orthogonal idempotents with $\sum_\beta\pi_\beta=I$, and they are symmetric, as polynomials in $A$. So they are the orthogonal
projections onto the pairwise orthogonal subspaces $\pi_\beta\Q_p^n$, whose direct sum is $\Q_p^n$, so that the form is
nondegenerate on each of them. Since $\Phi_\beta f_\beta\equiv0\pmod\varphi$ and
$\Phi_\beta\mid f_{\beta'}$ for $\beta'\ne\beta$, a vector $y\in\Z_p^n$ satisfies $\Phi_\beta(A)y=0$ if and only if $y=\pi_\beta y$. Hence
$N_\beta=\pi_\beta\Z_p^n$, $V_\beta=\pi_\beta\Q_p^n$, and $\Z_p^n=\bigoplus_\beta N_\beta$ orthogonally. Each $N_\beta$ is $A$-stable, and
it is unimodular as an orthogonal direct summand of the unimodular lattice $\Z_p^n$. Over an algebraic closure of $\Q_p$, $A$ is
diagonalizable with the $n$ distinct roots of $\varphi$ as eigenvalues, and $\ker\Phi_\beta(A)$ is spanned by the eigenvectors for the
roots of $\Phi_\beta$, so $N_\beta$ has rank $\deg\Phi_\beta=a_\beta\deg\beta$. Reducing modulo $p$, we get $\F_p^n=\bigoplus_\beta\bar N_\beta$
with $\bar N_\beta\subseteq\ker\beta(\bar A)^{a_\beta}$. As $\F_p^n=\bigoplus_\beta\ker\beta(\bar A)^{a_\beta}$, all these inclusions are
equalities.

(b) As $\pi_\beta\in\Z_p[A]$, we have $C\otimes\Z_p=\Z_p[A]e=\bigoplus_\beta C_\beta$ with $C_\beta\subseteq N_\beta$, and $C_\beta$ is a
lattice in $V_\beta$ as $\Z_p[A]e$ spans $\Q_p^n$. Since $N_\beta$ is unimodular, dualizing $C_\beta\subseteq N_\beta$ gives
$N_\beta\subseteq C_\beta^*$. The dual of $C\otimes\Z_p$ in $\Q_p^n$ is $C^*\otimes\Z_p$, since the dual basis of a basis of $C$ is a
basis of both, and it equals $\bigoplus_\beta C_\beta^*$, because the decomposition $\Q_p^n=\bigoplus_\beta V_\beta$ is orthogonal. Hence
\[
\mathfrak M_{(p)}=\Z_p^n/(C\otimes\Z_p)=\bigoplus_\beta N_\beta/C_\beta,\qquad D_{(p)}=\bigoplus_\beta C_\beta^*/C_\beta .
\]
The summands at $\beta$ are finite $p$-groups killed by $\Phi_\beta(A)$, which vanishes on $V_\beta$. As
$\Phi_\beta\equiv\tilde\beta^{a_\beta}\pmod{p\Z_p[x]}$, the matrix $\tilde\beta(A)^{a_\beta}$ maps each of these summands into $p$ times
itself, so they are killed by a power of $\tilde\beta(A)$ and by a power of $p$, hence by a power of $\m$. Thus the summands at $\beta$ lie
in the $\m$-parts, and since distinct $\beta$ give distinct $\m$, they are the $\m$-parts $\mathfrak M_\m$ and $D_\m$. Finally
$D_\m\cong\mathcal T_\m$ by Proposition~\ref{prop:D}, and~\eqref{eq:loc-index} with $\Lambda_1=C_\beta$ and $\Lambda_2=N_\beta=N_\beta^*$
gives $|C_\beta^*/N_\beta|=|N_\beta/C_\beta|$, so that $|D_\m|=|\mathfrak M_\m|^2$.
\end{proof}

\subsection{Levels are local}\label{subsec:loc-levels}

Let $Q\in\QG$ and let $p$ be a prime. With the notation of Lemma~\ref{lem:loccomp}, put $\Lambda_p:=Q\Z_p^n$ and
$\Lambda_\beta:=\pi_\beta\Lambda_p$ for every irreducible factor $\beta$ of $\bar\varphi$, and let
\[
\mathcal F(Q,p):=\big\{\m=(p,\tilde\beta):\ \beta\mid\bar\varphi\text{ and }\Lambda_\beta\neq N_\beta\big\}
\]
be the set of maximal ideals above $p$ at which $Q$ and the identity matrix give different summands. The following proposition is the main result
of this section. It shows that the $p$-part of the level of $Q$ is controlled by the components in $\mathcal F(Q,p)$, that these components
are not chain rings, and that local failures occur at them.

\begin{proposition}\label{lem:levellocal}
Let $G$ be controllable, $Q\in\QG$ and $p$ a prime.
\begin{enumerate}[label=(\roman*)]
\item We have $p\mid\ell(Q)$ if and only if $\mathcal F(Q,p)\ne\emptyset$, and
\[
v_p(\ell(Q))\le\max_{\m\in\mathcal F(Q,p)}\log_p\exp(\mathfrak M_\m)\ \le\ v_p(d_n),
\]
the maximum over the empty set being $0$.
\item Let $\m=(p,\tilde\beta)\in\mathcal F(Q,p)$. Then $\Lambda_\beta/C_\beta$ and $\mathfrak M_\m=N_\beta/C_\beta$ are two distinct ideals
of $\mathcal T_\m\cong D_\m$ of order $|\mathcal T_\m|^{1/2}$. Consequently $\mathcal T_\m$ is not a chain ring, and $\mathfrak M_\m\ne0$.
\item Let $\m=(p,\tilde\beta)\in\mathcal F(Q,p)$. Then some simple $\bar A$-submodule $U\subseteq\F_p^n$ of type $\beta$ is
totally isotropic and orthogonal to $\bar e$. In particular, if $p$ is odd, then $E_p^{(\deg\beta)}$ holds.
\end{enumerate}
\end{proposition}

\begin{proof}
\step{1}{We decompose $\Lambda_p$ into unimodular summands $\Lambda_\beta$ between $C_\beta$ and $C_\beta^*$, and we compare them with the
summands $N_\beta$.} By Lemma~\ref{lem:ideals}, the lattice $\Lambda_p=\Z_p\Lambda_Q$ in $\Q_p^n$ is unimodular, contains $e$ and is
$A$-stable. As the $\pi_\beta$ are polynomials in $A$ with $\sum_\beta\pi_\beta=I$, we get $\Lambda_p=\bigoplus_\beta\Lambda_\beta$ with
$\Lambda_\beta=\Lambda_p\cap V_\beta$. Each $\Lambda_\beta$ is $A$-stable and contains $e_\beta$, hence $C_\beta$. Like $N_\beta$ in
Lemma~\ref{lem:loccomp}(a), it is unimodular in $V_\beta$, since it is an orthogonal direct summand of the unimodular lattice $\Lambda_p$.
Dualizing $C_\beta\subseteq\Lambda_\beta$ gives $\Lambda_\beta\subseteq C_\beta^*$, and~\eqref{eq:loc-index} gives
$|C_\beta^*/\Lambda_\beta|=|\Lambda_\beta/C_\beta|$, so that $|\Lambda_\beta/C_\beta|^2=|D_\m|$ for $\m=(p,\tilde\beta)$. The same holds for
$N_\beta$ by Lemma~\ref{lem:loccomp}(b). Finally, if $\Lambda_\beta\subseteq N_\beta$, then
$N_\beta=N_\beta^*\subseteq\Lambda_\beta^*=\Lambda_\beta$, so $\Lambda_\beta=N_\beta$.

\step{2}{We prove (i) by computing the $p$-part of $\ell(Q)$ summand by summand.} For an integer $t\ge1$, $tQ$ is integral if and only if
$tQ\Z^n\subseteq\Z^n$, so $\ell(Q)$ is the exponent of ${(Q\Z^n+\Z^n)/\Z^n}$. Hence $p^{v_p(\ell(Q))}$ is the exponent of its $p$-primary part
\[
(\Lambda_p+\Z_p^n)/\Z_p^n=\bigoplus_\beta(\Lambda_\beta+N_\beta)/N_\beta .
\]
By Step~1 the summand at $\beta$ is nonzero if and only if $\Lambda_\beta\ne N_\beta$, that is, $(p,\tilde\beta)\in\mathcal F(Q,p)$. For
$\m=(p,\tilde\beta)$ let $p^k:=\exp(\mathfrak M_\m)=\exp(N_\beta/C_\beta)$. Then $p^kN_\beta\subseteq C_\beta$, and dualizing gives
$C_\beta^*\subseteq p^{-k}N_\beta^*=p^{-k}N_\beta$. As $\Lambda_\beta\subseteq C_\beta^*$, the summand at $\beta$ is killed by $p^k$. This
proves (i), as $\exp(\mathfrak M_\m)\le\exp(\mathfrak M_{(p)})=p^{v_p(d_n)}$.

\step{3}{We prove (ii) by comparing the two ideals of $\mathcal T_\m$ given by $\Lambda_\beta$ and by $N_\beta$.} The $A$-stable subgroups
$\Lambda_\beta/C_\beta$ and $N_\beta/C_\beta=\mathfrak M_\m$ of $D_\m=C_\beta^*/C_\beta$ correspond to ideals of $\mathcal T_\m\cong D_\m$.
They have order $|D_\m|^{1/2}=|\mathcal T_\m|^{1/2}$ by Step~1, and they are distinct as $\m\in\mathcal F(Q,p)$. Two distinct ideals of the
same order are not comparable, so $\mathcal T_\m$ is not a chain ring. If $\mathfrak M_\m=0$, then $|\mathcal T_\m|=1$ and both ideals
vanish, which is impossible. So $\mathfrak M_\m\ne0$.

\step{4}{We prove (iii) by reducing a suitable multiple of $\Lambda_\beta$ modulo $p$.} By Steps~1 and~2,
$(\Lambda_\beta+N_\beta)/N_\beta$ is a nonzero finite $p$-group. Let $p^s$ be its exponent, so that $s\ge1$ and
$p^s\Lambda_\beta\subseteq N_\beta$, and let $U_0\subseteq\F_p^n$ be the reduction modulo $p$ of $p^s\Lambda_\beta$. Then $U_0\ne0$, since
otherwise $p^{s-1}\Lambda_\beta\subseteq\Z_p^n\cap V_\beta=N_\beta$, contrary to the choice of $s$. Moreover
$U_0\subseteq\bar N_\beta=\ker\beta(\bar A)^{a_\beta}$, so every simple $\bar A$-submodule of $U_0$ has type $\beta$. The subspace $U_0$
is $\bar A$-stable as $A\Lambda_\beta\subseteq\Lambda_\beta$. It is totally isotropic, since
$(p^sy)^{\top}(p^sy')=p^{2s}y^{\top}y'\in p\Z_p$ for $y,y'\in\Lambda_\beta$, as $\Lambda_\beta$ is integral. And $U_0\perp\bar e$, since
$e\in\Lambda_p$ and $\Lambda_p$ is integral, so that $(p^sy)^{\top}e=p^s\,y^{\top}e\in p\Z_p$ for $y\in\Lambda_\beta\subseteq\Lambda_p$.
Hence every simple $\bar A$-submodule $U\subseteq U_0$ has the asserted properties, and for odd $p$ the event $E_p^{(\deg\beta)}$ holds,
as $\dim U=\deg\beta$.
\end{proof}

Proposition~\ref{lem:levellocal}(iii) refines the key lemma of Guo--Wang--Wang~\cite[Lemma~11]{GWW25} (Lemma~\ref{lem:matesfail}), where a nonzero
$\bar A$-stable totally isotropic subspace orthogonal to $\bar e$ is obtained from the column space of $\ell(Q)Q$ modulo $p$. Localizing at
the components in $\mathcal F(Q,p)$ also determines the type $\beta$ of the simple submodules involved. By Proposition~\ref{lem:levellocal}(ii), a
local component which is a chain ring never contributes to a level. The next corollary gives a simple sufficient condition for this.

\begin{corollary}\label{cor:chain}
Let $\m=(p,\tilde\beta)\supseteq(\varphi,\psi)$ be a maximal ideal.
\begin{enumerate}[label=(\alph*)]
\item If the maximal ideal of $\mathcal T_\m$ is principal, then $\mathcal T_\m$ is a chain ring, and $\m\notin\mathcal F(Q,p)$ for every
$Q\in\QG$.
\item If some element of the ideal $(\varphi,\psi)$ of $\Z[x]$ does not lie in $\m^2$, that is, if $\varphi\notin\m^2$ or
$\psi\notin\m^2$, then the maximal ideal of $\mathcal T_\m$ is principal.
\end{enumerate}
\end{corollary}

\begin{proof}
Let $\mathfrak n:=\m\mathcal T_\m$ be the maximal ideal of $\mathcal T_\m$ and $k:=\mathcal T_\m/\mathfrak n$ its residue field.

(a) Write $\mathfrak n=t\mathcal T_\m$. As $\mathcal T_\m$ is finite, $\mathfrak n$ is nilpotent. Let $\mathfrak a\ne0$ be an ideal, and
let $j$ be the largest integer with $\mathfrak a\subseteq t^j\mathcal T_\m$. Choose $u\in\mathfrak a$ with $u\notin t^{j+1}\mathcal T_\m$
and write $u=t^jv$. Then $v\notin\mathfrak n$, so $v$ is a unit, $t^j\in\mathfrak a$ and $\mathfrak a=t^j\mathcal T_\m$. So the ideals of
$\mathcal T_\m$ are the $t^j\mathcal T_\m$, which are totally ordered by inclusion. The last assertion follows from
Proposition~\ref{lem:levellocal}(ii), by which the components at the maximal ideals in $\mathcal F(Q,p)$ are not chain rings.

(b) Since $\mathcal T_\m=\Z[x]_\m/(\varphi,\psi)$, we have $\mathfrak n/\mathfrak n^2\cong\m/(\m^2+(\varphi,\psi))$, the localization being
harmless because this module is killed by $\m$. The $k$-vector space $\m/\m^2$ is spanned by the classes of $p$ and $\tilde\beta$, and by
assumption the image of $(\varphi,\psi)$ in it is nonzero. So $\dim_k\mathfrak n/\mathfrak n^2\le1$, that is,
$\mathfrak n=t\mathcal T_\m+\mathfrak n^2$ for some $t\in\mathfrak n$, and $\mathfrak n=t\mathcal T_\m$ by Nakayama's lemma.
\end{proof}

Since $\chi_{A+J}=\varphi-\psi$ by the matrix determinant lemma, the condition $(\varphi,\psi)\not\subseteq\m^2$ of
Corollary~\ref{cor:chain}(b) says that $\chi_A\notin\m^2$ or $\chi_{A+J}\notin\m^2$. It holds, for instance, when $\beta^2\nmid\bar\varphi$,
since then $\varphi\notin\m^2$. Corollary~\ref{cor:chain} also explains arithmetic criteria of the kind recalled at the beginning of this
section. For instance, if a prime $p$ satisfies $p\,\|\,\det W$, then $|\mathcal T_{(p)}|=p^2$ by Proposition~\ref{prop:D}. As
$|\mathcal T_\m|=|\mathfrak M_\m|^2$ by Lemma~\ref{lem:loccomp}(b), a single local component above $p$ is nonzero. It has order $p^2$, and
it is not a field, since it contains the ideal $\mathfrak M_\m$ of order $p$. So its maximal ideal has order $p$ and is principal. By
Corollary~\ref{cor:chain}(a) for this component, and by Proposition~\ref{lem:levellocal}(ii) for the vanishing components above $p$,
$\mathcal F(Q,p)$ is empty for every $Q\in\QG$, so that $p$ divides no level by Proposition~\ref{lem:levellocal}(i) (compare Wang
(2013)~\cite{W13}). At $p=2$ a similar argument recovers the $2$-adic part of the criterion of Wang
(2017)~\cite{W17} (Remark~\ref{rem:loc-wang} in Section~\ref{sec:two}).

\subsection{Krylov presentations and short columns}\label{subsec:loc-hensel}

This subsection gives a concrete $p$-adic model of the components $\mathfrak M_\m$ of $\coker W$, which Section~\ref{sec:small} uses to show
that their exponents are bounded in probability, and it derives the short vectors used in Section~\ref{sec:smooth}. As in
Section~\ref{subsec:loc-lagr}, for a finite $\Z[x]$-module $X$ we identify $X\otimes\Z_p$ with $X_{(p)}$. Just as $\coker W$ is presented
by the walk matrix $W$, which is the Krylov matrix of $(A,e)$ (Krylov (1931)~\cite{Kry31}), the component $\mathfrak M_\m$ is presented by
the Krylov matrix of a pair $(B,w)$ of size $a_\beta\deg\beta$.

\begin{lemma}\label{lem:hensel}
Let $p$ be a prime, let $\beta$ be an irreducible factor of $\bar\varphi$, and put $\m:=(p,\tilde\beta)$, $a:=a_\beta$ and $r:=a\deg\beta$.
Let $\Phi:=\Phi_\beta$ be the Hensel factor of $\varphi$ at $\beta$, and let $N:=N_\beta=\ker\Phi(A)$ be as in Lemma~\ref{lem:loccomp}, an
$A$-stable unimodular sublattice of $\Z_p^n$ of rank $r$ with $\bar N=\ker\beta(\bar A)^a$. Let $V\in\Z_p^{n\times r}$ be a matrix whose
columns form a basis of $N$, and define $B,\Theta\in M_r(\Z_p)$ and $w\in\Q_p^r$ by
\[
AV=VB,\qquad\Theta:=V^{\top}V,\qquad w:=\Theta^{-1}V^{\top}e .
\]
Then $\bar\Theta$ is invertible, so that $w\in\Z_p^r$. Moreover $\mathfrak M_\m\cong\Z_p^r/\mathcal K_r(B,w)\Z_p^r$, where
$\mathcal K_r(B,w):=[w,Bw,\dots,B^{r-1}w]$ is the Krylov matrix of $(B,w)$. In particular $|\mathfrak M_\m|=p^{v_p(\det\mathcal K_r(B,w))}$,
so that for every $E\ge1$
\[
\exp(\mathfrak M_\m)\ge p^E\quad\Longrightarrow\quad\det\mathcal K_r(B,w)\equiv0\pmod{p^E},
\]
and the right-hand side depends only on $B$ and $w$ modulo $p^E$.
\end{lemma}

\begin{proof}
The matrix $\bar\Theta$ is invertible as $N$ is unimodular. By Lemma~\ref{lem:loccomp}(b), $\mathfrak M_\m=N/\Z_p[A]e_\beta$. As
$e_\beta=\pi_\beta e$ lies in $N$, we have $e_\beta=Vu$ for some $u\in\Z_p^r$, and $\Theta u=V^{\top}e_\beta=V^{\top}e$ (as
$e-e_\beta\perp N$), so that $e_\beta=Vw$. Moreover $A^je_\beta=VB^jw$. Since $V$ maps $\Z_p^r$ isomorphically onto $N$, this gives the
presentation by $\mathcal K_r(B,w)$, where $j<r$ suffices by the Cayley--Hamilton theorem. As $\mathfrak M_\m$ is finite,
$\det\mathcal K_r(B,w)\ne0$, and the Smith normal form over $\Z_p$ gives $|\mathfrak M_\m|=p^{v_p(\det\mathcal K_r(B,w))}$. Finally
$|\mathfrak M_\m|\ge\exp(\mathfrak M_\m)$.
\end{proof}

Finally, the columns of $\ell(Q)Q$ give short vectors whose reductions modulo $p$ are killed by a polynomial in $\bar A$ that is
determined by the components in $\mathcal F(Q,p)$. This is how Section~\ref{sec:smooth} reaches a contradiction once the level is known to
be bounded.

\begin{lemma}\label{lem:shortcol}
Let $G$ be controllable, let $Q\in\QG$ have level $\ell$, let $p\mid\ell$ be a prime, and put
\[
g:=\prod_{(p,\tilde\beta)\in\mathcal F(Q,p)}\beta^{a_\beta}\in\F_p[x],\qquad\text{where }a_\beta=v_\beta(\bar\varphi).
\]
Then some column of $\ell Q$ is an integer vector $z$ with $0<\wt(\bar z)\le\ell^2$ and $g(\bar A)\bar z=0$.
\end{lemma}

\begin{proof}
Let $y$ be a column of $Q$, so that $y\in\Lambda_p$ and $z:=\ell y\in\Z^n$. For every irreducible factor $\beta$ of $\bar\varphi$ we have
$\pi_\beta z=\ell\,\pi_\beta y\in\ell\Lambda_\beta$, with the notation of Section~\ref{subsec:loc-levels}. If
$(p,\tilde\beta)\notin\mathcal F(Q,p)$, then $\Lambda_\beta=N_\beta\subseteq\Z_p^n$, so $\pi_\beta z\in\ell\Z_p^n\subseteq p\Z_p^n$ as
$p\mid\ell$. If $(p,\tilde\beta)\in\mathcal F(Q,p)$, then $\pi_\beta z\in\pi_\beta\Z_p^n=N_\beta$. Summing over $\beta$, we see that
$\bar z$ lies in $\bigoplus_{(p,\tilde\beta)\in\mathcal F(Q,p)}\bar N_\beta=
\bigoplus_{(p,\tilde\beta)\in\mathcal F(Q,p)}\ker\beta(\bar A)^{a_\beta}=\ker g(\bar A)$, by Lemma~\ref{lem:loccomp}(a) and as the
$\beta^{a_\beta}$ are pairwise coprime. Since $\ell$ is the level, $\ell Q\not\equiv0\pmod p$, so some column $z=\ell y$ has $\bar z\ne0$.
Finally $z$ is an integer vector with $\|z\|^2=\ell^2\|y\|^2=\ell^2$, so $z$ has at most $\ell^2$ nonzero entries, and
$\wt(\bar z)\le\ell^2$.
\end{proof}

\section{Fourier analysis modulo \texorpdfstring{$p^E$}{p\textasciicircum E} and short vectors}\label{sec:fourier}

Through the deterministic theory of Section~\ref{sec:local}, the proof of the second component, Theorem~\ref{thm:smooth}, reduces to three
kinds of events at the small primes $p$ (and to one more at $p=2$, treated in Section~\ref{sec:two}): local failures
(Definition~\ref{def:Ep}), large exponents of the components of $\coker W$ or large multiplicities of the irreducible factors of
$\bar\varphi$, and short nonzero vectors in the kernel of a polynomial in $A$ modulo $p$. All three concern the adjacency matrix $A$ of
$G(n,1/2)$ modulo a power of $p$. The last kind is controlled by Proposition~\ref{lem:suppd}, and this section provides the tools for the
other two. The main tool is a Fourier bound for the probability that $AY$ takes a prescribed value $Z$ modulo $p^E$, for a fixed matrix $Y$
whose columns are linearly independent modulo $p$ and span no short nonzero vector modulo $p$ (Proposition~\ref{lem:fourier}). In
Section~\ref{sec:small} the Fourier bound is combined with the law of the Gram data of random matrices (Lemma~\ref{lem:gram}), in order to
bound first moments of local failures and to prove tightness.

The Fourier bound is proved by expanding the indicator of the event $AY\equiv Z\pmod{p^E}$ into additive characters, with the character
method for random matrices of Wood (2017)~\cite{Woo17} and Nguyen--Wood (2025)~\cite{NW25}, adapted to the products $AY$ needed here.

\subsection{Notation and robust matrices}\label{subsec:four-notation}

Throughout this section, $p$ is a prime, $E\ge1$, and $\Z/p^E$ is the ring of residues modulo $p^E$. For $c\in\Z/p^E$ the number
$e^{2\pi ic/p^E}$ is well defined, and we put
\[
\tau_{p^E}:=\cos(\pi/p^E).
\]
If $a$ is uniform on $\{0,1\}$ and $c\in\Z/p^E$ is nonzero, represented by an integer in $[1,p^E-1]$, then
$|\E\,e^{2\pi iac/p^E}|=|1+e^{2\pi ic/p^E}|/2=\lvert\cos(\pi c/p^E)\rvert$, so
\begin{equation}\label{eq:four-cos}
|\E\,e^{2\pi iac/p^E}|\le\tau_{p^E}\le e^{-4/p^{2E}},
\end{equation}
where the last inequality follows from $\cos\theta\le e^{-\theta^2/2}$ for $0\le\theta\le\pi/2$ and $\pi^2/2>4$. For a matrix $Y$ we write
$Y_i$ for its $i$th row, and we put $\langle u,v\rangle:=\sum_ku_kv_k$. For $m\ge1$ the set of antisymmetric matrices
\[
\mathcal A_m(p^E):=\{K\in M_m(\Z/p^E):\ K^{\top}=-K\}
\]
has $|\mathcal A_m(p^E)|=p^{Em(m-1)/2}$ elements if $p$ is odd (then the diagonal of $K$ vanishes) and
$|\mathcal A_m(2^E)|=2^{Em(m-1)/2+m}$ elements if $p=2$ (then the diagonal entries lie in $\{0,2^{E-1}\}$). For a commutative ring $R$ we
write $\mathrm{Sym}_m(R)$ for the set of symmetric matrices in $M_m(R)$.

The Fourier method applies to matrices whose columns are linearly independent and span no short nonzero vectors modulo $p$, in the
following quantitative sense.

\begin{definition}\label{def:robust}
Let $s\ge1$. A matrix $Y\in\F_p^{N\times m}$ is \emph{$s$-robust} if $\wt(Yc)\ge s$ for every $c\in\F_p^m\setminus\{0\}$, that is, if its
columns are linearly independent and their span contains no nonzero vector of weight~${<s}$. A matrix $Y\in(\Z/p^E)^{N\times m}$ is
\emph{$s$-robust} if its reduction modulo $p$ is.
\end{definition}

The following range of parameters is used in this section and in Sections~\ref{sec:small} and~\ref{sec:smooth}.

\begin{definition}\label{def:Rn}
For $\delta\in(0,\frac12)$ let $s_\delta:=\lceil n^{1-\delta}\rceil$ and
\[
\mathcal R_n(\delta):=\big\{(p,d):\ p\text{ prime},\ 1\le d\le n^{\delta}/\log^2n,\ pd\le n^{1/2-\delta}/\log n\big\}.
\]
\end{definition}

\subsection{Special characters and the Fourier bound}\label{subsec:four-special}

The next lemma identifies the \emph{special} characters, which are indexed by the matrices $X$ for which all coefficients $c_{ij}(X)$ of
$a_{ij}$ in $\operatorname{tr}(X^{\top}AY)$ vanish, and it shows that all other characters have many nonzero coefficients.

\begin{lemma}\label{lem:special}
Let $p$ be a prime, $E\ge1$, $m\ge1$ and $s\ge16m$, and let $Y\in(\Z/p^E)^{N\times m}$ be $s$-robust. For $X\in(\Z/p^E)^{N\times m}$ put
\[
c_{ij}(X):=\langle X_i,Y_j\rangle+\langle X_j,Y_i\rangle\in\Z/p^E,\qquad \nu(X):=\#\{i<j:\ c_{ij}(X)\ne0\},
\]
\[
t(X):=\min_{K\in\mathcal A_m(p^E)}\#\{i:\ X_i\ne Y_iK\}.
\]
Then $\nu(X)\ge\min\bigl(t(X)\,s/2,\ s^2/(64m)\bigr)$. In particular $\nu(X)=0$ if and only if $X=YK$ for some $K\in\mathcal A_m(p^E)$.
\end{lemma}

\begin{proof}
For $K\in\mathcal A_m(p^E)$ we have $c_{ij}(YK)=Y_iKY_j^{\top}+Y_jKY_i^{\top}=Y_i(K+K^{\top})Y_j^{\top}=0$. Since $c_{ij}$ is additive in
$X$, it follows that $c_{ij}(X-YK)=c_{ij}(X)$. This gives the ``if'' part of the last assertion, and the ``only if'' part follows from the
inequality. Choose $K_0\in\mathcal A_m(p^E)$ attaining $t:=t(X)$, put $X':=X-YK_0$, so that
$c_{ij}(X)=\langle X'_i,Y_j\rangle+\langle X'_j,Y_i\rangle$, and let $P:=\{i:\ X'_i\ne0\}$, so that $|P|=t$. Write $\nu:=\nu(X)$.

\step{1}{We prove the inequality when $t\le s/2$, by counting the pairs $\{i,j\}$ with $i\in P$ and $j\notin P$.} Let $i\in P$, and write
$X'_i=p^rX''_i$ with $0\le r<E$ and $\bar X''_i\ne0$. By robustness, $\langle\bar X''_i,\bar Y_j\rangle\ne0$ for at least $s$ indices $j$
(the nonzero coordinates of $\bar Y(\bar X''_i)^{\top}$), hence for at least $s-t$ indices $j\notin P$. For these,
$c_{ij}(X)=\langle X'_i,Y_j\rangle=p^r\langle X''_i,Y_j\rangle\ne0$. Each pair $\{i,j\}$ with $i\in P$ and $j\notin P$ is counted once, so
$\nu\ge t(s-t)\ge ts/2$.

\step{2}{We show by contradiction that $\nu\ge s^2/(64m)$ when $t>s/2$.} Suppose that $\nu<s^2/(64m)$. Consider the graph on $[N]$ whose
edges are the pairs $\{i,j\}$, $i\ne j$, with $c_{ij}(X)\ne0$. It has $\nu$ edges, so at most $2\nu\cdot8m/s<s/4$ vertices have degree
$>s/(8m)$. We choose $b_1,\dots,b_m\in[N]$ of degree $\le s/(8m)$, pairwise non-adjacent, such that $\bar Y_{b_1},\dots,\bar Y_{b_m}$ are
linearly independent. Having chosen $b_1,\dots,b_{k-1}$ ($k\le m$), pick $0\ne u\in\F_p^m$ orthogonal to
$\bar Y_{b_1},\dots,\bar Y_{b_{k-1}}$. By robustness, at least $s$ indices $j$ satisfy $\langle\bar Y_j,u\rangle\ne0$, and then $\bar Y_j$
lies outside the span of $\bar Y_{b_1},\dots,\bar Y_{b_{k-1}}$. Fewer than $s/4$ of these indices have degree $>s/(8m)$, and at most
$(k-1)s/(8m)<s/8$ of them are adjacent to some $b_l$, $l<k$. So a valid choice of $b_k$ exists. Put $B:=\{b_1,\dots,b_m\}$, and let $\Gamma$
be the set of indices outside $B$ that are not adjacent to any element of $B$. Then $|[N]\setminus\Gamma|\le m+s/8$.

The rows $Y_b$, $b\in B$, form a basis of $(\Z/p^E)^m$, since their reductions form a basis of $\F_p^m$. Hence there is a unique bilinear
form $\kappa$ on $(\Z/p^E)^m$ with $\kappa(a,Y_b)=-\langle X'_b,a\rangle$ for all $a$ and all $b\in B$. Let $K$ be its matrix, so that
$\kappa(a,a')=aKa'^{\top}$. For $j\in\Gamma$ and $b\in B$ we have $c_{jb}(X)=0$, that is,
$\langle X'_j,Y_b\rangle=-\langle X'_b,Y_j\rangle=\kappa(Y_j,Y_b)$. The linear forms $a'\mapsto\langle X'_j,a'\rangle$ and
$a'\mapsto\kappa(Y_j,a')$ agree at the basis vectors $Y_b$, $b\in B$, so they are equal, which means that $X'_j=Y_jK$. Consequently, for
distinct $j,j'\in\Gamma$,
\[
c_{jj'}(X)=Y_jKY_{j'}^{\top}+Y_{j'}KY_j^{\top}=Y_jLY_{j'}^{\top},\qquad L:=K+K^{\top}.
\]
If $L\ne0$, write $L=p^rL'$ with $0\le r<E$ and $L'$ symmetric with $\bar L'\ne0$. By robustness applied to a nonzero row of $\bar L'$, at
least $s$ indices $j'$ satisfy $\bar L'\bar Y_{j'}^{\top}\ne0$, and at least $s-m-s/8$ of them lie in $\Gamma$. For each such $j'$,
robustness applied to $\bar L'\bar Y_{j'}^{\top}$ gives at least $s-m-s/8-1$ indices $j\in\Gamma\setminus\{j'\}$ with
$\bar Y_j\bar L'\bar Y_{j'}^{\top}\ne0$, and then $c_{jj'}(X)=p^rY_jL'Y_{j'}^{\top}\ne0$. Since $m\le s/16$ and $s\ge16$,
\[
\nu\ge\tfrac12\big(s-m-\tfrac s8\big)\big(s-m-\tfrac s8-1\big)\ge\tfrac12\cdot\tfrac{13s}{16}\cdot\tfrac{3s}{4}
\ge\frac{s^2}{8}\ge\frac{s^2}{64m},
\]
a contradiction. If $L=0$, then $K\in\mathcal A_m(p^E)$ and $X_j=Y_j(K_0+K)$ for all $j\in\Gamma$, with $K_0+K\in\mathcal A_m(p^E)$. So
$t\le|[N]\setminus\Gamma|\le m+s/8<s/2$, again a contradiction. Hence $\nu\ge s^2/(64m)$ when $t>s/2$.
\end{proof}

Expanding the indicator of the event $AY\equiv Z$ into additive characters and applying Lemma~\ref{lem:special}, we obtain the following
bound. Part~(b) is its asymptotic form, uniform in the modulus $p^E$ and in $m$ as long as $p^Em$ is a little smaller than $\sqrt n$.

\begin{proposition}\label{lem:fourier}
Let $A$ be the adjacency matrix of $G(n,1/2)$.
\begin{enumerate}[label=(\alph*)]
\item Let $p$ be a prime, $E,m\ge1$, $s\ge16m$, let $Y\in(\Z/p^E)^{n\times m}$ be $s$-robust, and let $Z\in(\Z/p^E)^{n\times m}$. Then, with
$\tau:=\tau_{p^E}$,
\begin{equation}\label{eq:four-bound}
\Pr\bigl(AY\equiv Z\pmod{p^E}\bigr)\le p^{-Enm}\Big(|\mathcal A_m(p^E)|\big(1+p^{Em}\tau^{s/2}\big)^n+p^{Enm}\tau^{s^2/(64m)}\Big).
\end{equation}
\item Fix $\delta\in(0,\frac12)$. There are $\varepsilon_n=\varepsilon_n(\delta)\to0$ such that
\[
\Pr\bigl(AY\equiv Z\pmod{p^E}\bigr)\le(1+\varepsilon_n)\,p^{-Enm}\,|\mathcal A_m(p^E)|
\]
for all primes $p$ and all $E,m\ge1$ with $m\le n^{\delta}/\log^2n$ and $p^Em\le n^{1/2-\delta}/\log n$, all $s_\delta$-robust
$Y\in(\Z/p^E)^{n\times m}$ and all $Z\in(\Z/p^E)^{n\times m}$.
\end{enumerate}
\end{proposition}

\begin{proof}
(a) By the orthogonality of the additive characters of $(\Z/p^E)^{n\times m}$,
\[
\Pr(AY\equiv Z)=p^{-Enm}\sum_{X\in(\Z/p^E)^{n\times m}}\E\exp\Bigl(\frac{2\pi i\operatorname{tr}(X^{\top}(AY-Z))}{p^E}\Bigr).
\]
Since $A=(a_{ij})$ is symmetric with zero diagonal,
$\operatorname{tr}(X^{\top}AY)=\sum_{i,j}a_{ij}\langle X_i,Y_j\rangle=\sum_{i<j}a_{ij}c_{ij}(X)$, and the $a_{ij}$, $i<j$, are independent
and uniform on $\{0,1\}$. By~\eqref{eq:four-cos} the term of $X$ has modulus at most $\tau^{\nu(X)}$. By Lemma~\ref{lem:special},
$\tau^{\nu(X)}\le\tau^{t(X)s/2}+\tau^{s^2/(64m)}$ (with $0^0=1$). Every $X$ with $t(X)=t$ has the form $YK+X'$ with $K\in\mathcal A_m(p^E)$
and $X'$ having exactly $t$ nonzero rows, so there are at most $|\mathcal A_m(p^E)|\binom ntp^{Emt}$ such $X$. Hence
\[
\sum_X\tau^{\nu(X)}\le|\mathcal A_m(p^E)|\sum_{t=0}^n\binom nt\big(p^{Em}\tau^{s/2}\big)^t+p^{Enm}\tau^{s^2/(64m)},
\]
which is~\eqref{eq:four-bound}.

(b) Put $s:=s_\delta$ and $\tau:=\tau_{p^E}$, and let $n$ be large. Then $s\ge n^{1-\delta}\ge16n^{\delta}\ge16m$, so (a) applies, and
$\tau\le e^{-4/p^{2E}}$ by~\eqref{eq:four-cos}. Since $p^E\le p^Em\le n^{1/2-\delta}/\log n$ and $s\ge n^{1-\delta}$, we have
$s/p^{2E}\ge n^{\delta}\log^2n$, and as $E\log p=\log p^E\le\log n$, also $mE\log p\le n^{\delta}/\log n$. Hence
\[
p^{Em}\tau^{s/2}\le\exp\bigl(mE\log p-2s/p^{2E}\bigr)\le\exp\bigl(-n^{\delta}\log^2n\bigr),
\]
and $(1+p^{Em}\tau^{s/2})^n\le\exp(n\,e^{-n^{\delta}\log^2n})$, which tends to $1$. Next,
$nm^2p^{2E}E\log p\le n\cdot n^{1-2\delta}\log^{-2}n\cdot\log n\le s^2/\log n$, so
\[
p^{Enm}\tau^{s^2/(64m)}\le\exp\bigl(nmE\log p-s^2/(16mp^{2E})\bigr)\le\exp\bigl(nmE\log p\,(1-\tfrac1{16}\log n)\bigr),
\]
which tends to $0$, as $nmE\log p\ge n\log2$ and $\log n>16$. Both bounds are uniform in $p$, $E$ and $m$. As $|\mathcal A_m(p^E)|\ge1$,
the bracket in~\eqref{eq:four-bound} is at most $(1+\varepsilon_n)|\mathcal A_m(p^E)|$ for suitable $\varepsilon_n\to0$, which proves (b).
\end{proof}

In the terminology of Wood (2017)~\cite{Woo17}, a matrix $Y\in\F_p^{N\times m}$ is $s$-robust if and only if the map $v\mapsto Y^{\top}v$
from $\F_p^N$ to $\F_p^m$ is a code of distance $s$. Proposition~\ref{lem:fourier}(a) is the analogue, for symmetric matrices with zero
diagonal, of her estimate for codes~\cite[Lemma~4.1]{Woo17}, which is stated for symmetric matrices whose entries on and above the
diagonal are independent. With a zero diagonal, the special characters at $p=2$ are indexed by the group $\mathcal A_m(2^E)$ of
antisymmetric matrices, which is larger by a factor $2^m$ than the group of alternating matrices that arises when the diagonal is random
as well.

\subsection{Short vectors}\label{subsec:four-short}

The kernels of polynomials in $\bar A$ of small degree contain, with high probability, no short nonzero vectors. Section~\ref{sec:small}
uses this to show that the frames describing simple submodules are robust, and it is a special case of Proposition~\ref{lem:suppd}.

\begin{lemma}\label{prop:shortany}
Fix $\delta\in(0,\frac14]$, and let $A$ be the adjacency matrix of $G(n,1/2)$. For all large $n$ (depending only on $\delta$), with
probability at least $1-e^{-n/8}$, for every prime $p$ (including $p=2$) and every monic $g\in\F_p[x]$ with
$(p,\deg g)\in\mathcal R_n(\delta)$ there is no $z\in\F_p^n$ with $0<\wt(z)<s_\delta$ and $g(\bar A)z=0$.
\end{lemma}

\begin{proof}
Let $n$ be large, and let $p$ and $g$ be as in the statement, with $d:=\deg g$. As $(p,d)\in\mathcal R_n(\delta)$ and $\delta\le\frac14$, we
have $p\le n^{1/2}\le e^{n^{1/8}}$, $d\le n^{1/4}$ and $n/(4(d+1))\ge n^{1-\delta}\log^2n/8\ge s_\delta$. So the lemma follows from
Proposition~\ref{lem:suppd} with $M=A$, $m=n$ and $\{a,b\}=\{0,1\}$.
\end{proof}

In Lemma~\ref{prop:shortany} the polynomial $g$ need not be irreducible, and Section~\ref{sec:small} also applies it to powers of
irreducible polynomials.

\subsection{Gram data of random matrices}\label{subsec:four-gram}

Finally, we need the law of the Gram data $(V^{\top}V,V^{\top}e)$ of a uniformly random matrix $V$ modulo $p^E$.
Section~\ref{sec:small} uses it to count the candidates for admissible frames, that is, the matrices $V$ with $V^{\top}V=0$ and
$V^{\top}\bar e=0$, in the proof of Theorem~\ref{thm:firstmoment}, and, in the proof of Proposition~\ref{prop:tight}, to count matrices
with prescribed Gram matrices and to identify the limits. Detecting the conditions with additive characters leads to Gauss sums of
quadratic polynomials.

\begin{lemma}\label{lem:gram}
Let $p$ be a prime, $E,r\ge1$, and let $V$ be uniform in $(\Z/p^E)^{n\times r}$. Put $\mathcal G:=\mathrm{Sym}_r(\Z/p^E)\times(\Z/p^E)^r$,
let $\mathcal G_r(p^E):=\mathcal G$ if $p$ is odd, and
\[
\mathcal G_r(2^E):=\big\{(\Theta,u)\in\mathcal G:\ \Theta_{ii}\equiv u_i\pmod2\ \text{for all }i\big\}.
\]
Let $\mu$ be the uniform distribution on $\mathcal G_r(p^E)$, viewed as a probability measure on $\mathcal G$.
\begin{enumerate}[label=(\alph*)]
\item There is $c=c(p,E,r)<1$ such that $\lvert\Pr((V^{\top}V,V^{\top}e)=(\Theta,u))-\mu(\Theta,u)\rvert\le c^n$ for all $n\ge1$ and all
$(\Theta,u)\in\mathcal G$. In particular $(V^{\top}V,V^{\top}e)$ converges in total variation, exponentially fast in $n$, to $\mu$, and
$V^{\top}V$ converges to the uniform distribution on $\mathrm{Sym}_r(\Z/p^E)$.
\item If $p$ is odd and $E=1$, then (a) holds with $c=p^{-1/2}$, for every $r\ge1$. In particular, with $k_r:=r(r+3)/2$, the number of
$V\in\F_p^{n\times r}$ with $V^{\top}V=0$ and $V^{\top}\bar e=0$ differs from $p^{nr-k_r}$ by at most $p^{nr-n/2}$.
\end{enumerate}
\end{lemma}

\begin{proof}
The characters of $\mathcal G$ are $\chi_{\sigma,t}(\Theta,u)=\exp\bigl(2\pi i\,(\sum_{i\le j}\sigma_{ij}\Theta_{ij}+t^{\top}u)/p^E\bigr)$,
for $\sigma=(\sigma_{ij})_{i\le j}$ and $t$ with entries in $\Z/p^E$. Since the rows of $V$ are independent and uniform in $(\Z/p^E)^r$,
\[
\E\big(\chi_{\sigma,t}(V^{\top}V,V^{\top}e)\big)=\Big(\E\,e^{2\pi i\,F_{\sigma,t}(v)/p^E}\Big)^n,\qquad
F_{\sigma,t}(y):=\sum_{i\le j}\sigma_{ij}y_iy_j+t^{\top}y ,
\]
where $v$ is uniform in $(\Z/p^E)^r$.

\step{1}{We show that $F_{\sigma,t}$ vanishes identically if and only if $\chi_{\sigma,t}$ is trivial on $\mathcal G_r(p^E)$.} If
$F_{\sigma,t}\equiv0$, then evaluating at $y=e_i$, $y=2e_i$ and $y=e_i+e_j$ gives $\sigma_{ij}=0$ for $i<j$, $2\sigma_{ii}=0$ and
$t_i=-\sigma_{ii}$. For odd $p$ this means $(\sigma,t)=0$. For $p=2$ it means that $\sigma$ is diagonal with
$\sigma_{ii}=t_i\in\{0,2^{E-1}\}$, and conversely these $(\sigma,t)$ give $F_{\sigma,t}(y)=\sum_i\sigma_{ii}y_i(y_i+1)\equiv0$. The
corresponding characters $\chi(\Theta,u)=\prod_{i:\,\sigma_{ii}\ne0}(-1)^{\Theta_{ii}+u_i}$ are exactly the characters trivial on
$\mathcal G_r(2^E)$, since $(\Theta,u)\mapsto(\Theta_{ii}+u_i\bmod2)_{i\le r}$ identifies $\mathcal G/\mathcal G_r(2^E)$ with $(\Z/2)^r$.

\step{2}{We prove (a) by comparing the Fourier coefficients of the law of $(V^{\top}V,V^{\top}e)$ with those of $\mu$.} If
$F_{\sigma,t}\not\equiv0$, the mean $\E\,e^{2\pi i\,F_{\sigma,t}(v)/p^E}$ is an average of roots of unity that are not all equal (as
$F_{\sigma,t}(0)=0$), so it has modulus less than $1$. Let $c<1$ be the largest of these moduli. The Fourier coefficients of $\mu$ are $1$
at the characters trivial on $\mathcal G_r(p^E)$ and $0$ at all others. By Step~1, those of the law of $(V^{\top}V,V^{\top}e)$ are the same,
up to errors of modulus at most $c^n$. Fourier inversion on $\mathcal G$ gives
$\lvert\Pr((V^{\top}V,V^{\top}e)=(\Theta,u))-\mu(\Theta,u)\rvert\le c^n$, and summing over $(\Theta,u)\in\mathcal G$ bounds the total variation
distance by $\frac12|\mathcal G|c^n$. The image of $\mu$ under $(\Theta,u)\mapsto\Theta$ is uniform on $\mathrm{Sym}_r(\Z/p^E)$, since every
$\Theta$ has the same number of compatible $u$, and projecting does not increase the total variation distance.

\step{3}{We prove (b) by bounding the Gauss sums that arise when $p$ is odd and $E=1$.} Let $p$ be odd, $E=1$ and $(\sigma,t)\ne0$. If
$\sigma=0$, then $t\ne0$ and $\E\,e^{2\pi i\,t^{\top}v/p}=0$. If $\sigma\ne0$, then $\sum_{i\le j}\sigma_{ij}y_iy_j=y^{\top}My$ for a
nonzero symmetric matrix $M$, as $p$ is odd, and we bound the Gauss sum in the standard way (see Lidl--Niederreiter
(1997)~\cite[Ch.~5]{LN97}). The substitution $y=z+h$ gives
\[
\Bigl|\sum_{y\in\F_p^r}e^{2\pi i\,F_{\sigma,t}(y)/p}\Bigr|^2=\sum_he^{2\pi i\,F_{\sigma,t}(h)/p}\sum_ze^{4\pi i\,h^{\top}Mz/p}
\le p^r\,\lvert\ker M\rvert\le p^{2r-1},
\]
because the inner sum equals $p^r$ if $Mh=0$ and vanishes otherwise. So $|\E\,e^{2\pi i\,F_{\sigma,t}(v)/p}|\le p^{-1/2}$, and Step~2
applies with $c=p^{-1/2}$. For the count, note that $\mu(0,0)=|\mathcal G|^{-1}=p^{-k_r}$, and multiply by $p^{nr}$.
\end{proof}

\section{Small primes: first moments and tightness}\label{sec:small}

The proof of the second component, Theorem~\ref{thm:smooth}, must exclude matrices $Q\in\QG$, other than permutation matrices, whose level
has only small prime factors. By Proposition~\ref{lem:levellocal}, if $G$ is controllable, then every prime factor $p$ of such a level
comes with a local component above $p$ which is not a chain ring, because $Q$ and the identity matrix give two distinct ideals of the same
order there. For odd $p$, this component also gives a local failure $E_p^{(d)}$ in its degree $d$. This section provides two of the
probabilistic estimates that control these components. Theorem~\ref{thm:firstmoment} bounds the probability of $E_p^{(d)}$, up to small
error terms, by an explicit first moment $c_p^{(d)}$, uniformly in the range $\mathcal R_n(\delta)$ of Definition~\ref{def:Rn}, so that,
with probability close to $1$, $E_p^{(d)}$ fails for all $(p,d)$ in this range with $p^d$ large. Proposition~\ref{prop:tight} shows that
at a fixed prime $p$ (including $p=2$) the exponents of the components of $\coker W$ of fixed degree, and the multiplicities of the
irreducible factors of $\bar\varphi$ of fixed degree, are tight. In Section~\ref{sec:smooth}, Theorems~\ref{thm:firstmoment}
and~\ref{thm:largedegrees} leave only finitely many relevant pairs $(p,\beta)$ at the odd primes, and Proposition~\ref{prop:twolarge} does
the same at $p=2$. Proposition~\ref{prop:tight} then bounds the exponents of the corresponding components of $\coker W$, and hence the
level, as well as the multiplicities of the corresponding factors of $\bar\varphi$.

Both results rest on the Fourier bounds of Section~\ref{sec:fourier}. In the proof of Proposition~\ref{prop:tight} the limits are
described by Haar measures on spaces of $p$-adic matrices, as in the random matrix model of Friedman--Washington (1989)~\cite{FW89} for
the Cohen--Lenstra (1984)~\cite{CL84} heuristics, and our moment computations are in the spirit of the moment method for random integral
matrices of Wood (2017)~\cite{Woo17} and Nguyen--Wood (2025)~\cite{NW25}.

For a prime $p$ and $d\ge1$ let $\pi_p(d)$ be the number of monic irreducible polynomials of degree $d$ in $\F_p[x]$. We use the bound
$\pi_p(d)\le p^d/d$ (see Lidl--Niederreiter (1997)~\cite[Ch.~3]{LN97}). In this section and in Sections~\ref{sec:two} and~\ref{sec:smooth},
``irreducible polynomial'' always means monic irreducible polynomial.

\subsection{First moments at odd primes}
In this subsection $p$ is an odd prime, and we put
\[
c_p^{(d)}:=\frac{\pi_p(d)}{p^{2d}(p^d-1)}\qquad(d\ge1).
\]
Recall the parameter $s_\delta=\lceil n^{1-\delta}\rceil$, the range $\mathcal R_n(\delta)$ of Definition~\ref{def:Rn} and the $s$-robust
matrices of Definition~\ref{def:robust}. Simple submodules are described by frames, as follows. Let $\beta\in\F_p[x]$ be monic irreducible
of degree $d$, and fix $M_\beta\in M_d(\F_p)$ with characteristic polynomial $\beta$ (say, its companion matrix). A \emph{$\beta$-frame}
of $G$ is a nonzero matrix $V\in\F_p^{n\times d}$ with $\bar AV=VM_\beta$. Since $\F_p[M_\beta]\cong\F_p[x]/(\beta)$ is a field with $q$
elements, $\F_p^d$ is a one-dimensional vector space over it, and hence a simple $\F_p[M_\beta]$-module. As $\ker V$ is $M_\beta$-stable
and $V\ne0$, we get $\ker V=0$. So $V$ is an isomorphism of $\F_p[x]$-modules from $(\F_p^d,M_\beta)\cong\F_p[x]/(\beta)$ onto the column
space $U$ of $V$, a simple $\bar A$-submodule. Conversely, every simple $\bar A$-submodule $U\cong\F_p[x]/(\beta)$ of $\F_p^n$ is the
column space of exactly $q-1$ $\beta$-frames, as the automorphism group of the $\F_p[x]$-module $\F_p[x]/(\beta)$ is $\F_q^{\times}$. A
$\beta$-frame $V$ is \emph{admissible} if $V^{\top}V=0$ and $V^{\top}\bar e=0$, that is, if its column space is totally isotropic and
orthogonal to $\bar e$. Hence $E_p^{(d)}$ holds if and only if some $\beta$ of degree $d$ has an admissible $\beta$-frame, and then all
$q-1$ frames with the same column space are admissible. If $U$ contains no nonzero vector of weight $<s$, then all $\beta$-frames with
column space $U$ are $s$-robust.

The constant $c_p^{(d)}$ is the first moment obtained by treating $\bar AV$, for each $V\in\F_p^{n\times d}$ with $V^{\top}V=0$ and
$V^{\top}\bar e=0$, as uniform among the $Z\in\F_p^{n\times d}$ with $V^{\top}Z$ symmetric. Indeed, there are about $p^{nd-k_d}$ such $V$,
where $k_d=d(d+3)/2$ (Lemma~\ref{lem:gram}(b)), each of them satisfies $\bar AV=VM_\beta$ with probability about $p^{-nd+d(d-1)/2}$, and
each simple submodule has $q-1$ frames. So each of the $\pi_p(d)$ polynomials $\beta$ contributes $p^{-2d}/(p^d-1)$. The main result of this
subsection makes this heuristic rigorous, uniformly in the range $\mathcal R_n(\delta)$.

\begin{theorem}\label{thm:firstmoment}
Fix $\delta\in(0,\frac14]$. There are $\varepsilon_n=\varepsilon_n(\delta)\to0$ and $n_0(\delta)$ such that for all $n\ge n_0(\delta)$
and all $(p,d)\in\mathcal R_n(\delta)$ with $p$ odd,
\[
\Pr\bigl(E_p^{(d)}\bigr)\le c_p^{(d)}\,(1+\varepsilon_n)+e^{-n/8}.
\]
In particular $\limsup_n\Pr(E_p^{(d)})\le c_p^{(d)}$ for every fixed odd prime $p$ and every fixed $d\ge1$, and
\[
\lim_{q_0\to\infty}\ \limsup_{n\to\infty}\ \Pr\bigl(E_p^{(d)}\text{ holds for some }(p,d)\in\mathcal R_n(\delta)
\text{ with $p$ odd and }p^d>q_0\bigr)=0 .
\]
\end{theorem}

\begin{proof}
Let $n$ be large and $(p,d)\in\mathcal R_n(\delta)$ with $p$ odd. All error terms $o(1)$ below are uniform in such $(p,d)$. Let
$\mathcal N$ be the event of Lemma~\ref{prop:shortany} for this $\delta$, so that $\Pr(\mathcal N^c)\le e^{-n/8}$, and let $\mathcal V$ be
the (deterministic) set of $s_\delta$-robust $V\in\F_p^{n\times d}$ with $V^{\top}V=0$ and $V^{\top}\bar e=0$. Suppose that $\mathcal N$
holds and that some simple $\bar A$-submodule $U$ of some type $\beta$ of degree $d$ is totally isotropic and orthogonal to $\bar e$. As
$\beta(\bar A)$ kills $U$, the event $\mathcal N$ shows that $U$ contains no nonzero vector of weight $<s_\delta$. So the $q-1$
$\beta$-frames with column space $U$ are $s_\delta$-robust and admissible, that is, they lie in $\mathcal V$. Hence
\[
\mathbf 1_{E_p^{(d)}}\le\mathbf 1_{\mathcal N^c}+\frac1{q-1}\sum_\beta\sum_{V\in\mathcal V}\mathbf 1\big\{\bar AV=VM_\beta\big\},
\]
where $\beta$ runs over the irreducible polynomials of degree $d$. For $V\in\mathcal V$, Proposition~\ref{lem:fourier}(b) with $E=1$,
$m=d$, $Y=V$ and $Z=VM_\beta$ applies, as $(p,d)\in\mathcal R_n(\delta)$, and since $|\mathcal A_d(p)|=p^{d(d-1)/2}$ for odd $p$ it gives
\[
\Pr\bigl(\bar AV=VM_\beta\bigr)\le p^{-nd+d(d-1)/2}\big(1+o(1)\big).
\]
By Lemma~\ref{lem:gram}(b) with $r=d$, and with $k_d=d(d+3)/2$,
\[
|\mathcal V|\le p^{nd-k_d}\big(1+p^{k_d-n/2}\big)=p^{nd-k_d}\big(1+o(1)\big),
\]
because $k_d\le2d^2\le n^{2\delta}\le n/4$ for large $n$. Since $k_d-d(d-1)/2=2d$, each of the $\pi_p(d)$ polynomials $\beta$ contributes at
most $p^{-2d}(p^d-1)^{-1}(1+o(1))$ to the expectation of the last term. This gives the first assertion, with $1+\varepsilon_n$ the product
of the two factors $1+o(1)$ above. The second assertion follows, since $(p,d)\in\mathcal R_n(\delta)$ for all large $n$ when $p$ and $d$ are
fixed.

For the tail statement, note that $c_p^{(d)}\le p^d/(p^{2d}(p^d-1))\le2q^{-2}$, and that distinct pairs $(p,d)$ give distinct prime powers
$q=p^d$. As the event $\mathcal N$ does not depend on $(p,d)$, the bound on $\mathbf 1_{E_p^{(d)}}$ above gives
\[
\Pr\bigl(E_p^{(d)}\text{ for some }(p,d)\in\mathcal R_n(\delta),\ p\text{ odd},\ p^d>q_0\bigr)\le(1+\varepsilon_n)\sum_{t>q_0}\frac2{t^2}
+e^{-n/8},
\]
where $t$ runs over the integers $t>q_0$. The $\limsup$ of the right-hand side is at most $\sum_{t>q_0}2t^{-2}$, a tail of a convergent
series, which tends to $0$ as $q_0\to\infty$.
\end{proof}

\subsection{Tightness at a fixed prime}
For controllable $G$ recall the module $\mathfrak M=\coker W$ of Section~\ref{sec:local} and its component
$\mathfrak M_{(p,\beta)}=\mathfrak M_\m$ at $\m=(p,\tilde\beta)$, for a prime $p$ and an irreducible $\beta\in\F_p[x]$. Its exponent
$\exp(\mathfrak M_{(p,\beta)})$ is a power of $p$. Since $\mathfrak M_\m\subseteq D_\m\cong\mathcal T_\m$ (Proposition~\ref{prop:D}) and
$\mathcal T_\m=0$ unless $\m\supseteq(\varphi,\psi)$, we have $\mathfrak M_{(p,\beta)}=0$ unless $\beta$ divides $\bar\varphi$, the
reduction of the characteristic polynomial $\varphi=\chi_A$. Recall from Section~\ref{sec:fourier} the set $\mathcal A_m(p^E)$ of matrices
$K\in M_m(\Z/p^E)$ with $K^{\top}=-K$, for which $|\mathcal A_m(p^E)|=p^{Em(m-1)/2}$ if $p$ is odd and
$|\mathcal A_m(2^E)|=2^{Em(m-1)/2+m}$.

The second main result of this section shows that, at a fixed prime, the multiplicities of the irreducible factors of $\bar\varphi$ of fixed
degree and the exponents of the components of $\coker W$ of fixed degree are tight, that is, bounded in probability uniformly in $n$.

\begin{proposition}\label{prop:tight}
Let $p$ be a prime (possibly $p=2$) and $d\ge1$, both fixed.
\begin{enumerate}[label=(\alph*)]
\item For every $k\ge1$,
\[
\limsup_{n\to\infty}\Pr\bigl(\beta^k\mid\bar\varphi\text{ for some irreducible }\beta\in\F_p[x]\text{ of degree }d\bigr)
\le(k+1)\,\pi_p(d)\,p^{-\lfloor k/2\rfloor d}.
\]
\item Let $\mathcal X_E$ be the event that $G$ is controllable and $\exp(\mathfrak M_{(p,\beta)})\ge p^E$ for some irreducible
$\beta\in\F_p[x]$ of degree $d$. Then $\lim_{E\to\infty}\limsup_{n\to\infty}\Pr(\mathcal X_E)=0$.
\end{enumerate}
\end{proposition}

\begin{proof}
Throughout the proof, $p$ and $d$ are fixed, as are the parameters $k$, $E$ and $a$ introduced below, and $o(1)$ refers to $n\to\infty$ with
all of them fixed. Put $s:=s_{1/4}=\lceil n^{3/4}\rceil$. Let $\mathcal N$ be the event of Lemma~\ref{prop:shortany} with $\delta=\frac14$,
so that $\Pr(\mathcal N^c)\le e^{-n/8}$ for large $n$. As $p$, $d$ and $k$ are fixed, $(p,jd)\in\mathcal R_n(\frac14)$ for $1\le j\le k$ and
large $n$. So on $\mathcal N$, for every irreducible $\beta\in\F_p[x]$ of degree $d$ and every $j\le k$, the space $\ker\beta(\bar A)^j$
contains no nonzero vector of weight $<s$. We also use Proposition~\ref{lem:fourier}(b) with $\delta=\frac14$ (so that $s=s_\delta$) and a
fixed $m$. Its hypotheses hold for large $n$, as $p$, $E$ and $m$ are fixed. So, uniformly over $s$-robust $Y\in(\Z/p^E)^{n\times m}$ and
all $Y'\in(\Z/p^E)^{n\times m}$,
\begin{equation}\label{eq:fourierfixed}
\Pr\bigl(AY\equiv Y'\pmod{p^E}\bigr)\le p^{-Enm}\,|\mathcal A_m(p^E)|\,\big(1+o(1)\big).
\end{equation}

\step{1}{We bound, on the event $\mathcal N$, the expected number of vectors with annihilator $\beta^j$ in terms of the size of a
set $\mathcal S_j$ of Gram matrices.}
Fix an irreducible $\beta$ of degree $d$ (there are $\pi_p(d)$ of them), and let $\beta^b\,\|\,\bar\varphi$. The $\beta$-primary component
of the $\F_p[x]$-module $\F_p^n$ is $\bigoplus_i\F_p[x]/(\beta^{b_i})$ with $\sum_ib_i=b$, so $\dim\ker\beta(\bar A)^k=d\sum_i\min(b_i,k)$,
which is at least $kd$ if and only if $b\ge k$. Thus $\beta^k\mid\bar\varphi$ if and only if $\lvert\ker\beta(\bar A)^k\rvert\ge p^{kd}$.

The annihilator of a vector $z\in\F_p^n$ is the monic generator of the ideal of the polynomials $f$ with $f(\bar A)z=0$. Every
$z\in\ker\beta(\bar A)^k$ has annihilator $\beta^j$ for some $0\le j\le k$. Let $Z_j$ be the set of these $z$, so that $Z_0=\{0\}$. For
$j\ge1$ put $R_j:=\F_p[x]/(\beta^j)$, and let $C_j\in M_{jd}(\F_p)$ be the companion matrix of $\beta^j$, that is, the matrix of
multiplication by $x$ on $R_j$ in the basis $1,x,\dots,x^{jd-1}$. Then $z\mapsto V_z:=[z,\bar Az,\dots,\bar A^{jd-1}z]$ is a
bijection from $Z_j$ onto the set of matrices $V\in\F_p^{n\times jd}$ of rank $jd$ with $\bar AV=VC_j$: the identity $\bar AV_z=V_zC_j$
means $\beta^j(\bar A)z=0$, every $V$ with $\bar AV=VC_j$ equals $V_z$ for $z=Ve_1$, and $V_z$ has rank $jd$ if and only if the
annihilator of $z$ has degree $jd$. On $\mathcal N$ these matrices are $s$-robust, as they have full rank and their column spaces lie in
$\ker\beta(\bar A)^k$. Moreover, $\bar AV=VC_j$ implies that $V^{\top}VC_j=V^{\top}\bar AV$ is symmetric, and alternating if $p=2$ (as
$\bar A$ is then alternating). Let $\mathcal S_j$ be the set of symmetric $\Theta\in M_{jd}(\F_p)$ such that $\Theta C_j$ is symmetric,
and alternating if $p=2$. Then
\[
|Z_j|\,\mathbf 1_{\mathcal N}\le\sum_V\mathbf 1\big\{\bar AV=VC_j\big\},
\]
the sum over the $s$-robust $V\in\F_p^{n\times jd}$ with $V^{\top}V\in\mathcal S_j$. By Lemma~\ref{lem:gram}(a) (with $E=1$ and $r=jd$, for
$V^{\top}V$ alone), and as $\mathcal S_j$ is nonempty (it contains $0$), the number of $V\in\F_p^{n\times jd}$ with
$V^{\top}V\in\mathcal S_j$ is $p^{njd}\,|\mathcal S_j|\,|\mathrm{Sym}_{jd}(\F_p)|^{-1}(1+o(1))$. With~\eqref{eq:fourierfixed} for $E=1$ and
$m=jd$ we obtain
\begin{equation}\label{eq:Zj}
\E\big(|Z_j|\,\mathbf 1_{\mathcal N}\big)\le\frac{|\mathcal S_j|\,|\mathcal A_{jd}(p)|}{|\mathrm{Sym}_{jd}(\F_p)|}\big(1+o(1)\big).
\end{equation}

\step{2}{We bound the size of $\mathcal S_j$ and prove (a) by Markov's inequality.}
Identify $\F_p^{jd}$ with $R_j$ through the basis $1,x,\dots,x^{jd-1}$, so that $C_j$ becomes multiplication by $x$. A symmetric $\Theta$
defines the symmetric bilinear form $\langle y,y'\rangle_\Theta:=y^{\top}\Theta y'$ on $R_j$, and $\Theta C_j$ is symmetric if and only if
$\langle y,xy'\rangle_\Theta=\langle xy,y'\rangle_\Theta$ for all $y,y'$. In that case
$\langle fy,y'\rangle_\Theta=\langle y,fy'\rangle_\Theta$ for every polynomial $f$, so $\langle y,y'\rangle_\Theta=\lambda(yy')$ with
$\lambda:=\langle1,\cdot\rangle_\Theta$ in the dual space $R_j^*$ of $R_j$. Conversely, every $\lambda\in R_j^*$ defines such a form.
Hence, for odd $p$, $|\mathcal S_j|=p^{jd}$, and as $|\mathcal A_{jd}(p)|/|\mathrm{Sym}_{jd}(\F_p)|=p^{-jd}$, the right-hand side
of~\eqref{eq:Zj} is $1+o(1)$. For $p=2$ we have $|\mathcal A_{jd}(2)|=2^{jd(jd+1)/2}=|\mathrm{Sym}_{jd}(\F_2)|$, and $\Theta C_j$ is
moreover alternating if and only if $\lambda(xy^2)=0$ for all $y\in R_j$. The map $y\mapsto xy^2$ is $\F_2$-linear on $R_j$, so the
$\lambda$ satisfying this condition are those that vanish on its image, and they form a space whose dimension equals that of its kernel.
If $xy^2\equiv0\pmod{\beta^j}$, then $\beta^{j-1}\mid y^2$, as $\beta^2\nmid x$, so that $\beta^{\lceil(j-1)/2\rceil}\mid y$. Hence the
kernel lies in $\beta^{\lceil(j-1)/2\rceil}R_j$, which has dimension $(j-\lceil(j-1)/2\rceil)d=\lceil j/2\rceil d$, and
$|\mathcal S_j|\le2^{\lceil j/2\rceil d}\le2^{\lceil k/2\rceil d}$ for $j\le k$.

Since $|Z_0|=1$, summing~\eqref{eq:Zj} over $1\le j\le k$ gives $\E\big(\lvert\ker\beta(\bar A)^k\rvert\,\mathbf 1_{\mathcal N}\big)\le k+1+o(1)$
for odd $p$, and $\le(k+1)\,2^{\lceil k/2\rceil d}+o(1)$ for $p=2$. By Markov's inequality,
\[
\Pr\bigl(\beta^k\mid\bar\varphi\bigr)\le\Pr\bigl(\mathcal N^c\bigr)+p^{-kd}\,\E\big(\lvert\ker\beta(\bar A)^k\rvert\,\mathbf 1_{\mathcal N}\big),
\]
which is at most $(k+1)p^{-kd}+o(1)$ for odd $p$, and at most $(k+1)2^{-\lfloor k/2\rfloor d}+o(1)$ for $p=2$. Summing over the $\pi_p(d)$
polynomials $\beta$ gives (a), since $p^{-kd}\le p^{-\lfloor k/2\rfloor d}$.

\step{3}{We reduce (b) to events $\mathcal Y_{a,E}$ on which $\beta^a$ exactly divides $\bar\varphi$, for a fixed $\beta$ and a fixed $a$.}
Let $\eta>0$, and choose $k$ with $(k+1)\pi_p(d)p^{-\lfloor k/2\rfloor d}<\eta$. Then $\mathcal X_E$ is contained in the union of the
event of (a) for this $k$ and the event that $G$ is controllable and $\exp(\mathfrak M_{(p,\beta)})\ge p^E$ for some irreducible $\beta$
of degree $d$ with $\beta^k\nmid\bar\varphi$. Since $\mathfrak M_{(p,\beta)}=0$ unless $\beta\mid\bar\varphi$, and
$\Pr(\mathcal N^c)\le e^{-n/8}$, it therefore suffices to show, for each irreducible $\beta$ of degree $d$ and each $1\le a<k$, that
$\lim_E\limsup_n\Pr(\mathcal Y_{a,E})=0$, where $\mathcal Y_{a,E}$, or $\mathcal Y$ for short, is the event that $G$ is controllable,
$\beta^a\,\|\,\bar\varphi$, $\mathcal N$ holds and $\exp(\mathfrak M_{(p,\beta)})\ge p^E$. Put $r:=ad$.

\step{4}{We bound $\Pr(\mathcal Y)$ by the mean of a function of the Gram data of a uniformly random matrix modulo $p^E$.}
On $\mathcal Y$ let $\Phi:=\Phi_\beta$ be the Hensel factor of $\varphi$ at $\beta$, so that $\bar\Phi=\beta^a$, and let
$N:=N_\beta=\ker\Phi(A)\subseteq\Z_p^n$ be as in Lemma~\ref{lem:hensel}: $N$ is $A$-stable and unimodular of rank $r$, and its reduction
modulo $p$ is $\bar N=\ker\beta(\bar A)^a$. As $N$ is a direct summand of $\Z_p^n$ (Lemma~\ref{lem:loccomp}(a)), $N/p^EN$ is a submodule of
$(\Z/p^E)^n$. Call $V\in(\Z/p^E)^{n\times r}$ a \emph{basis of $N$ modulo $p^E$} if its columns form a basis of $N/p^EN$. There are
$|GL_r(\Z/p^E)|=p^{Er^2}\prod_{i=1}^r(1-p^{-i})$ of them, and each such $V$ has the following properties. It is the reduction of a
$\Z_p$-basis of $N$ (lift its columns and apply Nakayama's lemma). It is $s$-robust, since its reduction modulo $p$ has rank $r$ and column
space $\ker\beta(\bar A)^a$, and $\mathcal N$ holds. Its Gram matrix $\Theta_V:=V^{\top}V$ is invertible modulo $p$, since $N$ is
unimodular. Finally, ${AV\equiv VB_V\pmod{p^E}}$ for a unique $B_V\in M_r(\Z/p^E)$, since $N$ is $A$-stable and $V$ maps $(\Z/p^E)^r$
isomorphically onto $N/p^EN$. By Lemma~\ref{lem:hensel}, applied to a basis of $N$ lifting $V$, and since
$\exp(\mathfrak M_{(p,\beta)})\ge p^E$, we have $\det\mathcal K_r(B_V,\Theta_V^{-1}V^{\top}e)\equiv0\pmod{p^E}$.

If $AV\equiv VB\pmod{p^E}$, then $\Theta_VB\equiv V^{\top}AV$ is symmetric, with even diagonal if $p=2$ (the diagonal entries
$y^{\top}Ay=2\sum_{i<l}a_{il}y_iy_l$ are even). For $V$ with $\Theta_V$ invertible modulo $p$, call $B\in M_r(\Z/p^E)$
\emph{compatible with $V$} if $\Theta_VB$ is symmetric, with even diagonal if $p=2$. As $B\mapsto\Theta_VB$ is a bijection, the number of
compatible $B$ is $p^{Er(r+1)/2}$ for odd $p$ and $2^{Er(r+1)/2-r}$ for $p=2$. In both cases it equals $p^{Er^2}/|\mathcal A_r(p^E)|$.
Therefore
\[
\mathbf 1_{\mathcal Y}\le\frac1{|GL_r(\Z/p^E)|}\sum_V\sum_B\mathbf 1\big\{AV\equiv VB\pmod{p^E}\big\}\,
\mathbf 1\bigl\{\det\mathcal K_r(B,\Theta_V^{-1}V^{\top}e)\equiv0\pmod{p^E}\bigr\},
\]
where $V$ runs over the $s$-robust matrices in $(\Z/p^E)^{n\times r}$ with $\Theta_V$ invertible modulo $p$, and $B$ over the matrices
compatible with $V$. For $(\Theta,u)\in\mathrm{Sym}_r(\Z/p^E)\times(\Z/p^E)^r$ let $h(\Theta,u):=0$ if $\Theta$ is not invertible modulo
$p$, and otherwise let $h(\Theta,u)$ be the proportion of the matrices $B$ with $\Theta B$ symmetric (with even diagonal if $p=2$) for which
$\det\mathcal K_r(B,\Theta^{-1}u)\equiv0\pmod{p^E}$. Taking expectations with~\eqref{eq:fourierfixed} for $m=r$, and dropping the robustness
condition, which only increases the sum, we get
\[
\Pr(\mathcal Y)\le c_{p,r}\,\E\big(h(\Theta_V,V^{\top}e)\big)\big(1+o(1)\big),\qquad c_{p,r}:=\prod_{i=1}^r\big(1-p^{-i}\big)^{-1},
\]
where $V$ is uniform in $(\Z/p^E)^{n\times r}$.

\step{5}{We identify the limit of this mean as the Haar measure of a set $\Omega_E$ of $p$-adic data.}
By Lemma~\ref{lem:gram}(a), $(\Theta_V,V^{\top}e)$ converges in total variation to the uniform distribution on $\mathcal G_r(p^E)$, which is
$\mathrm{Sym}_r(\Z/p^E)\times(\Z/p^E)^r$ for odd $p$ and $\{(\Theta,u):\ \Theta_{ii}\equiv u_i\pmod2\ \text{for all }i\}$ for $p=2$. We
describe the limit $p$-adically, by a Haar measure, as for the random $p$-adic matrices of Friedman--Washington (1989)~\cite{FW89} in
connection with the heuristics of Cohen--Lenstra (1984)~\cite{CL84}. Let
$X_r:=\mathrm{Sym}_r(\Z_p)\times\Z_p^r\times\mathrm{Sym}_r(\Z_p)$, with points denoted $(\Theta,u,H)$. For $p=2$ we replace $X_r$ by its
open subset on which $\Theta_{ii}\equiv u_i$ and $H_{ii}\equiv0\pmod2$ for all $i$. Let $\mu_r$ be the normalized Haar measure on $X_r$,
and put
\[
\Omega_E:=\big\{(\Theta,u,H)\in X_r:\ \bar\Theta\text{ is invertible and }\det\mathcal K_r(\Theta^{-1}H,\Theta^{-1}u)\equiv0
\pmod{p^E}\big\}.
\]
The reduction of $\mu_r$ modulo $p^E$ is the product of the uniform distribution on $\mathcal G_r(p^E)$ and the uniform distribution on the
symmetric matrices $H$ modulo $p^E$ (with even diagonal if $p=2$). Moreover, for fixed $\Theta$ invertible modulo $p$, $B=\Theta^{-1}H$ is
uniform among the $B$ with $\Theta B$ symmetric (with even diagonal if $p=2$) when $H$ is uniform. As $\Omega_E$ depends only on
reductions modulo $p^E$, it follows that $\E(h(\Theta_V,V^{\top}e))\to\mu_r(\Omega_E)$, and hence
$\limsup_n\Pr(\mathcal Y_{a,E})\le c_{p,r}\,\mu_r(\Omega_E)$.

\step{6}{We show that $\mu_r(\Omega_E)\to0$ as $E\to\infty$, which completes the proof of (b).}
The sets $\Omega_E$ decrease in $E$, and for $(\Theta,u,H)\in\bigcap_E\Omega_E$ the determinant of
$\mathcal K_r(\Theta^{-1}H,\Theta^{-1}u)$ is divisible by every power of $p$, so it vanishes. Hence $\bigcap_E\Omega_E$ is contained in the
zero set of
\begin{align*}
F(\Theta,u,H)&:=\det\Bigl(\bigl(\operatorname{adj}(\Theta)H\bigr)^i\operatorname{adj}(\Theta)\,u\Bigr)_{0\le i<r}\\
&=\det(\Theta)^{r(r+1)/2}\det\mathcal K_r\big(\Theta^{-1}H,\Theta^{-1}u\big),
\end{align*}
the determinant of the matrix with columns $(\operatorname{adj}(\Theta)H)^i\operatorname{adj}(\Theta)u$, $0\le i<r$. It is a polynomial in
the entries of $(\Theta,u,H)$, and it is not the zero polynomial: at $\Theta=I$, $u=(1,\dots,1)^{\top}$ and
$H=\operatorname{diag}(0,2,4,\dots,2r-2)$ it is a nonzero Vandermonde determinant. The zero set of a nonzero polynomial in $t$ variables has
Haar measure $0$ in $\Z_p^{t}$ (by induction on $t$ and Fubini's theorem, since a nonzero polynomial in one variable has finitely many zeros
and Haar measure on $\Z_p$ has no atoms), and $\mu_r$ is a multiple of the restriction of Haar measure to an open set. Hence
$\mu_r(\bigcap_E\Omega_E)=0$, so $\mu_r(\Omega_E)\to0$ as $E\to\infty$. With Steps~3 and~5, $\limsup_n\Pr(\mathcal X_E)<2\eta$ for all
large $E$, and as $\eta>0$ was arbitrary, this proves (b).
\end{proof}

Van Werde (2025)~\cite{VW25} determined the limiting distribution of the $p$-parts of $\coker W$ for random directed graphs with random edge
weights, and pointed out that undirected graphs, whose adjacency matrices are symmetric, are harder to handle. Proposition~\ref{prop:tight}
is a step in the undirected case: at a fixed prime, the exponents of the components of $\coker W$ of fixed degree are tight, and the proof
compares these exponents with the orders of the cokernels of the Krylov matrices $\mathcal K_r(\Theta^{-1}H,\Theta^{-1}u)$ of Haar-random
$p$-adic data.

\section{The prime \texorpdfstring{$2$}{2}}\label{sec:two}

The second component, Theorem~\ref{thm:smooth}, asserts that, with probability $1-o(1)$, no matrix in $\QG$ other than a permutation matrix
has a level whose prime factors are all at most $n^{1/2-\varepsilon}$. In its proof the prime $2$ must be treated separately from the odd
primes, and this is the purpose of this section. Seidel switching (Lemma~\ref{lem:switch}) is useless at $p=2$, because the Seidel matrix
$J-I-2A$ is constant modulo $2$. Local failures (Definition~\ref{def:Ep}) are defined only at odd primes, and their analogue at $p=2$ always
occurs. The $\bar A$-stable subspace $(\F_2[\bar A]\bar e)^{\perp}$ is orthogonal to $\bar e$ and nonzero, since
$\rank_2W\le\lceil n/2\rceil<n$ (Lemma~\ref{lem:krylovlaw}(a), whose proof applies to every graph), so it contains a simple
$\bar A$-submodule $U$. Every such $U$ is totally isotropic, because $U=\F_2[\bar A]u$ for every nonzero $u\in U$, and for $l\ge0$ and
$y:=\bar A^lu\in U$ we have $u^{\top}\bar A^{2l+1}u=y^{\top}\bar Ay=0$, as $\bar A$ is symmetric with zero diagonal, and
$u^{\top}\bar A^{2l}u=y^{\top}y=\bar e^{\top}y=0$, as squaring is the identity on $\F_2$ and $U\perp\bar e$. The $2$-adic criterion of
Wang (2017)~\cite{W17} makes every level odd when
$2^{-\lfloor n/2\rfloor}\det W$ is odd (Remark~\ref{rem:loc-wang}), but this condition fails with probability bounded away from $0$
(Remark~\ref{rem:wangfails}), so it cannot exclude the prime $2$ from all levels with high probability.

Instead, we show that the local components of large degree above $2$ are unlikely to contribute to a level. Recall from
Section~\ref{sec:local} that a prime divides the level of $Q\in\QG$ only through local components $\mathcal T_\m$ of the finite ring
$\mathcal T\cong D$ that are not chain rings (Proposition~\ref{lem:levellocal}(i),(ii)). For an irreducible $\beta\in\F_2[x]$ we write
$D_{(2,\beta)}$ and $\mathcal T_{(2,\beta)}$ for the components of $D$ and $\mathcal T$ at $\m=(2,\tilde\beta)$, as in
Section~\ref{subsec:loc-lagr}. The main result of the section, Proposition~\ref{prop:twolarge}, states that, for every $d_0\ge2$, with
probability at least $1-2^{2-d_0}-O(n^22^{-n/2})$, either $G$ is not controllable or $\mathcal T_{(2,\beta)}$ is a chain ring for every
irreducible $\beta$ of degree at least $d_0$ with $D_{(2,\beta)}\ne0$. Such components lie in no set $\mathcal F(Q,2)$, so they contribute
to no level. In the proof of Theorem~\ref{thm:smooth} in Section~\ref{sec:smooth}, this leaves only the finitely many components above $2$
of degree less than $d_0$. Their contributions to the levels are bounded by the exponents of the corresponding components of $\coker W$
(Proposition~\ref{lem:levellocal}(i)), which are bounded in probability by Proposition~\ref{prop:tight}.

The proof is an exact computation over $\F_2$. We first recall the parity facts behind the $2$-adic theory of the walk matrix
(Section~\ref{subsec:loc-two}). At $p=2$ there is no universality problem: $\bar A$ is uniformly distributed on the space $\mathrm{Alt}_n$
of alternating (that is, symmetric with zero diagonal) $n\times n$ matrices over $\F_2$, and $A$ is its $0/1$ lift. The walk matrix $W$ is
the Krylov matrix of $(A,e)$: its columns are the first $n$ terms of the Krylov sequence $e,Ae,A^2e,\dots$, which span the Krylov subspace
of $A$ generated by $e$ (Krylov (1931)~\cite{Kry31}). Modulo $2$ this sequence can be exposed step by step with explicit transition
probabilities: given linearly independent $e,\bar Ae,\dots,\bar A^{k-1}e$, the next vector $\bar A^ke$ is uniformly distributed on an
explicit subspace (Lemma~\ref{lem:krylov}). This determines the law of the Krylov space $\F_2[\bar A]e$ and of the minimal polynomial of $e$
(Lemma~\ref{lem:krylovlaw}). By Corollary~\ref{cor:chain} and Lemma~\ref{lem:wangLC}, $\mathcal T_{(2,\beta)}$ is a chain ring as soon as a
linear form built from the parities of the halved walk counts $\frac12e^{\top}A^ke$ does not vanish on an explicit space of polynomials, and
Lemma~\ref{lem:Lambda} expresses this linear form through the Krylov data and a quadratic form. Conditionally on the Krylov data, the
relevant parity bits are uniform on an explicit affine space, for every value of the dimension of the Krylov space, and this space is large
outside an event of probability $O(n^22^{-n/2})$ (Lemma~\ref{lem:condlaw}). Apart from two exceptional events, a fixed $\beta$ of degree
$d\ge d_0$ then gives a component that is not a chain ring with probability at most $2^{1-2d}$, and summing over $\beta$ gives
Proposition~\ref{prop:twolarge}.

\subsection{The ring \texorpdfstring{$\mathcal T/2\mathcal T$}{T/2T}}\label{subsec:loc-two}

The $2$-adic theory of the walk matrix is due to Wang, who proved that $2^{\lfloor n/2\rfloor}\mid\det W$ (Wang (2013)~\cite{W13}) and that
$G$ is DGS if $2^{-\lfloor n/2\rfloor}\det W$ is odd and squarefree (Wang (2017)~\cite{W17}). In this subsection, in which $G$ is any graph
on $n$ vertices, we recall the parity facts behind the first result and recover the $2$-adic part of the second in the language of
Section~\ref{sec:local}. Over $\F_2$ the matrix $\bar A$ is symmetric with zero diagonal, so $y^{\top}\bar Ay=0$ for all $y\in\F_2^n$.
Throughout this section we write $\langle y,y'\rangle:=y^{\top}y'$ for $y,y'\in\F_2^n$, and note that
$\langle y,y\rangle=\langle\bar e,y\rangle$. The first lemma says that the numbers of walks of positive length in a graph are even.

\begin{lemma}\label{lem:par}
For every graph and every $i\ge1$ we have $\langle\bar e,\bar A^i\bar e\rangle=0$ in $\F_2$, that is, $e^{\top}A^ie$ is even. Consequently,
for all $i,j\ge0$, $\langle\bar A^i\bar e,\bar A^j\bar e\rangle=0$ unless $i=j=0$.
\end{lemma}

\begin{proof}
We prove the first assertion by induction on $i$. If $i=2l+1$ is odd, then $\langle\bar e,\bar A^i\bar e\rangle=y^{\top}\bar Ay=0$ for
$y:=\bar A^l\bar e$. If $i=2l$ with $l\ge1$, then $\langle\bar e,\bar A^i\bar e\rangle=\langle y,y\rangle=\langle\bar e,y\rangle$ for
$y:=\bar A^l\bar e$, which vanishes by induction, as $l<i$. The last assertion follows from
$\langle\bar A^i\bar e,\bar A^j\bar e\rangle=\langle\bar e,\bar A^{i+j}\bar e\rangle$.
\end{proof}

We use the walk counts $N_k:=e^{\top}A^ke$, the numbers of walks of length $k$ in $G$. Thus $N_0=n$, and $N_k$ is even for $k\ge1$ by
Lemma~\ref{lem:par}. Since $\operatorname{adj}(xI-A)=\varphi(x)(xI-A)^{-1}$ and $(xI-A)^{-1}=\sum_{k\ge0}A^kx^{-k-1}$ in $\Z((x^{-1}))$,
we have
\begin{equation}\label{eq:walkgf}
\frac{\psi(x)}{\varphi(x)}=e^{\top}(xI-A)^{-1}e=\sum_{k\ge0}N_kx^{-k-1}.
\end{equation}
The second lemma computes $\mathcal T/2\mathcal T$ and decides which local components of $\mathcal T$ above $2$ are nonzero.

\begin{lemma}\label{lem:Ttwo}
Let $G$ be any graph on $n$ vertices, and write $\varphi=\sum_{i=0}^nc_ix^{n-i}$.
\begin{enumerate}[label=(\alph*)]
\item $c_i$ is even for every odd $i$. Hence $\bar\varphi=\gamma^2$ if $n$ is even and $\bar\varphi=x\gamma^2$ if $n$ is odd, for a unique
monic $\gamma\in\F_2[x]$ of degree $\lfloor n/2\rfloor$, namely $\gamma=\sum_j\bar c_{2j}x^{\lfloor n/2\rfloor-j}$.
\item $\bar\psi=0$ if $n$ is even and $\bar\psi=\gamma^2$ if $n$ is odd. Consequently
$\mathcal T/2\mathcal T=\F_2[x]/(\bar\varphi,\bar\psi)=\F_2[x]/(\gamma^2)$.
\item If $G$ is controllable, $\beta\in\F_2[x]$ is irreducible and $\m:=(2,\tilde\beta)$, then $D_\m\ne0$ if and only if $\beta\mid\gamma$.
\end{enumerate}
\end{lemma}

\begin{proof}
(a) The number $(-1)^ic_i$ is the sum of the principal minors of order $i$ of $A$. A principal submatrix of odd order $i$ is symmetric with
zero diagonal, so its reduction modulo $2$ is alternating. Alternating matrices have even rank, so this reduction is singular and the minor
is even. (Equivalently, each elementary subgraph on an odd number of vertices contains a cycle, which contributes a factor $2$ to its
term in Sachs' theorem, see Cvetkovi\'c--Doob--Sachs (1980)~\cite{CDS80}.) Hence $\bar\varphi=\sum_j\bar c_{2j}x^{n-2j}$. As squaring
is additive on $\F_2[x]$ and fixes $\F_2$, this equals $\bigl(\sum_j\bar c_{2j}x^{n/2-j}\bigr)^2$ for even $n$ and
$x\bigl(\sum_j\bar c_{2j}x^{(n-1)/2-j}\bigr)^2$ for odd $n$. The polynomial $\gamma$ is monic since $c_0=1$, and unique since squaring is
injective on $\F_2[x]$. (The evenness of the odd-index coefficients is used in the same way by Wang (2017)~\cite{W17}.)

(b) By~\eqref{eq:walkgf} and $N_0=n$, the Laurent series $\psi-nx^{-1}\varphi=\varphi\sum_{k\ge1}N_kx^{-k-1}$ has even coefficients, that
is, $x\bar\psi=\bar n\,\bar\varphi$ in $\F_2[x]$. For even $n$ this gives $\bar\psi=0$, and for odd $n$ it gives $x\bar\psi=x\gamma^2$, that
is, $\bar\psi=\gamma^2$. Finally
$\mathcal T/2\mathcal T=\Z[x]/(2,\varphi,\psi)=\F_2[x]/(\bar\varphi,\bar\psi)$, and both $(\gamma^2,0)$ and $(x\gamma^2,\gamma^2)$ generate the
ideal $(\gamma^2)$.

(c) By Proposition~\ref{prop:D}, $D_\m\cong\mathcal T_\m$, which is nonzero if and only if $\m\supseteq(\varphi,\psi)$, that is, if and only if
$\beta$ divides $\bar\varphi$ and $\bar\psi$. By (a) and (b) this means $\beta\mid\gamma$.
\end{proof}

For controllable $G$, Lemma~\ref{lem:Ttwo} and Corollary~\ref{cor:chain} recover the $2$-adic results of Wang as follows.

\begin{remark}\label{rem:loc-wang}
Let $G$ be controllable. By Lemma~\ref{lem:Ttwo}(b), $\mathcal T/2\mathcal T\cong\F_2[x]/(\gamma^2)$ has order $2^{2\lfloor n/2\rfloor}$.
Since $\mathcal T/2\mathcal T=\mathcal T_{(2)}/2\mathcal T_{(2)}$ and $|\mathcal T_{(2)}|=2^{2v_2(\det W)}$ by Proposition~\ref{prop:D}, we
get $v_2(\det W)\ge\lfloor n/2\rfloor$, with equality if and only if $2\mathcal T_{(2)}=0$. This recovers the divisibility
$2^{\lfloor n/2\rfloor}\mid\det W$ of Wang (2013)~\cite{W13}. If equality holds, that is, if $2^{-\lfloor n/2\rfloor}\det W$ is odd, then
$\mathcal T_{(2)}=\mathcal T/2\mathcal T\cong\F_2[x]/(\gamma^2)$, so every nonzero local component above $2$ is isomorphic to
$\F_2[x]/(\beta^{2c})$ for some $\beta\mid\gamma$, with $c:=v_\beta(\gamma)$, and its maximal ideal is generated by $\beta$. By
Corollary~\ref{cor:chain}(a) for the nonzero components, and by Proposition~\ref{lem:levellocal}(ii) for the others, $\mathcal F(Q,2)$ is
then empty for every $Q\in\QG$, so that every level is odd by Proposition~\ref{lem:levellocal}(i). This is the $2$-adic part of the main
theorem of Wang (2017)~\cite{W17}.
\end{remark}

\subsection{Notation and the Krylov chain}
In the rest of this section, $G=G(n,1/2)$ with $n\ge3$, and we write $e$ also for $\bar e\in\F_2^n$. For a subspace $U\subseteq\F_2^n$ put
$\mathrm{Alt}_U:=\{B\in\mathrm{Alt}_n:\ BU=0\}$. We write $y*z$ for the coordinatewise product of $y,z\in\F_2^n$, and $U*U':=\{y*z:\
y\in U,\ z\in U'\}$. For $B\in\mathrm{Alt}_n$ let $q_B(y):=\sum_{i<l}B_{il}y_iy_l$. This is a quadratic form on $\F_2^n$ with polar form
$q_B(y+z)-q_B(y)-q_B(z)=\langle y,Bz\rangle$. We put $q_A:=q_{\bar A}$ and $q_J:=q_{\bar J-I}$, so that $q_J(y)=\sum_{i<l}y_iy_l$, with polar
form $\langle y,z\rangle+\langle e,y\rangle\langle e,z\rangle$.

Let $v_k:=\bar A^ke$ and $K_k:=\operatorname{span}(v_0,\dots,v_{k-1})$. Let $K:=\F_2[\bar A]e$ be the Krylov space of $e$ and $k_0:=\dim K$
(so $k_0=\rank_2W$), and let $\mu\in\F_2[x]$ be the minimal polynomial of $e$ with respect to $\bar A$, so that $\deg\mu=k_0$ and
$v_0,\dots,v_{k_0-1}$ is a basis of $K$. The \emph{Krylov data} of $G$ is the tuple $\mathcal D:=(v_0,\dots,v_{k_0})$. It determines $k_0$,
$K$, $\mu$ and the action of $\bar A$ on $K$. Put $\kappa:=\lceil n/2\rceil$ and $j:=\kappa-k_0$. By Lemma~\ref{lem:Ttwo}(a),
$\bar\varphi=\gamma^2$ if $n$ is even and $\bar\varphi=x\gamma^2$ if $n$ is odd, for a unique monic $\gamma\in\F_2[x]$. We put
$\gamma_\circ:=\gamma$ if $n$ is even and $\gamma_\circ:=x\gamma$ if $n$ is odd, so that $\gamma_\circ^2=\bar\varphi$, respectively
$\gamma_\circ^2=x\bar\varphi$.

In this notation, the last assertion of Lemma~\ref{lem:par} reads
\begin{equation}\label{eq:vorth}
\langle v_i,v_l\rangle=\langle e,\bar A^{i+l}e\rangle=0\qquad\text{unless }i=l=0 ,
\end{equation}
while $\langle v_0,v_0\rangle=\langle e,e\rangle\equiv n\pmod2$.
The next lemma describes one step of the exposure of the Krylov sequence.

\begin{lemma}\label{lem:krylov}
Let $k\ge1$. Conditionally on $(v_0,\dots,v_{k-1})$, on the event $\dim K_k=k$, the vector $v_k$ is uniformly distributed on $K_k^{\perp}$.
\end{lemma}

\begin{proof}
Fix the values of $v_0=e,v_1,\dots,v_{k-1}$, linearly independent. The conditioning event is $\{\bar A:\ \bar Av_i=v_{i+1},\ 0\le i\le k-2\}$,
a coset $\bar A_0+\mathrm{Alt}_{K_{k-1}}$ (with $K_0=0$), on which $\bar A$ is uniformly distributed, and
$v_k=\bar Av_{k-1}=\bar A_0v_{k-1}+Bv_{k-1}$ with $B$ uniform on $\mathrm{Alt}_{K_{k-1}}$. The values of $v_k$ lie in $K_k^{\perp}$:
$\langle\bar Av_{k-1},v_i\rangle=\langle v_{k-1},v_{i+1}\rangle=0$ for $i\le k-2$ by~\eqref{eq:vorth}, and $\langle\bar Av_{k-1},v_{k-1}\rangle=0$
as $\bar A$ is alternating. The linear map $B\mapsto Bv_{k-1}$ maps $\mathrm{Alt}_{K_{k-1}}$ onto $K_k^{\perp}$: given $z\in K_k^{\perp}$, choose
$y$ with $\langle y,v_{k-1}\rangle=1$ and $\langle y,v_i\rangle=0$ for $i\le k-2$, and take $B=zy^{\top}+yz^{\top}$, which is alternating,
kills $K_{k-1}$ and maps $v_{k-1}$ to $z$. Hence $v_k$ is uniform on $K_k^{\perp}$.
\end{proof}

Exposing the whole Krylov sequence in this way gives the law of the Krylov data.

\begin{lemma}\label{lem:krylovlaw}
\begin{enumerate}[label=(\alph*)]
\item If $n$ is even, then $K\subseteq K^{\perp}$. If $n$ is odd, then $\mu=x\mu_1$, where $\mu_1$ (of degree $k_0-1$) is the minimal
polynomial of $v_1=\bar Ae$, and $K\cap K^{\perp}=\bar AK=\operatorname{span}(v_1,\dots,v_{k_0-1})$. In both cases $j\ge0$, and there is a
monic $\gamma'\in\F_2[x]$ of degree $j$ with
\[
\bar\varphi=\mu^2\gamma'^2,\quad\gamma=\mu\gamma'\quad(n\text{
even}),\qquad\qquad\bar\varphi=x\mu_1^2\gamma'^2,\quad\gamma=\mu_1\gamma'\quad(n\text{ odd}).
\]
In both cases $\gamma_\circ=\mu\gamma'$.
\item $\Pr(j\ge r)\le\frac43\,4^{-r}$ for every integer $r\ge0$.
\item Conditionally on $(k_0,v_0,\dots,v_{k_0-1})$, the polynomial $\mu$ is uniformly distributed among the monic polynomials of degree
$k_0$ if $n$ is even, and $\mu_1$ is uniformly distributed among the monic polynomials of degree $k_0-1$ if $n$ is odd.
\item Let $P:=e^{\perp}/\F_2e$ and $K^{\circ}:=K/\F_2e$ if $n$ is even, and $P:=e^{\perp}$ and $K^{\circ}:=\bar AK$ if $n$ is odd. Then
$\langle\,,\rangle$ induces a nondegenerate alternating form on $P$, which has dimension $2(\kappa-1)$, and $K^{\circ}$ is a totally isotropic
subspace of $P$ of dimension $k_0-1$. Conditionally on $k_0$, $K^{\circ}$ is uniformly distributed among the totally isotropic subspaces of
$P$ of dimension $k_0-1$.
\end{enumerate}
\end{lemma}

\begin{proof}
(a) If $n$ is even, all $\langle v_i,v_l\rangle$ vanish by~\eqref{eq:vorth}, so $K\subseteq K^{\perp}$. If $n$ is odd, a vector
$y=\sum_{i<k_0}c_iv_i$ satisfies $\langle y,v_0\rangle=c_0$ and $\langle y,v_l\rangle=0$ for $l\ge1$, so
$K\cap K^{\perp}=\operatorname{span}(v_1,\dots,v_{k_0-1})$. Moreover $v_{k_0}\in K$, and $v_{k_0}\perp v_i$ for all $i<k_0$
by~\eqref{eq:vorth}, so $v_{k_0}\in K\cap K^{\perp}$. Writing $v_{k_0}=\sum_{i=1}^{k_0-1}a_iv_i$ shows that $x\mid\mu$, that
$\bar AK=\operatorname{span}(v_1,\dots,v_{k_0})=K\cap K^{\perp}$, and that $\mu_1=\mu/x$ is the minimal polynomial of $v_1$ (it kills $v_1$,
and $v_1,\dots,v_{k_0-1}$ are independent).

All the subspaces below are $\bar A$-stable ($\bar A$ is symmetric). Let $n$ be even. The characteristic polynomial of $\bar A$ on the cyclic
space $K$ is $\mu$. The pairing identifies $\F_2^n/K^{\perp}$ with the dual of $K$, the action of $\bar A$ becoming the transpose, so the
characteristic polynomial on $\F_2^n/K^{\perp}$ is also $\mu$. Let $\chi'$ be the characteristic polynomial on $K^{\perp}/K$, of degree
$n-2k_0\ge0$. Then $\gamma^2=\bar\varphi=\mu^2\chi'$, so $\mu\mid\gamma$, and $\chi'=\gamma'^2$ with $\gamma':=\gamma/\mu$ of degree
$n/2-k_0=j$. Now let $n$ be odd, and put $K':=\bar AK$. In the filtration $0\subseteq K'\subseteq K^{\perp}\subseteq K'^{\perp}\subseteq\F_2^n$
the characteristic polynomials of $\bar A$ on the successive quotients are $\mu_1$ (on the cyclic space $K'$ generated by $v_1$), some
$\chi'$ of degree $n+1-2k_0\ge0$, $x$ (the pairing identifies $K'^{\perp}/K^{\perp}$ with the dual of $K/K'$, on which $\bar A$ acts as $0$
because $\bar AK=K'$), and $\mu_1$ (on $\F_2^n/K'^{\perp}$, dual to $K'$). Hence $x\gamma^2=\bar\varphi=x\mu_1^2\chi'$, so $\mu_1\mid\gamma$, and
$\chi'=\gamma'^2$ with $\gamma':=\gamma/\mu_1$ of degree $(n+1)/2-k_0=j$. Then $\gamma_\circ=x\mu_1\gamma'=\mu\gamma'$.

(b), (c) Fix $k\ge1$ and values $v_0=e,v_1,\dots,v_{k-1}$ with $\dim K_k=k$. By Lemma~\ref{lem:krylov}, $v_k$ is then uniform on
$K_k^{\perp}$, a space of size $2^{n-k}$, and $k_0=k$ if and only if $v_k\in K_k\cap K_k^{\perp}$. By the computation in (a) (with $k$ in
place of $k_0$), $K_k\cap K_k^{\perp}$ is $K_k$ if $n$ is even and $\operatorname{span}(v_1,\dots,v_{k-1})$ if $n$ is odd. Hence the
conditional probability that $k_0=k$ is $2^{2k-n}$ for even $n$ and $2^{2k-1-n}$ for odd $n$, that is, $4^{k-\kappa}$ in both cases.
Conditionally on $k_0=k$, the vector
$v_k=\sum_{i<k}a_iv_i$ is uniform on $K_k\cap K_k^{\perp}$, that is, $(a_0,\dots,a_{k-1})$ is uniform on $\F_2^k$ for even $n$, while $a_0=0$
and $(a_1,\dots,a_{k-1})$ is uniform on $\F_2^{k-1}$ for odd $n$. As $\mu=x^k+\sum_{i<k}a_ix^i$, this proves (c). For (b), $j\ge r$ means
$k_0\le\kappa-r$, and summing the conditional probabilities of stopping at the steps $k\le\kappa-r$ gives
\[
\Pr(j\ge r)\le\sum_{k=1}^{\kappa-r}4^{k-\kappa}\le\sum_{i\ge r}4^{-i}=\tfrac43\,4^{-r}.
\]

(d) If $n$ is even, then $e\in e^{\perp}$ and the radical of the restriction of $\langle\,,\rangle$ to $e^{\perp}$ is
$e^{\perp}\cap(e^{\perp})^{\perp}=\F_2e$. If $n$ is odd, then $e\notin e^{\perp}$ and the restriction to $e^{\perp}$ is nondegenerate. In both
cases the induced form on $P$ is alternating, since $\langle y,y\rangle=\langle e,y\rangle=0$ for $y\in e^{\perp}$, and $\dim P=2(\kappa-1)$.
By (a), $K^{\circ}$ is totally isotropic of dimension $k_0-1$ (for odd $n$, $K'=\bar AK\subseteq e^{\perp}$ by~\eqref{eq:vorth}).

Let $1\le k\le\kappa$ and let $U$ be a totally isotropic subspace of $P$ of dimension $k-1$. Let $U'$ be the preimage of $U$ in $e^{\perp}$
if $n$ is even, and $U':=U$ if $n$ is odd. For a sequence $w=(w_1,\dots,w_{k-1})$ with
$w_i\in U'\setminus\operatorname{span}(e,w_1,\dots,w_{i-1})$ for all $i$, let $F(w)$ be the event that $v_i=w_i$ for $1\le i<k$ and
$v_k\in K_k$. The event $\{k_0=k,\ K^{\circ}=U\}$ is the disjoint union of the events $F(w)$. Indeed, on this event $v_1,\dots,v_{k-1}$
lie in $U'$ (by (a) if $n$ is odd) and form such a sequence. Conversely, on $F(w)$ we have $k_0=k$, and $K^{\circ}$ is spanned by the
images of $w_1,\dots,w_{k-1}$ in $P$ (by (a) if $n$ is odd), which are linearly independent and lie in $U$, so that $K^{\circ}=U$. For
each such $w$, the vector $w_i$ is orthogonal to $\operatorname{span}(e,w_1,\dots,w_{i-1})$, because $U'$ is totally isotropic and, if $n$
is odd, contained in $e^{\perp}$. Hence Lemma~\ref{lem:krylov}, applied successively, and the
stopping probability computed in (b) give $\Pr(F(w))=4^{k-\kappa}\prod_{i=1}^{k-1}2^{-(n-i)}$ for both parities of $n$.
The number of sequences $w$ is $\prod_{i=1}^{k-1}(2^k-2^i)$ for even $n$ and
$\prod_{i=1}^{k-1}(2^{k-1}-2^{i-1})$ for odd $n$. Neither depends on $U$, which proves (d).
\end{proof}

The proof of Lemma~\ref{lem:krylovlaw} also gives the law of $\rank_2W=k_0$, and we note a consequence for the criterion of Wang.

\begin{remark}\label{rem:wangfails}
By the proof of Lemma~\ref{lem:krylovlaw}(b), the conditional probability of stopping at step $k$ is $4^{k-\kappa}$ for both parities
of $n$. Hence
$\Pr(k_0=\kappa)=\prod_{i=1}^{\kappa-1}(1-4^{-i})$, which tends to $\prod_{i\ge1}(1-4^{-i})<0.69$. By the Smith normal form,
$2^{n-\rank_2W}$ divides $\det W$, so the hypothesis of Wang (2017)~\cite{W17} that $2^{-\lfloor n/2\rfloor}\det W$ is odd forces
$k_0\ge\lceil n/2\rceil$, that is, $k_0=\kappa$. So this hypothesis fails for $G(n,1/2)$ with probability at least $0.31$ for large $n$ (compare
the heuristics of Lvov--Van Werde (2026)~\cite{LVW26}), whereas, for every $d_0\ge2$, Proposition~\ref{prop:twolarge} shows that its
$2$-adic conclusion holds at all local components above $2$ of degree at least $d_0$, except with probability $O(2^{-d_0})+o(1)$.
\end{remark}

\subsection{Components of \texorpdfstring{$\mathcal T$}{T} at the prime \texorpdfstring{$2$}{2}}
We now give a sufficient condition for $\mathcal T_{(2,\beta)}$ to be a chain ring in terms of the parities of the halved walk counts
$N_k/2$, $k\ge1$, where $N_k=e^{\top}A^ke$ is as in~\eqref{eq:walkgf}. For $f\in\F_2[x]$ choose a lift $\tilde f\in\Z[x]$ and put
\[
\omega(f):=\tfrac12\Big(e^{\top}\tilde f(A)e-\epsilon_n\,n\tilde f(0)\Big)\bmod2,\qquad\epsilon_n:=\begin{cases}1,&n\text{
odd},\\0,&n\text{ even}.\end{cases}
\]
Since $e^{\top}g(A)e=\sum_kg_kN_k\equiv ng(0)\pmod2$ for $g\in\Z[x]$, the expression in parentheses is even, and replacing $\tilde f$ by
$\tilde f+2g$ changes it by $2(e^{\top}g(A)e-\epsilon_nng(0))\equiv2(1-\epsilon_n)ng(0)\equiv0\pmod4$. So $\omega$ is a well-defined
$\F_2$-linear form on $\F_2[x]$, with $\omega(x^k)=N_k/2\bmod2$ for $k\ge1$, $\omega(1)=n/2\bmod2$ for even $n$, and $\omega(1)=0$ for odd
$n$. Recall the polynomial $\gamma_\circ$, with $\gamma_\circ^2=x^{\epsilon_n}\bar\varphi$. In terms of $\omega$ and $\gamma_\circ$, the
sufficient condition reads as follows.

\begin{lemma}\label{lem:wangLC}
Let $G$ be controllable, let $\beta\in\F_2[x]$ be irreducible with $\beta\mid\gamma$, and let $\m:=(2,\tilde\beta)$. If $\omega$ does not
vanish identically on $(\gamma_\circ^2/\beta)\F_2[x]$, then $\mathcal T_\m$ is a chain ring.
\end{lemma}

\begin{proof}
By Lemma~\ref{lem:Ttwo}(a),(b), $\beta$ divides $\bar\varphi$ and $\bar\psi$, so $\m\supseteq(\varphi,\psi)$. By Corollary~\ref{cor:chain} it
suffices to show that $\varphi\notin\m^2$ or $\psi\notin\m^2$, that is, to find an element of the ideal $(\varphi,\psi)$ that does not lie
in $\m^2$.

\step{1}{We find a polynomial $h$ with $2h\in(\varphi,\psi)$ that encodes $\omega$.} Multiplying~\eqref{eq:walkgf} by
$x^{\epsilon_n}\varphi$ and using $N_0=n$, we obtain
\[
x^{\epsilon_n}\psi-\epsilon_nn\varphi=\varphi\sum_{k\ge\epsilon_n}N_kx^{\epsilon_n-k-1}.
\]
The numbers $N_k$ in this sum are even (for even $n$ also $N_0=n$), so the right-hand side has even coefficients. Hence
$h:=(x^{\epsilon_n}\psi-\epsilon_nn\varphi)/2$ lies in $\Z[x]$, and $2h\in(\varphi,\psi)$. Dividing the last identity by
$2x^{\epsilon_n}\varphi$ in $\Z((x^{-1}))$ and reducing modulo $2$, with $x^{\epsilon_n}\bar\varphi=\gamma_\circ^2$, we get
$\bar h/\gamma_\circ^2=\sum_{k\ge\epsilon_n}\omega(x^k)x^{-k-1}$. As
$\omega(1)=0$ for odd $n$, it follows that, for every $g\in\F_2[x]$, $\omega(g)$ is the coefficient of $x^{-1}$ in the Laurent series
$g\bar h/\gamma_\circ^2\in\F_2((x^{-1}))$.

\step{2}{We show that $2h\notin\m^2$, which proves the lemma by Corollary~\ref{cor:chain}.} For $g=(\gamma_\circ^2/\beta)t$ with
$t\in\F_2[x]$, Step~1 shows that $\omega(g)$ is the coefficient of $x^{-1}$ in $t\bar h/\beta$. If $\beta$ divided $\bar h$, then
$t\bar h/\beta$ would be a polynomial, without a term in $x^{-1}$, and $\omega$ would vanish on $(\gamma_\circ^2/\beta)\F_2[x]$. Hence
$\beta\nmid\bar h$, that is, $h\notin\m$. Suppose that $2h=4a+2\tilde\beta b+\tilde\beta^2c$ with $a,b,c\in\Z[x]$, that is, $2h\in\m^2$.
Reducing modulo $2$ gives $\beta^2\bar c=0$, so $c=2c'$ with $c'\in\Z[x]$, and then $h=2a+\tilde\beta b+\tilde\beta^2c'\in\m$, a
contradiction.
\end{proof}

\subsection{Quadratic forms and the conditional law}
Next we express $\omega$ through the quadratic forms $q_A$ and $q_J$ on the Krylov space. As squaring is additive on $\F_2[x]$ and fixes
$\F_2$, every $f\in\F_2[x]$ can be written uniquely as $f=f_0^2+xf_1^2$ with $f_0,f_1\in\F_2[x]$, where $f_0^2$ and $xf_1^2$ collect the
terms of $f$ of even and of odd degree. In particular, $f_0$ and $f_1$ depend linearly on $f$.

\begin{lemma}\label{lem:Lambda}
\begin{enumerate}[label=(\roman*)]
\item The restrictions of $q_A$ and $q_J$ to $K$ are $\F_2$-linear, and so is the restriction of $q_B$ to $K$ for every
$B\in\mathrm{Alt}_K$.
\item Let $f=f_0^2+xf_1^2$ with $f_0,f_1\in\F_2[x]$, and if $n$ is odd assume that $f(0)=0$. Then
$\omega(f)=q_J\big(f_0(\bar A)e\big)+q_A\big(f_1(\bar A)e\big)$.
\item $\omega$ vanishes on $\mu^2\F_2[x]$.
\end{enumerate}
\end{lemma}

\begin{proof}
(i) A quadratic form is additive on every subspace on which its polar form vanishes. By Lemma~\ref{lem:par},
$\langle e,g(\bar A)e\rangle=g(0)n$ in $\F_2$ for $g\in\F_2[x]$. Hence, for $a,b\in\F_2[x]$,
$\langle a(\bar A)e,\bar Ab(\bar A)e\rangle=\langle e,(xab)(\bar A)e\rangle=0$, and
\[
\langle a(\bar A)e,b(\bar A)e\rangle+\langle e,a(\bar A)e\rangle\langle e,b(\bar A)e\rangle=(ab)(0)\,n+a(0)b(0)\,n^2=0 .
\]
These are the polar forms of $q_A$ and $q_J$ on $K$. For $B\in\mathrm{Alt}_K$ the polar form $\langle y,Bz\rangle$ vanishes for $z\in K$.

(ii) For $y\in\Z^n$ we have $\frac12y^{\top}Ay=\sum_{i<l}a_{il}y_iy_l\equiv q_A(\bar y)\pmod2$. Moreover $y_i^2\equiv1$ or $0\pmod4$
according as $y_i$ is odd or even, so $\|y\|^2\equiv\wt(\bar y)\pmod4$. If $\wt(\bar y)$ is even, then
$\frac12\|y\|^2\equiv\frac12\wt(\bar y)\equiv\binom{\wt(\bar y)}2=q_J(\bar y)\pmod2$. Let $\tilde f_0,\tilde f_1$ be the $0/1$ lifts of
$f_0,f_1$. Then $\tilde f:=\tilde f_0^2+x\tilde f_1^2$ is a lift of $f$, with $\tilde f(0)=0$ if $n$ is odd (as then $f_0(0)=f(0)=0$). With
$u_0:=\tilde f_0(A)e$ and $u_1:=\tilde f_1(A)e$ we get $e^{\top}\tilde f(A)e=\|u_0\|^2+u_1^{\top}Au_1$, and
$\wt(\bar u_0)\equiv\langle\bar u_0,\bar u_0\rangle=\langle e,f_0^2(\bar A)e\rangle=f_0(0)n\equiv0\pmod2$. Hence
$\omega(f)=\frac12\|u_0\|^2+\frac12u_1^{\top}Au_1\equiv q_J(\bar u_0)+q_A(\bar u_1)$, where $\bar u_i=f_i(\bar A)e$.

(iii) For $h=h_0^2+xh_1^2$ we have $\mu^2h=(\mu h_0)^2+x(\mu h_1)^2$ and $(\mu h_i)(\bar A)e=0$. If $n$ is odd, then $(\mu^2h)(0)=0$ since
$x\mid\mu$ (Lemma~\ref{lem:krylovlaw}(a)). So $\omega(\mu^2h)=0$ by (ii).
\end{proof}

Conditionally on the Krylov data, the restriction of $q_A$ to $K$ is uniformly distributed on an explicit affine space, and outside a rare
event the underlying linear space is large.

\begin{lemma}\label{lem:condlaw}
\begin{enumerate}[label=(\roman*)]
\item Conditionally on $\mathcal D$, $\bar A$ is uniformly distributed on a coset of $\mathrm{Alt}_K$, and the linear form $q_A|_K$ is
uniformly distributed on a coset of the subspace
\[
\Psi(\mathrm{Alt}_K)=\big\{\langle z,\cdot\rangle|_K:\ z\in\operatorname{span}(K^{\perp}*K^{\perp})\big\}\subseteq K^*,
\qquad\text{where }\Psi(B):=q_B|_K ,
\]
and $K^*$ is the dual space of $K$.
\item The event $\mathcal B$ that $\operatorname{span}(K^{\perp}*K^{\perp})\not\supseteq e^{\perp}$ is determined by $\mathcal D$, and
$\Pr(\mathcal B)=O(n^22^{-n/2})$ with an absolute implied constant.
\end{enumerate}
\end{lemma}

\begin{proof}
(i) The event that the Krylov data equals a given value $(w_0,\dots,w_r)$ is $\{\bar A:\ \bar Aw_i=w_{i+1},\ 0\le i<r\}$, a coset
$\bar A_0+\mathrm{Alt}_K$, on which $\bar A$ is uniformly distributed. The map $B\mapsto q_B$ is linear, $q_{\bar A_0+B}=q_{\bar A_0}+q_B$, and
the restrictions to $K$ are linear (Lemma~\ref{lem:Lambda}(i)). So $q_A|_K$ is uniform on the coset $q_{\bar A_0}|_K+\Psi(\mathrm{Alt}_K)$.
The space $\mathrm{Alt}_K$ is the space of alternating bilinear forms on $\F_2^n$ that vanish on $K\times\F_2^n$, that is, of alternating
forms on $\F_2^n/K$,
so its dimension is $\binom m2$ with $m:=n-k_0$. If $z_1,\dots,z_m$ is a basis of $K^{\perp}$, the matrices $z_az_b^{\top}+z_bz_a^{\top}$
($a<b$) lie in $\mathrm{Alt}_K$, and they are linearly independent (evaluate the corresponding forms at $(y_c,y_d)$, where
$\langle z_a,y_c\rangle=\delta_{ac}$). Hence they span $\mathrm{Alt}_K$. For $B=uu'^{\top}+u'u^{\top}$ with $u,u'\in K^{\perp}$ and $h\in K$,
using $h_i^2=h_i$ and $\langle u,h\rangle=0$,
\[
q_B(h)=\sum_{i\ne l}u_iu'_lh_ih_l=\langle u,h\rangle\langle u',h\rangle+\sum_iu_iu'_ih_i=\langle u*u',h\rangle .
\]
Since the matrices $uu'^{\top}+u'u^{\top}$ with $u,u'\in K^{\perp}$ lie in $\mathrm{Alt}_K$ and span it, and $\Psi$ is linear, this gives
the formula for $\Psi(\mathrm{Alt}_K)$.

(ii) The event $\mathcal B$ is determined by $K$, hence by $\mathcal D$. Put $\theta_t:=\prod_{i=1}^t(2^i+1)$, so that $\theta_0=1$. We use
two standard facts about a nondegenerate alternating (that is, symplectic) space $Y$ over $\F_2$ of dimension $2t$, for which we refer to
Taylor (1992)~\cite{Tay92}. First, for a totally isotropic subspace $U$, the map $L\mapsto L/U$ is a bijection from the Lagrangian (that is,
totally isotropic $t$-dimensional) subspaces containing $U$ onto the Lagrangian subspaces of the nondegenerate alternating space
$U^{\perp}/U$ of dimension $2t-2\dim U$. Second, $Y$ has exactly $\theta_t$ Lagrangian subspaces. (The second fact follows from the first by
counting the pairs $(y,L)$ with $0\ne y\in L$, every vector being isotropic: $\theta_t(2^t-1)=(2^{2t}-1)\theta_{t-1}$.)

\step{1}{We count the totally isotropic subspaces of maximal dimension in $\F_2^m$.} For $m\ge1$ let $e\in\F_2^m$ also denote the all-ones
vector. A totally isotropic subspace of $\F_2^m$ (for the standard form) has dimension at most $\lfloor m/2\rfloor$ and lies in $e^{\perp}$.
If $m$ is odd, $e^{\perp}$ is nondegenerate alternating of dimension $m-1$. If $m$ is even, a totally isotropic subspace $L$ of dimension
$m/2$ satisfies $L=L^{\perp}\ni e$, and $L\mapsto L/\F_2e$ identifies these subspaces with the Lagrangians of $e^{\perp}/\F_2e$, of
dimension $m-2$. In both cases $\F_2^m$ has exactly $\theta_{\lceil m/2\rceil-1}$ totally isotropic subspaces of dimension
$\lfloor m/2\rfloor$. In particular, for $m=n$ these subspaces correspond bijectively to the Lagrangians of the space $P$ of
Lemma~\ref{lem:krylovlaw}(d), via $L\mapsto L/\F_2e$ for even $n$ and $L\mapsto L$ for odd $n$.

\step{2}{We find inside $K^{\perp}$ a uniformly random totally isotropic subspace $L$ of maximal dimension, and reduce (ii) to a statement
about $L$.} Conditionally on $G$, let $L^{\circ}$ be a uniformly random Lagrangian of $P$ containing $K^{\circ}$, and let $L\subseteq\F_2^n$
be the corresponding totally isotropic subspace of dimension $\lfloor n/2\rfloor$. Then $L\subseteq K^{\perp}$. Indeed, $L^{\circ}$ is
orthogonal to $K^{\circ}$ in $P$, as $L^{\circ}\supseteq K^{\circ}$ is totally isotropic. For even $n$ this gives $\langle L,K\rangle=0$, and
for odd $n$ it gives $L=L^{\circ}\subseteq e^{\perp}\cap(\bar AK)^{\perp}=K^{\perp}$, as $K=\F_2e+\bar AK$. Conditionally on $k_0$,
$L^{\circ}$ is uniformly distributed among all Lagrangians of $P$: for a Lagrangian $L_1$ of $P$,
\[
\Pr\bigl(L^{\circ}=L_1\bigm|k_0\bigr)=\sum_U\frac{\Pr(K^{\circ}=U\mid k_0)}{\theta_{\kappa-k_0}},
\]
the sum over the $(k_0-1)$-dimensional subspaces $U$ of $L_1$ (all of them totally isotropic). Here $\theta_{\kappa-k_0}$ is the number of
Lagrangians of $P$ that contain a given totally isotropic subspace of dimension $k_0-1$, by the two facts above. By
Lemma~\ref{lem:krylovlaw}(d) all summands are equal, and their number depends only on $k_0$ and $\kappa$. Hence $L$ is uniformly
distributed among the $\theta_{\kappa-1}$ totally isotropic subspaces of $\F_2^n$ of dimension $\lfloor n/2\rfloor$. Since
$L\subseteq K^{\perp}$, we have
$\operatorname{span}(L*L)\subseteq\operatorname{span}(K^{\perp}*K^{\perp})$, and it suffices to show that
$\Pr(\operatorname{span}(L*L)\not\supseteq e^{\perp})=O(n^22^{-n/2})$.

\step{3}{We show that if $\operatorname{span}(L*L)\ne e^{\perp}$, then $L$ splits along a nontrivial partition of the coordinates, and we
count such $L$.} For $y,z\in L$ we have $\langle y*z,e\rangle=\langle y,z\rangle=0$, so $\operatorname{span}(L*L)\subseteq e^{\perp}$.
Suppose $\operatorname{span}(L*L)\ne e^{\perp}$. Then some $w\notin(e^{\perp})^{\perp}=\{0,e\}$ is orthogonal to $L*L$. Let $S$ be its
support, so that $\emptyset\ne S\ne[n]$. For $y,z\in L$ we get $\sum_{i\in S}y_iz_i=\langle w,y*z\rangle=0$ and
$\sum_{i\notin S}y_iz_i=\langle y,z\rangle-\sum_{i\in S}y_iz_i=0$, so the coordinate projections $\pi_S(L)\subseteq\F_2^S$ and
$\pi_{S^c}(L)\subseteq\F_2^{S^c}$ are totally isotropic, and
$\dim\pi_S(L)+\dim\pi_{S^c}(L)\le\lfloor|S|/2\rfloor+\lfloor|S^c|/2\rfloor\le\lfloor n/2\rfloor=\dim L$. As $L$ embeds into
$\pi_S(L)\oplus\pi_{S^c}(L)$, it follows that $L$ splits along $S$, that is, $L=\pi_S(L)\oplus\pi_{S^c}(L)$, with both summands
totally isotropic of the maximal dimensions $\lfloor|S|/2\rfloor$ and $\lfloor|S^c|/2\rfloor$, whose sum is $\lfloor n/2\rfloor$. Hence,
for $|S|=m$, the number of such $L$ is at most $\theta_{a-1}\theta_{b-1}$ with $a:=\lceil m/2\rceil$ and $b:=\lceil(n-m)/2\rceil$, and
it is $0$ unless $a+b=\kappa$ (if $n$ is even and $m$ is odd, the dimensions add up to $n/2-1$).

\step{4}{We sum the bound of Step~3 over the possible supports $S$ and conclude the proof of (ii).} With $c:=\prod_{i\ge1}(1+2^{-i})<3$ we
have $2^{t(t+1)/2}\le\theta_t\le c\,2^{t(t+1)/2}$. If $a+b=\kappa$, then
\[
\frac{\theta_{a-1}\theta_{b-1}}{\theta_{\kappa-1}}\le c^2\,2^{\frac12\left((a-1)a+(b-1)b-(a+b-1)(a+b)\right)}=c^2\,2^{-ab}.
\]
As $L$ is uniform among $\theta_{\kappa-1}$ subspaces, the probability that $L$ splits along some $S$ of size $m$ is at most
$\binom nmc^22^{-ab}$, and this bound is symmetric under $m\mapsto n-m$. For $1\le m\le n/2$ we have
$a\ge1$, $m\le2a$ and $b\ge(n-m)/2\ge\max(n/4,(n-2a)/2)$, so $ab=b+(a-1)b\ge\frac n2-a+(a-1)\frac n4$. With $\binom nm\le n^{2a}$ this gives
\[
\binom nm\,c^2\,2^{-ab}\le c^2\,2^{-n/2}\cdot2n^2\cdot\big(2n^22^{-n/4}\big)^{a-1}.
\]
Each $a\ge1$ arises from at most two values of $m$. Hence, if $2n^22^{-n/4}\le\frac12$, the probability that $L$ splits along some $S$ is at
most $2\cdot2\cdot c^2\,2^{-n/2}\cdot2n^2\cdot\sum_{i\ge0}2^{-i}=16c^2n^22^{-n/2}$. For the finitely many remaining $n$ the bound
$O(n^22^{-n/2})$ holds trivially.
\end{proof}

The last lemma of this subsection is a computation in $\F_2[x]$ that is used in Step~2 of the proof of Proposition~\ref{prop:twolarge}.

\begin{lemma}\label{lem:socle}
Let $\mu\in\F_2[x]$ be nonzero, let $\beta$ be an irreducible factor of $\mu$, and put $c:=v_\beta(\mu)$. If $f=(\mu^2/\beta)t$ with
$t\in\F_2[x]$, $\beta\nmid t$, and $f=f_0^2+xf_1^2$ with $f_0,f_1\in\F_2[x]$, then $v_\beta(f_1)=c-1$.
\end{lemma}

\begin{proof}
Differentiating $f=f_0^2+xf_1^2$ gives $f'=f_1^2$, since squares have derivative $0$ in characteristic $2$. Write
$\mu^2/\beta=\beta^{2c-1}g$ with $\beta\nmid g$. Then $f=\beta^{2c-1}gt$ and, as $2c-1$ is odd, $f'=\beta^{2c-2}\big(\beta'gt+\beta(gt)'\big)$.
Here $\beta'\ne0$, because $\beta$ is irreducible and hence not a square, and $\deg\beta'<\deg\beta$, so $\beta\nmid\beta'gt$. Hence
$2v_\beta(f_1)=v_\beta(f')=2c-2$.
\end{proof}

\subsection{Large degrees at the prime \texorpdfstring{$2$}{2}}
Recall that, for controllable $G$ and an irreducible $\beta\in\F_2[x]$, $D_{(2,\beta)}=D_\m$ and $\mathcal T_{(2,\beta)}=\mathcal T_\m$, where
$\m=(2,\tilde\beta)$, and that $D_{(2,\beta)}\ne0$ if and only if $\beta\mid\gamma$ (Lemma~\ref{lem:Ttwo}(c)). For $d_0\ge2$ let
$\mathcal E_2(d_0)$ be the event that $G$ is controllable and some irreducible $\beta\in\F_2[x]$ with $\deg\beta\ge d_0$ satisfies
$D_{(2,\beta)}\ne0$, while $\mathcal T_{(2,\beta)}$ is not a chain ring. We can now prove the main result of this section.

\begin{proposition}\label{prop:twolarge}
For all $n\ge3$ and $d_0\ge2$,
\[
\Pr\bigl(\mathcal E_2(d_0)\bigr)\le2^{2-d_0}+O\big(n^22^{-n/2}\big),
\]
with an absolute implied constant. If $G$ is controllable and $\mathcal E_2(d_0)$ does not hold, then $(2,\tilde\beta)\notin\mathcal F(Q,2)$
for every $Q\in\QG$ and every irreducible $\beta\in\F_2[x]$ with $\deg\beta\ge d_0$.
\end{proposition}

\begin{proof}
By Lemma~\ref{lem:krylovlaw}(a), $\gamma_\circ=\mu\gamma'$ with $\deg\gamma'=j$, and $x\mid\mu$ if $n$ is odd. Let $\mathcal B$ be the event
of Lemma~\ref{lem:condlaw}(ii). For an irreducible $\beta\ne x$ let $\mathcal Z_\beta$ be the event that $\beta\mid\mu$ and that $\omega$
vanishes identically on the space
\[
S_\beta:=\big\{(\mu^2/\beta)r:\ r\in\F_2[x],\ \deg r<\deg\beta\big\}.
\]
Only polynomials $\beta$ of degree at least $d_0\ge2$ occur below, so $\beta\ne x$.

\step{1}{We show that $\mathcal E_2(d_0)\subseteq\{j\ge d_0\}\cup\bigcup_{\deg\beta\ge d_0}\mathcal Z_\beta$.} Suppose that
$\mathcal E_2(d_0)$ holds through $\beta$ of degree $d\ge d_0$, and that $j<d_0$. As $D_{(2,\beta)}\ne0$, $\beta$ divides $\gamma$ and hence
$\gamma_\circ=\mu\gamma'$, and $\beta\nmid\gamma'$ because $\deg\gamma'=j<d$. So $\beta\mid\mu$. As $\mathcal T_{(2,\beta)}$ is not a chain
ring, Lemma~\ref{lem:wangLC} shows that $\omega$ vanishes on $(\gamma_\circ^2/\beta)\F_2[x]$. By Lemma~\ref{lem:Lambda}(iii) it also
vanishes on $\mu^2\F_2[x]$, hence on the ideal generated by $\gamma_\circ^2/\beta=(\mu^2/\beta)\gamma'^2$ and $\mu^2$, which is
$(\mu^2/\beta)\gcd(\gamma'^2,\beta)\F_2[x]=(\mu^2/\beta)\F_2[x]$. This ideal contains $S_\beta$, so $\mathcal Z_\beta$ holds.

\step{2}{We show that $\Pr(\mathcal Z_\beta\mid\mathcal D)\le2^{1-\deg\beta}$ whenever the Krylov data $\mathcal D$ is such that
$\beta\mid\mu$ and $\mathcal B$ fails.} Let $d:=\deg\beta$, so that $\dim S_\beta=d$, and $c:=v_\beta(\mu)$. For $f\in S_\beta$ write
$f=f_0^2+xf_1^2$. If $n$ is odd,
then $f(0)=0$, as $x\mid\mu$ and $\beta\ne x$. So Lemma~\ref{lem:Lambda}(ii) gives $\omega(f)=q_J(f_0(\bar A)e)+q_A(f_1(\bar A)e)$. The
maps $f\mapsto f_0$ and $f\mapsto f_1$ are $\F_2$-linear, the vectors $f_i(\bar A)e\in K$ are determined by $\mathcal D$, and $q_J$, $q_A$
are linear on $K$ (Lemma~\ref{lem:Lambda}(i)). Hence $\omega|_{S_\beta}=\lambda_{\mathcal D}+\varpi\circ\iota$, where
$\lambda_{\mathcal D}(f):=q_J(f_0(\bar A)e)$ is determined by $\mathcal D$, $\varpi:=q_A|_K\in K^*$, and $\iota:S_\beta\to K$,
$f\mapsto f_1(\bar A)e$, is linear and determined by $\mathcal D$. If $f=(\mu^2/\beta)r\in S_\beta$ is nonzero, then $\beta\nmid r$ as
$\deg r<d$, so $v_\beta(f_1)=c-1$ by Lemma~\ref{lem:socle}. Then $\mu\nmid f_1$ and $\iota(f)\ne0$, as $\mu$ is the minimal polynomial of
$e$. So $\iota$ is injective, and $\mathcal Z_\beta$ is the event that $\lambda_{\mathcal D}+\varpi\circ\iota=0$. By
Lemma~\ref{lem:condlaw}(i), conditionally on $\mathcal D$, $\varpi$ is uniform on a coset of $\Psi(\mathrm{Alt}_K)$. As $\mathcal B$ fails,
this subspace contains $\{\langle z,\cdot\rangle|_K:\ z\in e^{\perp}\}=\{\vartheta\in K^*:\ \vartheta(e)=0\}$ (the map
$z\mapsto\langle z,\cdot\rangle|_K$ is onto $K^*$, and $e\in K$). The image of this hyperplane under the surjection
$\vartheta\mapsto\vartheta\circ\iota$ of $K^*$ onto $S_\beta^*$ has codimension at most $1$. Hence, conditionally on $\mathcal D$,
$\varpi\circ\iota$ is uniform on a coset of a subspace of $S_\beta^*$ of dimension at least $d-1$, and
$\Pr(\mathcal Z_\beta\mid\mathcal D)\le2^{1-d}$.

\step{3}{We sum the bounds of Steps~1 and~2 over $\beta$ and obtain the bound on $\Pr(\mathcal E_2(d_0))$.} By Lemma~\ref{lem:krylovlaw}(c),
conditionally on $k_0$ the polynomial $\mu$ is uniformly distributed among the monic polynomials of degree $k_0$ if $n$ is even, and
$\mu=x\mu_1$ with $\mu_1$ uniformly distributed among the monic polynomials of degree $k_0-1$ if $n$ is odd. A uniformly random monic
polynomial of degree at least $\deg\beta$ is divisible by $\beta$ with probability $2^{-\deg\beta}$, and one of smaller degree is not
divisible by $\beta$. As $\beta\ne x$, we get $\Pr(\beta\mid\mu)\le2^{-\deg\beta}$ in both cases. The events $\{\beta\mid\mu\}$ and
$\mathcal B$ are determined by $\mathcal D$. By Steps~1 and~2, Lemma~\ref{lem:krylovlaw}(b), Lemma~\ref{lem:condlaw}(ii) and the bound
$\pi_2(d)\le2^d/d$ of Section~\ref{sec:small},
\begin{align*}
\Pr\bigl(\mathcal E_2(d_0)\bigr)&\le\Pr(j\ge d_0)+\Pr(\mathcal B)
+\sum_{\deg\beta\ge d_0}\E\Big(\mathbf 1\{\beta\mid\mu\}\,\mathbf 1_{\mathcal B^c}\,\Pr\bigl(\mathcal Z_\beta\bigm|\mathcal D\bigr)\Big)\\
&\le\tfrac43\,4^{-d_0}+O\big(n^22^{-n/2}\big)+\sum_{d\ge d_0}\frac{2^d}d\,2^{-d}\,2^{1-d}.
\end{align*}
The last sum is at most $2^{2-d_0}/d_0\le2^{1-d_0}$, and $\frac43\,4^{-d_0}\le2^{-d_0}$. So
$\Pr(\mathcal E_2(d_0))\le3\cdot2^{-d_0}+O(n^22^{-n/2})\le2^{2-d_0}+O(n^22^{-n/2})$.

Finally, let $G$ be controllable, suppose that $\mathcal E_2(d_0)$ fails, and let $\beta$ be irreducible with $\deg\beta\ge d_0$. Then
$D_{(2,\beta)}=0$ or $\mathcal T_{(2,\beta)}$ is a chain ring. In both cases Proposition~\ref{lem:levellocal}(ii) shows that
$(2,\tilde\beta)\notin\mathcal F(Q,2)$ for every $Q\in\QG$.
\end{proof}

\section{No mates with smooth levels}\label{sec:smooth}

This section completes the proof of the second component, Theorem~\ref{thm:smooth}: for every fixed $\varepsilon\in(0,\frac12)$, the
probability that some $Q\in\QG$ other than a permutation matrix has a level whose prime factors are all at most $n^{1/2-\varepsilon}$ tends
to $0$. As in number theory, we call such levels \emph{smooth}. In the strategy of Section~\ref{sec:plan}, this theorem disposes at once of
all small primes, including the prime $2$. Local failures are not expected to be rare at these primes, which is why the theorem is a
statement about levels. The proof below assembles the results of Sections~\ref{sec:prelim}--\ref{sec:two}, including
Theorem~\ref{thm:largedegrees}, proved in Section~\ref{sec:reduction}.

Wang--Zhao (2025)~\cite{WZ25} proved that for every fixed $\ell\ge2$, with probability $1-o(1)$ the binomial random graph $G$ with any
fixed edge probability has no cospectral mate $H$ with $Q^{\top}A(G)Q=A(H)$ for a rational orthogonal matrix $Q$ of level $\ell$, and
Zhao (2026)~\cite{Z26} extended this to rational orthogonal matrices of height $o((n/\log n)^{1/10})$, the height being the largest
denominator of an entry. Every prime factor of $\ell(Q)$ divides the denominator of some entry of $Q$, so it is at most the height of $Q$.
Hence, for the generalized cospectral mates of $G(n,1/2)$, Theorem~\ref{thm:smooth} contains these results and extends them to all heights
up to $n^{1/2-\varepsilon}$ and to all levels, however large, whose prime factors are at most $n^{1/2-\varepsilon}$ (for instance
$\ell=2^a3^b$ with arbitrary $a,b$), with an argument that does not rely on theirs.

In the proof, if $G$ is controllable, $Q\in\QG$ has level $\ell$ and $p\mid\ell$, then by Proposition~\ref{lem:levellocal} some local
component $\mathcal T_\m$, $\m=(p,\tilde\beta)$, is not a chain ring, and for odd $p$ the local failure $E_p^{(\deg\beta)}$ occurs.
Theorem~\ref{thm:largedegrees} excludes local failures of degree at least $3\log_2n$ at all odd primes, and Theorem~\ref{thm:firstmoment}
then excludes large values of $p^{\deg\beta}$ at the odd primes $p\le n^{1/2-\varepsilon}$. At $p=2$, Proposition~\ref{prop:twolarge}
excludes components of large degree that are not chain rings. So, outside an event of small probability, only components $(p,\beta)$ from
a fixed finite set can contribute to the level, and by Proposition~\ref{prop:tight} their exponents and the multiplicities
$v_\beta(\bar\varphi)$ are bounded. The level is then bounded, and a column of $\ell Q$ gives a short vector whose reduction modulo $p$ is
killed by a polynomial of bounded degree in $\bar A$ (Lemma~\ref{lem:shortcol}), which Proposition~\ref{lem:suppd} excludes. The proof
also uses controllability, which holds with probability $1-o(1)$ by the theorem of O'Rourke--Touri (2016)~\cite{OT16}
(Theorem~\ref{thm:OT}).

\begin{proof}[Proof of Theorem~\ref{thm:smooth}]
Fix $\varepsilon\in(0,\frac12)$, put $\delta:=\varepsilon/2\in(0,\frac14)$, and let $\eta>0$. In this proof, limits are taken as
$n\to\infty$.

\step{1}{We choose the parameters $q_0$, $d_0$, $a_*$ and $h$, together with a finite set $\mathcal P$ of pairs $(p,\beta)$.} By
Theorem~\ref{thm:firstmoment} (with this $\delta$) choose $q_0\ge2$ such that
\[
\limsup_{n\to\infty}\Pr\bigl(E_p^{(d)}\text{ holds for some }(p,d)\in\mathcal R_n(\delta)\text{ with $p$ odd and }p^d>q_0\bigr)<\eta,
\]
and choose $d_0\ge2$ with $2^{2-d_0}<\eta$. Let $\mathcal P$ be the finite set of pairs $(p,\beta)$, with $p$ prime and $\beta\in\F_p[x]$
irreducible, such that either $p$ is odd and $p^{\deg\beta}\le q_0$, or $p=2$ and $\deg\beta<d_0$. Only finitely many pairs $(p,\deg\beta)$
occur in $\mathcal P$, so by Proposition~\ref{prop:tight}(a) there is $a_*\ge1$ such that
\[
\limsup_{n\to\infty}\Pr\bigl(\beta^{a_*}\mid\bar\varphi\text{ for some }(p,\beta)\in\mathcal P\bigr)<\eta,
\]
and by Proposition~\ref{prop:tight}(b) there is $h\ge1$ such that
\[
\limsup_{n\to\infty}\Pr\bigl(G\text{ is controllable and }\exp(\mathfrak M_{(p,\beta)})\ge p^h
\text{ for some }(p,\beta)\in\mathcal P\bigr)<\eta .
\]
Put $L_0:=\prod_{p\le q_0}p^{h-1}$, the product over the primes $p\le q_0$, and $d_*:=a_*\sum_{(p,\beta)\in\mathcal P}\deg\beta$.

\step{2}{We define an exceptional event $\mathcal X$ and bound its probability.} Let $\mathcal X$ be the union of the following events.
\begin{enumerate}[label=(X\arabic*)]
\item $G$ is not controllable.
\item $E_p^{(d)}$ holds for some $(p,d)\in\mathcal R_n(\delta)$ with $p$ odd and $p^d>q_0$.
\item $E_p^{(d)}$ holds for some odd prime $p$ and some $d\ge3\log_2n$.
\item $G$ is controllable, and $D_{(2,\beta)}\ne0$ while $\mathcal T_{(2,\beta)}$ is not a chain ring, for some irreducible
$\beta\in\F_2[x]$ with $\deg\beta\ge d_0$.
\item $\beta^{a_*}\mid\bar\varphi$ for some $(p,\beta)\in\mathcal P$, or $G$ is controllable and $\exp(\mathfrak M_{(p,\beta)})\ge p^h$
for some $(p,\beta)\in\mathcal P$.
\item The conclusion of Proposition~\ref{lem:suppd} fails for $M=A$ and $m=n$, the entries of $A$ above the diagonal being uniform on
$\{0,1\}$.
\end{enumerate}
The probability of (X1) tends to $0$ by Theorem~\ref{thm:OT}. That of (X2) has $\limsup$ less than $\eta$ by the choice of $q_0$, and that
of (X3) tends to $0$ by Theorem~\ref{thm:largedegrees}. The event (X4) is the event $\mathcal E_2(d_0)$ of Proposition~\ref{prop:twolarge},
of probability at most $2^{2-d_0}+O(n^22^{-n/2})$, so its $\limsup$ is less than $\eta$ by the choice of $d_0$. By the choices of $a_*$ and
$h$, the probability of (X5) has $\limsup$ less than $2\eta$. Finally, (X6) has probability at most $e^{-n/8}$ for large $n$ by
Proposition~\ref{lem:suppd}. Hence
\[
\limsup_{n\to\infty}\Pr(\mathcal X)\le4\eta .
\]

\step{3}{We show that, outside $\mathcal X$, a matrix in $\QG$ whose level has only small prime factors can differ from the identity only at
components given by $\mathcal P$.} Let $n$ be large, let $G\notin\mathcal X$, and suppose that some $Q\in\QG$ which is not a permutation
matrix has a level $\ell:=\ell(Q)$ all of whose prime factors are at most $n^{1/2-\varepsilon}$. Since (X1) fails, $G$ is controllable.
Moreover $\ell\ge2$ by Theorem~\ref{thm:wangxu}, as $Q$ is not a permutation matrix. Let $p$ be a prime factor of $\ell$. By
Proposition~\ref{lem:levellocal}(i) the set $\mathcal F(Q,p)$ is nonempty. Let $\m=(p,\tilde\beta)\in\mathcal F(Q,p)$ and put
$d:=\deg\beta$.

If $p$ is odd, then $E_p^{(d)}$ holds by Proposition~\ref{lem:levellocal}(iii). Since (X3) fails, $d<3\log_2n$. For large $n$ (depending
only on $\varepsilon$) we have $3\log_2n\le n^{\delta}/\log^2n$ and $3n^{1/2-\varepsilon}\log_2n\le n^{1/2-\delta}/\log n$. As
$p\le n^{1/2-\varepsilon}$, it follows that $(p,d)\in\mathcal R_n(\delta)$. Since (X2) fails, we get $p^d\le q_0$, that is,
$(p,\beta)\in\mathcal P$.

If $p=2$, then $d<d_0$ by the last assertion of Proposition~\ref{prop:twolarge}, as $G$ is controllable and (X4) fails. Hence
$(2,\beta)\in\mathcal P$.

Thus $\mathcal F(Q,p)\subseteq\{(p,\tilde\beta):\ (p,\beta)\in\mathcal P\}$ for every prime $p\mid\ell$. In particular every prime factor of
$\ell$ is at most $q_0$ (recall that $q_0\ge2$).

\step{4}{We show that the level of such a matrix $Q$ divides $L_0$.} Let $p\mid\ell$. For $\m=(p,\tilde\beta)\in\mathcal F(Q,p)$ we have
$(p,\beta)\in\mathcal P$ and $\mathfrak M_\m=\mathfrak M_{(p,\beta)}$. As $G$ is controllable and (X5) fails, $\exp(\mathfrak M_\m)$ is a
power of $p$ smaller than $p^h$, so Proposition~\ref{lem:levellocal}(i) gives $v_p(\ell)\le h-1$. Hence $\ell$ divides $L_0$, and
$\ell\le L_0$.

\step{5}{We derive a contradiction from the short columns of $\ell Q$.} Let $p\mid\ell$, and put
\[
g:=\prod_{(p,\tilde\beta)\in\mathcal F(Q,p)}\beta^{a_\beta}\in\F_p[x],\qquad\text{where }\beta^{a_\beta}\,\|\,\bar\varphi .
\]
By Step~3, $(p,\beta)\in\mathcal P$ for all these $\beta$, so $a_\beta<a_*$ as (X5) fails, and $\deg g\le d_*$. By
Lemma~\ref{lem:shortcol}, some column $z$ of $\ell Q$ satisfies $0<\wt(\bar z)\le\ell^2\le L_0^2$ and $g(\bar A)\bar z=0$. In particular
$\deg g\ge1$, as $\bar z\ne0$. For large $n$ we have $p\le q_0\le e^{n^{1/8}}$, $d_*\le n^{1/4}$ and
$L_0^2<n/(4(d_*+1))\le n/(4(\deg g+1))$, so $\bar z$ violates the conclusion of Proposition~\ref{lem:suppd}, and (X6) holds, which
contradicts $G\notin\mathcal X$.

Consequently, for all large $n$, the event that some $Q\in\QG$ other than a permutation matrix has a level whose prime factors are all at
most $n^{1/2-\varepsilon}$ is contained in $\mathcal X$. Its probability therefore has $\limsup$ at most $4\eta$, and as $\eta>0$ was
arbitrary, it tends to $0$.
\end{proof}

\section{Anti-concentration tools}\label{sec:tools}

In the proof strategy of Section~\ref{sec:plan}, local failures of degree less than $3\log_2n$ at the moderate primes
$n^{1/2-\varepsilon}<p\le e^{n^{1/10}}$ are excluded in Sections~\ref{sec:deg1}--\ref{sec:deg3}, which together with the present section
prove the third component, Theorem~\ref{thm:moderate}. Several arguments there are union bounds over the possible eigenvectors and
invariant subspaces of a random Seidel matrix modulo $p$, in which the additively structured vectors, whose coefficients make sums of
random signs concentrated, must be counted separately. This section collects the anti-concentration estimates that these sections use.
The main results are block estimates, which bound the measure of the vectors whose concentration exceeds a given threshold. We prove
them in dimension one (Proposition~\ref{lem:blockdelta}) and in dimensions two to four (Proposition~\ref{lem:blockd}). The block
estimate in dimension one serves the eigenvalues in $\F_p$ (Section~\ref{sec:deg1}), and the one in dimensions two to four serves the
pencils of quadratic forms of Section~\ref{sec:deg1} and the eigenvalues of degrees two to four (Sections~\ref{sec:deg2}
and~\ref{sec:deg3}).

The results do not depend on the rest of the paper. Apart from Lemmas~\ref{lem:seqexp} and~\ref{lem:seqbases}, which concern a
uniformly random Seidel matrix, they are elementary estimates and statements about sums $\sum_j\xi_jw_j$ of independent random
signs with coefficients in an abelian group, about quadratic forms in random signs and about product measures.

After the elementary bounds of Section~\ref{subsec:toolsbounds}, we show that vectors with large concentration are rare, using inverse
Littlewood--Offord theory, developed by Tao--Vu (2009)~\cite{TV09} (see also Tao--Vu (2006)~\cite{TVbook}), Rudelson--Vershynin
(2008)~\cite{RV08} and Nguyen--Vu (2011)~\cite{NV11}, in the counting form of Ferber--Jain--Luh--Samotij (2021)~\cite{FJLS}. A version of
the inequality of Hal\'asz (1977)~\cite{Hal77} bounds the concentration by the density of vanishing signed sums, and a counting lemma for
product measures shows that vectors with many vanishing sums are rare (Section~\ref{subsec:toolshalasz}). In dimensions two to four a
block can also be concentrated because most of its coordinates lie in a proper subspace (Section~\ref{subsec:toolsblockd}). Finally, a
squaring trick reduces quadratic forms to linear ones (Section~\ref{subsec:toolssquaring}), and Section~\ref{subsec:toolsLO} gives
Littlewood--Offord bounds for subspaces whose nonzero vectors have large support, which improve the bound of Lemma~\ref{lem:odlyzko}.

\subsection{Notation, characters and vanishing sums}\label{subsec:toolsnotation}
Throughout this section, $\xi_1,\xi_2,\dots$ are independent uniform random signs, as in Section~\ref{sec:intro}. We extend the
concentration $\rho_p$ of Section~\ref{subsec:plan-switch} to abelian groups. For an abelian group $G$, a finite index set $b$ and
$w=(w_j)_{j\in b}\in G^b$ let
\[
\rho_G(w):=\max_{a\in G}\Pr\Bigl(\sum_{j\in b}\xi_jw_j=a\Bigr),
\]
and write $\rho_N:=\rho_{\Z/N}$. Thus $\rho_p=\rho_{\F_p}$, and $\rho_\Z$ is the case $G=\Z$. For $b'\subseteq b$ let
$w_{b'}:=(w_j)_{j\in b'}$. Since adding an independent random variable cannot increase the largest atom of a distribution,
$\rho_G(w_{b'})\ge\rho_G(w)$.

For an odd prime $p$ and a vector $y$ with entries in $\F_p$, an \emph{integer lift} of $y$ is an integer vector $z$ with $\bar z=y$, as in
Section~\ref{subsec:plan-moderate}, and the \emph{lift} $\tilde y$ of $y$ is the integer lift with entries in $(-\frac p2,\frac p2)$. It is
the shortest integer lift of $y$, that is, $\|\tilde y\|\le\|z\|$ for every integer lift $z$ of $y$. For a vector $w$ with entries in
$\F_p$ we put
\[
\mu(w):=\min_{c\in\F_p^{\times}}\big\|\widetilde{cw}\big\|^2 ,
\]
the smallest squared length of the lift of a \emph{dilate} $cw$, $c\in\F_p^{\times}$, of $w$. Lifts and the quantity $\mu$ are also used in
Sections~\ref{sec:deg1}--\ref{sec:deg3}.

We bound concentrations by Fourier analysis on finite abelian groups, with the following conventions. Let $G$ be a finite
abelian group of order $N$, and let $\widehat G:=\operatorname{Hom}(G,\R/\Z)$ be its group of characters, written
additively. Then $|\widehat G|=N$, $\sum_{\chi\in\widehat G}e^{2\pi i\chi(g)}=N\,\mathbf 1\{g=0\}$ for $g\in G$, and, for a subgroup
$L\le G$, every character of $L$ is the restriction of exactly $|G|/|L|$ characters of $G$. For $G=\Z/N$ the characters are $y\mapsto ty/N$
($t\in\Z/N$), and for $G=\F_p^d$ they are $w\mapsto\tau^{\top}w/p$ ($\tau\in\F_p^d$). Put
\[
\hat\rho_G(w):=\frac1N\sum_{\chi\in\widehat G}\ \prod_{j\in b}\bigl|\cos\bigl(2\pi\chi(w_j)\bigr)\bigr|,\qquad\hat\rho_N:=\hat\rho_{\Z/N},\quad
\hat\rho_p:=\hat\rho_{\F_p}.
\]
As $\E\,e^{2\pi i\chi(\xi_jw_j)}=\cos(2\pi\chi(w_j))$, Fourier inversion gives
\[
\Pr\Bigl(\sum_j\xi_jw_j=a\Bigr)=\frac1N\sum_{\chi\in\widehat G}e^{-2\pi i\chi(a)}\prod_j\cos\bigl(2\pi\chi(w_j)\bigr),
\]
and hence $\rho_G(w)\le\hat\rho_G(w)$. We call $\hat\rho_G(w)$ the \emph{Fourier majorant} of $\rho_G(w)$. Every factor $\lvert\cos(\cdot)\rvert$ is
at most $1$, so $\hat\rho_G(w_{b'})\ge\hat\rho_G(w)$ for $b'\subseteq b$. If all $w_j$ lie in a subgroup $L$, then
$\hat\rho_G(w)=\hat\rho_L(w)$ by the extension property of characters. In particular, for a subspace $L$ of $\F_p^d$, $\hat\rho_L$ is
defined intrinsically, through the characters of $L$.

The concentration of a vector is controlled by the number of vanishing signed sums of its coordinates, which we count as follows. For
$l\ge1$ and $w\in G^b$ with $|b|=j\ge1$ let
\[
R_l(w):=\#\Big\{(i,\sigma)\in b^{2l}\times\{\pm1\}^{2l}:\ \sum_{r=1}^{2l}\sigma_rw_{i_r}=0\Big\},\qquad r_l(w):=\frac{R_l(w)}{(2j)^{2l}},
\]
so that $r_l(w)$ is the probability that a signed sum of $2l$ coordinates of $w$, chosen independently and uniformly, vanishes. A pair
$(i,\sigma)$ with fewer than $\frac32l$ distinct indices $i_r$ is \emph{degenerate}, and $R'_l(w)$ denotes the number of nondegenerate
solutions counted by $R_l(w)$. If $j\ge\frac32l$, a degenerate pair has at most $D:=\lceil\frac32l\rceil-1$ distinct indices, all in
some $D$-subset of $b$, so there are at most $4^l\binom jDD^{2l}\le4^lj^D(1.5l)^{2l}$ degenerate pairs. As $D-2l\le-\frac{l+1}2$, they
form a fraction at most $j^{D-2l}(1.5l)^{2l}\le j^{-(l+1)/2}(1.5l)^{2l}\le j^{-l/2}(1.5l)^{2l}$ of all $(2j)^{2l}$ pairs, and therefore
\begin{equation}\label{eq:degen}
R'_l(w)\ \ge\ \big(r_l(w)-j^{-l/2}(1.5l)^{2l}\big)\,(2j)^{2l}\qquad\big(j=|b|\ge\tfrac32l\big).
\end{equation}
Below, \eqref{eq:degen} is applied only to blocks with $j\ge s$, where $s\ge16l\log(2N/\delta)$ by the first condition
in~\eqref{eq:scondsdelta}, which Propositions~\ref{lem:blockdelta} and~\ref{lem:blockd} below assume. As $N\ge3$ and $\delta\le1$ there,
this gives $s\ge16l\log6>\frac32l$, so the condition $j\ge\frac32l$ always holds.

A vanishing signed sum determines every coordinate that occurs in it exactly once from the other coordinates, and we record this by
witnesses. Let $|b|=j\ge2$. A \emph{vanishing-sum witness} at $b$ is a tuple $\omega=(i,r,\iota,\sigma)$ with $i\in b$, $r\in[2l]$,
$\iota=(\iota_{r'})_{r'\ne r}\in(b\setminus\{i\})^{[2l]\setminus\{r\}}$ and $\sigma\in\{\pm1\}^{2l}$, and we write $i(\omega):=i$. There
are $j\cdot2l\cdot(j-1)^{2l-1}\cdot4^l$ of them. The witness is \emph{valid} for $w\in G^b$ if
$\sigma_rw_i+\sum_{r'\ne r}\sigma_{r'}w_{\iota_{r'}}=0$. Then $w_i=-\sigma_r\sum_{r'\ne r}\sigma_{r'}w_{\iota_{r'}}$ is determined by
$w_{b\setminus\{i\}}$. A solution $(i,\sigma)$ counted by $R'_l(w)$ has $d'\ge\frac32l$ distinct indices, of which at least
$2d'-2l\ge l$ occur exactly once (if $u$ of them occur once, then $u+2(d'-u)\le2l$). Each such occurrence, at position $r$ say,
gives the valid witness $(i_r,r,(i_{r'})_{r'\ne r},\sigma)$, and different pairs (solution, position) give different witnesses. Hence,
if $G$ is finite of order $N$, $t>0$ and $R'_l(w)\ge t(2j)^{2l}/N$, then $w$ has at least $\lambda_1(j):=l\,t(2j)^{2l}/N$ valid
vanishing-sum witnesses at $b$, while
\begin{equation}\label{eq:witnesses}
\frac{\#\{\text{vanishing-sum witnesses at }b\}}{\lambda_1(j)}
=\frac{j\cdot2l\,(j-1)^{2l-1}\,4^l\,N}{l\,t\,4^l\,j^{2l}}\ \le\ \frac{2N}t.
\end{equation}

\subsection{Elementary bounds}\label{subsec:toolsbounds}
In the applications the integer lifts of a candidate eigenvector are weighted by a discrete Gaussian. Gaussian weights on lattice points
are a classical tool in the geometry of numbers, see Banaszczyk (1993)~\cite{Ban93}. We need the following elementary estimates for the
lattice $p\Z$ and its cosets.

\begin{lemma}\label{lem:gauss}
Let $p$ be an odd prime, $0<c\le10^{-5}$ and $\sigma:=1/(2cp^2)$. Then
\[
\theta_0:=\sum_{z\in p\Z}e^{-\sigma z^2}\le1.001,\qquad \mu_0:=\sum_{z\in\Z\setminus p\Z}e^{-\sigma z^2}\le\sqrt{2\pi c}\,p\ \ (\le p),
\]
and $\theta_0+\mu_0\le1+\sqrt{2\pi c}\,p$. Moreover, the image on $\F_p^\times$ of the probability law proportional to $e^{-\sigma z^2}$ on
$\Z\setminus p\Z$ has largest atom $\nu_{\max}$ with $\mu_0\nu_{\max}\le1.001$ (and, trivially, $\nu_{\max}\ge1/(p-1)$).
\end{lemma}
\begin{proof}
Note that $\sigma p^2=1/(2c)\ge5\cdot10^4$. Then
$\theta_0=\sum_{k\in\Z}e^{-k^2/(2c)}\le1+2\sum_{k\ge1}e^{-k/(2c)}\le1+3e^{-1/(2c)}\le1.001$.
As $z\mapsto e^{-\sigma z^2}$ is decreasing on $[0,\infty)$,
\[
\theta_0+\mu_0=\sum_{z\in\Z}e^{-\sigma z^2}\le1+2\int_0^\infty e^{-\sigma z^2}\,dz=1+\sqrt{\pi/\sigma}=1+\sqrt{2\pi c}\,p ,
\]
and as $\theta_0\ge1$ (the term $z=0$), $\mu_0\le\sqrt{2\pi c}\,p$, where $\sqrt{2\pi c}<1$. The image law gives $a\in\F_p^\times$ the
mass $\mu_0^{-1}\sum_{z\equiv a}e^{-\sigma z^2}$. Let $z_0\in(-\frac p2,\frac p2)$ represent $a$. The other elements of its residue class
are $z_0+jp$, $j\ne0$, and $|z_0+jp|\ge(|j|-\frac12)p$. Using $(k+\frac12)^2\ge k+\frac14$,
\[
\sum_{z\equiv a\ (\mathrm{mod}\ p)}e^{-\sigma z^2}\le1+2\sum_{k\ge0}e^{-(k+\frac12)^2/(2c)}\le1+2\sum_{k\ge0}e^{-(k+\frac14)/(2c)}
\le1+3e^{-1/(8c)}\le1.001 .
\]
So $\mu_0\nu_{\max}\le1.001$. Finally $\nu_{\max}\ge1/(p-1)$, as the image law lives on $p-1$ points.
\end{proof}

Exposing a random Seidel matrix one row at a time turns concentration bounds into bounds for linear systems. This is the row-by-row
exposure used by Costello--Tao--Vu (2006)~\cite{CTV06} for random symmetric matrices.

\begin{lemma}\label{lem:seqexp}
Let $T$ be a uniformly random Seidel matrix of order $m$, let $G$ be an abelian group (for instance $\Z$, $\Z/N$ or $\F_p^d$), and let
$w,t\in G^m$ be fixed. Write $Tw=t$ for the system $\sum_{j}T_{ij}w_j=t_i$, $i\in[m]$ (for $G=\F_p^d$ this is the system $\bar Tw=t$). Then
\[
\Pr(Tw=t)\ \le\ \prod_{i=1}^{m-1}\rho_G\big(w_{(i,m]}\big),\qquad w_{(i,m]}:=(w_{i+1},\dots,w_m).
\]
The same bound holds after any reordering of the coordinates, that is, with $(w_{\pi(i+1)},\dots,w_{\pi(m)})$ in place of $w_{(i,m]}$ for
any permutation $\pi$ of $[m]$. Consequently, $\Pr(Tw=t)\le\rho_G(w_b)^{m-|b|}$ for every $b\subseteq[m]$.
\end{lemma}
\begin{proof}
Reveal the upper triangle of $T$ one row at a time, at step $i$ the entries $T_{ij}$, $j>i$. The $i$th equation reads
$\sum_{j>i}T_{ij}w_j=t_i-\sum_{j<i}T_{ij}w_j$. Its right side is determined by the entries revealed before step $i$ (as $T_{ij}=T_{ji}$ for
$j<i$ and $T_{ii}=0$), and its left side has the law of $\sum_{j>i}\xi_jw_j$, since the signs $T_{ij}$, $j>i$, are independent of the
entries revealed before step $i$. So, conditionally on these entries, the $i$th equation holds with probability at most
$\rho_G(w_{(i,m]})$. Multiplying these bounds over $i\le m-1$ gives the first claim (the last equation is not used). The law of $T$ is
invariant under simultaneous permutations of rows and columns, and $Tw=t$ is equivalent to $T'w'=t'$ for the permuted matrix $T'$ and
the permuted vectors $w',t'$, which gives the second claim.
For the third, order the coordinates so that $b$ comes last. Then $w_{(i,m]}$ contains $w_b$ for $i\le m-|b|$, so the corresponding factors
are at most $\rho_G(w_b)$, and we bound the other factors by $1$.
\end{proof}

The only obstruction to equidistribution modulo $p$ is a dilate with a short integer lift, by the following elementary bound.

\begin{lemma}\label{lem:dilates}
Let $p$ be an odd prime, $b$ a finite set, and $\mu$ as in Section~\ref{subsec:toolsnotation} (so $\mu(0)=0$). Then for all $w\in\F_p^b$ and
$c\in\F_p^\times$
\[
\prod_{j\in b}\Bigl|\cos\frac{2\pi cw_j}p\Bigr|\le e^{-2\mu(w)/p^2},\qquad\text{hence}\qquad
\rho_p(w)\le\hat\rho_p(w)\le\frac1p+e^{-2\mu(w)/p^2}.
\]
\end{lemma}
\begin{proof}
On $[-\frac12,\frac12]$, where $\cos(\pi\theta)\ge0$, we have $\lvert\cos(\pi\theta)\rvert\le1-2\theta^2\le e^{-2\theta^2}$: the function
${1-2\theta^2-\cos(\pi\theta)}$ is even, vanishes at $0$, and its derivative $\pi\sin(\pi\theta)-4\theta$ is nonnegative on $[0,\frac12]$, as
$\sin(\pi\theta)\ge2\theta$ there. By periodicity, $\lvert\cos(\pi\theta)\rvert\le e^{-2\|\theta\|_{\R/\Z}^2}$ for all real $\theta$. For $c\ne0$
apply this with $\theta=2cw_j/p$. As $\|2cw_j/p\|_{\R/\Z}=|(\widetilde{2cw})_j|/p$ and $2c\ne0$ in $\F_p$,
\[
\prod_{j\in b}\Bigl|\cos\frac{2\pi cw_j}p\Bigr|\le\exp\Bigl(-\frac2{p^2}\big\|\widetilde{2cw}\big\|^2\Bigr)\le e^{-2\mu(w)/p^2}.
\]
Finally $\hat\rho_p(w)=\frac1p\sum_{c\in\F_p}\prod_j|\cos(2\pi cw_j/p)|$, where the term $c=0$ equals $\frac1p$ and each of the other
$p-1$ terms is at most $\frac1pe^{-2\mu(w)/p^2}$.
\end{proof}

\subsection{Hal\'asz's inequality and structured blocks}\label{subsec:toolshalasz}
We first bound the concentration by the density of vanishing signed sums, and then count the vectors with many vanishing
sums. The first bound is a version of the inequality of Hal\'asz (1977)~\cite{Hal77} for finite abelian groups of odd order, with explicit
constants. In its proof, the growth of the level sets of the Fourier transform is controlled by the theorem of Kneser (1953)~\cite{Kne53}.

\begin{proposition}\label{lem:halasz}
Let $G$ be a finite abelian group of odd order $N$, let $m,l\ge1$, and let $w\in G^m$ be such that every proper subgroup $K<G$ satisfies
$\#\{j:\ w_j\notin K\}\ge0.18m/l$. Then
\[
\hat\rho_G(w)\ \le\ \frac1N+8\sqrt{\frac{\pi l}m}\;r_l(w)+e^{-m/(16l)} .
\]
\end{proposition}
\begin{proof}
Since $\log\cos z\le-z^2/2$ for $|z|<\pi/2$, we have $\lvert\cos(\pi\theta)\rvert\le e^{-c\|\theta\|_{\R/\Z}^2}$ for all $\theta$, with
$c:=\pi^2/2$. Hence $\prod_j|\cos(2\pi\chi(w_j))|\le e^{-cf(\chi)}$, where $f(\chi):=\sum_j\|2\chi(w_j)\|_{\R/\Z}^2$. For $x\ge0$ let
$\mathcal L(x):={\{\chi\in\widehat G:\ f(\chi)\le x\}}$. Then
\[
N\hat\rho_G(w)\le\sum_{\chi\in\widehat G}e^{-cf(\chi)}=\int_0^\infty ce^{-cx}\,|\mathcal L(x)|\,dx .
\]

\step{1}{We show that sums of level sets are contained in level sets.} The map $\phi:\widehat G\to(\R/\Z)^m$, $\chi\mapsto(2\chi(w_j))_j$,
is a homomorphism, and $f(\chi)=\operatorname{dist}(\phi(\chi),0)^2$ for the translation-invariant metric
$\operatorname{dist}(u,v):=(\sum_j\|u_j-v_j\|_{\R/\Z}^2)^{1/2}$ on $(\R/\Z)^m$. Hence
$\sqrt{f(\chi+\chi')}\le\sqrt{f(\chi)}+\sqrt{f(\chi')}$, and the $h$-fold sumset satisfies $h\mathcal L(x)\subseteq\mathcal L(h^2x)$.
Also $0\in\mathcal L(x)=-\mathcal L(x)$.

\step{2}{We bound the size of one level set by a moment computation.} As $N$ is odd, $\sum_{\chi}e^{4\pi i\chi(g)}=\sum_\chi
e^{2\pi i\chi(2g)}=N\,\mathbf 1\{g=0\}$
for $g\in G$. Writing $\sum_j\cos(4\pi\chi(w_j))=\frac12\sum_{j,\sigma}e^{4\pi i\sigma\chi(w_j)}$ with $\sigma\in\{\pm1\}$ and expanding the
$2l$th power, we get
\[
\sum_{\chi\in\widehat G}\Bigl(\sum_{j=1}^m\cos\bigl(4\pi\chi(w_j)\bigr)\Bigr)^{2l}=4^{-l}N\,R_l(w)=N\,m^{2l}\,r_l(w).
\]
For $\chi\in\mathcal L(x)$, the bound $\cos(2\pi\theta)\ge1-2\pi^2\|\theta\|_{\R/\Z}^2$ gives $\sum_j\cos(4\pi\chi(w_j))\ge m-2\pi^2x$. For
$x_0:=m/(8\pi^2l)$ this is at least $m(1-\frac1{4l})>0$, and all terms of the sum above are nonnegative, so
$|\mathcal L(x_0)|\le N\,r_l(w)(1-\frac1{4l})^{-2l}\le2N\,r_l(w)$ by Bernoulli's inequality.

\step{3}{We show that $\mathcal L(x_0)$ contains no nonzero subgroup of $\widehat G$.} Let $H\le\widehat G$ be a nonzero subgroup. For each
$j$, $\chi\mapsto\chi(w_j)$ maps $H$ homomorphically onto a cyclic group $k_j^{-1}\Z/\Z$, whose order $k_j$ divides $|H|$ and is therefore
odd, and it takes each value $|H|/k_j$ times. As $k_j$ is odd, the same holds for $\chi\mapsto2\chi(w_j)$. So the average of
$\|2\chi(w_j)\|_{\R/\Z}^2$ over $\chi\in H$ is
\[
\frac1{k_j}\sum_{u\in\Z/k_j}\Big\|\frac u{k_j}\Big\|_{\R/\Z}^2=\frac{k_j^2-1}{12k_j^2},
\]
which vanishes if $k_j=1$, that is, if $w_j\in H^\perp:=\{g\in G:\ \chi(g)=0\text{ for all }\chi\in H\}$, and is at least $\frac2{27}$
otherwise. Now $H^\perp$ is a proper subgroup of $G$, since only the zero character vanishes on $G$. By hypothesis at least $0.18m/l$
indices have $w_j\notin H^\perp$, so the average of $f$ over $H$ is at least $\frac2{27}\cdot\frac{0.18m}l>\frac m{8\pi^2l}=x_0$.
Hence $H\not\subseteq\mathcal L(x_0)$.

\step{4}{We show with Kneser's theorem that the level sets below $x_0$ are small.} Let $0<x\le x_0$, $h:=\lfloor\sqrt{x_0/x}\rfloor\ge1$ and
$\Lambda:=\mathcal L(x)$. By Step~1, $h\Lambda\subseteq\mathcal L(x_0)$, and $(j+1)\Lambda\subseteq h\Lambda$ for $1\le j<h$ as
$0\in\Lambda$. The stabilizer $\{\chi:\ \chi+(j+1)\Lambda=(j+1)\Lambda\}$ of $(j+1)\Lambda$ is a subgroup, and it is contained in
$(j+1)\Lambda\subseteq\mathcal L(x_0)$ as $0\in(j+1)\Lambda$. Hence it is trivial by Step~3. Kneser's theorem~\cite{Kne53} (see also Tao--Vu
(2006)~\cite[Theorem~5.5]{TVbook}) states that $|X+Y|\ge|X+K|+|Y+K|-|K|$ for nonempty finite subsets $X,Y$ of an abelian group, where $K$ is
the stabilizer of $X+Y$. So $|(j+1)\Lambda|\ge|j\Lambda|+|\Lambda|-1$ for $1\le j<h$. Hence $h|\Lambda|-(h-1)\le|h\Lambda|\le2N\,r_l(w)$,
and since $h\ge\frac12\sqrt{x_0/x}$,
\[
|\mathcal L(x)|\le1+4N\,r_l(w)\sqrt{x/x_0}\qquad(0<x\le x_0).
\]

\step{5}{We integrate the bounds of Step~4 for the level sets and conclude the proof.} Using Step~4 on $(0,x_0]$ and $|\mathcal L(x)|\le N$
for $x>x_0$,
\[
N\hat\rho_G(w)\le1+\frac{4N\,r_l(w)}{\sqrt{x_0}}\int_0^\infty c\sqrt x\,e^{-cx}\,dx+Ne^{-cx_0}
=1+\frac{4N\,r_l(w)}{\sqrt{2\pi x_0}}+Ne^{-m/(16l)},
\]
since $\int_0^\infty c\sqrt x\,e^{-cx}\,dx=\Gamma(\frac32)c^{-1/2}=(2\pi)^{-1/2}$ and $cx_0=m/(16l)$. As
$4/\sqrt{2\pi x_0}=8\sqrt{\pi l/m}$, the proposition follows.
\end{proof}

For $G=\Z/p$, the sumset bound in Step~4 is the Cauchy--Davenport theorem. What matters below is the factor $\sqrt{l/m}$ in front of
$r_l(w)$. Some hypothesis on subgroups is necessary: for $G=\Z/9$ and $w_j=3$ for all $j$, all $w_j$ lie in the proper subgroup
$\{0,3,6\}$, and for fixed $l$ we have $\rho_G(w)\to\frac13$ as $m\to\infty$, while the right side tends to $\frac19$. Over $\F_p$, $p$ an
odd prime, an inequality of this type with an unspecified absolute constant is due to Ferber--Jain--Luh--Samotij (2021)~\cite{FJLS}, and
Proposition~\ref{lem:halasz} extends it to all finite abelian groups of odd order, with explicit constants.

Next we show that vectors with many vanishing sums are rare. The proof is a chain argument in the spirit of the counting lemma of
Ferber--Jain--Luh--Samotij~\cite{FJLS}: every such vector admits many chains of valid witnesses, while each fixed chain is valid only on
a set of small measure. We state the argument for several kinds of witnesses at once, because in dimension $d\ge2$ we also use witnesses
saying that a coordinate lies in a subspace (Proposition~\ref{lem:blockd}).

\begin{lemma}\label{lem:chains}
Let $G$ be a finite set, $k>s\ge0$ integers, and $\nu=\nu_1\otimes\dots\otimes\nu_k$ a product probability measure on $G^k$. A
\emph{witness} at $b\subseteq[k]$ is an object $\omega$ together with an index $i(\omega)\in b$ and a map $A_\omega$ from
$G^{b\setminus\{i(\omega)\}}$ to the set of subsets of $G$. It is \emph{valid} for $y\in G^k$ if
$y_{i(\omega)}\in A_\omega(y_{b\setminus\{i(\omega)\}})$. (For instance, if $G$ is an abelian group, a vanishing-sum witness
$\omega=(i,r,\iota,\sigma)$ of Section~\ref{subsec:toolsnotation} has $i(\omega)=i$ and
$A_\omega(z)=\{-\sigma_r\sum_{r'\ne r}\sigma_{r'}z_{\iota_{r'}}\}$.) Let $\tau_0\ge1$, $\kappa_1,\dots,\kappa_{\tau_0}\ge0$, and
$\lambda_\tau(j)>0$ for $\tau\le\tau_0$ and $s<j\le k$. Suppose that to every $b\subseteq[k]$ with $|b|>s$ there are assigned pairwise
disjoint finite sets $\mathcal W_1(b),\dots,\mathcal W_{\tau_0}(b)$ of witnesses at $b$ such that
$\nu_{i(\omega)}(A_\omega(z))\le\kappa_\tau$ for every $\omega\in\mathcal W_\tau(b)$ and every $z\in G^{b\setminus\{i(\omega)\}}$. Let
$\mathcal B\subseteq G^k$ be such that for every $y\in\mathcal B$ and every
$b\subseteq[k]$ with $|b|=j>s$, some $\mathcal W_\tau(b)$ contains at least $\lambda_\tau(j)$ witnesses that are valid for $y$. Then
\[
\nu(\mathcal B)\ \le\ \prod_{j=s+1}^k\ \max_{|b|=j}\ \sum_{\tau=1}^{\tau_0}\frac{\kappa_\tau\,|\mathcal W_\tau(b)|}{\lambda_\tau(j)}.
\]
\end{lemma}
\begin{proof}
Put $\mathcal W(b):=\bigcup_\tau\mathcal W_\tau(b)$, and write $\tau(\omega):=\tau$ for $\omega\in\mathcal W_\tau(b)$. For $|b|=j\ge s$,
a \emph{chain} from $b$ is a sequence $(\omega_j,\omega_{j-1},\dots,\omega_{s+1})$ with $\omega_{j'}\in\mathcal W(b_{j'})$, where
$b_j:=b$ and $b_{j'-1}:=b_{j'}\setminus\{i(\omega_{j'})\}$ (the empty sequence if $j=s$). Its weight is
$\prod_{j'=s+1}^j1/\lambda_{\tau(\omega_{j'})}(j')$, and it is valid for $y$ if all its witnesses are. Fix $y\in\mathcal B$, and let
$F(b)$ be the total weight of the chains from $b$ that are valid for $y$. Then $F(b)=1$ if $|b|=s$. If $|b|=j>s$, choose $\tau$
such that $\mathcal W_\tau(b)$ contains at least $\lambda_\tau(j)$ witnesses valid for $y$. By induction on $j$,
\[
F(b)\ \ge\ \sum_{\substack{\omega\in\mathcal W_\tau(b)\\ \omega\text{ valid for }y}}\frac{F(b\setminus\{i(\omega)\})}{\lambda_\tau(j)}\
\ge\ 1 .
\]
So the chains from $[k]$ that are valid for $y$ have total weight $F([k])\ge1$. Now fix a chain from $[k]$. The validity of $\omega_j$
is the event $y_{i(\omega_j)}\in A_{\omega_j}(y_{b_{j-1}})$, and the validity of $\omega_{j'}$ for $j'<j$ depends only on $y_{b_{j'}}$,
which does not involve $y_{i(\omega_j)}$. Integrating out $y_{i(\omega_k)}$, then $y_{i(\omega_{k-1})}$, and so on (this uses that $\nu$
is a product measure), we find that the chain is valid on a set of $\nu$-measure at most $\prod_j\kappa_{\tau(\omega_j)}$. Hence
\[
\nu(\mathcal B)\le\int\ \sum_{\text{chains valid for }y}\operatorname{weight}\ d\nu(y)
\le\sum_{\text{chains from }[k]}\ \prod_{j=s+1}^k\frac{\kappa_{\tau(\omega_j)}}{\lambda_{\tau(\omega_j)}(j)} ,
\]
and summing first over $\omega_{s+1}\in\mathcal W(b_{s+1})$ for fixed $\omega_k,\dots,\omega_{s+2}$, then over $\omega_{s+2}$, and so
on, gives the claim.
\end{proof}

For the vanishing-sum witnesses alone, Lemma~\ref{lem:chains} gives the following counting bound.

\begin{corollary}\label{lem:fjls}
Let $G$ be a finite abelian group of order $N$, $\nu$ a probability measure on $G$ with largest atom $\nu_{\max}$, and let $l\ge1$,
$1\le s<k$ and $t>0$. Let $\mathcal B$ be the set of $y\in G^k$ such that $R'_l(y_b)\ge t(2|b|)^{2l}/N$ for every $b\subseteq[k]$
with $|b|>s$. Then
\[
\nu^{k}(\mathcal B)\ \le\ \Big(\frac{2N\nu_{\max}}t\Big)^{k-s}.
\]
\end{corollary}
\begin{proof}
Apply Lemma~\ref{lem:chains} with $\tau_0:=1$, the family $\mathcal W_1(b)$ of all vanishing-sum witnesses at $b$, $\kappa_1:=\nu_{\max}$
(each $A_\omega(z)$ is a single point) and the function $\lambda_1$ of~\eqref{eq:witnesses}. By the discussion
before~\eqref{eq:witnesses}, every $y\in\mathcal B$ has at least $\lambda_1(|b|)$ valid vanishing-sum witnesses at every $b$
with $|b|>s$, and by~\eqref{eq:witnesses} every factor of the product in Lemma~\ref{lem:chains} is at most $2N\nu_{\max}/t$.
\end{proof}

For $\nu$ uniform on $G$ the corollary is a count of vectors, as in~\cite{FJLS}. For us $\nu$ will be the image of a discrete Gaussian
(Lemma~\ref{lem:gauss}), and the corollary then bounds the weighted number of short integer lifts that are additively structured.

Combining Proposition~\ref{lem:halasz} and Corollary~\ref{lem:fjls}, we obtain the block estimate in dimension one: under a product measure
with small atoms, a vector of $k$ units of $\Z/N$ is concentrated at a level slightly above the uniform level $1/N$ only with probability
exponentially small in $k-s$. Its key step, part~(a), is stated in all finite abelian groups of odd order, as
Proposition~\ref{lem:blockd} uses it in $\F_p^d$.

\begin{proposition}\label{lem:blockdelta}
Let $N\ge3$ be odd, $l\ge1$, $\delta\in(0,1]$, and let $s\le k$ be integers with
\begin{equation}\label{eq:scondsdelta}
s\ge16l\log(2N/\delta)\qquad\text{and}\qquad s^{(l+1)/2}\ge\frac{32(1+\delta)}{\delta}\,\sqrt{\pi l}\,N\,(1.5l)^{2l}.
\end{equation}
Let $\rho_0\ge(1+\delta)/N$ and $t:=N\delta\rho_0\sqrt s/(32(1+\delta)\sqrt{\pi l})$.
\begin{enumerate}[label=(\alph*),leftmargin=*]
\item Let $G$ be a finite abelian group of order $N$, and let $w\in G^b$ with $|b|=j\ge s$ and $\hat\rho_G(w)\ge\rho_0$ be such that
every proper subgroup of $G$ misses at least $0.18j/l$ of the $w_i$. Then $R'_l(w)\ge t(2j)^{2l}/N$.
\item Let $\nu$ be a probability measure on $(\Z/N)^{\times}$ with largest atom $\nu_{\max}$. Then
\[
\nu^k\big(\{y:\ \rho_N(y)\ge\rho_0\}\big)\ \le\ \nu^k\big(\{y:\ \hat\rho_N(y)\ge\rho_0\}\big)\ \le\
\Big(\frac{64(1+\delta)\sqrt{\pi l}\,\nu_{\max}}{\delta\,\rho_0\sqrt s}\Big)^{k-s}.
\]
\end{enumerate}
\end{proposition}
\begin{proof}
(a) As $j\ge s\ge16l\log(2N/\delta)$, we have $e^{-j/(16l)}\le\delta/(2N)$. So Proposition~\ref{lem:halasz}, applied with $m=j$, gives,
as $1/N\le\rho_0/(1+\delta)$,
\[
8\sqrt{\frac{\pi l}j}\,r_l(w)\ge\rho_0-\frac{1+\delta/2}N\ge\frac{\delta\rho_0}{2(1+\delta)}.
\]
By the second condition in~\eqref{eq:scondsdelta}, $j\ge s$ and $\rho_0\ge1/N$, the degenerate pairs form a fraction at most
\[
j^{-l/2}(1.5l)^{2l}\le\frac{\delta\sqrt j}{32(1+\delta)\sqrt{\pi l}\,N}\le\frac{\delta\rho_0\sqrt j}{32(1+\delta)\sqrt{\pi l}} .
\]
The first display says that $r_l(w)\ge\delta\rho_0\sqrt j/(16(1+\delta)\sqrt{\pi l})$, so~\eqref{eq:degen} gives
$R'_l(w)\ge\frac{\delta\rho_0\sqrt j}{32(1+\delta)\sqrt{\pi l}}(2j)^{2l}\ge t(2j)^{2l}/N$, where the last step uses $j\ge s$.

(b) The first inequality holds as $\rho_N\le\hat\rho_N$. For the second, let $s<k$ (otherwise there is nothing to prove), and note
that $s\ge16\log6>1$. If $y\in((\Z/N)^{\times})^k$ and $\hat\rho_N(y)\ge\rho_0$, then for every $b\subseteq[k]$ with $|b|\ge s$ we
have $\hat\rho_N(y_b)\ge\rho_0$, and the coordinates of $y_b$ are units, which lie in no proper subgroup of $\Z/N$ (a unit generates
$\Z/N$). So (a) gives $R'_l(y_b)\ge t(2|b|)^{2l}/N$, and Corollary~\ref{lem:fjls} for $G=\Z/N$ gives the bound
$(2N\nu_{\max}/t)^{k-s}$, which is the claim, as $2N\nu_{\max}/t=64(1+\delta)\sqrt{\pi l}\,\nu_{\max}/(\delta\rho_0\sqrt s)$.
\end{proof}

\subsection{Blocks in dimension at most four}\label{subsec:toolsblockd}
For the pencils $t_1B+t_2B^2$ of Section~\ref{sec:deg1} and for the invariant subspaces of degree $d\le4$ in
Sections~\ref{sec:deg2} and~\ref{sec:deg3}, the row equations take values in $\F_p^d$ with $d\le4$. A block of coordinates
$w_j\in\F_p^d$ can then be concentrated for a reason without one-dimensional analogue: if its coordinates lie in a proper subspace $L$,
then its concentration is at least $p^{-\dim L}$, whatever these coordinates are. For a block most of whose coordinates lie in $L$,
Hal\'asz's inequality on $L$ handles the levels from $1.1p^{-\dim L}$ on, and vanishing sums handle the levels up to
$C\sqrt{l/s}\,p^{-\dim L}$, $C$ a constant. In the window between, we also use witnesses saying that a coordinate lies in the heavy
subspace. If the windows of different dimensions are disjoint, which is automatic for $d=2$ and holds for $d\in\{3,4\}$ at primes
$p\ge\sqrt{s/l}$, the level determines the dimension of the heavy subspace, and the other coordinates of the block determine the subspace
itself, so that a chain of witnesses never needs to name it.

In the rest of this section $p$ is an odd prime. For a subspace $L$ of $\F_p^d$ and $w\in(\F_p^d)^b$ let
$\operatorname{out}_L(w):=\#\{i\in b:\ w_i\notin L\}$. Lines and hyperplanes are linear subspaces of dimension $1$ and $d-1$,
respectively. A hyperplane of a subspace $L$ is a subspace of $L$ of dimension $\dim L-1$ (for a line, $\{0\}$). A block most of whose
coordinates lie in a subspace $L$ has many vanishing sums, by the following lemma.

\begin{lemma}\label{lem:lineheavy}
Let $d\ge1$, let $L\le\F_p^d$ be a subspace of dimension $g$, $\theta\in[0,1)$, $l\ge1$, and let $w\in(\F_p^d)^b$ with
$\operatorname{out}_L(w)\le\theta|b|$. Then $r_l(w)\ge(1-\theta)^{2l}p^{-g}$.
\end{lemma}
\begin{proof}
Let $b':=\{i\in b:\ w_i\in L\}$, so that $|b'|\ge(1-\theta)|b|>0$. The pairs $(i,\sigma)$ with all indices in $b'$ form a fraction
$(|b'|/|b|)^{2l}\ge(1-\theta)^{2l}$ of all pairs, and among them the vanishing ones form the fraction $r_l(w_{b'})$. By Fourier inversion
on $L$, with $i_1,\dots,i_{2l}$ independent and uniform in $b'$ and independent signs $\sigma_r$,
\[
r_l(w_{b'})=\frac1{p^g}\sum_{\chi\in\widehat L}\E\exp\Bigl(2\pi i\,\chi\Bigl(\sum_r\sigma_rw_{i_r}\Bigr)\Bigr)
=\frac1{p^g}\sum_{\chi\in\widehat L}\Big(\frac1{|b'|}\sum_{i\in b'}\cos\bigl(2\pi\chi(w_i)\bigr)\Big)^{2l}\ge\frac1{p^g},
\]
as the term $\chi=0$ equals $1$ and the others are nonnegative. Hence $r_l(w)\ge(1-\theta)^{2l}r_l(w_{b'})\ge(1-\theta)^{2l}p^{-g}$.
\end{proof}

In the rest of this subsection $2\le d\le4$, $G:=\F_p^d$, $N:=p^d$, and $l\ge1$. For $1\le g<d$ put
$\vartheta_g:=0.18\cdot2.1^{d-1-g}/l$. A vector $w\in G^b$ is \emph{heavy} on a subspace $L$ with $1\le\dim L=g<d$ if
$\operatorname{out}_L(w)<\vartheta_g|b|$, and $g^*(w)$ is the least dimension of a nonzero proper subspace on which $w$ is heavy
($g^*(w):=d$ if there is none). Recall that $\hat\rho_L$ is defined through the characters of $L$, and that $\hat\rho_G(w)=\hat\rho_L(w)$
if all $w_i$ lie in $L$.

The following block estimate holds at every level, also inside the windows.

\begin{proposition}\label{lem:blockd}
Let $2\le d\le4$, so that $G=\F_p^d$ and $N=p^d$, let $\delta:=\frac1{100}$, and let $l\ge50$ and $s<k$ be integers that
satisfy~\eqref{eq:scondsdelta} (with this $N$), $s^{l/2}\ge100p^{d-1}(1.5l)^{2l}$, $s\ge1400\,l$ and, if $d\ge3$, $p\ge\sqrt{s/l}$. Put
$\varpi:=6(1+\delta)\sqrt{\pi l}/(\delta\sqrt s)$
and $W^{(g)}:=(\varpi p^{-g},1.1p^{-g})$ for $1\le g<d$. These intervals are pairwise disjoint. Let $\nu$ be a probability measure on
$G\setminus\{0\}$ with largest atom $\nu_{\max}$, and $\kappa_g:=\max\{\nu(L):\ L\le G,\ \dim L=g\}$. Then for every
$\rho_0\ge(1+\delta)/N$,
\[
\nu^k\big(\{w\in G^k:\ \hat\rho_G(w)\ge\rho_0\}\big)\ \le\ \big(\gamma_1(\rho_0)+1.1\,\kappa(\rho_0)\big)^{k-s},
\]
where $\gamma_1(\rho_0):=64(1+\delta)\sqrt{\pi l}\,\nu_{\max}/(\delta\rho_0\sqrt s)$, and $\kappa(\rho_0):=\kappa_g$ if
$\rho_0\in W^{(g)}$ and $\kappa(\rho_0):=0$ if $\rho_0$ lies in no window.
\end{proposition}
\begin{proof}
We first check that the windows are disjoint. For $d=2$ there is only one window. For $d\ge3$, the window $W^{(g+1)}$ lies to the left
of $W^{(g)}$ because $1.1p^{-g-1}\le\varpi p^{-g}$,
that is, $p\ge1.1/\varpi$, and $1.1/\varpi\le0.0011\sqrt{s/l}\le p$. This is the only use of the hypothesis $p\ge\sqrt{s/l}$. (The windows
are empty if $\varpi\ge1.1$, and then Step~5 is never used.) Note that
$\vartheta_g\le\vartheta_1\le0.18\cdot2.1^2/l\le0.794/l\le0.016$ for every $g$, and put
$t:=N\delta\rho_0\sqrt s/(32(1+\delta)\sqrt{\pi l})$, so that $\gamma_1(\rho_0)=2N\nu_{\max}/t$.

\step{1}{We apply Hal\'asz's inequality in a subspace.} Let $L\le G$ have dimension $g\ge1$, and let $w\in(L\setminus\{0\})^b$ satisfy
$\operatorname{out}_{L''}(w)\ge0.18|b|/l$ for every hyperplane $L''$ of $L$. The subgroups of $L$ are its subspaces, and every proper
subspace of $L$ lies in a hyperplane of $L$, so every proper subgroup of $L$ misses at least $0.18|b|/l$ of the $w_i$. As
$\hat\rho_G(w)=\hat\rho_L(w)$, Proposition~\ref{lem:halasz}, applied to the group $L$ of order $p^g$, gives
$\hat\rho_G(w)\le{p^{-g}+8\sqrt{\pi l/|b|}\,r_l(w)+\exp(-|b|/(16l))}$.

\step{2}{We bound the vanishing sums of heavy blocks from below.} If $\operatorname{out}_L(w)<\vartheta|b|$ with $\dim L=g$ and
$\vartheta\le\vartheta_1$, then $r_l(w)\ge(1-\vartheta)^{2l}p^{-g}\ge0.2p^{-g}$, by Lemma~\ref{lem:lineheavy} and since
$(1-0.794/l)^{2l}\ge(1-0.794/50)^{100}\ge0.2$ for $l\ge50$ ($(1-c/l)^l$ increases with $l$).

\step{3}{We show that a block satisfies the hypothesis of Step~1 on a minimal heavy subspace.} Let $|b|=j\ge s$,
$w\in(G\setminus\{0\})^b$, $g^*=g^*(w)<d$, let $L^*$ be a subspace of dimension $g^*$ on which $w$ is heavy, and
$b^*:=\{i\in b:\ w_i\in L^*\}$, so that $|b^*|>(1-\vartheta_{g^*})j\ge0.98j$. Then $\operatorname{out}_{L''}(w_{b^*})\ge0.18|b^*|/l$ for
every hyperplane $L''$ of $L^*$. For $L''=\{0\}$ this holds as all $w_i\ne0$. Otherwise $\dim L''=g^*-1\ge1$, and
$\operatorname{out}_{L''}(w)\le\operatorname{out}_{L^*}(w)+\operatorname{out}_{L''}(w_{b^*})$. If
$\operatorname{out}_{L''}(w_{b^*})<0.18|b^*|/l$, this would be less than $(\vartheta_{g^*}+0.18/l)j\le\vartheta_{g^*-1}j$, and $w$ would
be heavy on $L''$, contradicting the minimality of $g^*$ (note $\vartheta_{g-1}=2.1\vartheta_g\ge\vartheta_g+0.18/l$).

\step{4}{We show that, outside the window of dimension $g^*(w)$, a block has many valid vanishing-sum witnesses.} Let $|b|=j>s$,
$w\in(G\setminus\{0\})^b$ and $\hat\rho_G(w)\ge\rho_0$, and suppose that $g^*=g^*(w)=d$, or that $g^*<d$ and $\rho_0\notin W^{(g^*)}$. It
suffices to show that $R'_l(w)\ge t(2j)^{2l}/N$, because by the discussion before~\eqref{eq:witnesses} $w$ then has at least
$\lambda_1(j)=l\,t(2j)^{2l}/N$ valid vanishing-sum witnesses at $b$. The degenerate pairs form a fraction at most
$j^{-l/2}(1.5l)^{2l}\le0.01p^{1-d}$, by $s^{l/2}\ge100p^{d-1}(1.5l)^{2l}$.
\begin{itemize}[leftmargin=*]
\item $g^*=d$: then $\operatorname{out}_{L''}(w)\ge\vartheta_{d-1}j=0.18j/l$ for every hyperplane $L''$ of $G$, so every proper subgroup of
$G$ misses at least $0.18j/l$ of the $w_i$, and Proposition~\ref{lem:blockdelta}(a) applies.
\item $g^*<d$ and $\rho_0\le\varpi p^{-g^*}$: by Step~2 for a subspace of dimension $g^*$ on which $w$ is heavy, $r_l(w)\ge0.2p^{-g^*}$,
and the degenerate pairs form a fraction at most $0.01p^{1-d}\le0.01p^{-g^*}$. So $R'_l(w)\ge0.19p^{-g^*}(2j)^{2l}$, and
$0.19p^{-g^*}\ge\delta\rho_0\sqrt s/(32(1+\delta)\sqrt{\pi l})=t/N$, since $\rho_0\le\varpi p^{-g^*}$ and $0.19\cdot32>6$.
\item $g^*<d$ and $\rho_0\ge1.1p^{-g^*}$: let $L^*$ and $b^*$ be as in Step~3. By Steps~3 and~1 for $L^*$ and $w_{b^*}$, and as
$\hat\rho_G(w_{b^*})\ge\hat\rho_G(w)\ge\rho_0$,
\[
\rho_0\le p^{-g^*}+8\sqrt{\frac{\pi l}{|b^*|}}\,r_l(w_{b^*})+\exp\Bigl(-\frac{|b^*|}{16l}\Bigr).
\]
Here $\rho_0-p^{-g^*}\ge\rho_0/11$ and $\exp(-|b^*|/(16l))\le\exp(-0.98s/(16l))\le\frac1{20}p^{-g^*}\le\rho_0/22$, because
$s\ge16l\log(200p^d)\ge16.4\,l\log(20p^{d-1})$. So $r_l(w_{b^*})\ge\rho_0\sqrt{|b^*|}/(176\sqrt{\pi l})$, and since vanishing signed sums of
$w_{b^*}$ are vanishing signed sums of $w$,
\[
r_l(w)\ge\Big(\frac{|b^*|}j\Big)^{2l}r_l(w_{b^*})\ge\frac{0.2\sqrt{0.98}\,\rho_0\sqrt s}{176\sqrt{\pi l}}\ge0.00112\,\frac{\rho_0\sqrt
s}{\sqrt{\pi l}},
\]
using $(|b^*|/j)^{2l}\ge(1-\vartheta_{g^*})^{2l}\ge0.2$ as in Step~2. The degenerate pairs form a fraction at most
$0.01p^{1-d}\le0.0091\rho_0$. Finally $\sqrt{s/(\pi l)}\ge\sqrt{1400/\pi}>21$, so that
$0.00112\sqrt{s/(\pi l)}-0.0091\ge0.00031\sqrt{s/(\pi l)}\ge\delta\sqrt{s/(\pi l)}/(32(1+\delta))$, and~\eqref{eq:degen} gives
$R'_l(w)\ge t(2j)^{2l}/N$.
\end{itemize}

\step{5}{We introduce membership witnesses for the levels inside a window.} Let $\rho_0\in W^{(g)}$. For $|b|=j>s$ and $i\in b$ let
$\omega_{b,i}$ be a witness at $b$, distinct from the vanishing-sum witnesses, with $i(\omega_{b,i})=i$ and $A_{\omega_{b,i}}(z)$, for
$z\in G^{b\setminus\{i\}}$, equal to $L$ if $L$ is the only $g$-dimensional subspace with $\operatorname{out}_L(z)<\vartheta_gj$, and
$A_{\omega_{b,i}}(z)=\emptyset$ otherwise. So $A_{\omega_{b,i}}(z)$ depends only on $z$ (and $b$), and
$\nu(A_{\omega_{b,i}}(z))\le\kappa_g$. Suppose that $w\in(G\setminus\{0\})^b$ has $\hat\rho_G(w)\ge\rho_0$ and that neither case of
Step~4 applies, that is, $g^*(w)<d$ and $\rho_0\in W^{(g^*(w))}$. As the windows are disjoint, this means $g^*(w)=g$. Let $L^*$ and
$b^*$ be as in Step~3, and let $i\in b^*$ and $z:=w_{b\setminus\{i\}}$. Then
$\operatorname{out}_{L^*}(z)=\operatorname{out}_{L^*}(w)<\vartheta_gj$, and by minimality
$\operatorname{out}_{L'}(z)\ge\operatorname{out}_{L'}(w)-1\ge\vartheta_{\dim L'}j-1$ whenever $1\le\dim L'<g$. If some $L\ne L^*$ of
dimension $g$ had $\operatorname{out}_L(z)<\vartheta_gj$, then $L':=L\cap L^*$ would have $\operatorname{out}_{L'}(z)<2\vartheta_gj$. This
is impossible if $L'=\{0\}$, as then $\operatorname{out}_{L'}(z)=j-1\ge2\vartheta_gj$. If $\dim L'\ge1$, then $\dim L'\le g-1$, as
$L\ne L^*$. Since
$\vartheta_{g'}$ decreases in $g'$ and $j>s\ge1400l>10/\vartheta_g$, this contradicts
$\operatorname{out}_{L'}(z)\ge\vartheta_{\dim L'}j-1\ge\vartheta_{g-1}j-1=2.1\vartheta_gj-1>2\vartheta_gj$. So
$A_{\omega_{b,i}}(z)=L^*\ni w_i$, and $\omega_{b,i}$ is valid for each of the more than $\lambda_M(j):=(1-\vartheta_g)j$ indices
$i\in b^*$.

\step{6}{We combine Steps~4 and~5 by means of Lemma~\ref{lem:chains} and conclude the proof.} Let $\mathcal B$ be the set of
$w\in(G\setminus\{0\})^k$ with $\hat\rho_G(w)\ge\rho_0$. It has the $\nu^k$-measure of the set in the statement. We use the following
families at a block $b$: the vanishing-sum witnesses at $b$ (with $\kappa_\tau=\nu_{\max}$ and $\lambda_\tau=\lambda_1$) and, if
$\rho_0\in W^{(g)}$, the witnesses $\omega_{b,i}$, $i\in b$, of Step~5 (with $\kappa_\tau=\kappa_g$ and $\lambda_\tau=\lambda_M$). For every
$w\in\mathcal B$ and every $b\subseteq[k]$ with $|b|>s$ we have $\hat\rho_G(w_b)\ge\hat\rho_G(w)\ge\rho_0$, so by Steps~4 and~5, applied
to $w_b$, one of these families contains at least $\lambda_\tau(|b|)$ witnesses that are valid for $w$. At each size $j$, the witnesses
of Step~4 contribute at most $2N\nu_{\max}/t=\gamma_1(\rho_0)$ by~\eqref{eq:witnesses}, and those of Step~5 at most
$\kappa_gj/((1-\vartheta_g)j)\le1.1\kappa_g$. So Lemma~\ref{lem:chains} bounds
$\nu^k(\mathcal B)$ by $(\gamma_1(\rho_0)+1.1\,\kappa(\rho_0))^{k-s}$, which is the claim.
\end{proof}

The proof shows that a block $w$ with $\hat\rho_G(w)\ge\rho_0$ has many vanishing signed sums unless $\rho_0$ lies in the window $W^{(g)}$
of dimension $g=g^*(w)$, and then $w$ is heavy on a unique subspace of dimension $g$. In this form Proposition~\ref{lem:blockd} is a
counting version of inverse Littlewood--Offord theory in $\F_p^d$, $2\le d\le4$, valid at all levels down to $(1+\delta)p^{-d}$.
The counting problem of inverse Littlewood--Offord theory was studied, over $\F_p$ and for the uniform measure, by
Ferber--Jain--Luh--Samotij (2021)~\cite{FJLS}.

\subsection{Quadratic forms}\label{subsec:toolssquaring}
The following inequality, a squaring trick in the spirit of the decoupling argument of Costello--Tao--Vu (2006)~\cite{CTV06} for quadratic
forms (see also Vershynin (2014)~\cite{Ver14}), bounds the Fourier coefficients of a quadratic form in a random sign vector by those of
linear forms.

\begin{lemma}\label{lem:squaring}
Let $p$ be an odd prime, $B\in\F_p^{m\times m}$ symmetric, $t\in\F_p$, and $\xi$ uniform in $\{\pm1\}^m$. Let
$u\in\{-1,0,1\}^m$ have independent coordinates with $\Pr(u_j=0)=\frac12$ and $\Pr(u_j=1)=\Pr({u_j=-1})=\frac14$, and let $U:=\supp u$. Then
\[
\Big|\E\exp\Bigl(\frac{2\pi i\,t\,\xi^{\top}B\xi}p\Bigr)\Big|^2\ \le\
\E_u\prod_{j\notin U}\Bigl|\cos\frac{8\pi t\,(Bu)_j}p\Bigr| .
\]
Consequently, if $t\ne0$ and $\Pr\bigl(\mu((Bu)_{U^c})\le cmp^2\bigr)\le e^{-m/100}$ for some $c\in(0,\frac1{200}]$, where $(Bu)_{U^c}$
is the restriction of $Bu$ to the coordinates outside $U$, then the left side is at most $2e^{-2cm}$.
\end{lemma}
\begin{proof}
Let $\xi'$ be an independent copy of $\xi$, and put $u:=(\xi-\xi')/2$ and $v:=(\xi+\xi')/2$. The pairs $(u_j,v_j)$ are independent and
uniform on $\{(\pm1,0),(0,\pm1)\}$. So $u$ has the stated law, and given $u$, $v$ vanishes on $U$ and its coordinates outside $U$ are
independent uniform signs. With $Q(y):=t\,y^{\top}By$ and $B=B^{\top}$,
\[
Q(\xi)-Q(\xi')=t\,(\xi-\xi')^{\top}B(\xi+\xi')=4t\,(Bu)^{\top}v .
\]
Hence
\[
\big|\E\,e^{2\pi iQ(\xi)/p}\big|^2=\E\,e^{2\pi i(Q(\xi)-Q(\xi'))/p}=\E_u\prod_{j\notin U}\cos\bigl(8\pi t\,(Bu)_j/p\bigr),
\]
and the first claim follows. For the second, Lemma~\ref{lem:dilates}, applied with the nonzero multiplier $4t$, bounds the product on
the right by $e^{-2\mu((Bu)_{U^c})/p^2}$, which is less than $e^{-2cm}$ unless $\mu((Bu)_{U^c})\le cmp^2$. So the right side is at most
$e^{-m/100}+e^{-2cm}\le2e^{-2cm}$, as $2c\le\frac1{100}$.
\end{proof}

\subsection{Littlewood--Offord bounds for subspaces}\label{subsec:toolsLO}
For an integer vector $w$ with $k$ nonzero coordinates, the inequality of Erd\H{o}s (1945)~\cite[Theorem~1]{Erd45}, sharpening
Littlewood--Offord (1943)~\cite{LO43}, gives $\rho_\Z(w)\le\binom k{\lfloor k/2\rfloor}2^{-k}=O(k^{-1/2})$. The next lemma is the moment
bound behind its analogue modulo $p$, which is the case $d=1$ of Lemma~\ref{lem:LObases}(a) below.

\begin{lemma}\label{lem:LOp}
Let $p$ be an odd prime and $t\ge2$ an integer. Then
\[
\frac1p\sum_{\sigma\in\F_p}\Bigl|\cos\frac{2\pi\sigma}p\Bigr|^t\ \le\ \frac1p\sum_{\sigma\in\F_p}\Bigl(\cos\frac{2\pi\sigma}p\Bigr)^{2\lfloor
t/2\rfloor}
\ \le\ \frac1p+\big(\pi\lfloor t/2\rfloor\big)^{-1/2}.
\]
\end{lemma}
\begin{proof}
The first inequality holds termwise, as $\lvert\cos\rvert\le1$ and $2\lfloor t/2\rfloor\le t$. Let $h:=\lfloor t/2\rfloor\ge1$ and
$S:=\xi_1+\dots+\xi_{2h}$. Since $\cos(2\pi\sigma/p)^{2h}=\E\,e^{2\pi i\sigma S/p}$ and
$\frac1p\sum_{\sigma\in\F_p}e^{2\pi i\sigma S/p}=\mathbf 1\{p\mid S\}$,
the middle term equals $\Pr(S\in p\Z)$. The law of $S$ is supported on $2\Z$, symmetric, and unimodal: $\Pr(S=2i)=\binom{2h}{h+i}4^{-h}$ is
nonincreasing in $i\ge0$. As $p$ is odd and $S$ is even, $S\in p\Z\setminus\{0\}$ means $S=2jp$ with $j\ne0$. For $j\ge1$ the interval
$(2(j-1)p,2jp]$ contains $p$ even integers, all positive and the largest being $2jp$, so $\Pr(S=2jp)\le\frac1p\Pr(2(j-1)p<S\le2jp)$. Summing
over $j\ge1$ and using symmetry, $\Pr(S\in p\Z\setminus\{0\})\le\frac2p\Pr(S>0)\le\frac1p$. Finally
$\Pr(S=0)=\binom{2h}h4^{-h}\le(\pi h)^{-1/2}$.
\end{proof}

Lemmas~\ref{lem:LObases} and~\ref{lem:seqbases} below use only the supports of the vectors of a subspace. Lemma~\ref{lem:LObases} is a
modular analogue of the bound of Hal\'asz (1977)~\cite{Hal77} in $\R^d$. Compared with the uniform value $p^{-d}$ they lose a factor
$(1+p(\pi\lfloor t/2\rfloor)^{-1/2})^d$, which is harmless in a union bound over few events, or when $p=O(\sqrt t)$.

\begin{lemma}\label{lem:LObases}
Let $p$ be an odd prime and $d\ge1$.
\begin{enumerate}[label=(\alph*),leftmargin=*]
\item Let $\varphi\in(\F_p^d)^b$, and suppose that $b$ contains $t\ge2$ pairwise disjoint sets $B_1,\dots,B_t$ with $|B_c|=d$ such that
$(\varphi_j)_{j\in B_c}$ is a basis of $\F_p^d$ for each $c$. Then
$\rho_{\F_p^d}(\varphi)\le\hat\rho_{\F_p^d}(\varphi)\le\big(\frac1p+(\pi\lfloor t/2\rfloor)^{-1/2}\big)^d$.
\item If every hyperplane of $\F_p^d$ misses at least $k$ of the $\varphi_j$, $j\in b$, then $b$ contains $\lceil k/d\rceil$ pairwise disjoint
sets as in (a).
\item Let $U\subseteq\F_p^m$ be a subspace of dimension $d$ such that every nonzero $y\in U$ has $\lvert\supp y\rvert\ge k$, let $\xi$ be uniform in
$\{\pm1\}^m$, and put $t:=\lceil k/d\rceil$. If $t\ge2$, then
\[
\Pr(\bar\xi\perp U)\ \le\ \Big(\frac1p+\big(\pi\lfloor t/2\rfloor\big)^{-1/2}\Big)^d .
\]
\end{enumerate}
\end{lemma}
\begin{proof}
(a) Drop the factors with $j\notin B_1\cup\dots\cup B_t$ (each is at most $1$) and apply H\"older's inequality with $t$ factors:
\[
\hat\rho_{\F_p^d}(\varphi)\le\frac1{p^d}\sum_{\tau\in\F_p^d}\prod_{c=1}^t\prod_{j\in B_c}\Bigl|\cos\frac{2\pi\,\tau^{\top}\varphi_j}p\Bigr|
\le\prod_{c=1}^t\Big(\frac1{p^d}\sum_{\tau\in\F_p^d}\prod_{j\in B_c}\Bigl|\cos\frac{2\pi\,\tau^{\top}\varphi_j}p\Bigr|^t\Big)^{1/t}.
\]
For each $c$, $\tau\mapsto(\tau^{\top}\varphi_j)_{j\in B_c}$ is a bijection of $\F_p^d$ onto $\F_p^{B_c}$, as $(\varphi_j)_{j\in B_c}$ is a
basis, so the inner average equals
\[
\Big(\frac1p\sum_{\sigma\in\F_p}\Bigl|\cos\frac{2\pi\sigma}p\Bigr|^t\Big)^d,
\]
which is at most $\big(\frac1p+(\pi\lfloor t/2\rfloor)^{-1/2}\big)^d$ by Lemma~\ref{lem:LOp}.

(b) Choose the sets greedily. After $c<k/d$ sets of $d$ indices have been removed, every hyperplane still misses at least $k-cd\ge1$ of the
remaining $\varphi_j$. So the remaining vectors lie in no hyperplane and span $\F_p^d$, and there is a set $B_{c+1}$ of $d$ remaining indices
such that $(\varphi_j)_{j\in B_{c+1}}$ is a basis. This works for $c=0,1,\dots,\lceil k/d\rceil-1$.

(c) Let $y^{(1)},\dots,y^{(d)}$ be a basis of $U$ and $\varphi_j:=(y^{(1)}_j,\dots,y^{(d)}_j)\in\F_p^d$ for $j\in[m]$. Then $\bar\xi\perp U$
if and only if $\sum_j\xi_j\varphi_j=0$, and for $a\in\F_p^d\setminus\{0\}$ the number of $j$ with $a^{\top}\varphi_j\ne0$ is
$\lvert\supp\sum_ia_iy^{(i)}\rvert\ge k$, that is, every hyperplane misses at least $k$ of the $\varphi_j$. By (b) and (a),
$\Pr(\bar\xi\perp U)\le\rho_{\F_p^d}(\varphi)\le\big(\frac1p+(\pi\lfloor t/2\rfloor)^{-1/2}\big)^d$.
\end{proof}

Combined with sequential exposure (Lemma~\ref{lem:seqexp}), Lemma~\ref{lem:LObases}(a) gives the following bound for linear systems in a
random Seidel matrix.

\begin{lemma}\label{lem:seqbases}
Let $p$ be an odd prime, $d\ge1$, $G:=\F_p^d$, and let $W=(\varphi_1,\dots,\varphi_m)\in G^m$ and $Y\in G^m$ be fixed. Suppose that every
hyperplane of $G$ misses at least $k$ of the $\varphi_j$, and that $t:=\lceil k/d\rceil\ge2$. Let $T$ be a uniformly random Seidel matrix
of order $m$, and write $\bar TW=Y$ for the system $\sum_{j=1}^mT_{ij}\varphi_j=Y_i$, $i\in[m]$. Then
\[
\Pr(\bar TW=Y)\ \le\ p^{-dm}\big(1+p(\pi\lfloor t/2\rfloor)^{-1/2}\big)^{dm}\,e^{1.6d^2p\sqrt t}\,p^{2d^2} .
\]
\end{lemma}
\begin{proof}
By Lemma~\ref{lem:LObases}(b) there are pairwise disjoint sets $B_1,\dots,B_t\subseteq[m]$ with $|B_c|=d$ such that $(\varphi_j)_{j\in B_c}$
is a basis of $G$ for each $c$. In particular $dt\le m$. Order the indices so that $B_t,B_{t-1},\dots,B_1$ come last, in this order, and
apply Lemma~\ref{lem:seqexp}. The rows $1,\dots,m-1$ have suffixes of lengths $m-1,\dots,1$, and a suffix of length $r$ contains
$B_1,\dots,B_{\min(t,\lfloor r/d\rfloor)}$. By Lemma~\ref{lem:LObases}(a), and as
$\lfloor\lfloor r/d\rfloor/2\rfloor=\lfloor r/(2d)\rfloor$, the concentration of such a suffix is at most
$\eta_r:=\big(\frac1p+(\pi\lfloor r/(2d)\rfloor)^{-1/2}\big)^d$ if $2d\le r<dt$, at most
$\eta_\infty:=\big(\frac1p+(\pi\lfloor t/2\rfloor)^{-1/2}\big)^d$ if $r\ge dt$, and at most $1$ if $r<2d$. Hence
\[
p^{dm}\Pr(\bar TW=Y)\ \le\ (p^d\eta_\infty)^{m-dt}\prod_{2d\le r<dt}\big(p^d\eta_r\big)\cdot p^{2d^2},
\]
where the factor $p^{2d^2}$ accounts for the $2d-1$ rows with suffixes shorter than $2d$ and for the unused last equation. Moreover
$1\le p^d\eta_\infty=(1+p(\pi\lfloor t/2\rfloor)^{-1/2})^d$, so the first factor is at most ${(1+p(\pi\lfloor t/2\rfloor)^{-1/2})^{dm}}$.
Each value $i=\lfloor r/(2d)\rfloor$ occurs for at most $2d$ values of $r$, and $i<t/2$ when $r<dt$. As
$\sum_{1\le i\le x}i^{-1/2}\le2\sqrt x$ and $1+y\le e^y$,
\[
\prod_{2d\le r<dt}p^d\eta_r\le\exp\Bigl(dp\sum_{2d\le r<dt}\big(\pi\lfloor r/(2d)\rfloor\big)^{-1/2}\Bigr)
\le\exp\Bigl(\frac{2d^2p}{\sqrt\pi}\cdot2\sqrt{t/2}\Bigr)\le e^{1.6d^2p\sqrt t},
\]
because $4/\sqrt{2\pi}<1.6$. Together with the bound for the first factor, this proves the lemma.
\end{proof}

\section{Eigenvalues in \texorpdfstring{$\F_p$}{F\_p}}\label{sec:deg1}

This section and the next two prove the third component of the proof, Theorem~\ref{thm:moderate}: with probability $1-o(1)$ there is no
local failure $E_p^{(d)}$ of degree $d<3\log_2n$ at a moderate prime $n^{1/2-\varepsilon}<p\le e^{n^{1/10}}$. The degrees $d=1$, $d=2$
and $3\le d<3\log_2n$ are treated in turn, with the anti-concentration tools of Section~\ref{sec:tools}. Here we treat $d=1$, that is,
eigenvalues in $\F_p$ itself. Our goal is Proposition~\ref{thm:deg1}, whose proof verifies the hypotheses of the criterion of
Proposition~\ref{prop:reduction} for the pairs $(p,1)$. Together with Propositions~\ref{thm:deg2} and~\ref{thm:deg3plus} of the next two
sections it proves Theorem~\ref{thm:moderate} (Section~\ref{subsec:proofmoderate}). Along the way we bound the corank of $\bar S-\lambda I$
for a uniformly random Seidel matrix $S$ of order $n$ (Theorem~\ref{thm:rankp}). This bound, for a fixed $\lambda\in\F_p$, is also the
input of the proof of Theorem~\ref{thm:nofactor} in Section~\ref{sec:nofactor}, and the methods of this section are used again, over
larger fields, in Sections~\ref{sec:deg2} and~\ref{sec:deg3}.

By Seidel switching and the reduction of Section~\ref{sec:reduction} (Proposition~\ref{prop:reduction}), it suffices to bound the first
moments of the quantities $N_\Pi(S_m)$ of Definition~\ref{def:Npi}, $m\in\{n-1,n\}$, restricted to typical matrices. We now describe
these quantities in degree one. Let $T$ be a symmetric integer matrix of order $m$, $p$ an odd prime, and $\beta=x-\lambda\in\F_p[x]$ with
$\lambda\in\F_p$. By~\eqref{eq:abetaeig}, the weight $a_\beta(T)$ is the largest probability that a uniform sign vector is orthogonal to
an eigenvector of $\bar T$ at $\lambda$. If $\Pi$ consists of pairs $(p,1)$, then $N_\Pi(T)$ is the sum of these probabilities over the
pairs $(p,\lambda)$ with $(p,1)\in\Pi$ and $\lambda$ a repeated root of $\chi_T$ modulo $p$. By Lemma~\ref{lem:simple}(d), $\lambda$ is a
repeated root if and only if $\corank(\bar T-\lambda I)\ge2$, or the eigenspace is a line $\F_pw$ with $w^{\top}w=0$.

For a uniformly random symmetric matrix over $\F_p$ the expected number of repeated roots in $\F_p$ is of order $1/p$, and a
uniformly random vector is orthogonal to a fixed nonzero vector with probability $1/p$. So one expects $\E(N_\Pi)$ to be of order
$\sum_pp^{-2}$, which is small over the primes $p>n^{1/2-\varepsilon}$. We prove both facts for random Seidel matrices:
\begin{itemize}[leftmargin=*]
\item Eigenvectors are equidistributed: with probability $1-O(2^{-n/8})$, for every prime $23\le p\le e^{n^{1/8}}$, every
eigenvector $w\in\F_p^n$ of $\bar S$ with eigenvalue in $\F_p$ satisfies $\rho_p(w)\le\frac1p+e^{-2c_1n}$, where $c_1=10^{-5}$
(Theorem~\ref{thm:SEplus} and Corollary~\ref{cor:ACplus}).
\item Repeated roots are rare: for $23\le p\le e^{n^{1/9}}$ the expected number of repeated roots in $\F_p$ of $\chi_S$ modulo $p$ is at
most $\frac{3.5}p+e^{-c_4n}$ for an absolute constant $c_4>0$ (Proposition~\ref{prop:isocond}).
\end{itemize}
These two facts together give Proposition~\ref{thm:deg1}.

The only obstruction to the equidistribution of $\sum_j\xi_jw_j$ modulo $p$ is a dilate of $w$ with a short integer lift
(Lemma~\ref{lem:dilates}). Theorem~\ref{thm:SEplus} excludes integer lifts $z$ of eigenvectors with $\|z\|^2\le c_1np^2$, a constant
multiple of the natural scale $np^2$, by a union bound over the candidate integer lifts. In this union bound they are weighted by a discrete
Gaussian, as in the work of Banaszczyk (1993)~\cite{Ban93} on lattice points, and the additively structured candidates are counted with
the tools of Section~\ref{sec:tools}, which rest on an inequality of Hal\'asz (1977)~\cite{Hal77} and on the counting method of
Ferber--Jain--Luh--Samotij (2021)~\cite{FJLS}.

For the repeated roots we expose $S$ one row and column at a time, as Costello--Tao--Vu (2006)~\cite{CTV06} did for random symmetric
matrices, and follow the corank of $\bar S_m-\lambda I$, where $S_m$ is the leading principal submatrix of order $m$. If
$\bar S_m-\lambda I$ is invertible, with inverse $B$, then $\lambda$ is a repeated root of $\chi_{S_{m+1}}$ modulo $p$ if and only if
$\xi^{\top}B\xi=-\lambda$ and $\xi^{\top}B^2\xi=-1$, where $\xi$ is the new column (Lemmas~\ref{lem:border} and~\ref{lem:simple}(d)). So we
need Fourier bounds for the quadratic forms of the pencil $t_1B+t_2B^2$ (Proposition~\ref{prop:Ptrue}), which we prove with the squaring
trick of Lemma~\ref{lem:squaring}.

Cokernels and ranks of random symmetric integer matrices modulo primes, uniformly in the prime, were studied by Nguyen--Wood
(2025)~\cite{NW25}. For symmetric $\pm1$ matrices with random diagonal, Ferber--Jain--Sah--Sawhney (2023)~\cite{FJSS} obtained sharp
estimates for the rank distribution modulo all primes up to $\exp(cn^{1/4})$. Our matrices have zero diagonal, and besides the ranks
we need the eigenvectors and the isotropy condition at repeated roots.

\subsection{Eigenvectors without short integer lifts}
Throughout this section $p$ denotes an odd prime. Concentrations $\rho_G$, Fourier majorants $\hat\rho_G$, integer lifts, the lift
$\tilde y$ of a vector $y$ over $\F_p$ and the quantity $\mu$ are as in Section~\ref{subsec:toolsnotation}.

We first show that, at all primes $23\le p\le e^{n^{1/8}}$ simultaneously, the eigenvectors of a random Seidel matrix with eigenvalues
in $\F_p$ have no short integer lifts.

\begin{theorem}\label{thm:SEplus}
Let $S$ be a uniformly random Seidel matrix of order $n$ and $c_1:=10^{-5}$. Let $\mathcal G_{\mathrm{SE}}(S)$ be the event that for
every prime $p$ with $23\le p\le e^{n^{1/8}}$, every $\lambda\in\F_p$ and every $z\in\Z^n$ whose reduction $\bar z$ is nonzero and
satisfies $\bar S\bar z=\lambda\bar z$,
\[
\|z\|^2\ >\ c_1\,n\,p^2 .
\]
Then $\Pr(\mathcal G_{\mathrm{SE}}(S))=1-O(2^{-n/8})$ for all sufficiently large $n$.
\end{theorem}

\begin{proof}
\step{1}{We fix the parameters and check the conditions of Proposition~\ref{lem:blockdelta}.} Let $\mathcal G$ be the event of
Proposition~\ref{lem:support} for $S$. On $\mathcal G$ every $\bar z$ as in the statement has a support $D$ of size $|D|>n/20$. Fix a
prime $p$ with $23\le p\le e^{n^{1/8}}$ and put
\[
l:=\Big\lceil\frac{100\log p}{\log n}\Big\rceil+1,\quad s:=\lceil n^{0.6}\rceil,\quad
k:=\lceil n^{0.8}\rceil,\quad \delta:=\frac1{100},\quad R^2:=c_1np^2,\quad \sigma:=\frac n{2R^2}=\frac1{2c_1p^2}.
\]
Then $l\le100n^{1/8}/\log n+2$, and for large $n$ we have $s<k$ and the conditions~\eqref{eq:scondsdelta} hold with $N=p$, uniformly
in $23\le p\le e^{n^{1/8}}$. The first holds as $16l\log(200p)=O(n^{1/4})$, while $s\ge n^{0.6}$. For the second, take logarithms:
$\frac{l+1}2\log s\ge30\log p+0.6\log n$, while
\[
\log\bigl(3232\sqrt{\pi l}\,p\bigr)+2l\log(1.5l)\ \le\ 26\log p+1000\,\frac{\log p}{\log n}+0.57\log n+O(1),
\]
because $l\le\frac{100\log p}{\log n}+2$, $\log(1.5l)\le\frac18\log n+5$ and $\log\sqrt l\le\frac1{16}\log n+O(1)$. For large $n$
the right side is below $30\log p+0.6\log n$.

\step{2}{We bound the probability that a fixed vector is an eigenvector.} Fix $\lambda\in\F_p$, a set $D\subseteq[n]$ with $|D|>n/20$,
and $h:=\lfloor|D|/k\rfloor$ disjoint blocks $B_1,\dots,B_h\subseteq D$ of size $k$, depending on $D$ only. For large $n$ we have
$n/20>k$, so $h\ge1$. For $\bar z$ with support $D$ let $\rho^*(\bar z):=\min_i\rho_p(\bar z_{B_i})$, attained at a block $B^*$. The last
statement of Lemma~\ref{lem:seqexp}, with $b=B^*$ and the fixed right side $\lambda\bar z$, gives
$\Pr(\bar S\bar z=\lambda\bar z)\le\rho^*(\bar z)^{n-k}$.

\step{3}{We weight the integer lifts by a discrete Gaussian, which turns the sum over them into an expectation over a random vector
with independent coordinates.} Let $\theta_0$, $\mu_0$ and $\nu_{\max}$ be as in Lemma~\ref{lem:gauss} for $c=c_1$, so that
$\theta_0\le1.001$ and $\mu_0\le\sqrt{2\pi c_1}\,p$. Let
$Y=(Y_j)_{j\in D}$ have independent coordinates with law $\nu(z):=e^{-\sigma z^2}/\mu_0$ on $\Z\setminus p\Z$, and let $\bar\nu$ be the
image of $\nu$ on $\F_p^{\times}$. Its largest atom $\nu_{\max}$ satisfies $\frac1{p-1}\le\nu_{\max}$ and $\mu_0\nu_{\max}\le1.001$.
Bounding $\mathbf 1\{\|z\|^2\le R^2\}\le e^{\sigma(R^2-\|z\|^2)}$ and summing over the integer lifts,
\[
\begin{aligned}
\sum_{\substack{z\in\Z^n:\ \|z\|^2\le R^2\\ \supp\bar z=D}}\Pr(\bar S\bar z=\lambda\bar z)
&\le e^{\sigma R^2}\sum_{\substack{z\in\Z^n\\ \supp\bar z=D}}e^{-\sigma\|z\|^2}\rho^*(\bar z)^{n-k}
=e^{n/2}\,\theta_0^{\,n-|D|}\mu_0^{|D|}\;\E\big(\rho^*(\bar Y)^{n-k}\big).
\end{aligned}
\]

\step{4}{We split the expectation according to the level of $\rho^*(\bar Y)$.} Put $\rho_a:=2^a(1+\delta)/p$ for the integers $a\ge0$ with
$\rho_a\le1$. There are at most $2n^{1/8}$ of them. Either $\rho^*(\bar Y)<(1+\delta)/p$, or $\rho^*(\bar Y)\in[\rho_a,2\rho_a)$ for one
of them. The blocks are disjoint, so the vectors $\bar Y_{B_i}$ are independent, and Proposition~\ref{lem:blockdelta}(b) (with $N=p$,
the law $\bar\nu$ and $\rho_0=\rho_a\ge(1+\delta)/p$) gives
\[
\Pr\bigl(\rho^*(\bar Y)\ge\rho_a\bigr)\le\min\bigl(1,\zeta/\rho_a\bigr)^{h(k-s)},\qquad
\zeta:=\frac{6464\sqrt{\pi l}\,\nu_{\max}}{\sqrt s},
\]
where $h(k-s)\ge(|D|/k-1)(k-s)\ge|D|(1-\frac sk)-k$. For large $n$ we have $\zeta\le1$, and with
$M':=\min(n-k,h(k-s))\ge|D|(1-\frac sk)-k$ we get
\[
(2\rho_a)^{n-k}\min\bigl(1,\zeta/\rho_a\bigr)^{h(k-s)}\le2^n\zeta^{M'}\rho_a^{n-k-M'}\le
2^n\zeta^{|D|(1-s/k)-k}.
\]
Now $\mu_0\zeta\le\gamma:=6471\sqrt{\pi l/s}$ and
$\zeta^{-1}\le\sqrt s\,(p-1)/(6464\sqrt{\pi l})\le\sqrt s\,p$, so
\[
\mu_0^{|D|}(2\rho_a)^{n-k}\Pr\bigl(\rho^*(\bar Y)\ge\rho_a\bigr)\le2^n\gamma^{|D|}
(\sqrt s\,p)^{s|D|/k+k}\le2^n\gamma^{n/20}(\sqrt s\,p)^{3n^{0.8}},
\]
using $\gamma<1$ (for large $n$), $|D|>n/20$ and $sn/k+k\le3n^{0.8}$. For the lowest level, as
$\mu_0/p\le\sqrt{2\pi c_1}$,
\[
\mu_0^{|D|}\Big(\frac{1+\delta}p\Big)^{n-k}\le1.01^n(\mu_0/p)^{|D|}p^{-(n-|D|)}p^k\le
\big(1.01\big(\tfrac1p+\sqrt{2\pi c_1}\big)\big)^np^k .
\]

\step{5}{We sum over $\lambda$, $D$, the levels and the primes, and conclude the proof.} Since
$\E(\rho^*(\bar Y)^{n-k})\le((1+\delta)/p)^{n-k}+\sum_a(2\rho_a)^{n-k}\Pr(\rho^*(\bar Y)\ge\rho_a)$ and
$\theta_0^{n-|D|}\le1.001^n$, the last two displays bound the sum over integer lifts for fixed $\lambda$ and $D$. Summing over the at
most $p$ values of $\lambda$, the $2^n$ sets $D$ and the at most $2n^{1/8}$ levels, the probability that some $\bar z$ with
$\lvert\supp\bar z\rvert>n/20$ and an integer lift of squared norm at most $R^2$ satisfies $\bar S\bar z=\lambda\bar z$ for some $\lambda\in\F_p$
is at most
\[
p^{k+1}\big(2.02\sqrt e\cdot1.001\,(\tfrac1{23}+\sqrt{2\pi c_1})\big)^n
+2n^{1/8}p\,(4\sqrt e\cdot1.001)^n\gamma^{n/20}(\sqrt s\,p)^{3n^{0.8}}.
\]
Here $2.02\sqrt e\cdot1.001(\frac1{23}+\sqrt{2\pi c_1})<0.18$ and $p^{k+1}\le e^{2n^{0.925}}$, so the first term is at most
$e^{-1.7n+2n^{0.925}}$. In the second, $l=O(n^{1/8})$ gives $\gamma=O(n^{1/16-0.3})$, so
$\gamma^{n/20}\le e^{-0.0118\,n\log n+O(n)}$, while the other factors are $e^{O(n)}$ and $e^{O(n^{0.93})}$. Summing over the primes
$p\le e^{n^{1/8}}$ and adding $\Pr(\mathcal G^c)=O(2^{-n/8})$ gives the claim.
\end{proof}

With Lemma~\ref{lem:dilates}, the absence of short integer lifts gives the equidistribution of eigenvectors.

\begin{corollary}\label{cor:ACplus}
On the event $\mathcal G_{\mathrm{SE}}(S)$ of Theorem~\ref{thm:SEplus}, which has probability $1-O(2^{-n/8})$ for large $n$, we
have, for every prime $23\le p\le e^{n^{1/8}}$ and every eigenvector $w\in\F_p^n$ of $\bar S$ with eigenvalue in $\F_p$,
\[
\rho_p(w)\ \le\ \frac1p+e^{-2c_1n}.
\]
In particular $\Pr_\xi(\xi^{\top}w=0)\le\frac1p+e^{-2c_1n}$.
\end{corollary}

\begin{proof}
For $c\in\F_p^{\times}$ the dilate $cw$ is an eigenvector with the same eigenvalue, and its lift $\widetilde{cw}$ is a vector $z$ as in
Theorem~\ref{thm:SEplus}. So
$\mu(w)>c_1np^2$ on $\mathcal G_{\mathrm{SE}}(S)$, and Lemma~\ref{lem:dilates} gives
$\rho_p(w)\le\frac1p+e^{-2\mu(w)/p^2}<\frac1p+e^{-2c_1n}$.
\end{proof}

Since the lift of every vector of $\F_p^n$ has squared norm less than $np^2/4$, Theorem~\ref{thm:SEplus} is sharp up to the value of
$c_1$. For kernel vectors of symmetric $\pm1$ matrices with random diagonal, anti-concentration modulo primes up to $\exp(cn^{1/4})$
was established by Ferber--Jain--Sah--Sawhney (2023)~\cite{FJSS}.

\subsection{Quadratic forms of resolvents}
We now bound the Fourier coefficients of $t_1\xi^{\top}B\xi+t_2\xi^{\top}B^2\xi$ with $(t_1,t_2)\ne(0,0)$, where
$B:=(\bar T-\lambda I)^{-1}$ for some $\lambda\in\F_p$ that is not an eigenvalue of $\bar T$. The case $t_2=0$ controls the step from
corank zero in the exposure of Section~\ref{subsec:corank} below, and the general case controls the isotropy condition at repeated roots.
Since $t_1B+t_2B^2=(t_1(\bar T-\lambda I)+t_2I)(\bar T-\lambda I)^{-2}$, we treat more generally the matrices $P(\bar T)Q(\bar T)^{-1}$
with $Q\in\F_p[x]$ monic of degree $2$ and $P\in\F_p[x]$ nonzero of degree at most $1$. This also covers the traces of resolvents over
$\F_{p^2}$ in Section~\ref{sec:deg2}. By Lemma~\ref{lem:squaring}, with $u$ and its support $U$ as there, the Fourier bound reduces to
excluding short integer lifts of $y_{U^c}$, where $y$ is a nonzero multiple of $P(\bar T)Q(\bar T)^{-1}u$. With $v:=Q(\bar T)^{-1}u$, the
pair $(v,\bar Tv)$ satisfies a linear system with fixed right sides, and the union bound runs over vectors with entries in $\F_p^2$.
The Gaussian weight is available only on the coordinates of $y$ in $U^c$, so we first show that $y$ has many nonzero coordinates there.

\begin{lemma}\label{lem:pencilsupp}
Let $\tau':=\frac1{100}$, and let $T$ be a uniformly random Seidel matrix of order $m$. For all sufficiently large $m$, with
probability at least $1-e^{-m}$ over $T$ the following holds for all $U\subseteq[m]$ with $0.4m\le|U|\le0.6m$, all primes
$23\le p\le e^{m^{1/8}}$, all monic $Q\in\F_p[x]$ of degree $2$ such that $Q(\bar T)$ is invertible, all nonzero $P\in\F_p[x]$ of
degree at most $1$, and all $R\subseteq U^c$ with $|U^c\setminus R|\le\tau'm$: the matrix $C:=P(\bar T)Q(\bar T)^{-1}$ satisfies
\[
\rank C_{R,U}\ \ge\ 0.13m .
\]
Consequently, for such $T$, $U$, $p$ and $Q$, and for $u$ with support $U$ and independent uniform signs on $U$,
\begin{multline*}
\Pr_u\bigl(\lvert\supp(P(\bar T)Q(\bar T)^{-1}u)\cap U^c\rvert<\tau'm\ \text{for some nonzero }P\in\F_p[x]\text{ with }\deg P\le1\bigr)\\
\le\ p^2\,2^{-0.05m}.
\end{multline*}
\end{lemma}

\begin{proof}
\step{1}{We replace $C_{R,U}$ by a matrix $K$ that has the same nullity and does not involve $Q(\bar T)^{-1}$.} Put
$\delta:=\tau'=\frac1{100}$. Fix $U$, $p$, $Q$, $P$ and $R$ as in the statement, write $Q=x^2+b_1x+b_0$ and $P=a_1x+a_0$, and note that
$U\subseteq R^c$. We identify $\F_p^U$ and $\F_p^{R^c}$ with the spaces of the vectors of $\F_p^m$ supported in $U$ and in $R^c$. If
$Q(\bar T)$ is invertible, $u'\in\F_p^U$ and $(Cu')_R=0$, then $y':=Cu'$ lies in $\F_p^{R^c}$ and satisfies
$Q(\bar T)y'=P(\bar T)u'$. Conversely, such a pair $(y',u')$ gives $Cu'=y'$. So the nullity of $C_{R,U}$ equals that of the
$m\times(|R^c|+|U|)$ matrix
\[
K:=\big[\,Q(\bar T)_{\cdot,R^c}\ \big|\ {-P(\bar T)_{\cdot,U}}\,\big],
\]
and $\rank C_{R,U}=\rank K-|R^c|$. For a row vector of length $|R^c|+|U|$, such as a row of $K$, we call its last $|U|$ entries its
\emph{$U$-part}. The matrix $K$ is defined for every $T$, and we show that $\rank K\ge|R^c|+0.13m$ for all choices as above, outside
an event of probability at most $e^{-m}$. Invertibility of $Q(\bar T)$ is used only in this deterministic step.

\step{2}{We decouple by conditioning on all entries of $T$ outside one off-diagonal block, in the spirit of Costello--Tao--Vu
(2006)~\cite{CTV06}.} Choose disjoint $V_1,V_2\subseteq R$ with $|V_1|=|V_2|=\lfloor|R|/2\rfloor$, depending only on $R$, and
condition on all entries of $T$ except the block $T_{V_1,V_2}$ (and its transpose). The entries of this block and of its transpose have
both indices in $V_1\cup V_2\subseteq R$. Every term $\bar T_{il}\bar T_{lj}$ of $(\bar T^2)_{U,R^c}$ has $i,j\in R^c$, so that each of
its two factors has an index in $R^c$, and the entries of $\bar T_{U,R^c}$ and of $P(\bar T)_{U,U}$ have both indices in $R^c$. So the
rows $K_U$ do not involve $T_{V_1,V_2}$. For $i\in V_1$, splitting the sum over
$l$ according to $l\in V_2$ or not, $Q(\bar T)_{i,R^c}=\bar T_{i,V_2}W+f_i$ with $W:=\bar T_{V_2,R^c}$ and $f_i$ fixed (note that
$\bar T_{i,R^c}$ is fixed and $I_{i,R^c}=0$), while $P(\bar T)_{i,U}=a_1\bar T_{i,U}$ is fixed. So the rows $K_i$, $i\in V_1$, are
$\bar T_{i,V_2}[\,W\,|\,0\,]+h_i$ with $h_i$ fixed, and the vectors $T_{i,V_2}$, $i\in V_1$, are independent and uniform in
$\{\pm1\}^{V_2}$.

\step{3}{We bound the probability that a new row fails to raise the rank.} Let $r_P$ be the corank of
$P(\bar T)_{U,U}=a_1\bar T_{U,U}+a_0I$. Then $\rank K_U\ge|U|-r_P$, and the vectors with zero $U$-part in the span of $K_U$ and of $g$
further rows form a space of dimension at most $r_P+g$ (the $U$-parts of the rows of $K_U$ are the rows of $-P(\bar T)_{U,U}$). Add
the rows $K_i$, $i\in V_1$, one at a time, and let $g$ be the number of them that have raised the rank so far. The vectors
$v[\,W\,|\,0\,]$, $v\in\F_p^{V_2}$, have zero $U$-part, so the next row $K_i$ fails to raise the rank only if $v=\bar T_{i,V_2}$ lies
in a translate of a subspace of $\F_p^{V_2}$ of dimension at most $(|V_2|-\rank W)+r_P+g$. By Lemma~\ref{lem:odlyzko}, given the rows
added before, this has conditional probability at most $2^{r_P+g-\rank W}$. If $a_1=0$, then $a_0\ne0$ and $r_P=0$. If $a_1\ne0$, then
$r_P=\corank(\bar T_{U,U}-\vartheta I)$ with $\vartheta:=-a_0/a_1\in\F_p$.

\step{4}{We show that $W$ has almost full rank and that $r_P$ is small, outside very unlikely events.} The matrix $W$ has independent
sign entries, and $|R^c|\ge|U|\ge|V_2|$. Exposing its rows one at a time, each fails to raise the rank with probability at most
$2^{-\delta m}$ while the rank is below $|V_2|-\delta m$ (Lemma~\ref{lem:odlyzko}). So $\rank W\ge|V_2|-\delta m$ except
with probability at most $2^m2^{-\delta^2m^2}$. Next, for a fixed $\vartheta\in\F_p$ and a fixed set $U$, the matrix $T_{U,U}$ is a
uniformly random Seidel matrix of order $|U|$, so by Lemma~\ref{lem:apriori} the corank of $\bar T_{U,U}-\vartheta I$ exceeds $\delta m$
with probability at most $30\cdot2^{-\delta^2m^2/4}$. With probability $1-e^{-\Omega(m^2)}$, therefore, $\rank W\ge|V_2|-\delta m$
and $r_P\le\delta m$ for all $P$. Here the union bound over $\vartheta\in\F_p$, the sets $U$ and $R$, and $p$ costs $e^{O(m)}$, and
these events involve neither $T_{V_1,V_2}$ nor $Q$.

\step{5}{We count the rows that raise the rank and conclude the proof of the rank bound.} On these events, while $g<\rank W-r_P-\delta m$
each row fails with probability at most $2^{-\delta m}$. So at the end
$g\ge\min(|V_1|,\rank W-r_P-\delta m)-\delta m\ge|V_2|-3\delta m-r_P$, except with probability at most
$2^m2^{-\delta^2m^2}$. Then, as $|R^c|-|U|=|U^c\setminus R|\le\tau'm$,
\[
\rank K-|R^c|\ge|U|-r_P+g-|R^c|\ge g-r_P-\tau'm\ge|V_2|-5\delta m-\tau'm\ge0.13m,
\]
as $|R|\ge0.39m$ and so $|V_2|\ge(0.39m-1)/2$. Union bounds over the $e^{O(m)}$ choices of $U$, $p$, $Q$, $P$ and $R$ give the first
claim.

\step{6}{We deduce the second claim from the rank bound and Odlyzko's lemma.} Suppose that $\lvert\supp(Cu)\cap U^c\rvert<\tau'm$ for some $P$, where
$C=P(\bar T)Q(\bar T)^{-1}$. Then $R:=U^c\setminus\supp(Cu)$ satisfies $|U^c\setminus R|<\tau'm$ and $(Cu)_R=0$. For fixed $P$ and $R$,
Lemma~\ref{lem:odlyzko} gives $\Pr_u({(Cu)_R=0})\le2^{-\rank C_{R,U}}\le2^{-0.13m}$, as $(Cu)_R=C_{R,U}u_U$. The support of $Cu$ does not
change if $P$ is multiplied by a nonzero scalar, so it suffices to consider at most $p+1\le p^2$ polynomials $P$. With $N:=|U^c|$, which
lies in $[0.4m,0.6m]$, there are at most $\sum_{j\le\tau'm}\binom Nj\le2^{NH(\tau'm/N)}\le2^{0.6H(1/60)m}\le2^{0.074m}$ sets $R$, as
$\tau'm\le N/2$ and $NH(\tau'm/N)$ increases with $N$. So the probability is at most $p^2\,2^{0.074m-0.13m}\le p^2\,2^{-0.05m}$.
\end{proof}

With Lemma~\ref{lem:pencilsupp} and the block estimate in dimension two (Proposition~\ref{lem:blockd} with $d=2$) we obtain the Fourier bound.

\begin{proposition}\label{prop:Ptrue}
Let $T$ be a uniformly random Seidel matrix of order $m$, $c_2:=e^{-1500}$, and let $u$, $U$ be as in Lemma~\ref{lem:squaring}. For
all sufficiently large $m$, with probability at least $1-e^{-m/200}$ over $T$ the following holds for every prime
$23\le p\le e^{m^{1/8}}$ and every monic $Q\in\F_p[x]$ of degree $2$ such that $Q(\bar T)$ is invertible:
\[
\Pr_u\Bigl(\mu\big((P(\bar T)Q(\bar T)^{-1}u)_{U^c}\big)\le c_2mp^2\ \text{for some nonzero }P\in\F_p[x]\text{ with }\deg P\le1\Bigr)\
\le\ e^{-m/100},
\]
and consequently $|\E_\xi\exp(2\pi i\,\xi^{\top}P(\bar T)Q(\bar T)^{-1}\xi/p)|\le e^{-c_2m/2}$ for every nonzero $P\in\F_p[x]$ of degree at
most $1$. In particular, if $\lambda\in\F_p$ and $\bar T-\lambda I$ is invertible, then with $B:=(\bar T-\lambda I)^{-1}$, for all
$(t_1,t_2)\in\F_p^2\setminus\{(0,0)\}$,
\[
\Big|\E_\xi\exp\Bigl(\frac{2\pi i\,(t_1\,\xi^{\top}B\xi+t_2\,\xi^{\top}B^2\xi)}p\Bigr)\Big|\ \le\ e^{-c_2m/2}.
\]
\end{proposition}

\begin{proof}
\step{1}{We fix the parameters and reduce the first claim to a bound for the probability that a linear system for pairs of vectors
has a solution with a short integer lift and a large support.} Fix $(p,Q)$, write $Q=x^2+b_1x+b_0$, and fix a vector
$u\in\{-1,0,1\}^m$ whose support $U$ satisfies $0.4m\le|U|\le0.6m$. Let $\tau'=\frac1{100}$ be as in Lemma~\ref{lem:pencilsupp}, and put
$m':=|U^c|$, $G:=\F_p^2$, $R^2:=c_2mp^2$, $\sigma:=1/(2c_2p^2)$ (so $\sigma R^2=m/2$), $\delta:=\frac1{100}$, and
\[
l:=\Big\lceil\frac{200\log p}{\log m}\Big\rceil+50,\qquad s:=\lceil m^{0.6}\rceil,\qquad k:=\lceil m^{0.8}\rceil .
\]
For large $m$ these satisfy the conditions of Proposition~\ref{lem:blockd} with $d=2$, uniformly in $23\le p\le e^{m^{1/8}}$. Indeed
$l\ge50$, $s<k$ and $s\ge1400l$. As $2l\log(1.5l)\le\frac14l\log m$ for large $m$, $s^{l/2}(1.5l)^{-2l}\ge m^{0.05l}\ge p^{10}$, which
exceeds both $100p$ and $\frac{32(1+\delta)}\delta\sqrt{\pi l}\,p^2s^{-1/2}$ (the latter gives the second condition
in~\eqref{eq:scondsdelta} with $N=p^2$). Finally, $s=\lceil m^{0.6}\rceil$ is much larger than $16l\log(200p^2)$.

If $Q(\bar T)$ is invertible and $\mu((P(\bar T)Q(\bar T)^{-1}u)_{U^c})\le R^2$ for some nonzero $P=a_1x+a_0$, then some nonzero
multiple $y$ of $P(\bar T)Q(\bar T)^{-1}u$ is such that $y_{U^c}$ has an integer lift $z$ with $\|z\|^2\le R^2$. With
$v:=Q(\bar T)^{-1}u$ and $v':=\bar Tv$,
\begin{equation}\label{eq:pairP}
\bar Tv=v',\qquad\bar Tv'=u-b_1v'-b_0v,
\end{equation}
and $y_j=\ell(v_j,v'_j)$ for all $j$, where $\ell$ is a nonzero multiple of the linear form $(g_1,g_2)\mapsto a_0g_1+a_1g_2$ on $G$.
Let $\mathcal S$ be the event of Lemma~\ref{lem:pencilsupp} for $T$. It has probability at least $1-e^{-m}$ and does not depend on
$(p,Q)$, and on it, if $Q(\bar T)$ is invertible, then for every $U$ with $0.4m\le|U|\le0.6m$,
\[
\Pr_u\bigl(\lvert\supp(P(\bar T)Q(\bar T)^{-1}u)\cap U^c\rvert<\tau'm\text{ for some nonzero }P\text{ with }\deg P\le1\mid T,U\bigr)\le
p^2\,2^{-0.05m}.
\]
Let $\Psi(u)$ be the probability over $T$ that~\eqref{eq:pairP} holds for some pair $(v,v')$ and some nonzero linear form $\ell$ on $G$
such that the vector $y$ with $y_j=\ell(v_j,v'_j)$ satisfies $\lvert\supp y\cap U^c\rvert\ge\tau'm$ and $y_{U^c}$ has an integer lift of squared
norm at most $R^2$. In Steps~2--4 we show that $\Psi(u)\le e^{-m/2}$ for large $m$.

\step{2}{We change coordinates, sort the coordinates into classes and set up the weights.} Fix $\ell$ and an $\F_p$-linear bijection
$\varphi:G\to G$ whose second coordinate is $\ell$, and put $w_j:=\varphi(v_j,v'_j)=(\check v_j,y_j)$. Then~\eqref{eq:pairP} is
equivalent to $\sum_jT_{ij}w_j=\varphi(v'_i,u_i-b_1v'_i-b_0v_i)$ for all $i$, a linear system for $T$ whose right side is determined by
$w$ and $u$. Each coordinate $j$ lies in one of five classes, namely $Z$ ($w_j=0$), $D_1$ ($j\in U$, $y_j\ne0$), $D_2$ ($j\in U^c$,
$y_j\ne0$), $D_3$ ($j\in U$, $y_j=0\ne\check v_j$) and $D_4$ ($j\in U^c$, $y_j=0\ne\check v_j$). For each pattern
$\mathcal D=(D_1,\dots,D_4)$ with $|D_2|\ge\tau'm$ fix disjoint blocks of size $k$, $\lfloor|D_i|/k\rfloor$ of them inside each $D_i$,
depending only on $\mathcal D$. There is at least one, as $|D_2|\ge\tau'm\ge k$. Let $\rho^*(w)$ be the smallest value of $\rho_G$ on
these blocks. Ordering the coordinates so that a block attaining $\rho^*$ comes last, Lemma~\ref{lem:seqexp} (for vectors with entries
in $G$) shows that~\eqref{eq:pairP} has probability at most $\rho^*(w)^{m-k}$ over $T$, the right side being fixed. On $U^c$ we bound
$\mathbf 1\{\|z\|^2\le R^2\}\le e^{\sigma(R^2-\|z\|^2)}$ and sum over the integer lifts $z$ of $y_{U^c}$, while all other coordinates are
counted with weight $1$. Let $\theta_0$ and $\mu_0$ be as in Lemma~\ref{lem:gauss} for $c=c_2$ (note $c_2\le10^{-5}$), and let
$\nu^{\rm G}$ be the image on $\F_p^{\times}$ of the probability law proportional to $e^{-\sigma z^2}$ on $\Z\setminus p\Z$, with largest
atom $\nu^{\rm G}_{\max}$. Then $\theta_0\le1.001$, $\mu_0\le\sqrt{2\pi c_2}\,p$ and $\mu_0\nu^{\rm G}_{\max}\le1.001$. The total weight
of a class at one coordinate (its \emph{mass}) and the normalized law $\nu_i$ on $G\setminus\{0\}$ are as follows.
\begin{itemize}[leftmargin=*]
\item $D_1$: mass $p(p-1)$, and $\nu_1$ is uniform on $\F_p\times\F_p^{\times}$.
\item $D_2$: mass $p\mu_0$, and $\nu_2$ is the product of the uniform law on $\F_p$ and $\nu^{\rm G}$.
\item $D_3$: mass $p-1$, and $\nu_3$ is uniform on $\F_p^{\times}\times\{0\}$.
\item $D_4$: mass $(p-1)\theta_0$, and $\nu_4=\nu_3$.
\item $Z$: mass $1$ on $U$ and $\theta_0$ on $U^c$.
\end{itemize}
Every term also carries the global factor $e^{\sigma R^2}=e^{m/2}$. By the list, $\mathrm{mass}_i\nu_{i,\max}\le1.001$ for every class
$i\in\{1,2,3,4\}$, where $\nu_{i,\max}$ is the largest atom of $\nu_i$. The lines of $G$ are $\{y=0\}$ and the lines
$\{\check v=\alpha y\}$, $\alpha\in\F_p$. Each line $\{\check v=\alpha y\}$ has $\nu_1$- and $\nu_2$-measure $\frac1p$, and $\{y=0\}$
has measure $0$. So the largest measure $\kappa_{1,i}$ of a line under $\nu_i$, which is the quantity $\kappa_1$ of
Proposition~\ref{lem:blockd} for $\nu_i$, is $\frac1p$ for $i=1,2$, while $\kappa_{1,3}=\kappa_{1,4}=1$. Hence
$\mathrm{mass}_i\kappa_{1,i}\le1.001p$ for $i\ne2$, and $\mathrm{mass}_2\kappa_{1,2}=\mu_0\le\sqrt{2\pi c_2}\,p$.

\step{3}{We bound the contribution of the lowest level.} The total mass at a coordinate is $p^2$ on $U$ and $p(\theta_0+\mu_0)$ on
$U^c$. Summing over the fewer than $p^2$ forms $\ell$ and over all $w$ with $\rho^*(w)<(1+\delta)/p^2$, the contribution is at most
\[
p^2e^{m/2}\,p^{2|U|}\big(p(\theta_0+\mu_0)\big)^{m'}\Big(\frac{1+\delta}{p^2}\Big)^{m-k}\le
p^{2k+2}\Big(e^{1/2}(1+\delta)\big(\tfrac{1.001}{23}+\sqrt{2\pi c_2}\big)^{0.4}\Big)^m\le
p^{2k+2}\,0.51^m .
\]

\step{4}{We bound the contribution of the structured levels.} For the integers $a\ge0$ with $\rho_a:=2^a(1+\delta)/p^2\le1$, fix a
pattern with $|D_2|\ge\tau'm$ and consider the $w$ with $\rho^*(w)\in[\rho_a,2\rho_a)$. Every block then has
$\hat\rho_G\ge\rho_G\ge\rho_a$, and under the product of the laws $\nu_i$ the blocks are independent. By Proposition~\ref{lem:blockd}
with $d=2$, a block of class $i$ has $\hat\rho_G\ge\rho_a$ with probability at most $\zeta_i^{k-s}$, where
$\zeta_i:=\gamma_{1,i}(\rho_a)+1.1\kappa_{1,i}\mathbf 1\{\rho_a\in W^{(1)}\}$, $\gamma_{1,i}$ is the function $\gamma_1$ of that
proposition for $\nu_i$, and $W^{(1)}\subseteq(0,1.1/p)$ is its window. Let $M_i$ be $k-s$ times the number of blocks in $D_i$, and
$M_{\rm bl}:=\sum_iM_i$. As $p^{-2}\le\rho_a\le1$ and $M_{\rm bl}\le m$, we have $(2\rho_a)^{m-k}\rho_a^{-M_{\rm bl}}\le2^mp^{2k}$. In
each block, $k-s$ of the mass factors are paired with the factor $\zeta_i$ of the block bound and with one factor $\rho_a$ (we call these
the \emph{chain coordinates}, and only their number matters). Since $\sum_i(|D_i|-M_i)\le4k+s\lceil m/k\rceil\le6k$, at most $6k$
coordinates outside $Z$ are not paired, each of mass at most $p^2$. The masses of $Z$ contribute at most $1.001^m$. So the contribution
of the pattern and the level, summed over $\ell$, is at most
\[
p^2\,e^{m/2}\,1.001^m\,p^{12k}\cdot2^mp^{2k}\prod_i\big(\mathrm{mass}_i\,\zeta_i\,\rho_a\big)^{M_i}.
\]
Per chain coordinate of class $i$,
$\mathrm{mass}_i\,\gamma_{1,i}(\rho_a)\,\rho_a=\mathrm{mass}_i\nu_{i,\max}\cdot64(1+\delta)\sqrt{\pi l}/(\delta\sqrt s)\le1.001\gamma_0$
with $\gamma_0:=6464\sqrt{\pi l/s}$, and $1.1\,\mathrm{mass}_i\kappa_{1,i}\rho_a\le1.21\,\mathrm{mass}_i\kappa_{1,i}/p$ when
$\rho_a\in W^{(1)}$. So if $\rho_a\notin W^{(1)}$, every factor is at most $1.001\gamma_0$. If $\rho_a\in W^{(1)}$, the factors are at
most $1.001\gamma_0+1.22\le1.23$ for $i\ne2$ and at most $1.001\gamma_0+1.21\sqrt{2\pi c_2}\le2.42\sqrt{2\pi c_2}$ for $i=2$, once $m$
is so large that $\gamma_0\le\sqrt{2\pi c_2}$. As $M_2\ge(k-s)(\tau'm/k-1)\ge\tau'm/2$, the product is at most
$(1.001\gamma_0)^{\tau'm/2}$ in the first case and at most $1.23^m(2.42\sqrt{2\pi c_2})^{\tau'm/2}$ in the second. Finally
$p^{14k+2}=e^{O(m^{0.93})}$, there are at most
$3^m$ patterns and at most $3m^{1/8}$ levels, $2\cdot1.001e^{1/2}\cdot3<9.91$ and $2\cdot1.001e^{1/2}\cdot3\cdot1.23<12.2$. Hence, with
the bound of Step~3,
\[
\begin{aligned}
\Psi(u)&\le p^{2k+2}0.51^m+e^{o(m)}\Big(9.91^m(1.001\gamma_0)^{m/200}
+12.2^m\big(2.42\sqrt{2\pi c_2}\big)^{m/200}\Big)\\
&\le e^{-m/2}
\end{aligned}
\]
for large $m$. Indeed, the first term is $e^{-0.67m+o(m)}$, in the second $\gamma_0\to0$, and in the third
$2.42\sqrt{2\pi c_2}<e^{-745}$, so that $12.2\,(2.42\sqrt{2\pi c_2})^{1/200}<12.2\,e^{-3.7}<e^{-1.1}$.

\step{5}{We combine the bound on $\Psi(u)$ with Lemma~\ref{lem:pencilsupp} and Markov's inequality, and deduce the claims of the
proposition.} Fix $(p,Q)$, and let $\Phi(T)$ be the probability over $u$ in the statement if $Q(\bar T)$ is invertible, and $\Phi(T):=0$
otherwise. If the event in the statement occurs for $u$, then either ${\lvert\supp(P(\bar T)Q(\bar T)^{-1}u)\cap U^c\rvert}<\tau'm$ for the polynomial
$P$ in question, or~\eqref{eq:pairP} holds for a pair counted in $\Psi(u)$ (note that $\supp y=\supp(P(\bar T)Q(\bar T)^{-1}u)$). As
$|U|$ is the sum of the $m$ independent indicators $\mathbf 1\{u_j\ne0\}$ of mean $\frac12$, Hoeffding's inequality
(1963)~\cite[Theorem~2]{Hoe63} gives $\Pr\bigl(\big||U|-\frac m2\big|>0.1m\bigr)\le2e^{-m/50}$. As $u$ is independent of $T$, this bound,
Lemma~\ref{lem:pencilsupp} and the bound on $\Psi(u)$ give
\[
\E_T\Phi(T)\le2e^{-m/50}+\Pr(\mathcal S^c)+p^2\,2^{-0.05m}+\max_u\Psi(u)\le3e^{-m/50}
\]
for large $m$, the maximum being over the $u$ with $0.4m\le|U|\le0.6m$. By Markov's inequality, $\Phi(T)>e^{-m/100}$ with
probability at most $3e^{-m/100}$, and a union bound over the at most $e^{3m^{1/8}}$ pairs $(p,Q)$ gives the first claim. For the
second, let $P\ne0$ have degree at most $1$. The matrix $C:=P(\bar T)Q(\bar T)^{-1}$ is symmetric, as $\bar T$ is, and the first claim
is the hypothesis of the last statement of Lemma~\ref{lem:squaring} for $C$, with $t=1$ and $c=c_2$. That statement gives
$|\E_\xi\exp(2\pi i\,\xi^{\top}C\xi/p)|\le\sqrt2\,e^{-c_2m}\le e^{-c_2m/2}$ for large $m$. Finally, if $\lambda\in\F_p$ and
$\bar T-\lambda I$ is invertible, then $t_1B+t_2B^2=P(\bar T)Q(\bar T)^{-1}$ with $Q:=(x-\lambda)^2$, for which
$Q(\bar T)=(\bar T-\lambda I)^2$ is invertible, and $P:=t_1(x-\lambda)+t_2$, which is nonzero as $(t_1,t_2)\ne(0,0)$.
\end{proof}

\subsection{Coranks}\label{subsec:corank}
We now expose $S$ one row and column at a time, following Costello--Tao--Vu (2006)~\cite{CTV06}. By Lemma~\ref{lem:border}, the corank
of $\bar S_m-\lambda I$, where $S_m$ is the leading principal submatrix of order $m$, changes by at most one when $m$ grows by one, and it
decreases unless the new column is orthogonal to the current kernel. We first state the resulting bookkeeping for eigenvalues in every
finite field $\F_q$ of odd characteristic with $q\ge23$, as Sections~\ref{sec:deg2} and~\ref{sec:deg3} use it again. It combines two
transition bounds with the a priori bound of Lemma~\ref{lem:apriori} through a potential function.

\begin{lemma}\label{lem:corank}
Let $p$ be an odd prime, $d\ge1$, $q:=p^d$ with $q\ge23$, and $\lambda\in\F_q$. Let $S$ be a uniformly random Seidel matrix of order
$n\ge4$, and for $m\le n$ let $S_m$ be its leading principal submatrix of order $m$, $L_m:=\bar S_m-\lambda I$ over $\F_q$,
$K_m:=\ker L_m\subseteq\F_q^m$ and $r_m:=\dim_{\F_q}K_m$. For $m<n$ let $\xi\in\{\pm1\}^m$ be the part of column $m+1$ of $S$ above the
diagonal. Let $m_0:=\lceil3n/4\rceil$ and $0\le\varpi\le10^{-5}$, and for $m_0\le m<n$ let $\mathcal E_m$ be an event determined by
$S_m$ on which, conditionally on $S_m$,
\begin{enumerate}[label=(\roman*)]
\item $\Pr(r_{m+1}=1)\le q^{-1}+\varpi$ if $r_m=0$, and
\item $\Pr(\bar\xi\perp K_m)\le q^{-r_m}+\varpi$ if $r_m\ge1$.
\end{enumerate}
Put $\mathcal E_{<m}:=\bigcap_{m_0\le j<m}\mathcal E_j$, the whole space for $m=m_0$, and let $C_0$ be the constant of
Lemma~\ref{lem:apriori}. Then, for $m_0\le m\le n$,
\[
\Pr\bigl(r_m\ge2;\ \mathcal E_{<m}\bigr)\le\frac{1.2}{q^2}+\varpi+C_0\,0.163^{(m-m_0-1)/2},
\]
and $\Pr(r_n\ge1)\le\frac1q+\frac{1.2}{q^2}+2\varpi+C_0\,0.163^{(n-m_0-2)/2}+\Pr(\mathcal E_{<n}^c)$.
\end{lemma}

\begin{proof}
\step{1}{We bound the conditional expectation of a potential after one step.} For $m<n$ we have
\[
L_{m+1}=\begin{pmatrix}L_m&\bar\xi\\ \bar\xi^{\top}&-\lambda\end{pmatrix},
\]
and $\xi$ is uniform and independent of $S_m$. By Lemma~\ref{lem:border}, applied over $\F_q$ with $L=L_m$ and $c=-\lambda$, we have
$|r_{m+1}-r_m|\le1$, $r_{m+1}\le1$ if $r_m=0$, and $r_{m+1}\ge r_m$ only if $\bar\xi\perp K_m$. Let $V(r):=10^{r-2}$ for $r\ge2$ and
$V(r):=0$ for $r\in\{0,1\}$, and put $\theta_1:=q^{-1}+\varpi$, $\theta_2:=q^{-2}+\varpi$ and $\kappa:=0.1+10\theta_2$. We show that, for
$m_0\le m<n$, on $\mathcal E_m$
\[
\E\big(V(r_{m+1})\mid S_m\big)\le\theta_1\mathbf 1\{r_m=1\}+\kappa V(r_m),\qquad\Pr(r_{m+1}\ge1\mid S_m)\le\theta_1+\mathbf 1\{r_m\ge2\},
\]
where the second inequality is trivial if $r_m\ge2$. If $r_m=0$, then $V(r_{m+1})=0$, and $\Pr(r_{m+1}\ge1\mid S_m)\le\theta_1$ by (i). If
$r_m=1$, then $V(r_{m+1})=\mathbf 1\{r_{m+1}=2\}$, and both $r_{m+1}=2$ and $r_{m+1}\ge1$ require $\bar\xi\perp K_m$, which has
conditional probability at most $\theta_1$ by (ii). If $r_m=2$, then $V(r_{m+1})\le10\cdot\mathbf 1{\{\bar\xi\perp K_m\}}$, of conditional
expectation at most $10\theta_2\le\kappa V(2)$. If $r_m=r\ge3$, then $V(r_{m+1})\le{V(r-1)+V(r+1)\mathbf 1\{\bar\xi\perp K_m\}}$, of
conditional expectation at most $0.1V(r)+10(q^{-r}+\varpi)V(r)\le\kappa V(r)$.

\step{2}{We solve the resulting recursion and deduce both claims.} Put $v_m:=\E(V(r_m);\ \mathcal E_{<m})$ and
$a_m:=\Pr(r_m\ge1;\ \mathcal E_{<m})$. As $\mathcal E_{<m+1}\subseteq\mathcal E_{<m}\cap\mathcal E_m$ is determined by $S_m$, Step~1 and
$\mathbf 1\{r\ge2\}\le V(r)$ give $v_{m+1}\le\theta_1a_m+\kappa v_m$ and $a_{m+1}\le\theta_1+v_m$ for $m_0\le m<n$. Hence
$v_{m+1}\le\theta_1^2+\theta_1v_{m-1}+\kappa v_m$ for $m_0<m<n$. Lemma~\ref{lem:apriori}, applied to the uniformly random Seidel matrix
$S_m$, for which $g_\lambda(S_m)=r_m$, gives $\E(10^{r_m})\le C_0$, so that $v_{m_0},v_{m_0+1}\le C_0$. Put
$\vartheta:=\theta_1+\kappa=q^{-1}+0.1+10q^{-2}+11\varpi$. As $q\ge23$ and $\varpi\le10^{-5}$, we have $\vartheta\le0.163$. As
$\theta_1^2+\vartheta\,\theta_1^2/(1-\vartheta)=\theta_1^2/(1-\vartheta)$ and $\theta_1+\kappa\sqrt{0.163}\le\vartheta\le0.163$, induction
on $m$, starting from $m\in\{m_0,m_0+1\}$, gives
\[
v_m\le\frac{\theta_1^2}{1-\vartheta}+C_0\,0.163^{(m-m_0-1)/2}\qquad(m_0\le m\le n).
\]
As $\Pr(r_m\ge2;\ \mathcal E_{<m})\le v_m$, $\theta_1^2\le q^{-2}+0.1\varpi$ and $1/(1-\vartheta)\le1.195$, this proves the first claim.
Finally $\Pr(r_n\ge1)\le a_n+\Pr(\mathcal E_{<n}^c)\le\theta_1+v_{n-1}+\Pr(\mathcal E_{<n}^c)$, which gives the second claim.
\end{proof}

With Theorem~\ref{thm:SEplus} and Proposition~\ref{prop:Ptrue}, the lemma bounds the coranks of $\bar S-\lambda I$ for
$\lambda\in\F_p$.

\begin{theorem}\label{thm:rankp}
Let $S$ be a uniformly random Seidel matrix of order $n$, and for a prime $p$ and $\lambda\in\F_p$ let $\corank_{\F_p}(\bar S-\lambda I)$ be
the dimension of the kernel of $\bar S-\lambda I$ over $\F_p$. There is an absolute constant $c_3>0$ such that for all sufficiently large
$n$, every prime $23\le p\le e^{n^{1/9}}$ and every $\lambda\in\F_p$,
\[
\Pr\bigl(\corank_{\F_p}(\bar S-\lambda I)\ge2\bigr)\le\frac{1.2}{p^2}+e^{-c_3n},\qquad
\Pr\bigl(\corank_{\F_p}(\bar S-\lambda I)\ge1\bigr)\le\frac1p+\frac{1.2}{p^2}+e^{-c_3n}.
\]
\end{theorem}

\begin{proof}
We use the notation of Lemma~\ref{lem:corank} with $d=1$, so that $q=p$, $L_m=\bar S_m-\lambda I$ and
$r_m=\corank_{\F_p}(\bar S_m-\lambda I)$. For $m_0\le m<n$ let $\mathcal E_m$ be the intersection of the events of
Theorem~\ref{thm:SEplus} and Proposition~\ref{prop:Ptrue} for $S_m$. For large $n$ they cover $p$, as $n^{1/9}\le m^{1/8}$ for $m\ge m_0$,
and $\Pr(\mathcal E_m^c)\le O(2^{-m/8})+e^{-m/200}$. Put $\varpi:=e^{-c_2m_0/2}$, so that $\varpi\le10^{-5}$ for large $n$ and
$e^{-2c_1m}\le\varpi$. We check the hypotheses (i) and (ii) of Lemma~\ref{lem:corank} on $\mathcal E_m$, conditionally on $S_m$.

If $r_m=0$, let $B:=L_m^{-1}$. By Lemma~\ref{lem:border}, $r_{m+1}=1$ if and only if $\xi^{\top}B\xi=-\lambda$. By Fourier inversion,
$\Pr(\xi^{\top}B\xi=-\lambda)\le\frac1p+\max_{t\ne0}|\E_\xi\exp(2\pi i\,t\,\xi^{\top}B\xi/p)|$, and Proposition~\ref{prop:Ptrue}, with
$t_1=t$ and $t_2=0$, bounds the maximum by $e^{-c_2m/2}\le\varpi$, which gives (i).

If $r_m=r\ge1$, then Fourier inversion over $K_m$, that is, $\mathbf 1\{\bar\xi\perp K_m\}=p^{-r}\sum_{v\in K_m}e^{2\pi i\,v^{\top}\xi/p}$,
gives
\[
\Pr(\bar\xi\perp K_m)\le p^{-r}+\max_{0\ne v\in K_m}\prod_j\Bigl|\cos\frac{2\pi v_j}p\Bigr|\le p^{-r}+e^{-2c_1m}\le p^{-r}+\varpi .
\]
This is (ii), since Lemma~\ref{lem:dilates} (with $c=1$) bounds the product by $e^{-2\mu(v)/p^2}$, while $\mu(v)>c_1mp^2$ by
Theorem~\ref{thm:SEplus}, as every nonzero multiple of $v$ is an eigenvector of $\bar S_m$ with eigenvalue $\lambda$.

So Lemma~\ref{lem:corank} applies, and moreover $\Pr(r_n\ge2)\le\Pr(r_n\ge2;\ \mathcal E_{<n})+\Pr(\mathcal E_{<n}^c)$,
$\varpi=e^{-\Omega(n)}$, $n-m_0-2\ge n/4-3$ and $\Pr(\mathcal E_{<n}^c)\le\sum_{m_0\le m<n}(O(2^{-m/8})+e^{-m/200})=e^{-\Omega(n)}$. As all
these bounds are uniform in $p$ and $\lambda$, both claims follow for a suitable absolute constant $c_3>0$.
\end{proof}

\subsection{Repeated roots and local failures of degree one}
We now combine the corank bounds of Lemma~\ref{lem:corank} and Theorem~\ref{thm:rankp} with the pencil bound of
Proposition~\ref{prop:Ptrue} to count the repeated roots in $\F_p$.

\begin{proposition}\label{prop:isocond}
There is an absolute constant $c_4>0$ such that for all sufficiently large $n$ and every prime $23\le p\le e^{n^{1/9}}$, a uniformly
random Seidel matrix $S$ of order $n$ satisfies
\[
\sum_{\lambda\in\F_p}\Pr\bigl(\lambda\text{ is a repeated root of }\chi_S\bmod p\bigr)\ \le\ \frac{3.5}p+e^{-c_4n},
\]
that is, the expected number of repeated roots of $\chi_S$ modulo $p$ in $\F_p$ is at most $\frac{3.5}p+e^{-c_4n}$.
\end{proposition}

\begin{proof}
\step{1}{We set up the exposure and record the transitions.} Fix $\lambda\in\F_p$, and let $S_m$, $L_m$, $r_m$, $K_m$, $\xi$, $m_0$,
$\mathcal E_m$, $\mathcal E_{<m}$ and $\varpi$ be as in the proof of Theorem~\ref{thm:rankp}. When $r_m=1$ let $\iota_m:=1$ if the
kernel vectors $w$ of $L_m$ satisfy $w^{\top}w=0$, and $\iota_m:=0$ otherwise. By Lemma~\ref{lem:simple}(d), $\lambda$ is a repeated
root of $\chi_{S_m}$ modulo $p$ if and only if $r_m\ge2$, or $r_m=1$ and $\iota_m=1$. For $m_0\le m<n$, the following holds on
$\mathcal E_m$, conditionally on $S_m$.
\begin{itemize}[leftmargin=*]
\item If $r_m=0$, let $B:=L_m^{-1}$. By Lemma~\ref{lem:border} (with $L=L_m$ and $c=-\lambda$), $r_{m+1}=1$ if and only if
$\xi^{\top}B\xi=-\lambda$, and then $K_{m+1}$ is spanned by $(B\xi,-1)$, so $\iota_{m+1}=1$ if and only if moreover $\xi^{\top}B^2\xi=-1$.
By Fourier inversion in two variables and Proposition~\ref{prop:Ptrue},
\[
\Pr(r_{m+1}=1,\ \iota_{m+1}=1)\le\frac1{p^2}+\max_{(t_1,t_2)\ne(0,0)}
\Big|\E_\xi\exp\Bigl(\frac{2\pi i\,(t_1\xi^{\top}B\xi+t_2\xi^{\top}B^2\xi)}p\Bigr)\Big|\le\frac1{p^2}+\varpi .
\]
\item If $r_m=1$ and $r_{m+1}=1$, then $\bar\xi\perp K_m$, and Lemma~\ref{lem:border}(b) shows that $K_{m+1}=K_m\times\{0\}$, so that
$\iota_{m+1}=\iota_m$. Hence $\Pr(r_{m+1}=1,\ \iota_{m+1}=1)\le\iota_m\Pr(\bar\xi\perp K_m)\le\iota_m(\frac1p+\varpi)$, by (ii) in the proof
of Theorem~\ref{thm:rankp}.
\end{itemize}

\step{2}{We solve the resulting recursion and sum over $\lambda$.} Let $d_m:=\Pr(r_m=1,\ \iota_m=1;\ \mathcal E_{<m})$. As
$\mathcal E_{<m+1}\subseteq\mathcal E_{<m}\cap\mathcal E_m$ is determined by $S_m$, Step~1 gives
\[
d_{m+1}\le p^{-2}+\varpi+\Big(\frac1p+\varpi\Big)d_m+\Pr(r_m\ge2;\ \mathcal E_{<m})\qquad(m_0\le m<n),
\]
and by Lemma~\ref{lem:corank}, $\Pr(r_m\ge2;\ \mathcal E_{<m})\le1.2p^{-2}+\omega_n$ for $m\ge m_1:=m_0+\lceil n/8\rceil$, where
$\omega_n:=\varpi+C_0\,0.163^{n/16-1}$. Iterating from $m_1$, where $d_{m_1}\le1$, we get
\[
d_n\le\frac{2.2p^{-2}+\varpi+\omega_n}{1-\frac1p-\varpi}+\Big(\frac1p+\varpi\Big)^{n-m_1}\le2.3\,p^{-2}+e^{-\Omega(n)},
\]
as $2.2/(1-\frac1p)\le2.3$ for $p\ge23$, while $\varpi$, $\omega_n$ and $(\frac1p+\varpi)^{n-m_1}\le0.05^{n/8-2}$ are $e^{-\Omega(n)}$.
Hence, by Theorem~\ref{thm:rankp},
\[
\Pr\bigl(\lambda\text{ is a repeated root of }\chi_S\bmod p\bigr)\le\Pr(r_n\ge2)+d_n+\Pr(\mathcal E_{<n}^c)\le3.5\,p^{-2}+e^{-\Omega(n)}.
\]
All the error terms are uniform in $p$ and $\lambda$, and summing over the $p\le e^{n^{1/9}}$ values of $\lambda$ gives the claim.
\end{proof}

We now prove the goal of this section. Recall the quantities $a_\beta(T)$ and $N_\Pi(T)$ of Definition~\ref{def:Npi}.

\begin{proposition}\label{thm:deg1}
Fix $\varepsilon\in(0,\frac12)$. Then the probability that $E_p^{(1)}$ holds for some odd prime $p$ with
$n^{1/2-\varepsilon}<p\le e^{n^{1/10}}$ tends to $0$.
\end{proposition}

\begin{proof}
Let $\Pi_n:=\{(p,1):\ p\text{ prime},\ n^{1/2-\varepsilon}<p\le e^{n^{1/10}}\}$, and let $n$ be so large that
$n^{1/2-\varepsilon}\ge23$ and $n^{1/10}\le(n-1)^{1/9}$. Then every prime $p$ with $(p,1)\in\Pi_n$ is odd and satisfies
$23\le p\le e^{m^{1/9}}\le e^{m^{1/8}}$ for $m\in\{n-1,n\}$. For $m\in\{n-1,n\}$ let $\mathcal G_m$ be the set of Seidel matrices $T$ of
order $m$ for which the event $\mathcal G_{\mathrm{SE}}(T)$ of Theorem~\ref{thm:SEplus} holds. We verify the hypotheses of
Proposition~\ref{prop:reduction}, which then proves the proposition. Let $m\in\{n-1,n\}$. By Theorem~\ref{thm:SEplus},
$\Pr(S_m\notin\mathcal G_m)=O(2^{-m/8})\to0$. Fix a prime $p$ with $(p,1)\in\Pi_n$, and $\lambda\in\F_p$ such that $\beta=x-\lambda$ is a
repeated factor of $S_m$ at $p$. If $S_m\in\mathcal G_m$, Corollary~\ref{cor:ACplus} (for $S_m$) gives
$\Pr_\xi(\xi^{\top}w=0)\le\frac1p+e^{-2c_1m}$ for every eigenvector $w\in\F_p^m$ of $\bar S_m$ with eigenvalue $\lambda$, so
$a_\beta(S_m)\le\frac1p+e^{-2c_1m}$ by~\eqref{eq:abetaeig}. Writing $R_p(S_m)$ for the number of repeated roots $\lambda\in\F_p$ of
$\chi_{S_m}$ modulo $p$, we get, if $S_m\in\mathcal G_m$,
\[
N_{\Pi_n}(S_m)\le\sum_{n^{1/2-\varepsilon}<p\le e^{n^{1/10}}}\Big(\frac1p+e^{-2c_1m}\Big)R_p(S_m).
\]
By Proposition~\ref{prop:isocond}, applied to $S_m$ (valid as $23\le p\le e^{m^{1/9}}$), $\E(R_p(S_m))\le\frac{3.5}p+e^{-c_4m}$.
Hence, with $c':=\min(2c_1,c_4)$,
\[
\begin{aligned}
\E\big(N_{\Pi_n}(S_m);\ S_m\in\mathcal G_m\big)&\le\sum_{n^{1/2-\varepsilon}<p\le e^{n^{1/10}}}
\Big(\frac1p+e^{-2c_1m}\Big)\Big(\frac{3.5}p+e^{-c_4m}\Big)\\
&\le\sum_{k>n^{1/2-\varepsilon}}\frac{3.5}{k^2}+3e^{n^{1/10}-c'm},
\end{aligned}
\]
which tends to $0$, as $m\ge n-1$.
\end{proof}

\section{Eigenvalues in \texorpdfstring{$\F_{p^2}$}{F\_\{p\^{}2\}}}\label{sec:deg2}

This section continues the proof of the third component, Theorem~\ref{thm:moderate}, with the local failures of degree two at the moderate
primes $n^{1/2-\varepsilon}<p\le e^{n^{1/10}}$. After Seidel switching, such a failure involves an eigenvalue of the random Seidel matrix in
$\F_{p^2}\setminus\F_p$, that is, a root of an irreducible quadratic factor of its characteristic polynomial modulo $p$, and this factor
must moreover be a repeated factor (Lemma~\ref{lem:Epswitch}). Our goal is Proposition~\ref{thm:deg2}, whose proof verifies the
hypotheses of the criterion of Proposition~\ref{prop:reduction} for the pairs $(p,2)$. With Proposition~\ref{thm:deg1} of
Section~\ref{sec:deg1} and Proposition~\ref{thm:deg3plus} of Section~\ref{sec:deg3} it proves Theorem~\ref{thm:moderate}.

In degree two the isotropy condition of Lemma~\ref{lem:simple}(d) is not needed. For an irreducible quadratic factor $\beta$, the quantity
$a_\beta$ of Definition~\ref{def:Npi} is the largest probability that a uniform sign vector is orthogonal to an eigenvector $w\in\F_{p^2}^m$
at a root of $\beta$ (equation~\eqref{eq:abetaeig}). It therefore suffices to prove two statements. First, these eigenvectors are
anti-concentrated at the level $p^{-2}$ (Proposition~\ref{prop:SEd}). Second, eigenvalues in $\F_{p^2}\setminus\F_p$ are as rare as for
random matrices, so that the expected number of irreducible quadratic factors of the characteristic polynomial modulo $p$ is bounded
(Proposition~\ref{thm:rank2}). Since a repeated factor is in particular a factor, the expected contribution of the prime $p$ to $N_\Pi$
(Definition~\ref{def:Npi}) is then $O(p^{-2})$, and $\sum_{p>n^{1/2-\varepsilon}}p^{-2}\to0$.

Both statements are proved with the methods of Section~\ref{sec:deg1}, transported to $\F_p$ by the trace map. For ranks modulo primes
of random symmetric matrices with random diagonal, compare Ferber--Jain--Sah--Sawhney (2023)~\cite{FJSS}.

Throughout this section $T$ denotes a uniformly random Seidel matrix of order $m$ and $p$ an odd prime. Except in
Section~\ref{subsec:SEd}, which treats the degrees $2\le d\le4$ together, we also let $q:=p^2$, $\lambda\in\F_q\setminus\F_p$, and
\[
\beta:=(x-\lambda)(x-\lambda^p)\in\F_p[x],
\]
which is irreducible. Every monic irreducible quadratic over $\F_p$ arises in this way from each of its two roots. Recall that
$c_2=e^{-1500}$ is the constant of Proposition~\ref{prop:Ptrue}. Concentrations $\rho_G$, Fourier majorants $\hat\rho_G$, lifts and the
quantity $\mu$ are as in Section~\ref{subsec:toolsnotation}.

\subsection{The trace identity and quadratic forms of resolvents}

We state the trace identity for every degree, as Section~\ref{sec:deg3} uses it in degree three. In this paragraph only, let
$\lambda\in\F_{p^d}$ have degree $d\ge1$ over $\F_p$, with minimal polynomial $\beta$, so that $\Tr$ is the trace of $\F_{p^d}$ over
$\F_p$, and let $\bar T$ be any matrix over $\F_p$ such that $\beta(\bar T)$ is invertible over $\F_p$, or equivalently, such that
$\bar T-\lambda$ is invertible over $\F_{p^d}$, as $\det(\bar T-\lambda^{p^i})=\det(\bar T-\lambda)^{p^i}$. The Frobenius map, applied
entrywise, takes $t(\bar T-\lambda)^{-1}$ to $t^p(\bar T-\lambda^p)^{-1}$. Hence, for $t\in\F_{p^d}$ and with $\Tr$ applied entrywise,
\begin{equation}\label{eq:trres}
\Tr\bigl(t(\bar T-\lambda)^{-1}\bigr)=\sum_{i=0}^{d-1}t^{p^i}\big(\bar T-\lambda^{p^i}\big)^{-1}=P_t(\bar T)\,\beta(\bar T)^{-1},
\end{equation}
where $P_t\in\F_p[x]$ is obtained by applying $\Tr$ to the coefficients of $t\beta(x)/(x-\lambda)\in\F_{p^d}[x]$. Indeed, as $\beta$
has coefficients in $\F_p$, the Frobenius map takes the coefficients of $t\beta(x)/(x-\lambda)$ to those of $t^p\beta(x)/(x-\lambda^p)$, so
that $P_t=\sum_it^{p^i}\beta/(x-\lambda^{p^i})$, and substituting $\bar T$ for $x$ gives the second equality. So $\deg P_t<d$, and
$t\mapsto P_t$ is $\F_p$-linear. It is injective by the uniqueness of the partial fraction decomposition of
$P_t/\beta=\sum_it^{p^i}/(x-\lambda^{p^i})$, as the conjugates $\lambda^{p^i}$ are distinct. As both sets have $p^d$ elements, it is a
bijection from $\F_{p^d}$ onto the polynomials of degree less than $d$ over $\F_p$. For $d=2$, $\beta=(x-\lambda)(x-\lambda^p)$ as above
and $P_t(x)=\Tr(t)\,x-\Tr(t\lambda^p)$.

By the trace identity, Proposition~\ref{prop:Ptrue} applies to the resolvents over $\F_{p^2}$. This gives the Fourier bound for their
quadratic forms that the exposure of Lemma~\ref{lem:corank} needs.

\begin{lemma}\label{prop:res2}
For all large $m$, on the event of Proposition~\ref{prop:Ptrue}, which has probability at least $1-e^{-m/200}$, the following holds
for every prime $23\le p\le e^{m^{1/8}}$ and every $\lambda\in\F_{p^2}\setminus\F_p$ such that $\bar T-\lambda$ is invertible over
$\F_{p^2}$. With $B:=(\bar T-\lambda)^{-1}$, we have $|\E_\xi\exp(2\pi i\Tr(t\,\xi^{\top}B\xi)/p)|\le e^{-c_2m/2}$ for all
$t\in\F_{p^2}^{\times}$, and $\Pr_\xi(\xi^{\top}B\xi=a)\le p^{-2}+e^{-c_2m/2}$ for every $a\in\F_{p^2}$.
\end{lemma}

\begin{proof}
Let $t\ne0$. By~\eqref{eq:trres}, $C:=\Tr(tB)=P_t(\bar T)\beta(\bar T)^{-1}$, where $\beta$ is monic of degree $2$ with $\beta(\bar T)$
invertible and $P_t\ne0$ has degree at most $1$. Moreover $\xi^{\top}C\xi=\Tr(t\,\xi^{\top}B\xi)$ for $\xi\in\{\pm1\}^m$, as $\Tr$ is
$\F_p$-linear. So the first bound is the second claim of Proposition~\ref{prop:Ptrue}, for $Q=\beta$ and $P=P_t$. The additive characters
of $\F_q$ are $a\mapsto e^{2\pi i\Tr(ta)/p}$, $t\in\F_q$ (Lidl--Niederreiter (1997)~\cite{LN97}), so by Fourier inversion over $\F_q$,
$\Pr_\xi(\xi^{\top}B\xi=a)\le q^{-1}+\max_{t\ne0}|\E_\xi\exp(2\pi i\Tr(t\,\xi^{\top}B\xi)/p)|$, and the second bound follows.
\end{proof}

\subsection{Eigenvectors in invariant subspaces of degree at most four}\label{subsec:SEd}

The next result says that vectors in invariant subspaces of degree $2\le d\le4$ have no short integer lifts, and hence that
eigenvectors whose eigenvalue has degree $d$ over $\F_p$ are equidistributed at the level $p^{-d}$. We use it for $d=2$ below and for
$d\in\{3,4\}$ in Section~\ref{sec:deg3}. The proof is the union bound over integer lifts from the proof of Theorem~\ref{thm:SEplus},
now for vectors with entries in $\F_p^d$, with the block estimate in dimension $d$ (Proposition~\ref{lem:blockd}).

\begin{proposition}\label{prop:SEd}
Let $d\in\{2,3,4\}$ and $c_5:=e^{-300}$. For all large $m$, with probability at least $1-e^{-m/9}$ the following holds for every prime
$p$ with $23\le p\le e^{m^{1/8}}$, and $p\ge m^{0.4}$ if $d\ge3$: for every monic irreducible $\beta\in\F_p[x]$ of degree $d$ and every
nonzero $y\in\ker\beta(\bar T)$, $\mu(y)>c_5mp^2$. Consequently, on this event, for every such $p$ and every $\lambda\in\F_{p^d}$ of
degree $d$ over $\F_p$, we have $\Pr_\xi(\bar\xi\perp K)\le p^{-dr}+e^{-2c_5m}$ for every subspace $K\subseteq\ker(\bar T-\lambda)$ over
$\F_{p^d}$ of dimension $r\ge1$.
In particular $\Pr_\xi(\xi^{\top}w=0)\le p^{-d}+e^{-2c_5m}$ for every eigenvector $w\in\F_{p^d}^m$ of $\bar T$ whose eigenvalue has degree
$d$ over $\F_p$.
\end{proposition}

\begin{proof}
\step{1}{We fix the parameters and check the conditions of Proposition~\ref{lem:blockd}.} Let $G:=\F_p^d$, $N:=p^d$,
$\delta:=\frac1{100}$, and
\[
l:=\Big\lceil\frac{100d\log p}{\log m}\Big\rceil+50,\qquad s:=\lceil m^{0.6}\rceil,\qquad k:=\lceil m^{0.8}\rceil .
\]
For $d=2$ these are the parameters of the proof of Proposition~\ref{prop:Ptrue}. For large $m$ the conditions of
Proposition~\ref{lem:blockd} hold uniformly in the primes $p$ of the statement, namely $l\ge50$, $s<k$, $s\ge16l\log(2N/\delta)$,
$s^{(l+1)/2}\ge32\delta^{-1}(1+\delta)\sqrt{\pi l}\,N(1.5l)^{2l}$, $s^{l/2}\ge100p^{d-1}(1.5l)^{2l}$, $s\ge1400l$ and, if $d\ge3$,
$p\ge\sqrt{s/l}$. Indeed, $1.5l\le m^{1/8}$ for large $m$, as $\log p\le m^{1/8}$, so $2l\log(1.5l)\le\frac14l\log m$ and
$s^{l/2}(1.5l)^{-2l}\ge m^{0.05l}\ge p^{5d}$, which exceeds $100p^{d-1}$ and $32\delta^{-1}(1+\delta)\sqrt{\pi l}\,Ns^{-1/2}$ (the latter
gives the condition on $s^{(l+1)/2}$). Moreover $s$ is much larger than $16l\log(200N)$ and $1400l$, and $\sqrt{s/l}\le m^{0.3}\le p$ if
$d\ge3$. Let $R^2:=c_5mp^2$ and $\sigma:=1/(2c_5p^2)$, so that $\sigma R^2=m/2$. Let $\theta_0$, $\mu_0$ be as in Lemma~\ref{lem:gauss}
with $c=c_5$ (note $c_5\le10^{-5}$), and let $\nu^{\rm G}$ be the image on $\F_p^{\times}$ of the probability law proportional
to $e^{-\sigma z^2}$ on $\Z\setminus p\Z$, with largest atom $\nu^{\rm G}_{\max}$. So $\theta_0\le1.001$, $\mu_0\le\sqrt{2\pi c_5}\,p$
and $\mu_0\nu^{\rm G}_{\max}\le1.001$. Let $\mathcal S$ be the event of Proposition~\ref{lem:suppd}, applied with $M=T$ and
$\{a,b\}=\{\pm1\}$, of probability at least $1-e^{-m/8}$.

\step{2}{We reduce the claim to linear systems for $T$ with fixed right sides.} Fix $p$ and $\beta=x^d+\sum_{i<d}b_ix^i$ ($p^d$
choices). If $\mathcal S$ holds, $y\ne0$, $\beta(\bar T)y=0$ and $\mu(y)\le R^2$, a nonzero multiple of $y$, still in
$\ker\beta(\bar T)$, has an integer lift $z$ with $\|z\|^2\le R^2$. We call it $y$ again and put $Y:=(y,\bar Ty,\dots,\bar T^{d-1}y)$,
with rows $w_j\in G$. Then
\[
\bar TY=\Big(\bar Ty,\ \dots,\ \bar T^{d-1}y,\ -\sum_{i<d}b_i\bar T^iy\Big),
\]
a right side determined by $Y$, and $\lvert\supp y\rvert>m/(4(d+1))\ge m/20$ by Proposition~\ref{lem:suppd}.

\step{3}{We sort the coordinates into classes and set up the weights.} The classes are $Z$ ($w_j=0$), $D_2$ ($y_j\ne0$) and $D_4$
($y_j=0\ne w_j$). For each pattern $(D_2,D_4)$ with $|D_2|>m/20$ fix disjoint blocks of size $k$, $\lfloor|D_h|/k\rfloor$ of them
inside $D_h$ for $h=2,4$, depending only on the pattern. There is at least one, as $m/20\ge k$. Let $\rho^*(Y)$ be the smallest value
of $\rho_G$ on these blocks. The last statement of Lemma~\ref{lem:seqexp} (for vectors with entries in $G$), applied to a block
attaining $\rho^*(Y)$, bounds the probability over $T$ of the system of Step~2 by $\rho^*(Y)^{m-k}$, the right side being fixed.
We bound $\mathbf 1\{\|z\|^2\le R^2\}\le e^{\sigma(R^2-\|z\|^2)}$ and sum over all integer lifts $z$ of $y$. This gives a global factor
$e^{\sigma R^2}=e^{m/2}$. Let $G_0\le G$ be the hyperplane of the vectors with first coordinate $0$. The total weight of a class at one
coordinate (its \emph{mass}) and the normalized law $\nu_h$ of $w_j$ on $G\setminus\{0\}$ are as follows.
\begin{itemize}[leftmargin=*]
\item $Z$: mass $\theta_0$.
\item $D_2$: mass $p^{d-1}\mu_0$, and $\nu_2:=\nu^{\rm G}\otimes(\text{uniform law on }\F_p^{d-1})$.
\item $D_4$: mass $(p^{d-1}-1)\theta_0$, and $\nu_4$ is uniform on $G_0\setminus\{0\}$.
\end{itemize}
So $\mathrm{mass}_h\,\nu_{h,\max}\le1.001$ for $h=2,4$, where $\nu_{h,\max}$ is the largest atom of $\nu_h$. A subspace $L\le G$ of
dimension $1\le g<d$ has $\nu_2(L)=p^{g-d}$ if $L\not\subseteq G_0$ (for each $\alpha\ne0$, the vectors of $L$ with first coordinate
$\alpha$ form an affine subspace of dimension $g-1$), and $\nu_2(L)=0$ if $L\subseteq G_0$, while $\nu_4(L)\le(p^g-1)/(p^{d-1}-1)$.
Hence, writing $\kappa_g(\nu)$ for the largest $\nu$-measure of a $g$-dimensional subspace, $\mathrm{mass}_2\,\kappa_g(\nu_2)=\mu_0p^{g-1}$
and $\mathrm{mass}_4\,\kappa_g(\nu_4)\le1.001p^g$.

\step{4}{We bound the contribution of the lowest level.} The masses add up to $p^{d-1}(\theta_0+\mu_0)$. Summing over the $p^d$
polynomials $\beta$ and over the $Y$ with $\rho^*(Y)<(1+\delta)/N$, the contribution is at most
\[
p^de^{m/2}\big(p^{d-1}(\theta_0+\mu_0)\big)^m\Big(\frac{1+\delta}{p^d}\Big)^{m-k}
\le p^{d(k+1)}\Big(e^{1/2}(1+\delta)\Big(\frac{1.001}{23}+\sqrt{2\pi c_5}\Big)\Big)^m\le p^{d(k+1)}\,0.08^m .
\]

\step{5}{We bound the contribution of the structured levels.} For $a\ge0$ with $\rho_a:=2^a(1+\delta)/N\le1$, fix a pattern with
$|D_2|>m/20$, and consider the $Y$ with $\rho^*(Y)\in[\rho_a,2\rho_a)$. Every block then has $\hat\rho_G\ge\rho_G\ge\rho_a$, and under the
product of the laws $\nu_h$ the blocks are independent. By Proposition~\ref{lem:blockd}, a block of class $h$ has $\hat\rho_G\ge\rho_a$ with
probability at most $\zeta_h^{k-s}$, where $\zeta_h:=\gamma_{1,h}(\rho_a)+1.1\,\kappa_g(\nu_h)$ if $\rho_a$ lies in the window $W^{(g)}$ of
that proposition, $\zeta_h:=\gamma_{1,h}(\rho_a)$ if it lies in no window, and $\gamma_{1,h}$ is the function $\gamma_1$ of
Proposition~\ref{lem:blockd} for $\nu_h$. Let $M_h$ be $k-s$ times the number of blocks in $D_h$, and $M_{\rm bl}:=M_2+M_4$. As
$p^{-d}\le\rho_a\le1$ and $M_{\rm bl}\le m$, we have $(2\rho_a)^{m-k}\rho_a^{-M_{\rm bl}}\le2^mp^{dk}$. In each block, $k-s$ of the mass
factors are paired with one factor $\zeta_h$ of the block bound and with one factor $\rho_a$ (we call these coordinates the \emph{chain
coordinates}, and only their number matters). The at most $2k+s\lceil m/k\rceil\le4k$ remaining coordinates outside $Z$ have mass at most
$p^d$ each. So the contribution of the pattern and the level, summed over $\beta$, is at most
\[
p^de^{m/2}\,1.001^m\,p^{4dk}\cdot2^mp^{dk}\prod_{h=2,4}\big(\mathrm{mass}_h\,\zeta_h\,\rho_a\big)^{M_h}.
\]
Per chain coordinate,
$\mathrm{mass}_h\,\gamma_{1,h}(\rho_a)\rho_a=\mathrm{mass}_h\nu_{h,\max}\cdot64(1+\delta)\sqrt{\pi l}/(\delta\sqrt s)\le1.001\gamma_0$ with
$\gamma_0:=6464\sqrt{\pi l/s}$. If $\rho_a\in W^{(g)}$, then $\rho_a<1.1p^{-g}$, so that
$1.1\,\mathrm{mass}_2\kappa_g(\nu_2)\rho_a\le1.21\mu_0/p\le1.21\sqrt{2\pi c_5}$ and $1.1\,\mathrm{mass}_4\kappa_g(\nu_4)\rho_a\le1.22$.
Since $l=O(m^{1/8})$ uniformly in $p$ and $s\ge m^{0.6}$, we have $\gamma_0\to0$. Once $m$ is so large that
$\gamma_0\le\sqrt{2\pi c_5}$, every factor $\mathrm{mass}_h\zeta_h\rho_a$ is therefore at most $1.23$, and for $h=2$ at most
$1.001\gamma_0+1.21\sqrt{2\pi c_5}\le2.42\sqrt{2\pi c_5}$.
As $M_2\ge(k-s)(|D_2|/k-1)\ge m/21$, the product is at most $1.23^m(2.42\sqrt{2\pi c_5})^{m/21}$. Finally $p^{5dk+d}=e^{O(m^{0.93})}$,
there are at most $3^m$ patterns and at most $d\log_2p+1$ levels, and $2\cdot1.001\,e^{1/2}\cdot3\cdot1.23<12.2$.

\step{6}{We add the bounds of Steps~4 and~5, take a union bound over the primes, and deduce the bound for eigenvectors.} Let $\mathcal P_p$
be the probability that $\mathcal S$ holds and that some $\beta$ and $y$ as above exist for the prime $p$. By Steps~4 and~5, for large $m$,
\[
\mathcal P_p\le p^{d(k+1)}\,0.08^m+e^{o(m)}\,12.2^m\big(2.42\sqrt{2\pi c_5}\big)^{m/21}\le e^{-m/2},
\]
since $2.42\sqrt{2\pi c_5}\le e^{-148}$ and so $12.2\,(2.42\sqrt{2\pi c_5})^{1/21}\le12.2\,e^{-7}<e^{-4.4}$. The union over the at most
$e^{m^{1/8}}$ primes and $\Pr(\mathcal S^c)\le e^{-m/8}$ show that the first claim fails with probability at most
$e^{m^{1/8}}e^{-m/2}+e^{-m/8}\le e^{-m/9}$.

For the consequence, let $q:=p^d$, let $\lambda$ and $K$ be as in the statement, and let $\beta$ be the minimal polynomial of $\lambda$.
For $\xi\in\{\pm1\}^m$ and $v\in\F_q^m$ we have $\Tr(\xi^{\top}v)=\sum_j\xi_j\Tr(v_j)$, and $v\mapsto\Tr(\xi^{\top}v)$ is $\F_p$-linear.
As $K$ is a vector space over $\F_q$ and the trace form of $\F_q$ over $\F_p$ is nondegenerate, $\bar\xi\perp K$ if and only if
$\Tr(\xi^{\top}v)=0$ for all $v\in K$. So Fourier inversion over the group $K$, which has $q^r$ elements, gives
\[
\begin{aligned}
\Pr_\xi\bigl(\bar\xi\perp K\bigr)&=\frac1{q^r}\sum_{v\in K}\prod_j\cos\frac{2\pi\Tr(v_j)}p
\le\frac1{q^r}+\max_{0\ne v\in K}\prod_j\Bigl|\cos\frac{2\pi\Tr(v_j)}p\Bigr|\\
&\le\frac1{q^r}+\max_{0\ne v\in K}e^{-2\mu(\Tr(v))/p^2},
\end{aligned}
\]
the last step by Lemma~\ref{lem:dilates}. For $0\ne v\in K$, $\Tr(v)$ is a nonzero vector of $\ker\beta(\bar T)$
(Lemma~\ref{lem:simple}(a)), so $\mu(\Tr(v))>c_5mp^2$. The case of the line $K$ spanned by an eigenvector gives the last claim.
\end{proof}

\subsection{The rank over \texorpdfstring{$\F_{p^2}$}{F\_\{p\^{}2\}} and local failures of degree two}

We now show that eigenvalues in $\F_{p^2}\setminus\F_p$ are rare. As in the proof of Theorem~\ref{thm:rankp}, we expose the matrix
one row and column at a time and apply Lemma~\ref{lem:corank}, now with the bounds of Lemma~\ref{prop:res2} and
Proposition~\ref{prop:SEd}.

\begin{proposition}\label{thm:rank2}
There is an absolute constant $c_6>0$ such that for all large $n$, every prime $23\le p\le e^{n^{1/9}}$, $q=p^2$, and every
$\lambda\in\F_q\setminus\F_p$, a uniformly random Seidel matrix $S$ of order $n$ satisfies
\[
\Pr\bigl(\corank_{\F_q}(\bar S-\lambda I)\ge1\bigr)\le\frac1q+\frac{1.2}{q^2}+e^{-c_6n}.
\]
\end{proposition}

\begin{proof}
We use the notation of Lemma~\ref{lem:corank} with $d=2$, so that $L_m=\bar S_m-\lambda I$ over $\F_q$ and
$r_m=\corank_{\F_q}(\bar S_m-\lambda I)$. For $m_0\le m<n$ let $\mathcal E_m$ be the event that the conclusions of
Propositions~\ref{prop:Ptrue} and~\ref{prop:SEd} (with $d=2$) hold for $S_m$. For large $n$, $\Pr(\mathcal E_m^c)\le e^{-m/200}+e^{-m/9}$,
and these events cover $p$, as $e^{n^{1/9}}\le e^{m^{1/8}}$ for $m\ge m_0$. Put $\varpi:=e^{-c_2m_0/2}$, so that $\varpi\le10^{-5}$ for
large $n$, and $e^{-2c_5m}\le\varpi$ as $c_5>c_2$. We check the hypotheses (i) and (ii) of Lemma~\ref{lem:corank} on $\mathcal E_m$,
conditionally on $S_m$. If $r_m=0$, then by Lemma~\ref{lem:border}, with $B:=L_m^{-1}$, we have $r_{m+1}=1$ if and only if
$\xi^{\top}B\xi=-\lambda$. By Lemma~\ref{prop:res2} this has conditional probability at most $q^{-1}+e^{-c_2m/2}\le q^{-1}+\varpi$,
which is (i). If $r_m=r\ge1$, then $K_m=\ker L_m$ has dimension $r$ over $\F_q$ and $\lambda$ has degree $2$ over $\F_p$, so
Proposition~\ref{prop:SEd} (with $K=K_m$) gives $\Pr(\bar\xi\perp K_m)\le q^{-r}+e^{-2c_5m}\le q^{-r}+\varpi$, which is (ii).

So Lemma~\ref{lem:corank} applies, and moreover $\varpi=e^{-\Omega(n)}$, $n-m_0-2\ge n/4-3$ and
$\Pr(\mathcal E_{<n}^c)\le\sum_{m_0\le m<n}(e^{-m/200}+e^{-m/9})=e^{-\Omega(n)}$. As all these bounds are uniform in $p$ and $\lambda$, the
second claim of Lemma~\ref{lem:corank} gives the proposition for a suitable absolute constant $c_6>0$.
\end{proof}

We now prove the goal of this section. Recall the quantities $a_\beta$ and $N_\Pi$ of Definition~\ref{def:Npi}.

\begin{proposition}\label{thm:deg2}
Fix $\varepsilon\in(0,\frac12)$. Then the probability that $E_p^{(2)}$ holds for some odd prime $p$ with
$n^{1/2-\varepsilon}<p\le e^{n^{1/10}}$ tends to $0$.
\end{proposition}

\begin{proof}
Let $\Pi_n:=\{(p,2):\ p\text{ prime},\ n^{1/2-\varepsilon}<p\le e^{n^{1/10}}\}$, and for $m\in\{n-1,n\}$ let $\mathcal G_m$ be the event
that the conclusion of Proposition~\ref{prop:SEd} with $d=2$ holds for $S_m$, so that $\Pr(S_m\notin\mathcal G_m)\le e^{-m/9}$ for large
$n$. We verify the remaining hypothesis of Proposition~\ref{prop:reduction}, which then proves the proposition. Let $n$ be large and
$m\in\{n-1,n\}$. Then every $(p,2)\in\Pi_n$ has $23<p\le e^{n^{1/10}}\le e^{m^{1/9}}$. Let $(p,2)\in\Pi_n$, $q=p^2$, and let
$\beta\in\F_p[x]$ be monic irreducible of degree $2$, with a root $\lambda\in\F_q\setminus\F_p$. On $\mathcal G_m$,
Proposition~\ref{prop:SEd} and~\eqref{eq:abetaeig} give
\[
a_\beta(S_m)\le p^{-2}+e^{-2c_5m}.
\]
Let $f_p(S_m)$ be the number of monic irreducible quadratic factors of $\chi_{S_m}$ modulo $p$. Their roots are exactly the
$\lambda\in\F_q\setminus\F_p$ with $\corank_{\F_q}(\bar S_m-\lambda I)\ge1$, two for each factor. So by Proposition~\ref{thm:rank2}
(for the order $m$),
\[
\E(f_p(S_m))=\frac12\sum_{\lambda\in\F_q\setminus\F_p}\Pr\bigl(\corank_{\F_q}(\bar S_m-\lambda I)\ge1\bigr)
\le\frac q2\Big(\frac1q+\frac{1.2}{q^2}+e^{-c_6m}\Big)\le1+\frac q2\,e^{-c_6m}.
\]
A repeated factor is in particular a factor, so $N_{\Pi_n}(S_m)\le\sum_pf_p(S_m)\,(p^{-2}+e^{-2c_5m})$ on $\mathcal G_m$, the sum running
over the primes $p\in(n^{1/2-\varepsilon},e^{n^{1/10}}]$. Hence
\[
\begin{aligned}
\E\big(N_{\Pi_n}(S_m);\ S_m\in\mathcal G_m\big)&\le\sum_p\Big(1+\frac{p^2}2e^{-c_6m}\Big)\big(p^{-2}+e^{-2c_5m}\big)\\
&\le\sum_{p>n^{1/2-\varepsilon}}p^{-2}+\sum_{p\le e^{n^{1/10}}}\big(e^{-2c_5m}+p^2e^{-c_6m}\big).
\end{aligned}
\]
The first sum is at most $2n^{-1/2+\varepsilon}$, and the second at most $e^{n^{1/10}-2c_5(n-1)}+e^{3n^{1/10}-c_6(n-1)}$. Both sums
tend to $0$, so the remaining hypothesis of Proposition~\ref{prop:reduction} holds.
\end{proof}

\section{Eigenvalues of degree at least three}\label{sec:deg3}

This is the last of the four sections (Sections~\ref{sec:tools}--\ref{sec:deg3}) that prove the third component of the proof strategy of
Section~\ref{sec:plan}, Theorem~\ref{thm:moderate}, which excludes local failures $E_p^{(d)}$ (Definition~\ref{def:Ep}) at the moderate
primes $n^{1/2-\varepsilon}<p\le e^{n^{1/10}}$ in the degrees $d<3\log_2n$. Sections~\ref{sec:deg1} and~\ref{sec:deg2} treat the degrees $d=1$
and $d=2$. Here we treat the remaining degrees $3\le d<3\log_2n$ (larger degrees are excluded at all odd primes by
Theorem~\ref{thm:largedegrees}). The main result of the section is Proposition~\ref{thm:deg3plus}, which together with
Propositions~\ref{thm:deg1} and~\ref{thm:deg2} proves Theorem~\ref{thm:moderate} (Section~\ref{subsec:proofmoderate}).

As in Sections~\ref{sec:deg1} and~\ref{sec:deg2}, by the reduction of Section~\ref{sec:reduction} (Proposition~\ref{prop:reduction}) it
suffices to find events $\mathcal G_m$ of probability $1-o(1)$, $m\in\{n-1,n\}$, on which the expectation of $N_\Pi(S_m)$ tends to $0$. Here
$N_\Pi(S_m)$ (Definition~\ref{def:Npi}) is the sum, over the repeated irreducible factors $\beta$ of $\chi_{S_m}$ modulo the primes $p$ with
$(p,\deg\beta)$ in the set $\Pi$ of pairs under consideration, of the largest probability $a_\beta(S_m)$ that a uniform sign vector is
orthogonal modulo $p$ to a simple submodule of type $\beta$. Such a submodule is a subspace of $\F_p^m$ of dimension $d=\deg\beta$, so the
natural value of $a_\beta(S_m)$ is $p^{-d}$. How close to this value we must get depends on the number of repeated factors in the sum, and
three arguments are combined.

First, in the degrees $d\ge5$ a crude bound for $a_\beta$, from the large supports of Proposition~\ref{lem:suppd} and the Littlewood--Offord
inequality for subspaces of Lemma~\ref{lem:LObases}, suffices against the discriminant budget of Lemma~\ref{lem:budget}
(Section~\ref{subsec:deg5}). Second, in the degrees three and four, Proposition~\ref{prop:SEd} and the trivial bound $m/(2d)$ on the number
of repeated factors of degree $d$ at a prime settle degree four, and, as $\sum_{p>\sqrt n}p^{-3}=O(1/(n\log n))$, degree three at the primes
$p>\sqrt n$ (Section~\ref{subsec:deg34}). Third, degree three at the primes $n^{1/2-\varepsilon}<p\le\sqrt n$ needs the analogue of
Proposition~\ref{thm:rank2} in degree three, which Section~\ref{subsec:deg3near} proves (Proposition~\ref{thm:rank3}).
Section~\ref{subsec:proofmoderate} completes the proof of Theorem~\ref{thm:moderate}.

In this section the Seidel matrix $T$ of order $m$ is uniformly random, the random vector $u\in\{-1,0,1\}^m$ and its support $U=\supp u$
are as in Lemma~\ref{lem:squaring}, sums over $p$ run over primes, and $c_5=e^{-300}$ is the constant of Proposition~\ref{prop:SEd}.

\subsection{Degrees five and more}\label{subsec:deg5}

By Proposition~\ref{lem:suppd} of Section~\ref{sec:prelim}, applied with $M=T$ and $\{a,b\}=\{\pm1\}$, with high probability the nonzero
vectors killed by a polynomial of small degree in $\bar T$ have large support, simultaneously at all odd primes $p\le e^{m^{1/8}}$. With
the Littlewood--Offord inequality for subspaces (Lemma~\ref{lem:LObases}), this settles the degrees $d\ge5$.

\begin{proposition}\label{prop:deg5}
Fix $\varepsilon\in(0,\frac1{10})$, and put
\[
\Pi_n^{\ge5}:=\{(p,d):\ p\text{ prime},\ n^{1/2-\varepsilon}<p\le e^{n^{1/10}},\ 5\le d<3\log_2n\}.
\]
For $m\in\{n-1,n\}$ let $\mathcal G^{(5)}_m$ be the event that $\Delta_{S_m}\ne0$ and that the conclusion of
Proposition~\ref{lem:suppd} holds for $S_m$. Then $\Pr(S_m\notin\mathcal G^{(5)}_m)\to0$, and for large $n$ and $m\in\{n-1,n\}$,
\[
\E\big(N_{\Pi_n^{\ge5}}(S_m);\ S_m\in\mathcal G^{(5)}_m\big)\le96\,n^{-1/2+5\varepsilon}.
\]
\end{proposition}

\begin{proof}
By Theorem~\ref{lem:simplespectrum}, $\Pr(\Delta_{S_m}=0)\to0$, and with Proposition~\ref{lem:suppd} this gives
$\Pr(S_m\notin\mathcal G^{(5)}_m)\to0$.

\step{1}{We bound $a_\beta(S_m)$ with Lemma~\ref{lem:LObases}.} Let $n$ be large, $m\in\{n-1,n\}$, and suppose that
$S_m\in\mathcal G^{(5)}_m$. Let $n^{1/2-\varepsilon}<p\le e^{n^{1/10}}$, let $\beta\in\F_p[x]$ be monic irreducible of degree $d$ with
$5\le d<3\log_2n$, and let $V$ be a simple $\bar S_m$-submodule of type $\beta$. By Lemma~\ref{lem:simple}(a), $V$ is a $d$-dimensional
subspace of $\ker\beta(\bar S_m)$. As $p\le e^{n^{1/10}}\le e^{m^{1/8}}$ and $d<3\log_2n\le m^{1/4}$, Proposition~\ref{lem:suppd} shows that
every nonzero vector of $V$ has more than $m/(4(d+1))$ nonzero coordinates. So Lemma~\ref{lem:LObases}(c) applies to $V$ with some
$k>m/(4(d+1))$ and $t=\lceil k/d\rceil\ge(n-1)/(4d(d+1))$. As $\lfloor t/2\rfloor\ge n/(74\log_2^2n)$ for $d<3\log_2n$ and large $n$
(in particular $t\ge2$), it gives
\[
\Pr_\xi(\bar\xi\perp V)\le\Big(\frac1p+\big(\pi\lfloor t/2\rfloor\big)^{-1/2}\Big)^d\le\Big(\frac1p+\frac{6\log_2n}{\sqrt
n}\Big)^d\le\big(2n^{-1/2+\varepsilon}\big)^d.
\]
This holds for every simple submodule of type $\beta$, hence for $a_\beta(S_m)$.

\step{2}{We count the repeated factors and conclude the proof.} Since $\Delta_{S_m}\ne0$, Lemma~\ref{lem:budget}(b) with
$P=n^{1/2-\varepsilon}$ shows that there are at most $m^2\log(2m)/\log P\le3n^2$ pairs $(p,\beta)$ with $p>n^{1/2-\varepsilon}$ and $\beta$
a repeated factor of $S_m$ at $p$ (here $\varepsilon<\frac1{10}$ is used). For each of them that is counted in $N_{\Pi_n^{\ge5}}(S_m)$,
Step~1 gives $a_\beta(S_m)\le(2n^{-1/2+\varepsilon})^{\deg\beta}\le(2n^{-1/2+\varepsilon})^5$. So
$N_{\Pi_n^{\ge5}}(S_m)\le3n^2(2n^{-1/2+\varepsilon})^5=96\,n^{-1/2+5\varepsilon}$.
\end{proof}

\subsection{Degree three near \texorpdfstring{$\sqrt n$}{sqrt n}}\label{subsec:deg3near}

Near $\sqrt n$ the union bounds over short vectors need no inverse theorems, only lower bounds for supports. For the vectors
$\alpha(\bar T)\beta(\bar T)^{-1}u$ that arise from resolvents, the bound we need concerns all coordinates, and it is easy to prove.

\begin{lemma}\label{lem:fullsupp}
Fix $d\ge1$. For all large $m$, with probability at least $1-e^{-m^{1.05}}$, every odd prime $p\le m$ and every nonzero
$\alpha\in\F_p[x]$ of degree $<d$ satisfy $\dim\ker\alpha(\bar T)\le(d-1)m^{0.55}$. For such $T$, every odd prime $p\le m$, every monic
$\beta\in\F_p[x]$ of degree $d$ with $\beta(\bar T)$ invertible, every $U\subseteq[m]$ with $|U|\ge0.4m$, and $u$ with support $U$ and
uniform signs on $U$,
\[
\Pr_u\Bigl(\lvert\supp(\alpha(\bar T)\beta(\bar T)^{-1}u)\rvert<\frac m{20}\ \text{for some nonzero }\alpha\in\F_p[x]\text{ of degree}<d\Bigr)
\ \le\ 2p^{d-1}\,2^{-0.06m}.
\]
\end{lemma}

\begin{proof}
For the first statement, let $r_0:=\lceil m^{0.55}\rceil$, so that $r_0^2\ge m^{1.1}$ and $(d-1)(r_0-1)\le(d-1)m^{0.55}$. By
Corollary~\ref{cor:kernelpoly}, applied with $M=T$, $\{a,b\}=\{\pm1\}$ and $r=r_0$, and a union bound over the odd primes $p\le m$, the
first statement fails with probability at most $60\,m^d2^{-m^{1.1}/4}\le e^{-m^{1.05}}$.

For the second statement, fix a nonzero $\alpha$ of degree $<d$, which we may take monic (at most $2p^{d-1}$ choices). Put
$X:=\alpha(\bar T)\beta(\bar T)^{-1}$, and let $R\subseteq[m]$ with $|R^c|<m/20$. Then $\ker X=\beta(\bar T)\ker\alpha(\bar T)$, so the
columns of $X$ indexed by the set $U$ of the statement have rank at least $|U|-(d-1)m^{0.55}$, and
\[
\rank X_{R,U}\ge|U|-(d-1)m^{0.55}-m/20\ge0.35m-(d-1)m^{0.55}.
\]
By Odlyzko's lemma (Lemma~\ref{lem:odlyzko}), applied to $u_U$, which is uniform on $\{\pm1\}^U$, we have
$\Pr_u((Xu)_R=0)\le2^{-\rank X_{R,U}}$, as $(Xu)_R=X_{R,U}u_U$. If $\lvert\supp(Xu)\rvert<m/20$, then $(Xu)_R=0$ for $R={[m]\setminus\supp(Xu)}$, one of
at most $2^{H(1/20)m}\le2^{0.287m}$ sets. A union bound over $\alpha$ and $R$ bounds the probability in the claim by
$2p^{d-1}2^{(0.287-0.35)m+(d-1)m^{0.55}}$, and as $0.287-0.35<-0.06$, the claim follows for large $m$.
\end{proof}

The next proposition is the analogue over $\F_{p^3}$, for the primes $m^{0.4}\le p\le K\sqrt m$, of Lemma~\ref{prop:res2}. The
decoupling in the proof of Lemma~\ref{lem:pencilsupp} does not directly extend to cubic polynomials in $\bar T$, and we use instead
Lemma~\ref{lem:seqbases}, which only needs the support bound of Lemma~\ref{lem:fullsupp}. As before, $\mu$ is the quantity of
Section~\ref{subsec:toolsnotation}.

\begin{proposition}\label{prop:simple3}
Let $K\ge1$. There is $c_7=c_7(K)\in(0,c_5]$ such that for all large $m$, with probability at least $1-e^{-m/200}$, the
following holds for every prime $m^{0.4}\le p\le K\sqrt m$, with $q=p^3$, and every $\lambda\in\F_q$ of degree $3$ such that
$\bar T-\lambda$ is invertible: with $B:=(\bar T-\lambda)^{-1}$,
\[
\Pr_u\Bigl(\mu\big((\Tr(tB)u)_{U^c}\big)\le c_7mp^2\ \text{for some }t\in\F_q^{\times}\Bigr)\le e^{-m/100}.
\]
Hence $|\E_\xi\exp(2\pi i\Tr(t\,\xi^{\top}B\xi)/p)|\le e^{-c_7m/2}$ for $t\ne0$, and $\Pr_\xi(\xi^{\top}B\xi=a)\le q^{-1}+e^{-c_7m/2}$
for every $a\in\F_q$.
\end{proposition}

\begin{proof}
We first recall the trace identity, then fix the parameters of the union bound, and then prove the claim.

\step{1}{We recall the trace identity in degree three.} Let $\lambda\in\F_q$ have degree $3$, with minimal polynomial
$\beta=x^3+b_2x^2+b_1x+b_0$. By the trace identity~\eqref{eq:trres} with $d=3$, if $\bar T-\lambda$ is invertible, then
$\Tr(t(\bar T-\lambda)^{-1})=P_t(\bar T)\beta(\bar T)^{-1}$ for $t\in\F_q$, where $t\mapsto P_t$ is an $\F_p$-linear bijection from $\F_q$
onto the polynomials of degree at most $2$ over $\F_p$.

\step{2}{We fix the parameters of the union bound and specialize Lemma~\ref{lem:seqbases} to $d=3$.} Let $R^2:=cmp^2$ with $0<c\le c_5$ to
be chosen, and $\sigma:=1/(2cp^2)$, so that $\sigma R^2=m/2$. Let $\theta_0$, $\mu_0$ be as in Lemma~\ref{lem:gauss} for this $c$ (note
$c\le c_5<10^{-5}$), so that $(\theta_0+\mu_0)/p\le1.001m^{-0.4}+\sqrt{2\pi c}$. We use Lemma~\ref{lem:seqbases} with
$d=3$, and write $t_0$ for its parameter $t$. Since $p\le K\sqrt m$, for large $m$ and $m/60\le t_0\le m/60+1$ we have
$\lfloor t_0/2\rfloor\ge m/121$, so that $p(\pi\lfloor t_0/2\rfloor)^{-1/2}\le6.21K$, and $1.6\cdot9\,p\sqrt{t_0}\le1.9Km$. For such $t_0$,
Lemma~\ref{lem:seqbases} gives, for the systems $\bar TW=Y$ below,
\begin{equation}\label{eq:seqb3}
\Pr(\bar TW=Y)\le p^{-3m}(1+6.21K)^{3m}e^{1.9Km}p^{18}.
\end{equation}

\step{3}{We bound the probability of a short vector for fixed $p$, $\lambda$ and $u$, with $0.4m\le|U|\le0.6m$.} Suppose that
$\bar T-\lambda$, hence also $\beta(\bar T)$, is invertible. Put $v:=\beta(\bar T)^{-1}u$ and $W:=(v,\bar Tv,\bar T^2v)$, with rows
$\varphi_j\in\F_p^3$, so that
\[
\bar TW=\big(\bar Tv,\ \bar T^2v,\ u-b_2\bar T^2v-b_1\bar Tv-b_0v\big)
\]
is determined by $(W,u)$, and $\Tr(tB)u=P_t(\bar T)v=\ell_t(W)$, where $\ell_t$ is the linear form on $\F_p^3$ given by the coefficients
of $P_t$, applied to each row of $W$. Nonzero multiples of $\Tr(tB)u$ are again of this form. Let $\mathcal E$ be the event that
$\lvert\supp\alpha(\bar T)v\rvert\ge m/20$ for every nonzero $\alpha\in\F_p[x]$ with $\deg\alpha\le2$. As
$a^{\top}\varphi_j=((a_0+a_1\bar T+a_2\bar T^2)v)_j$, on $\mathcal E$ every hyperplane of $\F_p^3$ misses at least $m/20$ of the
$\varphi_j$, and Lemma~\ref{lem:seqbases} applies with $k=\lceil m/20\rceil$ and $t_0=\lceil k/3\rceil\in[m/60,m/60+1]$. Let
$\mathcal P(u)$ be the probability over $T$ that for some nonzero linear form $\ell$ on $\F_p^3$ and some $W$ with this property,
$\bar TW$ equals the right side above and $\ell(W)_{U^c}$ has an integer lift of squared norm at most $R^2$. As in the proof of
Proposition~\ref{prop:SEd}, we bound $\mathbf 1\{\|z\|^2\le R^2\}\le e^{\sigma(R^2-\|z\|^2)}$ and sum over the integer lifts $z$ of
$\ell(W)_{U^c}$. For each $\ell$, writing the rows of $W$ in coordinates on $\F_p^3$ whose last coordinate is $\ell$, the weighted number
of such $W$ is at most $e^{m/2}p^{3|U|}(p^2(\theta_0+\mu_0))^{|U^c|}$. With~\eqref{eq:seqb3} and $|U^c|\ge0.4m$ it follows that
\[
\begin{aligned}
\mathcal P(u)&\le p^3e^{m/2}p^{3m}\Big(\frac{\theta_0+\mu_0}p\Big)^{|U^c|}p^{-3m}(1+6.21K)^{3m}e^{1.9Km}p^{18}\\
&\le p^{21}\Big(e^{1/2}(1+6.21K)^3e^{1.9K}\big(1.001m^{-0.4}+\sqrt{2\pi c}\big)^{0.4}\Big)^m .
\end{aligned}
\]
Choose $c=c_7(K)\le c_5$ so small that $e^{1/2}(1+6.21K)^3e^{1.9K}(2\sqrt{2\pi c})^{0.4}\le e^{-2}$. For large $m$
we have $1.001m^{-0.4}\le\sqrt{2\pi c}$, and then $\mathcal P(u)\le p^{21}e^{-2m}\le(K\sqrt m)^{21}e^{-2m}\le e^{-m}$.

\step{4}{We average over $u$ and $T$, take a union bound over the pairs $(p,\lambda)$, and deduce the consequences.} Fix $p$ and $\lambda$,
and let $\Phi(T)$ be the probability over $u$ of the event in the first claim if $\bar T-\lambda$ is invertible, and $\Phi(T):=0$
otherwise. As $|U|$ is the sum of the $m$ independent indicators $\mathbf 1\{u_j\ne0\}$ of mean $\frac12$, Hoeffding's inequality
(1963)~\cite[Theorem~2]{Hoe63} gives $\Pr\bigl(\big||U|-\frac m2\big|>0.1m\bigr)\le2e^{-m/50}$. As $u$ is independent of $T$,
$\E_T\Phi(T)$ is at most this bound, plus the failure probability $e^{-m^{1.05}}$ of the first part of Lemma~\ref{lem:fullsupp}, plus
$2p^2\,2^{-0.06m}$ for the complement of $\mathcal E$ (the second part of Lemma~\ref{lem:fullsupp} with $d=3$, as $p\le K\sqrt m\le m$),
plus the maximum of $\mathcal P(u)$ over the $u$ with $0.4m\le|U|\le0.6m$. So $\E_T\Phi(T)\le3e^{-m/50}$ for large $m$. By Markov's
inequality, $\Phi(T)>e^{-m/100}$ with probability at most $3e^{-m/100}$, and a union bound over the at most $K^4m^2$ pairs $(p,\lambda)$
shows that the first claim fails with probability at most $3K^4m^2e^{-m/100}\le e^{-m/200}$. For the consequences, let $t\in\F_q^{\times}$.
The last statement of Lemma~\ref{lem:squaring}, for the symmetric $\F_p$-matrix $C:=\Tr(tB)$, for which
$\xi^{\top}C\xi=\Tr(t\,\xi^{\top}B\xi)$, with $1$ in place of $t$ and with $c=c_7$, gives
$|\E_\xi\exp(2\pi i\Tr(t\,\xi^{\top}B\xi)/p)|\le\sqrt2\,e^{-c_7m}\le e^{-c_7m/2}$ for large $m$. The last claim follows by Fourier
inversion over $\F_q$, as in the proof of Lemma~\ref{prop:res2}.
\end{proof}

As in degree two (Proposition~\ref{thm:rank2}), we expose the matrix one row and column at a time and deduce that a given element of
$\F_{p^3}$ of degree three is an eigenvalue with probability at most about $1/q$.

\begin{proposition}\label{thm:rank3}
Fix $K\ge1$ and $\varepsilon\in(0,\frac1{10})$. There is $c_8=c_8(K)>0$ such that for all large $n$, every prime
$p\in[n^{1/2-\varepsilon},K\sqrt n]$, $q=p^3$, and every $\lambda\in\F_q$ of degree $3$, a uniformly random Seidel matrix $S$ of order
$n$ satisfies
\[
\Pr\bigl(\corank_{\F_q}(\bar S-\lambda I)\ge1\bigr)\le\frac1q+\frac{1.2}{q^2}+e^{-c_8n}.
\]
\end{proposition}

\begin{proof}
We argue as in the proof of Proposition~\ref{thm:rank2}, with the following changes. Now $q=p^3$, $\lambda\in\F_q$ has degree $3$, and for
$m_0\le m<n$ the event $\mathcal E_m$ is that the conclusions of Proposition~\ref{prop:simple3}, with $2K$ in place of $K$, and of
Proposition~\ref{prop:SEd} with $d=3$ hold for $S_m$. For large $n$, $\Pr(\mathcal E_m^c)\le e^{-m/200}+e^{-m/9}$. These events cover
$p$: for $m\ge m_0$ we have $23\le m^{0.4}\le n^{1/2-\varepsilon}\le p\le K\sqrt n\le2K\sqrt m\le e^{m^{1/8}}$. In hypothesis (i) of
Lemma~\ref{lem:corank}, Proposition~\ref{prop:simple3} replaces Lemma~\ref{prop:res2}, and in hypothesis (ii), Proposition~\ref{prop:SEd}
is used with $d=3$ instead of $d=2$. Accordingly $\varpi:=e^{-c_7(2K)m_0/2}$, so that $\varpi\le10^{-5}$ for large $n$, and $\varpi$
bounds $e^{-c_7(2K)m/2}$ and $e^{-2c_5m}$ for $m\ge m_0$, as $c_7(2K)\le c_5$. The rest of the argument is unchanged, and it gives the
claim with a constant $c_8(K)>0$.
\end{proof}

\subsection{Degrees three and four}\label{subsec:deg34}

Proposition~\ref{prop:SEd} bounds $a_\beta$ for the irreducible factors $\beta$ of degrees three and four at all moderate primes. With
the trivial bound on the number of repeated factors of degree $d$ at a prime, this settles degree four, and degree three at the primes
$p>\sqrt n$, since $\sum_{p>\sqrt n}p^{-3}$ is small. Below $\sqrt n$ we use Proposition~\ref{thm:rank3}.

\begin{proposition}\label{prop:deg34}
Fix $\varepsilon\in(0,\frac1{10})$, and let
\[
\Pi_n':=\{(p,d):\ p\text{ prime},\ n^{1/2-\varepsilon}<p\le e^{n^{1/10}},\ d\in\{3,4\}\}.
\]
For $m\in\{n-1,n\}$ let $\mathcal G'_m$ be the event that the conclusion of Proposition~\ref{prop:SEd} holds for $S_m$, for $d=3$ and
for $d=4$. Then, for large $n$, $\Pr(S_m\notin\mathcal G'_m)\le2e^{-m/9}$, and $\E(N_{\Pi_n'}(S_m);\ S_m\in\mathcal G'_m)\to0$ for
$m\in\{n-1,n\}$.
\end{proposition}

\begin{proof}
The bound on $\Pr(S_m\notin\mathcal G'_m)$ is Proposition~\ref{prop:SEd}. Let $n$ be large, $m\in\{n-1,n\}$ and $(p,d)\in\Pi_n'$, so
that $23\le m^{0.4}\le n^{1/2-\varepsilon}<p\le e^{n^{1/10}}\le e^{m^{1/8}}$, and let $\beta$ be monic irreducible of degree $d$ over
$\F_p$. On $\mathcal G'_m$, Proposition~\ref{prop:SEd} and~\eqref{eq:abetaeig} give $a_\beta(S_m)\le p^{-d}+e^{-2c_5m}$.

\step{1}{We treat degree four, and degree three at the primes $p>\sqrt n$, by counting factors.} At each prime, at most $m/(2d)$ monic
irreducible $\beta$ of degree $d$ satisfy $\beta^2\mid\chi_{S_m}$ modulo $p$, as $\deg\chi_{S_m}=m$. By the elementary bound of
Chebyshev (1852)~\cite{Che52}, $\#\{p\le y\}\le2y/\log y$ for large $y$, so, splitting the primes $p>\sqrt n$ into the ranges $2^j\sqrt
n<p\le2^{j+1}\sqrt n$,
\[
\sum_{p>\sqrt n}p^{-3}\le\sum_{j\ge0}\frac{2^{j+3}\sqrt n}{\log n}\big(2^j\sqrt n\big)^{-3}=\frac{32}{3\,n\log n}.
\]
Hence, with the sums over the primes $p\le e^{n^{1/10}}$, these pairs contribute to $\E(N_{\Pi_n'}(S_m);\ S_m\in\mathcal G'_m)$ at most
\[
\sum_{p>n^{1/2-\varepsilon}}\frac n8\big(p^{-4}+e^{-2c_5m}\big)+\sum_{p>\sqrt n}\frac n6\big(p^{-3}+e^{-2c_5m}\big)
\le n^{-1/2+3\varepsilon}+\frac{16}{9\log n}+\frac n3\,e^{n^{1/10}-2c_5(n-1)},
\]
which tends to $0$.

\step{2}{We treat degree three at the primes $p\le\sqrt n$ with Proposition~\ref{thm:rank3}.} Here
$n^{1/2-\varepsilon}<p\le\sqrt n$, so that $m^{1/2-\varepsilon}<p\le2\sqrt m$. Put $q:=p^3$, and let $f_p(S_m)$ be the number of monic
irreducible cubic factors of $\chi_{S_m}$ modulo $p$. Each of them has three roots $\lambda\in\F_q$ of degree $3$, each with
$\corank_{\F_q}(\bar S_m-\lambda I)\ge1$, and there are fewer than $q$ such $\lambda$. So Proposition~\ref{thm:rank3} with $K=2$, applied
to $S_m$, gives, with $c_8=c_8(2)$,
\[
\E(f_p(S_m))\le\frac q3\Big(\frac1q+\frac{1.2}{q^2}+e^{-c_8m}\Big)\le\frac12+\frac q3\,e^{-c_8m}.
\]
As a repeated factor is in particular a factor, these primes contribute at most
\[
\sum_{n^{1/2-\varepsilon}<p\le\sqrt n}\Big(\frac12+\frac{p^3}3e^{-c_8m}\Big)\big(p^{-3}+e^{-2c_5m}\big)
\le n^{-1+2\varepsilon}+\sqrt n\big(e^{-2c_5(n-1)}+n^{3/2}e^{-c_8(n-1)}\big),
\]
which tends to $0$.
\end{proof}

The sets of pairs treated in Propositions~\ref{prop:deg5} and~\ref{prop:deg34} cover all degrees $3\le d<3\log_2n$ at the moderate
primes, and we obtain the main result of this section.

\begin{proposition}\label{thm:deg3plus}
Fix $\varepsilon\in(0,\frac1{10})$. Then the probability that $E_p^{(d)}$ holds for some odd prime $p$ with
$n^{1/2-\varepsilon}<p\le e^{n^{1/10}}$ and some $3\le d<3\log_2n$ tends to $0$.
\end{proposition}

\begin{proof}
The pairs $(p,d)$ in question lie in $\Pi_n^{\ge5}\cup\Pi_n'$, and for large $n$ all their primes are odd. By
Propositions~\ref{prop:deg5} and~\ref{prop:deg34}, the hypotheses of Proposition~\ref{prop:reduction} hold for $\Pi_n^{\ge5}$ with the
events $\mathcal G^{(5)}_m$ (the bound $96\,n^{-1/2+5\varepsilon}$ tends to $0$, as $\varepsilon<\frac1{10}$) and for $\Pi_n'$ with the
events $\mathcal G'_m$. So for each of these two sets the probability that
$E_p^{(d)}$ holds for one of its pairs tends to $0$, and a union bound proves the proposition.
\end{proof}

\subsection{Proof of Theorem~\ref{thm:moderate}}\label{subsec:proofmoderate}

\begin{proof}[Proof of Theorem~\ref{thm:moderate}]
Fix $\varepsilon\in(0,\frac1{10})$. The event of the theorem is the union of the events that $E_p^{(d)}$ holds for some odd prime $p$ with
$n^{1/2-\varepsilon}<p\le e^{n^{1/10}}$ and $d=1$, $d=2$, or some $3\le d<3\log_2n$, respectively. Their probabilities tend to $0$ by
Propositions~\ref{thm:deg1} and~\ref{thm:deg2}, which hold for every $\varepsilon\in(0,\frac12)$, and by Proposition~\ref{thm:deg3plus}.
\end{proof}

\section{Rational factors of small degree}\label{sec:nofactor}

This section proves the fourth component of the proof, Theorem~\ref{thm:nofactor}, which states that, with high probability, the
characteristic polynomial of a uniformly random Seidel matrix has no rational factor of small degree. Recall from
Section~\ref{subsec:plan-components} that a Seidel matrix of order $m$ is a symmetric matrix with zero diagonal and off-diagonal entries
$\pm1$, and that $S_m$ is a uniformly random Seidel matrix of order $m$, whose entries above the diagonal are independent uniform signs.
Precisely, Theorem~\ref{thm:nofactor} says that for every fixed $\eta\in(0,\frac1{27})$ and all sufficiently large $m$, the probability that
$\chi_{S_m}$ has a monic factor in $\Z[x]$ of degree in $[1,m^\eta]$ is at most $\exp(-m^{1/9-\eta}/9)$. It is the arithmetic input of the
proof of Theorem~\ref{thm:huge} in Section~\ref{sec:huge}, where a nonzero invariant subspace modulo a huge prime that consists of
structured vectors forces a nonzero invariant subspace over $\Q$, and hence a rational factor of small degree. Remark~\ref{rem:nofactor}
compares it with known irreducibility results for random matrices and random polynomials.

We prove Theorem~\ref{thm:nofactor} by a union bound over the candidate factors. A monic irreducible factor of degree $j$ of $\chi_{S_m}$
has all its roots among the eigenvalues of $S_m$, which have absolute value less than $m$, and this leaves at most $m^{2j^2}$ candidates
(Lemma~\ref{lem:candidates}). If a candidate $g$ has a root $a$ modulo a prime $p$, then $g\mid\chi_{S_m}$ forces the matrix
$\bar S_m-\bar aI$ to be singular, and by the rank estimate of Theorem~\ref{thm:rankp} this has probability at most about $1/p$. So it
suffices to find, for each candidate, a prime at which it has a root and which is much larger than the number of candidates, but small
enough for the rank estimate. Schur (1912)~\cite{Sch12} proved that every nonconstant polynomial in $\Z[x]$ has roots modulo infinitely many
primes. We need an effective form of his theorem, which we prove by the method of Chebyshev (1852)~\cite{Che52} in
Section~\ref{subsec:candidates} (Proposition~\ref{lem:schur}), and Section~\ref{subsec:nofactor} carries out the union bound. The only
input of this section from the rest of the paper is Theorem~\ref{thm:rankp}, which is proved in Section~\ref{sec:deg1}. The other tools are
classical facts from number theory.

\subsection{An effective form of Schur's theorem}\label{subsec:candidates}
We first prove the effective form of the theorem of Schur (1912)~\cite{Sch12} that provides the primes at which the candidate factors have
roots. We need two classical estimates for primes. In this section $\vartheta(y):=\sum_{p\le y}\log p$ denotes Chebyshev's function, and
sums and products over $p$ run over primes. Part~(a) of the next lemma is a bound of Chebyshev (1852)~\cite{Che52}, with the constant of the
proof of Erd\H{o}s (1932)~\cite{Erd32}. Part~(b) is a one-sided form of a theorem of Mertens (1874)~\cite{Mer74}.

\begin{lemma}[Chebyshev 1852~\cite{Che52}, Mertens 1874~\cite{Mer74}]\label{lem:chebyshev}
\begin{enumerate}[label=(\alph*)]
\item $\vartheta(y)<y\log4$ for every real $y\ge1$.
\item $\sum_{p\le y}\frac{\log p}p<\log y+\log4$ for every real $y\ge2$.
\end{enumerate}
\end{lemma}

\begin{proof}
(a) It suffices to show that $\prod_{p\le n}p<4^n$ for every integer $n\ge1$, and we argue by induction on $n$, following Erd\H{o}s. For
$n\le2$ the product is at most $2<4^n$. For even $n>2$ it follows from the case $n-1$, as $n$ is not prime. Let $n=2k+1\ge3$. Every prime
$p$ with $k+2\le p\le2k+1$ divides $\binom{2k+1}{k+1}=\frac{(2k+1)!}{(k+1)!\,k!}$, since it divides the numerator but not the denominator.
Moreover $2\binom{2k+1}{k+1}=\binom{2k+1}k+\binom{2k+1}{k+1}\le2^{2k+1}$, so $\binom{2k+1}{k+1}\le4^k$. With the case $k+1<n$ we get
$\prod_{p\le2k+1}p\le4^k\prod_{p\le k+1}p<4^k4^{k+1}=4^{2k+1}$. For real $y\ge1$ we then have
$\vartheta(y)=\vartheta(\lfloor y\rfloor)<\lfloor y\rfloor\log4\le y\log4$.

(b) Let $n\ge2$ be an integer. By Legendre's formula for the exponent of a prime in $n!$,
\[
\log n!=\sum_{p\le n}\log p\sum_{k\ge1}\Big\lfloor\frac n{p^k}\Big\rfloor\ \ge\ \sum_{p\le n}\Big\lfloor\frac np\Big\rfloor\log p\
>\ \sum_{p\le n}\Big(\frac np-1\Big)\log p=n\sum_{p\le n}\frac{\log p}p-\vartheta(n).
\]
Since $\log n!\le n\log n$, part~(a) gives $\sum_{p\le n}\frac{\log p}p<\log n+\frac{\vartheta(n)}n<\log n+\log4$. For real $y\ge2$ we
apply this with $n:=\lfloor y\rfloor$.
\end{proof}

The next lemma bounds the total $p$-adic valuation of the values of a polynomial along an interval. In its proof we use the extension of the
$p$-adic valuation to finite extensions of the field $\Q_p$ of $p$-adic numbers (see Neukirch (1999)~\cite[Chapter~II, \S\S4 and~8]{Neu99}).

\begin{lemma}\label{lem:valuations}
Let $g\in\Z[x]$ be monic of degree $j\ge1$, let $N\ge1$ be an integer such that $g(a)\ne0$ for $a=1,\dots,N$, and let
$X\ge\max_{1\le a\le N}|g(a)|$. Then for every prime $p$,
\[
\sum_{a=1}^Nv_p\big(g(a)\big)\ \le\ j\Big(\frac N{p-1}+\frac{\log X}{\log p}+1\Big).
\]
\end{lemma}

\begin{proof}
\step{1}{We factor $g$ over a $p$-adic field.} Let $L$ be a splitting field of $g$ over $\Q_p$. It is a finite extension of $\Q_p$, so
$v_p$ extends uniquely to a valuation $v:L\to\Q\cup\{\infty\}$ with $v(p)=1$, which satisfies $v(yy')=v(y)+v(y')$ and
$v(y+y')\ge\min(v(y),v(y'))$. Write $g=\prod_{i=1}^j(x-\gamma_i)$ with $\gamma_i\in L$. Every root satisfies $v(\gamma_i)\ge0$. Indeed, let
$\gamma$ be a root with $v(\gamma)<0$, and write $g=x^j+\sum_{k<j}b_kx^k$ with $b_k\in\Z$. In the identity $\gamma^j=-\sum_{k<j}b_k\gamma^k$
the left side has valuation $jv(\gamma)$, while every term on the right side has valuation $v(b_k)+kv(\gamma)\ge kv(\gamma)>jv(\gamma)$,
which is impossible. Hence $v(a-\gamma_i)\ge\min(v(a),v(\gamma_i))\ge0$ for every integer $a$ and every $i$, and
\[
v_p\big(g(a)\big)=\sum_{i=1}^jv(a-\gamma_i).
\]

\step{2}{We bound the contribution of one root.} Fix $i$ and put $\tau:=\log X/\log p$. For $1\le a\le N$ all the terms $v(a-\gamma_{i'})$
are nonnegative, so $v(a-\gamma_i)\le v_p(g(a))\le\tau$, because $p^{v_p(g(a))}$ divides the nonzero integer $g(a)$, whose absolute value is
at most $X$. For real $t>0$ let $C(t)$ be the number of $a\in\{1,\dots,N\}$ with $v(a-\gamma_i)\ge t$, so that $C(t)=0$ for $t>\tau$. If $a$
and $a'$ are two of them, then $v_p(a-a')=v\big((a-\gamma_i)-(a'-\gamma_i)\big)\ge t$, and since $v_p(a-a')$ is an integer or $\infty$,
$p^{\lceil t\rceil}$ divides $a-a'$. So these $a$ lie in one residue class modulo $p^{\lceil t\rceil}$, and
$C(t)\le N/p^{\lceil t\rceil}+1$. Hence, as $\lceil t\rceil=k$ for $k-1<t\le k$,
\[
\sum_{a=1}^Nv(a-\gamma_i)=\int_0^\tau C(t)\,dt\ \le\ \sum_{k=1}^{\lceil\tau\rceil}\Big(\frac N{p^k}+1\Big)\ <\ \frac N{p-1}+\tau+1.
\]

\step{3}{We sum the bounds of Step~2 over the roots and conclude the proof.} By Step~1,
$\sum_{a=1}^Nv_p(g(a))=\sum_{i=1}^j\sum_{a=1}^Nv(a-\gamma_i)$, and by Step~2 each of the $j$ inner sums is less than
$\frac N{p-1}+\tau+1$ with $\tau=\log X/\log p$, which gives the claim.
\end{proof}

With these estimates, Chebyshev's method gives the following effective form of the theorem of Schur~\cite{Sch12}. In its proof, the values
$g(1),\dots,g(N)$ with $N:=\lceil A^3\rceil$ are too large, in total, to be composed of primes below $A$, because by
Lemma~\ref{lem:valuations} the sum of their $p$-adic valuations is at most about $jN/(p-1)$ for each prime $p$. The method goes back to
Chebyshev and was applied to the values of general polynomials by Nagell (1921)~\cite{Nag21}. We need a version that is uniform in $g$,
whose bounds depend on $g$ only through its degree and the size of its roots. Irreducibility is used only to ensure that $g$ has no integer
roots.

\begin{proposition}\label{lem:schur}
There is an absolute constant $A_0\ge2$ with the following property. Let $g\in\Z[x]$ be monic and irreducible over $\Q$ of degree $j\ge2$,
suppose that all complex roots of $g$ have absolute value at most $R$, where $R\ge1$, and let $A\ge\max(A_0,R+1,j)$. Then there are an
integer $a$ with $1\le a\le\lceil A^3\rceil$ and a prime $p$ with
\[
p\mid g(a)\qquad\text{and}\qquad A\le p\le A^{4j} .
\]
\end{proposition}

\begin{proof}
Put $N:=\lceil A^3\rceil$, $X:=A^{4j}$ and $P:=\prod_{a=1}^Ng(a)$. Since $g$ is irreducible of degree at least $2$, it has no rational
root, so $g(a)\ne0$ for every integer $a$. Let $\alpha_1,\dots,\alpha_j$ be the complex roots of $g$. For $1\le a\le N$ we have
$|g(a)|=\prod_i|a-\alpha_i|\le(a+R)^j\le(A^3+1+A)^j\le X$, as $N\le A^3+1$, $R\le A$ and $A^3+A+1\le A^4$. So if a prime $p\ge A$ divides
$P$, then it divides some $g(a)$ with $1\le a\le N$, and $p\le|g(a)|\le X$, as claimed. Suppose, for a contradiction, that every prime
factor of $P$ is less than $A$.

\step{1}{We bound $\log|P|$ from above, using that all prime factors of $P$ are less than $A$.} By Lemma~\ref{lem:valuations}, applied to
each prime $p<A$, and by $\frac1{p-1}\le\frac2p$,
\[
\begin{aligned}
\log|P|&=\sum_{p<A}\log p\sum_{a=1}^Nv_p\big(g(a)\big)\le j\sum_{p<A}\Big(\frac{N\log p}{p-1}+\log X+\log p\Big)\\
&\le j\Big(2N\sum_{p\le A}\frac{\log p}p+A\log X+\vartheta(A)\Big),
\end{aligned}
\]
since there are at most $A$ primes below $A$. As $A\ge2$, Lemma~\ref{lem:chebyshev} gives
\begin{equation}\label{eq:Pupper}
\log|P|<j\big(2N(\log A+\log4)+A\log X+A\log4\big).
\end{equation}

\step{2}{We bound $\log|P|$ from below, using that the roots of $g$ are small.} Put $R':=\lceil R\rceil$, so that $R'\le R+1\le A<N$. For
$a\ge R'+1$ we have $|a-\alpha_i|\ge a-R'\ge1$ for all $i$, so $|g(a)|\ge(a-R')^j$. For $1\le a\le R'$ we have $|g(a)|\ge1$, as $g(a)$ is a
nonzero integer. Hence, with $M:=N-R'$,
\[
\log|P|\ \ge\ j\sum_{a=R'+1}^N\log(a-R')=j\log(M!)\ \ge\ j\,(M\log M-M),
\]
by $\log M!=\sum_{t=2}^M\log t\ge\int_1^M\log t\,dt=M\log M-M+1$.

\step{3}{We compare the two bounds and obtain a contradiction for large $A$.} As $M\ge A^3-A\ge1$ and $t\mapsto t\log t-t$ is increasing on
$[1,\infty)$, Step~2 gives $\log|P|\ge j(A^3-A)\bigl(\log(A^3-A)-1\bigr)$. On the other hand $N\le A^3+1$ and
$\log X=4j\log A\le4A\log A$, as $j\le A$, so by~\eqref{eq:Pupper}
\[
\log|P|\ <\ j\big(2(A^3+1)(\log A+\log4)+4A^2\log A+A\log4\big).
\]
After division by $j$, both bounds depend only on $A$. As $A\to\infty$, the lower bound is $(3+o(1))A^3\log A$, while the upper bound is
$(2+o(1))A^3\log A$. So the two bounds contradict each other for all $A\ge A_0$, where $A_0$ is a suitable absolute constant. This
contradiction shows that some prime $p\ge A$ divides $P$, which proves the proposition.
\end{proof}

We compare Proposition~\ref{lem:schur} with known bounds. For general irreducible polynomials, explicit bounds for the least prime that does
not divide a given integer $b$ and modulo which the polynomial has a root were obtained by Dubickas--Sha--Shparlinski (2015)~\cite{DSS15}.
For the polynomials of Proposition~\ref{lem:schur} and $b$ the product of the primes below $A$, their bound is of the form $A^{O(j^2)}$, while
Proposition~\ref{lem:schur} gives a prime $p$ with $A\le p\le A^{4j}$.

\subsection{Proof of Theorem~\ref{thm:nofactor}}\label{subsec:nofactor}
We now prove Theorem~\ref{thm:nofactor}, starting with a count of the candidate factors. Since the eigenvalues of a Seidel matrix of order
$m$ have absolute value at most $m-1$, all roots of a factor of its characteristic polynomial lie in the disc of radius $m$.

\begin{lemma}\label{lem:candidates}
Let $j\ge1$ and $R\ge6$. The number of monic polynomials $g\in\Z[x]$ of degree $j$ all of whose complex roots have absolute value at most
$R$ is at most $R^{2j^2}$.
\end{lemma}

\begin{proof}
Write $g(x)=\prod_{i=1}^j(x-\alpha_i)=x^j+\sum_{k=1}^j(-1)^k\sigma_k\,x^{j-k}$, where $\sigma_k$ is the $k$th elementary symmetric
polynomial in the roots $\alpha_1,\dots,\alpha_j$. As $\sigma_k$ is a sum of $\binom jk\le2^j$ products of $k$ roots, we have
$|\sigma_k|\le2^jR^k\le2^jR^j$. So each of the $j$ coefficients of $x^{j-1},\dots,x^0$ is an integer of absolute value at most $2^jR^j$,
and there are at most $2^{j+1}R^j+1\le3\cdot2^jR^j\le6^jR^j\le R^{2j}$ choices for it, as $R\ge6$. Multiplying over the $j$ coefficients
gives at most $R^{2j^2}$ polynomials.
\end{proof}

The only probabilistic input of the proof is the second bound of Theorem~\ref{thm:rankp}, applied with $n=m$. It concerns the event that
$\bar S_m-\lambda I$ has a nonzero kernel over $\F_p$, that is, $\det(\bar S_m-\lambda I)=0$ in $\F_p$. For large $m$ and primes
$p\le e^{m^{1/9}}$, the error term of this bound satisfies $e^{-c_3m}\le e^{-m^{1/9}}\le\frac1p$, and moreover
$\frac1p+\frac{1.2}{p^2}\le\frac2p$. So for all sufficiently large $m$, all primes $23\le p\le e^{m^{1/9}}$ and all $\lambda\in\F_p$,
\begin{equation}\label{eq:rankfixed}
\Pr\bigl(\det(\bar S_m-\lambda I)=0\text{ in }\F_p\bigr)\ \le\ \frac3p.
\end{equation}
The essential feature of~\eqref{eq:rankfixed} is that it holds for each fixed $\lambda\in\F_p$, and not only on average over $\lambda$.
This is what allows us to treat every candidate factor at a prime adapted to it.

\begin{proof}[Proof of Theorem~\ref{thm:nofactor}]
Fix $\eta\in(0,\frac1{27})$, let $m$ be large, and put $K:=\lfloor m^\eta\rfloor\ge1$. Let $\mathcal F_m$ be the event that $\chi_{S_m}$
has a monic factor in $\Z[x]$ of degree in $[1,m^\eta]$.

\step{1}{We reduce $\mathcal F_m$ to irreducible candidates.} For $j\ge1$ let $\mathcal C_j$ be the set of monic polynomials $g\in\Z[x]$
of degree $j$ that are irreducible over $\Q$ and all of whose complex roots have absolute value at most $m$. Suppose that $\mathcal F_m$
holds, and let $f\in\Z[x]$ be a monic factor of $\chi_{S_m}$ of degree $k\in[1,m^\eta]$, so that $1\le k\le K$. By Gauss's lemma, the
monic irreducible factors of $f$ in $\Q[x]$ have integer coefficients. Let $g$ be one of them, and let $j:=\deg g$, so that $1\le j\le K$.
Then $g$ divides $\chi_{S_m}$, and since $g$ is monic, the quotient lies in $\Z[x]$. Every root of $g$ is an eigenvalue of the real
symmetric matrix $S_m$, so its absolute value is at most $m-1$, the largest absolute row sum of $S_m$. Hence $g\in\mathcal C_j$, and
\begin{equation}\label{eq:union}
\mathcal F_m\ \subseteq\ \bigcup_{j=1}^K\ \bigcup_{g\in\mathcal C_j}\big\{g\text{ divides }\chi_{S_m}\big\}.
\end{equation}
By Lemma~\ref{lem:candidates} with $R=m$, we have $|\mathcal C_j|\le m^{2j^2}$ for all $j\ge1$.

\step{2}{We attach to every candidate a large prime at which it has a root.} For $1\le j\le K$ put $A_j:=\exp\bigl(m^{1/9}/(4j)\bigr)$, so
that $A_j^{4j}=e^{m^{1/9}}$ and $A_j\ge\exp\bigl(m^{1/9-\eta}/4\bigr)$. As $\eta<\frac19$, the exponent $\frac19-\eta$ is positive, and
$A_j$ is larger than every fixed power of $m$ for large $m$. Fix a prime $p_*$ with $A_1<p_*\le2A_1$, which exists by Bertrand's postulate
for the integer $\lfloor A_1\rfloor$, proved by Chebyshev (1852)~\cite{Che52} (see also Hardy--Wright (2008)~\cite[Chapter~XXII]{HW08}). For
$g=x-c\in\mathcal C_1$ put $p_g:=p_*$ and $a_g:=c$. For $2\le j\le K$ and $g\in\mathcal C_j$ we apply Proposition~\ref{lem:schur} with
$R=m$ and $A=A_j$, which is allowed for large $m$, as $A_j\ge\max(A_0,m+1,K)$ and $j\le K$. It provides an integer $a_g$ and a prime
$p_g$ with $p_g\mid g(a_g)$ and $A_j\le p_g\le A_j^{4j}$. As $2A_1\le A_1^4$, in both cases
\begin{equation}\label{eq:pgsize}
p_g\mid g(a_g)\qquad\text{and}\qquad A_j\ \le\ p_g\ \le\ A_j^{4j}=e^{m^{1/9}} .
\end{equation}
The pair $(p_g,a_g)$ depends only on $g$, and not on $S_m$.

\step{3}{We bound the probability that a fixed candidate divides $\chi_{S_m}$.} Let $1\le j\le K$ and $g\in\mathcal C_j$. If $g$ divides
$\chi_{S_m}$, then $g(a_g)$ divides $\chi_{S_m}(a_g)=\det(a_gI-S_m)$, as the quotient lies in $\Z[x]$, so $p_g$ divides $\det(a_gI-S_m)$
and the matrix $\bar S_m-\lambda I$ with $\lambda:=\bar a_g\in\F_{p_g}$ is singular over $\F_{p_g}$. As $A_j\ge23$ for large $m$, we have
$23\le p_g\le e^{m^{1/9}}$ by~\eqref{eq:pgsize}, so the bound~\eqref{eq:rankfixed} applies to the prime $p_g$ and to $\lambda=\bar a_g$, and
gives
\begin{equation}\label{eq:onecand}
\Pr\bigl(g\text{ divides }\chi_{S_m}\bigr)\ \le\ \frac3{p_g}\ \le\ 3e^{-m^{1/9}/(4j)}.
\end{equation}

\step{4}{We sum the bounds of Step~3 over all candidates and conclude the proof.} Put $y:=m^{1/9-\eta}$. For $1\le j\le K$ we have
$|\mathcal C_j|\le m^{2j^2}\le\exp(2m^{2\eta}\log m)$ and $e^{-m^{1/9}/(4j)}\le e^{-y/4}$. As $\eta<\frac1{27}$, we have
$2\eta<\frac19-\eta$, so that $2m^{2\eta}\log m\le y/8$ for large $m$. Hence, by~\eqref{eq:union} and~\eqref{eq:onecand},
\[
\Pr(\mathcal F_m)\ \le\ \sum_{j=1}^K3\,|\mathcal C_j|\,e^{-m^{1/9}/(4j)}\ \le\ 3Ke^{y/8-y/4}\ \le\ e^{-y/9}
\]
for large $m$, as $K\le m^\eta$. As $y=m^{1/9-\eta}$, this is the bound of Theorem~\ref{thm:nofactor}.
\end{proof}

We end this section with a remark on the limits of the method and on related results.

\begin{remark}\label{rem:nofactor}
For Theorem~\ref{thm:huge} it suffices to exclude the factors of degree at most $(\log m)^3$ (see Proposition~\ref{lem:acinv} in
Section~\ref{sec:huge}), so any fixed exponent $\eta>0$ would do in Theorem~\ref{thm:nofactor}. The exponent $\frac1{27}$ is where the
number of candidates catches up with the size of the primes. For a candidate of degree $j$, Proposition~\ref{lem:schur} guarantees a prime
$p$ with $A\le p\le A^{4j}$ at which it has a root, which lies in the range $p\le e^{m^{1/9}}$ of Theorem~\ref{thm:rankp} when
$\log A\le m^{1/9}/(4j)$. The union bound needs $A$ to be much larger than the bound $m^{2j^2}$ for the number of candidates, and both
conditions can be met only if $j^3\log m$ is small compared with $m^{1/9}$. Better constants in Lemma~\ref{lem:candidates} and
Proposition~\ref{lem:schur} do not change the exponent. Primes of size $A^{O(1)}$ at which the candidates have roots would improve it to
$\frac1{18}$, and a larger range of primes in Theorem~\ref{thm:rankp} would improve it as well. However, the error term $e^{-c_3m}$ of
Theorem~\ref{thm:rankp}, multiplied by the number $e^{O(j^2\log m)}$ of candidates, limits the union bound to degrees
$j=O((m/\log m)^{1/2})$, whatever the range of primes. So the argument does not give the irreducibility of $\chi_{S_m}$.

The argument uses the random matrix only through a bound for its eigenvalues and through the bound~\eqref{eq:rankfixed} for each fixed
$\lambda$. Let $M_n$ be a random symmetric matrix of order $n$ with independent uniform $\pm1$ entries on and above the diagonal. Its
eigenvalues have absolute value at most $n$. The rank theorem of Ferber--Jain--Sah--Sawhney (2023)~\cite[Theorem~1.2]{FJSS}, applied with
corank $0$, gives $\Pr(\det(\bar M_n-\lambda I)=0\text{ in }\F_p)\le\frac2p+O(e^{-cn/\log p})$ for every $\lambda\in\F_p$ and every prime
$2<p\le e^{cn^{1/4}}$, where $c>0$ is a constant. In this range the error term is $O(e^{-n^{3/4}})$, so that the analogue
of~\eqref{eq:rankfixed} holds for large $n$. The same argument, with $A_j:=\exp(cn^{1/4}/(4j))$, then shows that the characteristic
polynomial of $M_n$ has no factor over $\Q$ whose degree lies between $1$ and $n^\eta$, for every fixed $\eta\in(0,\frac1{12})$, with
probability tending to one.

Irreducibility with probability tending to one is known in several related settings. For matrices with independent entries, Eberhard
(2022)~\cite{Eb22} proved it unconditionally when the entries are uniform modulo a product of four distinct primes, and under the extended
Riemann hypothesis for every fixed nonconstant law of the entries with finite support. Ferber--Jain--Sah--Sawhney (2023)~\cite{FJSS} proved
it under the extended Riemann hypothesis for symmetric matrices with independent uniform $\pm1$ entries on and above the diagonal. For
random polynomials with independent coefficients, Breuillard--Varj\'u (2019)~\cite{BV19} proved it under the Riemann hypothesis for Dedekind
zeta functions. Unconditionally, Bary-Soroker--Kozma (2020)~\cite{BSK20} proved it for coefficients uniform on $\{1,\dots,L\}$ whenever $L$
is divisible by at least four distinct primes, and Bary-Soroker--Koukoulopoulos--Kozma (2023)~\cite{BKK23} for coefficients uniform on at
least $35$ consecutive integers. The latter authors also showed that a random monic polynomial whose other coefficients are independent with
a fixed finitely supported law that is not a point mass, conditioned on a nonzero constant term, has with probability tending to one no
factor of degree at most a small constant times its degree. This is an analogue of Theorem~\ref{thm:nofactor} with a much larger range.
Earlier, Konyagin (1999)~\cite{Kon99} showed that a proportion $\gg1/\log n$ of the $0$--$1$ polynomials of degree $n$ are irreducible.
\end{remark}

\section{Huge primes}\label{sec:huge}

This section proves the fifth and last component of the proof, Theorem~\ref{thm:huge}, which excludes local failures $E_p^{(d)}$
(Definition~\ref{def:Ep}) at the primes $p>e^{(\log n)^6}$, in all degrees. With it all five components are proved, and with them
Theorem~\ref{thm:main}, which was deduced from them in Section~\ref{subsec:plan-components}.

At the huge primes $p>e^{n^{1/10}}$ the union bounds over vectors modulo $p$ that we use at the moderate primes are not available. Instead
we pass to the integers, as Nguyen--Wood (2025)~\cite{NW25} do for cokernels of random matrices at large primes. Modulo a huge prime, a
nonzero invariant subspace of small dimension all of whose nonzero vectors have concentrated signed sums forces a rational factor of small
degree of the characteristic polynomial (Section~\ref{subsec:lifting}), and Theorem~\ref{thm:nofactor} excludes such factors.
Section~\ref{subsec:proofhuge} then proves Theorem~\ref{thm:huge} with the criterion of Proposition~\ref{prop:reduction} and the
discriminant budget of Lemma~\ref{lem:budget}.

\subsection{Lifting invariant subspaces}\label{subsec:lifting}
We use the concentration $\rho_p(w)$ of a vector $w\in\F_p^m$ defined in Section~\ref{subsec:plan-switch}. Recall from
Definition~\ref{def:Npi} that, for a symmetric integer matrix $T$, an odd prime $p$ and a monic irreducible $\beta\in\F_p[x]$, the number
$a_\beta(T)$ is the largest probability $\Pr_\xi(\bar\xi\perp U)$ over the simple $\bar T$-submodules $U$ of type $\beta$, or $0$ if there
is none, and that $N_\Pi(T)$ is the sum of these numbers over the repeated factors $\beta$ of $T$ at the primes $p$ with
$(p,\deg\beta)\in\Pi$.

Inverse Littlewood--Offord theorems describe the coefficient vectors whose signed sums are concentrated. For integer coefficients, the
optimal inverse theorem of Nguyen--Vu (2011)~\cite{NV11} shows that, if the concentration is at least $m^{-C}$ for a constant $C$, then all
but at most $\sqrt m$ coefficients lie in a generalized arithmetic progression of bounded rank and polynomial size. At the huge primes a
much weaker structure suffices, and it follows from an elementary argument with dissociated sets, which are classical in additive
combinatorics (see Tao--Vu (2006)~\cite{TVbook}). A vector with concentration at least $2^{-s}$ is a linear combination of the reductions of
at most $s$ vectors with entries in $\{-1,0,1\}$.

\begin{lemma}\label{lem:dissoc}
Let $p$ be an odd prime and $w\in\F_p^m$. Then there are a set $B\subseteq[m]$ with $2^{|B|}\le1/\rho_p(w)$ and vectors
$y_b\in\{-1,0,1\}^m$, $b\in B$, such that $w=\sum_{b\in B}w_b\,\bar y_b$.
\end{lemma}

\begin{proof}
Call a set $B\subseteq[m]$ \emph{dissociated} if $\sum_{b\in B}c_bw_b\ne0$ for every nonzero $c\in\{-1,0,1\}^B$. The empty set is
dissociated, and we fix a dissociated set $B$ that is maximal for inclusion. If $\sigma\ne\sigma'$ are vectors in $\{\pm1\}^B$, then
$c:=\frac12(\sigma-\sigma')$ is a nonzero vector in $\{-1,0,1\}^B$, and $\sum_b\sigma_bw_b-\sum_b\sigma'_bw_b=2\sum_bc_bw_b\ne0$, as $p$ is
odd. So the $2^{|B|}$ sums $\sum_{b\in B}\sigma_bw_b$ with $\sigma\in\{\pm1\}^B$ are pairwise distinct. For $a\in\F_p$,
conditionally on the signs $\xi_j$ with $j\notin B$, the event $\sum_{j=1}^m\xi_jw_j=a$ prescribes the value of $\sum_{b\in B}\xi_bw_b$, so
it has conditional probability at most $2^{-|B|}$. Hence $\rho_p(w)\le2^{-|B|}$.

Now let $j\in[m]\setminus B$. As $B\cup\{j\}$ is not dissociated, there is a nonzero $c\in\{-1,0,1\}^{B\cup\{j\}}$ with
$c_jw_j+\sum_{b\in B}c_bw_b=0$, and $c_j\ne0$, since $B$ is dissociated. Hence $w_j=\sum_{b\in B}t_{bj}w_b$ with
$t_{bj}:=-c_jc_b\in\{-1,0,1\}$. For $j\in B$ put $t_{bj}:=1$ if $b=j$, and $t_{bj}:=0$ otherwise. Then the vectors $y_b\in\{-1,0,1\}^m$
with the coordinates $t_{bj}$, $j\in[m]$, satisfy $w=\sum_{b\in B}w_b\,\bar y_b$.
\end{proof}

The following deterministic lemma shows that, if $p$ is large, a nonzero invariant subspace modulo $p$ that lies in the span of the
reductions of a few vectors with entries in $\{-1,0,1\}$, such as those provided by Lemma~\ref{lem:dissoc}, forces a nonzero invariant
subspace over $\Q$ in the span of these vectors.

\begin{lemma}\label{lem:liftinv}
Let $T\in\Z^{m\times m}$ be symmetric with entries in $\{-1,0,1\}$, and let $u_1,\dots,u_D\in\{-1,0,1\}^m$ with $D<m$. Let
$U:=\operatorname{span}_\Q(u_1,\dots,u_D)$, and suppose that $0$ is the only $T$-invariant subspace of $\Q^m$ contained in $U$, which
holds, for instance, if $\chi_T$ has no monic factor in $\Z[x]$ of degree in $[1,D]$. Then for every prime $p>m^{3D^2}$, the only
$\bar T$-invariant subspace of $\F_p^m$ contained in $\operatorname{span}_{\F_p}(\bar u_1,\dots,\bar u_D)$ is $0$.
\end{lemma}

\begin{proof}
If $V\ne0$ is a $T$-invariant subspace of $\Q^m$ contained in $U$, the characteristic polynomial of $T|_V$ is a monic factor of
$\chi_T$ in $\Q[x]$ of degree $\dim V\in[1,D]$, and it lies in $\Z[x]$ by Gauss's lemma. This proves the sufficient condition in the
statement. Now let $p>m^{3D^2}$.

\step{1}{We construct an integral matrix whose kernel is $U$ over $\Q$ and $\operatorname{span}_{\F_p}(\bar u_1,\dots,\bar u_D)$
over $\F_p$.} Let $d:=\dim U\le D$. If $d=0$ there is nothing to prove. After renumbering, $u_1,\dots,u_d$ is a basis of $U$. Let
$M:=[u_1,\dots,u_d]\in\Z^{m\times d}$. In this proof, for a set $I\subseteq[m]$, we write $y_I$ for the restriction of a vector $y$ to the
coordinates in $I$, and $Y_I$ for the submatrix formed by the rows in $I$ of a matrix $Y$ with $m$ rows. Choose a set $I$ of $d$ rows with
$\kappa:=\det M_I\ne0$. By Hadamard's inequality $|\kappa|\le d^{d/2}\le m^{3D^2}<p$, so $p\nmid\kappa$. Let
$\pi_U\in\Z^{m\times m}$ be the matrix of $y\mapsto\kappa y-M\operatorname{adj}(M_I)\,y_I$. Then
$\pi_UM=\kappa M-M\operatorname{adj}(M_I)M_I=0$, and $\pi_Uy=0$ implies $y=M\big(\kappa^{-1}\operatorname{adj}(M_I)y_I\big)$, over $\Q$
and, as $p\nmid\kappa$, also over $\F_p$. Hence $\ker\pi_U=U$ and $\ker\bar\pi_U=\operatorname{im}\bar M$. As $\pi_Uu_i=0$ for
$1\le i\le D$, we get $\bar u_i\in\ker\bar\pi_U$, and hence $\operatorname{span}_{\F_p}(\bar u_1,\dots,\bar u_D)=\operatorname{im}\bar M$.

\step{2}{We show that a Krylov matrix built from $\pi_U$, $T$ and $M$ has rank $d$ over $\Q$.} Let $\Lambda\in\Z^{md\times d}$ be the
matrix of the map $a\mapsto(\pi_UT^jMa)_{1\le j\le d}$, that is, the blocks $\pi_UT^jM$, $1\le j\le d$, stacked vertically. Let
$a\in\Q^d$ with $\Lambda a=0$ and put $y:=Ma\in U$. Then $T^jy\in\ker\pi_U=U$ for $1\le j\le d$. Suppose $y\ne0$, and for
$0\le j\le d$ let $K_j:=\operatorname{span}(y,Ty,\dots,T^jy)$, the Krylov subspaces of $y$ (after Krylov (1931)~\cite{Kry31}), which
are contained in $U$. As $\dim K_0=1$ and $\dim K_d\le d$, we have $K_j=K_{j+1}$ for some $j<d$. Then $TK_j\subseteq K_{j+1}=K_j$, so
$K_j$ is a nonzero $T$-invariant subspace contained in $U$, a contradiction. Hence $Ma=0$ and $a=0$. So some $d\times d$ minor
$\Lambda_0$ of $\Lambda$ is nonzero.

\step{3}{We bound the nonzero minor $\Lambda_0$ by Hadamard's inequality and deduce that $p$ does not divide it.} For $j\ge1$ the entries
of $T^j$ are at most $m^{j-1}$ in absolute value, so those of $T^jM$ are at most $m^j$. The entries of $\operatorname{adj}(M_I)$ are at
most $(d-1)!$, and $|\kappa|\le d!$. Hence, for $1\le j\le d$, the entries of $\pi_UT^jM=\kappa\,T^jM-M\operatorname{adj}(M_I)(T^jM)_I$
are at most
\[
d!\,m^j+d^2\,(d-1)!\,m^j=(d+1)!\,m^j\le m^{2d},
\]
as $(d+1)!\le(d+1)^d\le m^d$ by $d<m$. By Hadamard's inequality, $|\Lambda_0|\le(\sqrt d\,m^{2d})^d\le m^{3d^2}\le m^{3D^2}<p$,
as $\sqrt d<m$. So $p\nmid\Lambda_0$, and $\bar\Lambda$ has rank $d$ over $\F_p$.

\step{4}{We deduce from Step~3 that $\operatorname{im}\bar M$ contains no nonzero $\bar T$-invariant subspace.} Let
$V\subseteq\operatorname{im}\bar M$ be $\bar T$-invariant, and suppose that $0\ne w\in V$. Write $w=\bar Ma$ with $a\in\F_p^d$, $a\ne0$.
Then $\bar T^jw\in V\subseteq\operatorname{im}\bar M=\ker\bar\pi_U$, so $\bar\pi_U\bar T^j\bar Ma=0$ for $1\le j\le d$. That is,
$\bar\Lambda a=0$, contradicting $\rank\bar\Lambda=d$. Hence $V=0$, and since
$\operatorname{im}\bar M=\operatorname{span}_{\F_p}(\bar u_1,\dots,\bar u_D)$ by Step~1, this proves the lemma.
\end{proof}

We now combine Lemmas~\ref{lem:dissoc} and~\ref{lem:liftinv} into a deterministic proposition, which replaces anti-concentration
estimates at huge primes by the absence of rational factors of small degree. Invariant subspaces of large dimension are handled by
Odlyzko's lemma, so that only invariant subspaces of dimension at most $3\log_2m$ are treated by lifting.

\begin{proposition}\label{lem:acinv}
For all sufficiently large $m$ the following holds. Let $T$ be a Seidel matrix of order $m$ such that $\chi_T$ has no monic factor in
$\Z[x]$ of degree in $[1,(\log m)^3]$, and let $p>e^{(\log m)^6}$ be prime. Then every nonzero $\bar T$-invariant subspace
$V\subseteq\F_p^m$ satisfies $\Pr_\xi(\bar\xi\perp V)<m^{-3}$. In particular $a_\beta(T)<m^{-3}$ for every monic irreducible
$\beta\in\F_p[x]$.
\end{proposition}

\begin{proof}
Let $m$ be large, so that the prime $p>e^{(\log m)^6}$ is odd, let $V\ne0$ be $\bar T$-invariant, and put $k:=\dim V$.

\step{1}{We treat the invariant subspaces of dimension greater than $3\log_2m$.} If $k>3\log_2m$, then
$\Pr_\xi(\bar\xi\perp V)\le2^{-k}<m^{-3}$ by Odlyzko's lemma (Lemma~\ref{lem:odlyzko}).

\step{2}{We show that if $k\le3\log_2m$, then $V$ contains a vector $w$ with $\rho_p(w)<m^{-3}$.} This gives the claim, as
$\bar\xi\perp V$ implies $\sum_j\xi_jw_j=0$, so that $\Pr_\xi(\bar\xi\perp V)\le\rho_p(w)$. Let $v_1,\dots,v_k$ be a basis of $V$, and
suppose, for a contradiction, that $\rho_p(v_l)\ge m^{-3}$ for $l=1,\dots,k$. As $p$ is odd, Lemma~\ref{lem:dissoc} shows that each $v_l$
is a linear combination of the reductions of at most $\log_2(m^3)=3\log_2m$ vectors in $\{-1,0,1\}^m$. Altogether $V$ is contained in the
$\F_p$-span of the reductions of $D$ vectors $u_1,\dots,u_D\in\{-1,0,1\}^m$, where, as $m$ is large,
\[
D\le k\cdot3\log_2m\le9\,(\log_2m)^2\le(\log m)^3<m .
\]
By the hypothesis on $\chi_T$ and Lemma~\ref{lem:liftinv}, $0$ is the only $T$-invariant subspace of $\Q^m$ contained in
$\operatorname{span}_\Q(u_1,\dots,u_D)$, and since $V\ne0$ is $\bar T$-invariant, $p\le m^{3D^2}$. Hence
\[
\log p\le3D^2\log m\le243\,(\log_2m)^4\log m<(\log m)^6 ,
\]
again as $m$ is large. This contradicts $p>e^{(\log m)^6}$.

\step{3}{We deduce the bound for $a_\beta(T)$ from Steps~1 and~2.} A simple $\bar T$-submodule of type $\beta$ is a nonzero
$\bar T$-invariant subspace, so the bound of Steps~1 and~2 holds for each of them, and hence for $a_\beta(T)$.
\end{proof}

\subsection{Proof of Theorem~\ref{thm:huge}}\label{subsec:proofhuge}
We now prove Theorem~\ref{thm:huge}: the probability that $E_p^{(d)}$ holds for some odd prime $p>e^{(\log n)^6}$ and some $d\ge1$
tends to $0$.

\begin{proof}[Proof of Theorem~\ref{thm:huge}]
Let $\Pi_n$ be the set of pairs $(p,d)$ with $p>e^{(\log n)^6}$ an odd prime and $d\ge1$. For $m\in\{n-1,n\}$ let
$\mathcal G_m$ be the set of Seidel matrices $T$ of order $m$ with $\Delta_T\ne0$ such that $\chi_T$ has no monic factor in $\Z[x]$ of
degree in $[1,(\log m)^3]$. Since $(\log m)^3\le m^{1/30}$ for large $m$, Theorem~\ref{lem:simplespectrum} and
Theorem~\ref{thm:nofactor} (with $\eta=\frac1{30}$) show that $\Pr(S_m\notin\mathcal G_m)\to0$ for $m\in\{n-1,n\}$. By
Proposition~\ref{prop:reduction} it suffices to show that $\E(N_{\Pi_n}(S_m);\ S_m\in\mathcal G_m)\to0$ for $m\in\{n-1,n\}$.

Let $n$ be large, $m\in\{n-1,n\}$ and $T\in\mathcal G_m$. By Lemma~\ref{lem:budget}(b) with $P=e^{(\log n)^6}$, the number of pairs
$(p,\beta)$ with $p>e^{(\log n)^6}$ prime and $\beta$ a repeated factor of $T$ at $p$ is at most
\[
\frac{m^2\log(2m)}{(\log n)^6}\le n^2\log(2n).
\]
For each of them $p>e^{(\log n)^6}\ge e^{(\log m)^6}$, so Proposition~\ref{lem:acinv} gives $a_\beta(T)<m^{-3}\le(n-1)^{-3}$. Hence
$N_{\Pi_n}(T)\le n^2\log(2n)\,(n-1)^{-3}$, and therefore
\[
\E\big(N_{\Pi_n}(S_m);\ S_m\in\mathcal G_m\big)\le n^2\log(2n)\,(n-1)^{-3}\longrightarrow0\qquad(m\in\{n-1,n\}).
\]
So the hypotheses of Proposition~\ref{prop:reduction} hold for $\Pi_n$ and the sets $\mathcal G_m$, and the proposition proves the
theorem.
\end{proof}

Together with Corollary~\ref{cor:ACplus}, the argument also bounds, at all primes $p\ge23$ at once, the probability that a random sign
vector is orthogonal to an eigenvector with eigenvalue in $\F_p$, as we note in the following remark.

\begin{remark}\label{rem:hugeAC}
Let $m$ be large. By Theorem~\ref{thm:nofactor} (with $\eta=\frac1{30}$, as $(\log m)^3\le m^{1/30}$) and Proposition~\ref{lem:acinv}, with
probability at least $1-e^{-m^{7/90}/9}$ every nonzero $\bar S_m$-invariant subspace $V\subseteq\F_p^m$ satisfies
$\Pr_\xi(\bar\xi\perp V)<m^{-3}$, simultaneously for all primes $p>e^{(\log m)^6}$. In particular $\Pr_\xi(\xi^{\top}w=0)<m^{-3}$ for every
eigenvector $w\in\F_p^m$ of $\bar S_m$ with eigenvalue in $\F_p$. At the primes $23\le p\le e^{m^{1/8}}$, Corollary~\ref{cor:ACplus} for
$S_m$ shows that, outside an event of probability $O(2^{-m/8})$, every such eigenvector satisfies
$\Pr_\xi(\xi^{\top}w=0)\le\frac1p+e^{-2c_1m}\le\frac1p+m^{-3}$. As $e^{m^{1/8}}>e^{(\log m)^6}$ and
$e^{-m^{7/90}/9}+O(2^{-m/8})\le e^{-m^{7/90}/10}$, with probability at least $1-e^{-m^{7/90}/10}$ every such eigenvector satisfies
$\Pr_\xi(\xi^{\top}w=0)\le\frac1p+m^{-3}$, at all primes $p\ge23$ at once.
\end{remark}

\end{document}